\documentclass[reqno]{amsart}

\usepackage{macros}

\toggletrue{mackey}

\numberwithin{equation}{section}

\counterwithin{theorem}{section}

\makeatletter
\def\enumfix{%
\if@inlabel
 \noindent \par\nobreak\vskip-\topsep\hrule\@height\z@
\fi}

\let\olditemize\itemize
\def\itemize{\enumfix\olditemize}

\makeatother

\usepackage{calligra}

\begin{document}

\title{Derived Mackey functors and $\Cyclic_{p^n}$-equivariant cohomology}

\author{David Ayala, Aaron Mazel-Gee, and Nick Rozenblyum}

\date{\today}

\begin{abstract}
We establish a novel approach to computing $G$-equivariant cohomology for a finite group $G$, and demonstrate it in the case that $G = \Cyclic_{p^n}$.

For any commutative ring spectrum $R$, we prove a symmetric monoidal reconstruction theorem for genuine $G$-$R$-modules, which records them in terms of their geometric fixedpoints as well as gluing maps involving their Tate cohomologies. This reconstruction theorem follows from a symmetric monoidal stratification (in the sense of \cite{AMR-strat}); here we identify the gluing functors of this stratification in terms of Tate cohomology.

Passing from genuine $G$-spectra to genuine $G$-$\ZZ$-modules (a.k.a.\! derived Mackey functors) provides a convenient intermediate category for calculating equivariant cohomology. Indeed, as $\ZZ$-linear Tate cohomology is far simpler than $\SS$-linear Tate cohomology, the above reconstruction theorem gives a particularly simple algebraic description of genuine $G$-$\ZZ$-modules. We apply this in the case that $G = \Cyclic_{p^n}$ for an odd prime $p$, computing the Picard group of genuine $G$-$\ZZ$-modules (and therefore that of genuine $G$-spectra) as well as the $\RO(G)$-graded and Picard-graded $G$-equivariant cohomology of a point.
\end{abstract}

\maketitle

\setcounter{tocdepth}{2}
\tableofcontents

\setcounter{section}{-1}

\section{Introduction}
\label{section.intro}

\subsection{Overview}
\label{subsection.intro.overview}

Let $G$ be a finite group. 
Let $X$ be a topological space equipped with an action by $G$.
There is a natural notion of $G$-equivariant cohomology of $X$,
whose output is a collection of Mackey functors of abelian groups
(see e.g.\!~\cite{GM-handbook} for a motivated account of equivariant cohomology; see also~\Cref{obs.mackey.for.genuine.G.objects} below for a recollection of Mackey functors that is particularly suited to our purposes):
\[
\sH^{V}_G(X)
\in
\Mack_G(\Ab)
~,
\]
indexed by a virtual representation $V \in \RO(G)$.\footnote{
The usual integer grading is given by multiples of the trivial representation.
}
(The $\RO(G)$-grading is of central importance.
For instance, it plays a key role in equivariant duality \cite{Wirth-eqdual,LMM-ROG,CostWan-eqPD}.)
Equivariant cohomology is quite difficult to compute, even in the case that $X$ is a single point.
%
%
%

In this paper, we establish a novel approach to computing equivariant cohomology. 
To explain our approach, let us briefly outline a perspective on why equivariant cohomology is more difficult than non-equivariant cohomology.

Recall that ordinary cohomology of a space takes values in abelian groups;
it can be computed as the cohomology of a cochain complex of abelian groups.  
From the point of view of $\infty$-categories, this can be explained by the identification of the $\ZZ$-linearization of the stable homotopy $\infty$-category as the derived $\infty$-category of its heart:
\[
\Mod_\ZZ
~:=~
\ZZ \otimes \Spectra 
~\simeq~ 
\bD(\Ab)
~\simeq~
\bD\bigl( (\ZZ \otimes \Spectra)^{\heartsuit} \bigr)
~.
\]
As a result, standard homological algebra tools are directly applicable for computations of ordinary cohomology.  

Now, consider the $\ZZ$-linearization $\ZZ \otimes \Spectra^{\gen G}$ of the equivariant stable homotopy $\infty$-category $\Spectra^{\gen G}$; following \Cref{obs.mackey.for.genuine.G.objects}, we denote this by the equivalent $\infty$-category
\[
\Mack_G(\Mod_\ZZ)
\simeq
\ZZ \otimes \Spectra^{\gen G}
~.
\]
Now, equivariant cohomology takes values in the category $\Mack_G(\Ab)$ of Mackey functors of abelian groups, which is the heart of $\Mack_G(\Mod_\ZZ)$:
\[
\Mack_G(\Ab)
~\simeq~
\Mack_G(\Mod_\ZZ)^{\heartsuit}
~.
\]
However, $\Mack_G(\Mod_\ZZ)$ is \emph{not} the derived $\infty$-category of its heart:
\[
\Mack_G(\Mod_\ZZ)
~:=~
\ZZ \otimes \Spectra^{\gen G}
~\not\simeq~ 
\bD\bigl(\Mack_G(\Ab) \bigr)
~\simeq~
\bD\bigl( \Mack_G(\Mod_\ZZ)^{\heartsuit} \bigr)
~.
\]
In this sense, given an equivariant spectrum, it is not possible to systematically associate a cochain complex of Mackey functors that computes its equivariant cohomology.
Therefore, standard homological algebra tools do not obviously apply for computations of equivariant cohomology.

The triangulated category corresponding to the stable $\infty$-category $\Mack_G(\Mod_\ZZ)$ was first considered by Kaledin purely algebraically in \cite{Kaledin-Mack}, under the name \textit{derived Mackey functors}. He suggested that it should receive a functor from $\Spectra^{\gen G}$ compatible with both geometric and categorical fixedpoints; this is essentially the content of \cite[Conjectures 8.10 and 8.11]{Kaledin-Mack}. The work in this paper may be seen as resolving and applying those conjectures.

From our point of view, it is $\Mack_G(\Mod_\ZZ)$ that is a natural intermediate target for equivariant cohomology.
While it is not the derived $\infty$-category of its heart, we prove as \Cref{intro.thm.gluing.functors} below that it can be constructed (in a precise sense) from finitely many $\infty$-categories that are the derived $\infty$-categories of their hearts.  
Using this, we can import tools from standard homological algebra for computations in the equivariant stable homotopy category.

Specifically, we apply the theory of stratifications as developed in \cite{AMR-strat}:
for $\cX$ a presentable stable $\infty$-category, a \textit{stratification} of it over a poset $\pos$ is a collection $\{\cZ_p \subseteq \cX\}_{p\in \pos}$ of full subcategories satisfying certain natural conditions.
Given a stratification of $\cX$, we can consider the ``associated graded'' stable $\infty$-categories $\{\cX_p\}_{p\in \pos}$.
We prove as \cite[Theorem \ref{strat:intro.thm.cosms}]{AMR-strat} that, in favorable situations such as when the poset $\pos$ is finite, $\cX$ can be reconstructed from its associated graded $\infty$-categories together with, for each $p<q$ in $\pos$, a gluing functor
\[
\cX_p
\xra{\Gamma^p_q}
\cX_q
~,
\]
as well as coherence data.  

In \cite[Theorem \ref{strat:intro.thm.gen.G.spt}]{AMR-strat}, we constructed a stratification of the $\infty$-category $\Spectra^{\gen G}$ of genuine $G$-spectra by the poset $\pos_G$ of conjugacy classes of subgroups of $G$, called the \bit{geometric stratification}.
There, we identified the associated graded $\infty$-categories as
\[
(\Spectra^{\gen G})_{[H]}
~\simeq~
\Fun ( \sB \Weyl(H) , \Spectra)
~,
\]
where $\Weyl(H) := \Normzer(H)/H$ is the Weyl group (i.e.\! the normalizer of $H$ in $G$ modulo $H$).
Moreover, the projection functors to the associated gradeds
\[
\Spectra^{\gen G}
\xra{\Phi^H}
(\Spectra^{\gen G})_{[H]} \simeq \Fun ( \sB \Weyl(H) , \Spectra)
\]
are given by geometric $H$-fixedpoints.  
In particular, it follows that categorical fixedpoints can be expressed in terms of geometric fixedpoints and gluing data.

In fact, the stratification of $\Spectra^{\gen G}$ is a \textit{symmetric monoidal} stratification, so that we are also able to describe the symmetric monoidal structure of $\Spectra^{\gen G}$ in terms of those of its strata and gluing functors \cite[Theorem \ref{strat:intro.thm.O.mon.reconstrn}]{AMR-strat}.

However, in~\cite{AMR-strat}, we did not give an explicit general formula for the gluing functors of this stratification.  
Such a general formula is supplied by \Cref{intro.thm.gluing.functors} below in terms of a variant of Tate cohomology.

In fact, we prove \Cref{intro.thm.gluing.functors} in somewhat more generality. 
For $\cR$ a presentable stable $\infty$-category, consider the tensor product in $\PrL$
\[
\cR^{\gen G}
~:=~
\cR \otimes \Spectra^{\gen G}
~\simeq~
\Mack_G(\cR)
\]
(again see \Cref{obs.mackey.for.genuine.G.objects} for the equivalence).
The stratification of $\Spectra^{\gen G}$ over $\pos_G$ induces a stratification of $\cR^{\gen G}$ over $\pos_G$, with associated graded $\infty$-categories
\[
\Mack_G(\cR)_{[H]}
~\simeq~
\Fun ( \sB \Weyl(H) , \cR)
~.
\]
In particular, taking $\cR = \Mod_\ZZ$, we have that the $\ZZ$-linearization of the stable equivariant homotopy $\infty$-category $\Mack_G(\Mod_\ZZ)$ has a stratification by $\pos_G$, with associated graded $\infty$-categories given by
\[
\Mack_G(\Mod_\ZZ)_{[H]}
~\simeq~
\Fun ( \sB \Weyl(H) , \Mod_\ZZ)
~\simeq~
\bD( {\sf Rep}_{\Weyl(H) }( \Ab ) )
~,
\]
the derived $\infty$-category of abelian groups with an action by $\Weyl(H)$.

Now, equivariant cohomology is by definition the cohomology of categorical fixedpoints of genuine $G$-spectra.  By definition, the $\ZZ$-linearization functor
\begin{equation}\label{functor.from.G.spt.to.DMack}
\Spectra^{\gen G}
\longra
\Mack_G(\Mod_\ZZ)
\end{equation}
is compatible with both geometric and categorical fixedpoints. In practice, geometric fixedpoints are much easier to compute than categorical fixedpoints. In particular, the equivariant suspension spectrum functor takes fixedpoints to \emph{geometric} fixedpoints, and moreover the geometric fixedpoints functors are symmetric monoidal.

The compatibility of the $\ZZ$-linearization functor with both types of fixedpoint functors implies that the passage from geometric fixedpoints to categorical fixedpoints can be performed in the much simpler context of derived Mackey functors. Ultimately, this yields an algebraic description of the equivariant cohomology of a genuine $G$-space in terms of cochains on its fixedpoint spaces.

By contrast, a standard technique in equivariant homotopy theory is to compute equivariant cohomology using resolutions of Mackey functors (see e.g.\! \cite{Lewis-ROG-linear-Zp,Green-projmackey,Zeng-eq-EM}).  This is tantamount to performing homological algebra in the derived $\infty$-category $\bD(\Mack_G(\Ab))$ of the abelian category $\Mack_G(\Ab)$ of Mackey functors, which does not enjoy the favorable properties of $\Mack_G(\Mod_\ZZ)$.

Our first main result explicitly identifies the gluing functors of this stratification of $\Mack_G(\cR)$ in the case that $\cR$ is presentably symmetric monoidal and rigidly-compactly generated\footnote{That is, $\cR$ is a compactly generated presentably symmetric monoidal stable $\infty$-category whose compact and dualizable objects coincide.} (such as $\cR = \Mod_\ZZ$ and $\cR = \Spectra$).  

\begin{maintheorem}[\Cref{thm.stratn.of.genuine.G.objects}]
\label{intro.thm.gluing.functors}
For any containment $H  \subset  K$ between subgroups of $G$, the corresponding gluing functor
\[
\Fun ( \sB \Weyl(H) , \cR)
\xra{~\Gamma^H_K~}
\Fun ( \sB \Weyl(K) , \cR)
\]
of the stratification of $\Mack_G(\cR)$ evaluates as
\[
\Gamma^H_K
\colon
E
\longmapsto
\bigoplus_{[g] \in \Weyl(H) \backslash C(H,K) / \Weyl(K)} \Ind_{(\Normzer(H) \cap \Normzer(gKg^{-1})) / (gKg^{-1})}^{\Weyl(K)} E^{\tate ( gKg^{-1} / H) }
~,
\]
where
\begin{itemize}

\item $C(H,K)$ denotes a certain subset of $G/K$ that carries a natural $(\Weyl(H),\Weyl(K))$-bimodule structure (\Cref{notn.C.sub.H.upper.K}),

\item $(-)^{\tate (gKg^{-1}/H)}$ denotes a variant of Tate cohomology (\Cref{defn.proper.tate}), and

\item we implicitly use the isomorphism $\Weyl(gKg^{-1}) \xra{\cong} \Weyl(K)$.

\end{itemize}
\end{maintheorem}

\begin{remark}
If the finite group $G$ is abelian, the above gluing functors simplify as
$
\Gamma^H_K
\colon
E
~\mapsto~
E^{\tate (K/H)}
.
$

\end{remark}

Note that the functors $(-)^{\tate (gKg^{-1}/H)}$ do not preserve colimits, and as a result the gluing functors of the stratification of $\Mack_G(\cR)$ are \textit{not} given by tensoring those of $\Spectra^{\gen G}$ with $\cR$.

The decisive advantage of working with $\Mod^{\gen G}_\ZZ$ is that its stratification is dramatically simpler than that of $\Spectra^{\gen G}$. 
In addition to the associated gradeds being the derived $\infty$-categories of their hearts, the gluing functors are substantially simpler.
Indeed, the failure of the gluing functors of $\Mod_\ZZ^{\gen G}$ to be tensored up from those of $\Spectra^{\gen G}$ is a feature, and not a bug: Tate cohomology in $\Mod_\ZZ$ is far simpler than Tate cohomology in $\Spectra$. 
In the case that $G = \Cyclic_{p^n}$, this simplicity is especially pronounced, due additionally to the Tate vanishing results of Nikolaus--Scholze \cite{NS}.\footnote{Whereas in general the gluing functors of a stratification only laxly compose, in that of $\Mod^{\gen \Cyclic_{p^n}}_\ZZ$ they strictly compose. More than that, all of its nontrivial composite gluing functors are zero. Neither of these facts is the case for the stratification of $\Spectra^{\gen \Cyclic_{p^n}}$.
These facts support a relatively simple description of the $\infty$-category $\Mod_\ZZ^{\gen \Cyclic_{p^n}}$, which we articulate as \Cref{intro.thm.gen.Cpn.Z.mods}.
} All in all, we obtain the following.


\begin{maintheorem}[\Cref{thm.stratn.of.genuine.Cpn.Z.mods}]
\label{intro.thm.gen.Cpn.Z.mods}
The stratification of $\Mod^{\gen \Cyclic_{p^n}}_\ZZ$ determines an equivalence between presentably symmetric monoidal stable $\infty$-categories:
\[
\Mod^{\gen \Cyclic_{p^n}}_\ZZ
\xlongra{\sim}
\lim \left(
\begin{tikzcd}[ampersand replacement=\&]
\Mod_{\ZZ[\Cyclic_{p^n}]}
\arrow{rd}[sloped, swap]{(-)^{\st \Cyclic_p}}
\&
\Ar \left(\Mod_{\ZZ[\Cyclic_{p^{n-1}}]} \right)
\arrow{d}{\ev_1}
\arrow{rd}[sloped, swap]{(-)^{\st \Cyclic_p} \circ \ev_0}
\&
\cdots
\arrow{d}{\ev_1}
\&
\cdots
\arrow{rd}[sloped, swap]{(-)^{\st \Cyclic_p} \circ \ev_0}
\&
\Ar \left( \Mod_\ZZ \right)
\arrow{d}{\ev_1}
\\
\&
\Mod_{\ZZ[\Cyclic_{p^{n-1}}]}
\&
\Mod_{\ZZ[\Cyclic_{p^{n-2}}]}
\&
\cdots
\&
\Mod_\ZZ
\end{tikzcd} \right)
~,
\]
where $\Mod_{\ZZ[\Cyclic_{p^k}]}$ is the derived $\infty$-category $\bD( {\sf Rep}_{\Cyclic_{p^k}}(\Ab))$.  In particular, a genuine $\Cyclic_{p^n}$-$\ZZ$-module $E \in \Mod^{\gen \Cyclic_{p^n}}_\ZZ$ is equivalent to the data of
\begin{itemize}

\item its geometric $\Cyclic_{p^s}$-fixedpoints
\[
E_s
:=
\Phi^{\Cyclic_{p^s}}(E)
\in
\Mod_{\ZZ[\Cyclic_{p^{n-s}}]}
\]
for all $0 \leq s \leq n$, along with

\item its gluing maps
\[
E_s
\xra{\gamma_{s-1,s}^E}
(E_{s-1})^{\st \Cyclic_p}
\]
in 
$
\Mod_{\ZZ[\Cyclic_{p^{n-s-1}}]}
$
for all $1 \leq s \leq n$.

\end{itemize}
Moreover, this description is compatible with symmetric monoidal structures. 

\end{maintheorem}

\noindent \Cref{intro.thm.gen.Cpn.Z.mods} is closely related to \cite[Corollary II.4.7]{NS} (see also \cite[Remark II.4.8]{NS}).\\

\subsection{Computations}
\label{subsection.intro.computations}
As a demonstration of our general machinery, for $G=\Cyclic_{p^n}$ with $p$ an odd prime, we compute:
\begin{enumerate}
\item the Picard group $\Pic(\Mack_G(\Mod_\ZZ))$ and the homomorphism
\[
\RO(G)
\longra
\Pic(\Mack_G(\Mod_\ZZ))
~,
\]
as well as
\item the $\Pic(\Mack_G(\Mod_\ZZ))$-graded cohomology of a point (and hence in particular the $\RO(G)$-graded cohomology of a point).
\end{enumerate}
In fact, the recent paper \cite{Krause-Pic} proves that the homomorphism
\[
\Pic(\Spectra^{\gen G})
\xra{\Pic\Cref{functor.from.G.spt.to.DMack}}
\Pic(\Mack_G(\Mod_\ZZ))
\]
is an isomorphism. Consequently, our computations explicitly identify the Picard group of $\Spectra^{\gen \Cyclic_{p^n}}$.
For comparison, we survey some analogous existing computations in \Cref{subsection.relation.with.lit}.

Although the $\RO(G)$-grading of equivariant cohomology has become standard in the literature, it is more natural to grade over the Picard group $\Pic(\Mod^{\gen G}_\ZZ)$ of genuine $G$-$\ZZ$-modules.\footnote{Recall the isomorphism $\Pic(\Spectra^{\gen G}) \xra{\cong} \Pic(\Mod^{\gen G}_\ZZ)$ of \cite{Krause-Pic}, which implies that this is equivalent to grading over $\Pic(\Spectra^{\gen G})$.} Let $X$ be a $G$-space. For a Picard element $L \in \Pic(\Mod^{\gen G}_\ZZ)$ and a subgroup $H \subseteq G$, we write
\[
\sC^L_G(X)(H)
:=
\ulhom_{\Mod^{\gen G}_\ZZ} ( \Sigma^\infty_G (X \times G/H)_+ \otimes \ZZ , L \otimes_\ZZ \ul{\ZZ} )
\in
\Mod_\ZZ
\]
for the indicated hom-$\ZZ$-module (where we consider $\ul{\ZZ} \in \Mack_G(\Ab) \subset \Mack_G(\Mod_\ZZ) \simeq \Mod_\ZZ^{\gen G}$). Using this, we define $\Pic(\Mod^{\gen G}_\ZZ)$-graded equivariant cohomology as
\[
\sH^{i+L}_G(X)(H)
:=
\pi_{-i}(\sC^L_G(X)(H))
\in
\Ab
~.
\]
This recovers $\RO(G)$-graded equivariant cohomology via pullback along the composite abelian group homomorphism
\begin{equation}
\label{homomorphism.from.ROCpn.to.Pic}
\RO(G)
\xra{V \longmapsto \SS^V}
\Pic( \Spectra^{\gen G} )
\xra{(-) \otimes \ZZ}
\Pic(\Mod^{\gen G}_\ZZ)
~.
\end{equation}

We have the following two computational results. First, we compute the Picard group of genuine $\Cyclic_{p^n}$-$\ZZ$-modules.

\begin{maintheorem}[Theorems \ref{thm.picard.group} and \ref{thm.from.reps.to.Pic}]
\label{intro.thm.Pic.of.gen.Cpn.Z.mods} Let $p$ be an odd prime.
There is an isomorphism between abelian groups:
\[
\ZZ^{\oplus (n+1)}
\oplus
\left(
\bigoplus_{s=1}^n
(\ZZ/p^{n-s+1})^\times / \{ \pm 1 \}
\right)
\xra{~\cong~}
\Pic(\Mod^{\gen \Cyclic_{p^n}}_\ZZ)
~.
\]
Furthermore, the resulting homomorphism
\[
\RO(\Cyclic_{p^n})
\xra{~\Cref{homomorphism.from.ROCpn.to.Pic}~}
\Pic(\Mod^{\gen \Cyclic_{p^n}}_\ZZ)
~\cong~
\ZZ^{\oplus (n+1)}
\oplus
\left(
\bigoplus_{s=1}^n
(\ZZ/p^{n-s+1})^\times / \{ \pm 1 \}
\right)
\]
is given on irreducibles (which freely generate $\RO(\Cyclic_{p^n})$) by
\[
\rho_\triv
\longmapsto
{(e_0,\vec{1})}
\qquad
\text{and}
\qquad
\rho_j
\longmapsto
{
\left(
~
2 e_0 - e_{\nu(j)+1}
~,~
\left( \frac{j}{p^{\nu(j)}} \right)_{\nu(j)+1}
\right)
}
~,
\]
where
\begin{itemize}

\item $\rho_\triv$ denotes the trivial (1-dimensional) representation,

\item $\rho_j$ denotes the 2-dimensional representation in which the generator of $\Cyclic_{p^n}$ acts by rotation by $2\pi  j /p^n$ (for $1 \leq j < p^n$), 

\item 
$\{e_0,\dots,e_n\}\subset \ZZ^{\oplus (n+1)}$ is the standard basis,

\item
$\vec{1} \in \bigoplus_{s=1}^n
(\ZZ/p^{n-s+1})^\times / \{ \pm 1 \}$
is the identity element,

\item
$\nu(j)$ is the $p$-adic valuation of $j$ (i.e., $\nu(j)$ is the largest integer such that $p^{\nu(j)}$ divides $j$),

\item
$
\left( \frac{j}{p^{\nu(j)}} \right)_{\nu(j)+1}
\in 
\bigoplus_{s=1}^n
(\ZZ/p^{n-s+1})^\times / \{ \pm 1 \}
$ 
is the image of the element $\frac{j}{p^{\nu(j)}}$ under the canonical homomorphism $(\ZZ/p^{n-(\nu(j)+1)+1})^\times  \to \bigoplus_{s=1}^n
(\ZZ/p^{n-s+1})^\times / \{ \pm 1 \}$.


\end{itemize}
\end{maintheorem}

%
%
%
%
%
%

\noindent 
Additionally, we compute the $\Pic(\Mod^{\gen \Cyclic_{p^n}}_\ZZ)$-graded cohomology of a point.

\begin{maintheorem}[\Cref{thm.describe.ho.Mod.Z.valued.mackey.functor}]
\label{intro.thm.cohomology} 
Let $p$ be an odd prime.
Denote by ${\rm D}(\Ab)$ the (ordinary) derived category of abelian groups.
For each $L \in \Pic(\Mod^{\gen \Cyclic_{p^n}}_\ZZ)$, there is an explicit
chain-level description of the ${\rm D}(\Ab)$-valued Mackey functor
\[
\sC^L_{\Cyclic_{p^n}}(\pt)
\in
\Mack_{\Cyclic_{p^n}}\left ({\rm D}\left(\Ab\right) \right)
\]
whose $i\th$ cohomology is the (ordinary) Mackey functor 
$\sH^{i+L}_{\Cyclic_{p^n}}(\pt)$.

\end{maintheorem}

Even in the special case that $n=1$ (and restricting to $\RO(\Cyclic_p)$), \Cref{intro.thm.cohomology} gives a new proof of Stong's classical calculation of the $\RO(\Cyclic_p)$-graded cohomology of a point. We refer the reader to \Cref{subsection.relation.with.lit} for a discussion of related literature.

\subsection{Miscellaneous remarks}
\label{subsection.misc.rmks}

\begin{remark}
A key ingredient in the proof of \Cref{intro.thm.cohomology} is an explicit description of the constant Mackey functor $\ul{\ZZ} \in \Mod^{\gen \Cyclic_{p^n}}_\ZZ$ in terms of \Cref{intro.thm.gen.Cpn.Z.mods}, i.e.\! in terms of its geometric $\Cyclic_{p^r}$-fixedpoints and gluing maps (\Cref{prop.describe.gluing.diagram.of.Z.underline}). As a consequence of this computation, we obtain an equivalence
\[
\Phi^{\Cyclic_p} (\ul{\ZZ})
\simeq
\THH(\FF_p)
~,
\]
of which we would be very interested to have a conceptual description (see \Cref{cor.identify.THH.FP} \and \Cref{rmk.wonder.about.significance.of.Bokstedt}).
\end{remark}

\begin{remark}
\label{rmk.why.Zp.n.minus.r.shows.up}
Fix any $1 \leq s \leq n$. The summand $(\ZZ/p^{n-s+1})^\times / \{ \pm 1 \} \subseteq \Pic(\Mod^{\gen \Cyclic_{p^n}}_\ZZ)$ appearing in \Cref{intro.thm.Pic.of.gen.Cpn.Z.mods} arises from the homogeneous invertible elements of the ring $\pi_*(\ZZ^{\st \Cyclic_p})^{\htpy \Cyclic_{p^{n-s}}}$. Indeed, for any commutative ring $R$, there is a commutative monoid homomorphism
\[
\ZZ \times \prod_{s=1}^n (\pi_* (  R^{\st \Cyclic_p})^{\htpy \Cyclic_{p^{n-s}}} )^\times_\homog
\longra
\pi_0 (\iota_0 (  \Mod^{\gen \Cyclic_{p^n}}_R) )
\]
given by the evident generalization of \Cref{notn.potential.picard.elt.K} (see also \Cref{obs.multiplicativity.of.K.bullet}). In the case that $R = \ZZ$, this surjects onto the Picard group $\Pic(\Mod^{\gen \Cyclic_{p^n}}_\ZZ) \subseteq \pi_0 ( \iota_0 ( \Mod^{\gen \Cyclic_{p^n}}_\ZZ ) )$ as a result of the fact that every Picard element of $\Mod_\ZZ^{\htpy \Cyclic_{p^{n-s}}}$ is trivial up to de/suspension.
\end{remark}

\begin{remark}
Our methods allow for the computation of $\Pic(\Mod^{\gen G}_R)$ for more general finite groups $G$ and commutative ring spectra $R \in \CAlg(\Spectra)$ (generalizing \Cref{intro.thm.Pic.of.gen.Cpn.Z.mods}).\footnote{In the notation of \cite[\S\ref{strat:subsection.intro.eq.spt}]{AMR-strat}, this may be seen as the Picard group of the fiber product $\BBGG \times_{\Spec(\SS)} \Spec(R)$.} For instance, in the case that $R = \QQ$, we have that
\[
\Pic(\Mod^{\gen \Cyclic_{p^n}}_\QQ)
\cong
\prod_{s=0}^n \Pic(\Mod^{\htpy \Cyclic_{p^{n-s}}}_\QQ)
\cong
\ZZ^{\oplus (n+1)}
\]
as a result of the fact that the Tate construction vanishes on $\QQ$-modules.

Likewise, our methods allow for more general computations in equivariant cohomology (generalizing \Cref{intro.thm.cohomology}). Specifically, one can vary the (finite) group $G$, the $G$-space, and the coefficients. Moreover, we expect that one can use our techniques to describe the multiplicative structure (i.e.\! the Green functor) on equivariant cohomology. This would involve a more careful analysis of the multiplicative structure of Tate cohomology than is done in this paper.

We would be very interested to see any such computations along these lines.
\end{remark}

\begin{remark}
\label{rmk.use.strat.freely}
Although the theory of stratifications developed in \cite{AMR-strat} is a crucial ingredient in our work here, it does not play a major role from an expositional point of view: its main consequence that we use here is the explicit description of genuine $\Cyclic_{p^n}$-$\ZZ$-modules (as a symmetric monoidal $\infty$-category) of \Cref{intro.thm.gen.Cpn.Z.mods}. So, we use the theory freely here, and refer the interested reader to \cite[\S \ref{strat:subsection.intro.detailed.overview}]{AMR-strat} for a more detailed overview.
\end{remark}

\begin{remark}
As explained above, our work applies to the geometric stratification of $\cR^{\gen G}$. However, there exist other interesting stratifications of $\cR^{\gen G}$. For instance, by \cite[Theorem \ref{strat:intro.thm.balmer}]{AMR-strat}, $\cR^{\gen G}$ also admits an \textit{adelic stratification} over its Balmer spectrum (which is also a symmetric monoidal stratification). The Balmer spectra of $\Spectra^{\gen G}$ and $\Mod^{\gen G}_\ZZ$ are respectively studied in \cite{BS-BalmergenuineSmods} and \cite{PSW-BalmergenuineZmods}.
\end{remark}

\subsection{Relations with existing literature}
\label{subsection.relation.with.lit}

As mentioned in \Cref{subsection.intro.overview}, the importance of derived Mackey functors (i.e.\! genuine $G$-$\ZZ$-modules) goes back to Kaledin \cite{Kaledin-Mack}, who (in different terms) studied its geometric stratification. The idea that genuine $G$-objects can be expressed in terms of their geometric fixedpoints stems from the work of Greenlees and May; see in particular \cite{Greenlees-thesis,GreenMay-Tate}. There is also much work on similar expressions of rational $G$-spectra (which are simpler because the relevant Tate constructions vanish rationally), notably the reconstruction results of Greenlees--Shipley \cite{GreenShip}.  More recent works in this direction include \cite{MNN,Saul-strat}; see also \cite[Remark II.4.8]{NS}.

The first computation of $\RO(G)$-graded cohomology was for $G = \Cyclic_p$, due to Stong (see \cite{Lewis-ROG-linear-Zp}). The works \cite{HHR-kervaire,HHR-C4KO} of Hill--Hopkins--Ravenel give partial computations for $G = \Cyclic_{2^n}$, which play an essential role in their resolution of the Kervaire invariant one problem. Further computations include the works \cite{Zeng-eq-EM,Georgakop-C4} of Zeng and Georgakopoulos for $G = \Cyclic_{p^2}$, as well as the works \cite{HollerKriz-even,HollerKriz-odd} of Holler--Kriz for $G = (\Cyclic_p)^{\times n}$ (with coefficients in $\ul{\ZZ/p}$ and restricting to actual (not virtual) representations). Georgakopoulos also gives a computer program for $G = \Cyclic_{p^n}$.

As mentioned previously, in \cite{Krause-Pic} Krause proves the isomorphism $\Pic(\Spectra^{\gen G}) \xra{\cong} \Pic(\Mod^{\gen G}_\ZZ)$ for any finite group $G$, and gives a partial computation of this Picard group in a number of examples: $\Cyclic_p$, $(\Cyclic_p)^{\times 2}$, $\sD_{2p}$, and $\sA_5$.\footnote{Namely, Krause computes $\Pic(\Spectra^{\gen G})$ in these cases up to unaddressed extension problems, which arise due to the inductive nature of the approach.} Moreover, Fausk--Lewis--May \cite{FLM-PicSpgG} give an algebraic description of $\Pic(\Spectra^{\gen G})$ in terms of the Picard group of the Burnside ring. Using this and the results of tom Dieck--Petrie \cite{tDP-homotopyreps}, one can also compute the Picard group $\Pic(\Spectra^{\gen \Cyclic_{p^n}})$, and hence (by Krause's theorem) deduce \Cref{intro.thm.Pic.of.gen.Cpn.Z.mods}. Our approach is more direct; in particular, it produces an explicit construction of the Picard elements of $\Mod^{\gen \Cyclic_{p^n}}_\ZZ$, which is needed for \Cref{intro.thm.cohomology}.

\subsection{Outline}
\label{subsection.linear.outline}

This paper is divided into two parts. In \Cref{part.smstrat}, we study the geometric stratification of $\cR^{\gen G}$, culminating with the proof of \Cref{intro.thm.gluing.functors}; it is organized as follows.
\begin{itemize}

\item[\Cref{section.defn.gen.G.objects}:] For a compact Lie group $G$ and a presentable stable $\infty$-category $\cM$, we introduce the $\infty$-category $\cM^{\gen G}$ of genuine $G$-objects in $\cM$, and lay out some basic notations and conventions surrounding it. Taking $\cM = \cR$ to be a presentably symmetric monoidal stable $\infty$-category, the $\infty$-category $\cR^{\gen G}$ is also presentably symmetric monoidal.

\item[\Cref{section.stratn.of.tensor.products}:] We introduce the geometric stratification of $\cM^{\gen G}$, which is inherited from that of $\Spectra^{\gen G}$. The geometric stratification of $\cR^{\gen G}$ is a symmetric monoidal stratification.

\item[\Cref{section.stable.quotients}:] We establish some technical results regarding the interplay between small and presentable stable $\infty$-categories.

\item[\Cref{section.proper.tate.in.htpy.G.objects}:] We introduce the proper Tate construction for objects of $\cR^{\htpy G} := \Fun ( \BG , \cR)$ (in a way making no reference to genuine equivariant homotopy theory).

\item[\Cref{section.proper.tate.from.genuine.G.objects}:] We establish a formula for the proper Tate construction in terms of $\cR^{\gen G}$. Starting here, we make the assumption that $\cR$ is rigidly-compactly generated.

\item[\Cref{section.gluing.functors.for.RgG}:] We prove \Cref{intro.thm.gluing.functors}, which describes the gluing functors of the geometric stratification of $\cR^{\gen G}$ under the further assumption that $G$ is finite.

\end{itemize}
In \Cref{part.coh}, we apply the results of \Cref{part.smstrat} to the case that $G = \Cyclic_{p^n}$ and $\cR = \Mod_\ZZ$, proving Theorems \ref{intro.thm.gen.Cpn.Z.mods}-\ref{intro.thm.cohomology}; it is organized as follows.
\begin{itemize}

\item[\Cref{section.stratn.of.gen.Cpn.Z.mods}:] We prove \Cref{intro.thm.gen.Cpn.Z.mods}, which gives a simple and explicit description of the geometric stratification of $\Mod^{\gen \Cyclic_{p^n}}_\ZZ$.

\item[\Cref{section.Pic.of.gen.Cpn.Z.mods}:] We prove the first part of \Cref{intro.thm.Pic.of.gen.Cpn.Z.mods}, our computation of the Picard group of $\Mod^{\gen \Cyclic_{p^n}}_\ZZ$. Starting here, we make the assumption that the prime $p$ is odd.

\item[\Cref{section.fxn.reps.to.Pic}:] We prove the second part of \Cref{intro.thm.Pic.of.gen.Cpn.Z.mods}, which describes the Picard elements of $\Mod^{\gen \Cyclic_{p^n}}_\ZZ$ that underlie (virtual) representation spheres.

\item[\Cref{section.Z.underline}:] We study the gluing diagram of the constant Mackey functor $\ul{\ZZ} \in \Mod^{\gen \Cyclic_{p^n}}_\ZZ$ (the coefficients for equivariant cohomology).

\item[\Cref{section.Cpn.eqvrt.cohomology}:] We prove \Cref{intro.thm.cohomology}, our computation of the $\Pic(\Mod^{\gen \Cyclic_{p^n}}_\ZZ)$-graded cohomology of a point, based on the results of \Cref{section.homological.algebra}.

\item[\Cref{section.homological.algebra}:] We record some auxiliary results in homological algebra.
\end{itemize}

\begin{remark}
Most of the work in this paper takes place at the homotopical (i.e.\! $\infty$-categorical) level. However, we work at the point-set (i.e.\! chain) level as well. We compartmentalize the latter as \Cref{section.homological.algebra}, in which we produce various chain-level data (chain complexes, chain maps, and chain homotopies) and prove that they are presentations of our desired corresponding homotopical data.\footnote{In fact, these chain complexes are all quite simple: for the most part they are levelwise free of rank 0 or 1. This simplicity is ultimately afforded by certain Tate vanishing results (\Cref{obs.tate.vanishing.for.Z.mods}).}

The primary purpose of the material in \Cref{section.homological.algebra} is as input to \Cref{section.Cpn.eqvrt.cohomology} (in which we prove \Cref{intro.thm.cohomology}), and indeed the remainder of the paper (i.e.\! \S\S\ref{section.defn.gen.G.objects}-\ref{section.Z.underline}) can largely be read without reference to it. However, a single straightforward computation made in \Cref{section.homological.algebra} (\Cref{lem.compute.htpy.ring.of.htpy.of.tate}) is used in \Cref{section.Pic.of.gen.Cpn.Z.mods}, and in \Cref{section.fxn.reps.to.Pic} we make use of some basic techniques in homological algebra (our conventions for which are recorded in \Cref{subsection.notation.and.conventions.for.homological.algebra}).
\end{remark}

\subsection{Notation and conventions}
\label{subsection.notation.and.conventions}

\begin{enumerate}

\item \catconventions 

\item \functorconventions

\item \circconventions

\item \kanextnconventions

\item \spacescatsspectraconventions







\item \presentablyenrichedconventions

\item \codenamesformackey

\end{enumerate}

\begin{warning}
In this paper, we study generalizations of a number of notions introduced in \S S.\ref{strat:subsection.stratn.of.SpgG}. We occasionally reappropriate our notation without additional decoration.
\end{warning}

\subsection{Acknowledgments}
\label{subsection.acknowledgments}

It is our pleasure to acknowledge our intellectual debt to Kaledin, which is clear from the discussion of \Cref{subsection.intro.overview}; much of the material in this paper arose from thinking about the paper \cite{Kaledin-Mack}. We thank Akhil Mathew for a number of helpful conversations regarding the Tate construction, and in particular for pointing out \Cref{prop.proper.tate.formula} to us.

\acksupport \ Additionally, all three authors gratefully acknowledge the superb working conditions provided by the Mathematical Sciences Research Institute (which is supported by NSF award 1440140), where DA and AMG were in residence and NR was a visitor during the Spring 2020 semester.

\part{A symmetric monoidal stratification of derived Mackey functors}
\label{part.smstrat}

\section{Genuine $G$-objects in presentable stable $\infty$-categories}
\label{section.defn.gen.G.objects}

In this section, we introduce the $\infty$-categories of genuine and homotopy $G$-objects in a presentable stable $\infty$-category as well as various basic notions surrounding them.

\begin{notation}
We assume a basic familiarity with equivariant homotopy theory; we refer the reader to \S S.\ref{strat:subsection.stratn.of.SpgG} for a rapid review. In general, we use the notation and terminology laid out there (which is largely quite standard). Here we highlight a few conventions of particular interest in the present work.
\begin{enumerate}

\item We fix an arbitrary compact Lie group $G$ (which will sometimes be assumed to be finite).

\item We write $\pos_G$ for the poset of conjugacy classes of closed subgroups of $G$ ordered by subconjugacy. We denote relation of subconjugacy by $\leq$. When we wish to indicate literal containment, we use the notation $\subseteq$.

\item We write $H$ and $K$ for arbitrary closed subgroups of $G$. When discussing closed subgroups that are related by subconjugacy, we will always take $H$ to be subconjugate to $K$.

\end{enumerate}
\end{notation}

\begin{local}
In this section, we fix a presentable stable $\infty$-category $\cM \in \PrLSt$.
\end{local}

\needspace{2\baselineskip}
\begin{definition}
\label{defn.gen.G.objects}
\begin{enumerate}
\item[]

\item\label{item.defn.gen.G.objects}

We define the presentable stable $\infty$-category of \bit{genuine $G$-objects in $\cM$} to be the tensor product
\[
\cM^{\gen G}
:=
\Spectra^{\gen G}
\otimes
\cM
\]
in $\PrLSt$. Given an associative ring spectrum $R \in \Alg(\Spectra)$, we refer to $\Mod^{\gen G}_R$ as the presentable stable $\infty$-category of \bit{genuine $G$-$R$-modules}.

\item

We define the presentable stable $\infty$-category of \bit{homotopy $G$-objects in $\cM$} to be
\[
\cM^{\htpy G}
:=
\Fun ( \BG , \cM)
~.
\]
Given an associative ring spectrum $R \in \Alg(\Spectra)$, we refer to $\Mod^{\htpy G}_R$ as the presentable stable $\infty$-category of \bit{homotopy $G$-$R$-modules}.
\end{enumerate}
\end{definition}

\begin{observation}
\label{obs.tensoring.with.cpctly.gend}
Suppose that $\cC \in \PrLSt$ is a presentable stable $\infty$-category.  If $\cC$ is compactly generated, then there is a canonical equivalence
\[
\cC \otimes (-)
\simeq
\Fun^\ex ( (\cC^\omega)^\op , - )
\]
in $\Fun(\PrLSt,\PrLSt)$. In particular, we have a canonical equivalence
\[
\cC \otimes \cM
\simeq
\Fun^\ex ( (\cC^\omega)^\op , \cM )
~.
\]
\end{observation}

\begin{observation}
Note that $\Spectra^{\htpy G} \simeq \LMod_{\Sigma^\infty_+ G}(\Spectra)$ is compactly generated. Hence, by \Cref{obs.tensoring.with.cpctly.gend}, the presentable stable $\infty$-category of homotopy $G$-objects in $\cM$ admits an identification
\[
\cM^{\htpy G}
:=
\Fun( \BG , \cM)
\simeq
\Fun^\ex ( ( ( \Spectra^{\htpy G} )^\omega )^\op , \cM)
\simeq
\Spectra^{\htpy G} \otimes \cM
~.
\]
We use this fact without further comment.
\end{observation}

\begin{notation}
We simply write
\[
U
:
\cM^{\gen G}
:=
\Spectra^{\gen G} \otimes \cM
\xra{U \otimes \cM}
\Spectra^{\htpy G} \otimes \cM
\simeq
\cM^{\htpy G}
\]
for the tensor product with $\cM$ of the forgetful functor $\Spectra^{\gen G} \xra{U} \Spectra^{\htpy G}$ from genuine $G$-spectra to homotopy $G$-spectra (a morphism in $\PrLSt$). Moreover, we simply write
\[ \begin{tikzcd}[column sep=1.5cm]
\cM^{\gen G}
\arrow[transform canvas={yshift=0.9ex}]{r}{U}
\arrow[dashed, hookleftarrow, transform canvas={yshift=-0.9ex}]{r}[yshift=-0.2ex]{\bot}[swap]{\beta}
&
\cM^{\htpy G}
\end{tikzcd} \]
for the indicated right adjoint (which is fully faithful because the functor $(-) \otimes \cM$ preserves colimits (in particular the quotient $\Spectra^{\gen G} \xra{U} \Spectra^{\htpy G}$)).\footnote{This right adjoint may be referred to as the inclusion of the ``Borel-complete'' genuine $G$-objects in $\cM$.}
\end{notation}

\begin{remark}
A morphism $\cA \xra{F} \cB$ in $\PrLSt$ is the data of an adjunction
\[ \begin{tikzcd}[column sep=2cm]
\cA
\arrow[transform canvas={yshift=0.9ex}]{r}{F}
\arrow[leftarrow, transform canvas={yshift=-0.9ex}]{r}[yshift=-0.2ex]{\bot}[swap]{F^R}
&
\cB
\end{tikzcd} \]
in $\Cat$. Tensoring this morphism with any object $\cC \in \PrLSt$ therefore determines an adjunction
\begin{equation}
\label{tensored.up.adjunction}
\begin{tikzcd}[column sep=2cm]
\cA \otimes \cC
\arrow[transform canvas={yshift=0.9ex}]{r}{F \otimes \cC}
\arrow[leftarrow, transform canvas={yshift=-0.9ex}]{r}[yshift=-0.2ex]{\bot}[swap]{(F \otimes \cC)^R}
&
\cB \otimes \cC
\end{tikzcd}~.
\end{equation}
In general, there is no straightforward description of the right adjoint $(F \otimes \cC)^R$ (e.g.\! in terms of the right adjoint $F^R$).  However, if $\cC$ is compactly generated, then by \Cref{obs.tensoring.with.cpctly.gend} we may identify the adjunction \Cref{tensored.up.adjunction} as the adjunction
\[ \begin{tikzcd}[column sep=3cm]
\cA \otimes \cC
\simeq
\Fun^\ex((\cC^\omega)^\op,\cA)
\arrow[transform canvas={yshift=0.9ex}]{r}{\Fun^\ex((\cC^\omega)^\op,F)}
\arrow[leftarrow, transform canvas={yshift=-0.9ex}]{r}[yshift=-0.2ex]{\bot}[swap]{\Fun^\ex((\cC^\omega)^\op,F^R)}
&
\Fun^\ex((\cC^\omega)^\op,\cB)
\simeq
\cC \otimes \cM
\end{tikzcd} \]
(because right adjoints are unique when they exist).
\end{remark}

\begin{notation}
\label{notn.various.fixedpoints.on.genuine.G.objects.in.M}
We simply write
\[
(-)^H : \cM^{\gen G}
:=
\Spectra^{\gen G} \otimes \cM
\xra{(-)^H \otimes \cM}
\Spectra^{\gen \Weyl(H)} \otimes \cM
=: \cM^{\gen \Weyl(H)}~,
\]
\[
(-)^H : \cM^{\gen G}
:=
\Spectra^{\gen G} \otimes \cM
\xra{(-)^H \otimes \cM}
\Spectra^{\htpy \Weyl(H)} \otimes \cM
\simeq \cM^{\htpy \Weyl(H)}~,
\]
\[
\Phi^H_\gen
:
\cM^{\gen G}
:=
\Spectra^{\gen G} \otimes \cM
\xra{\Phi^H_\gen \otimes \cM}
\Spectra^{\gen \Weyl(H)} \otimes \cM
=:
\cM^{\gen \Weyl(H)}
~,
\]
and
\[
\Phi^H
:
\cM^{\gen G}
:=
\Spectra^{\gen G} \otimes \cM
\xra{\Phi^H \otimes \cM}
\Spectra^{\htpy \Weyl(H)} \otimes \cM
\simeq
\cM^{\htpy \Weyl(H)}
\]
for the tensor product with $\cM$ of the various indicated $H$-fixedpoints functors on genuine $G$-spectra (all of which are morphisms in $\PrLSt$).
\end{notation}

\begin{observation}
\label{obs.mackey.for.genuine.G.objects}
Suppose that $G$ is a finite group.  Then, by \cite{GM-gen,Bar-Mack} we have an equivalence
\[
\Spectra^{\gen G}
\simeq
\Mack_G(\Spectra)
:=
\Fun^\oplus ( \Burn_G , \Spectra)
\]
where $\Burn_G$ denotes the $(2,1)$-category of spans among finite $G$-sets, which is preadditive and so is canonically enriched in commutative monoid spaces.  It follows that the idempotent-complete stable envelope of $\Burn_G$ (i.e.\! that of its homwise $\infty$-group completion) admits a canonical identification
\[
\Env^\idem(\Burn_G)
\simeq
((\Spectra^{\gen G})^\omega)^\op
~.
\]
Hence, using \Cref{obs.tensoring.with.cpctly.gend} and the fact that $\Spectra^{\gen G}$ is compactly generated, we obtain a composite equivalence
\[
\cM^{\gen G}
\simeq
\Fun^\ex( ( ( \Spectra^{\gen G})^\omega )^\op , \cM)
\simeq
\Fun^\ex ( \Env^\idem(\Burn_G) , \cM )
\simeq
\Fun^\oplus ( \Burn_G , \cM )
=:
\Mack_G(\cM)
~.
\]
By construction, evaluating a genuine $G$-object $E \in \cM^{\gen G}$ on the finite $G$-set $G/H \in \Burn_G$ yields its categorical $H$-fixedpoints $E^H \in \cM$, with the homotopy $\Weyl(H)$-action coming from its action on $G/H \in \Burn_G$: in other words, this equivalence 
extends to a commutative diagram
\[ \begin{tikzcd}
\cM^{\gen G}
\arrow{rr}{\sim}
\arrow{rd}[sloped, swap]{(-)^H}
&
&
\Mack_G(\cM)
\arrow{ld}[sloped, swap]{\ev_{G/H}}
\\
&
\cM^{\htpy \Weyl(H)}
\end{tikzcd}
~.
\]
\end{observation}

\begin{definition}
\label{defn.inc.and.trf}
Suppose that $G$ is a finite group. Given subgroups $H \leq K \leq G$, we obtain a morphism $G/H \ra G/K$ between finite $G$-sets, which determines morphisms in both directions in $\Burn_G$. Via \Cref{obs.mackey.for.genuine.G.objects}, evaluating a genuine $G$-object $E \in \cM^{\gen G}$ on these morphisms determines natural morphisms
\[
E^H
\xra{\inc_H^K(E)}
E^K
\qquad
\text{and}
\qquad
E^K
\xra{\trf_H^K(E)}
E^H
\]
in $\cM$ in both directions between its categorical $H$- and $K$-fixedpoints, which we respectively refer to as the \bit{inclusion} and \bit{transfer} morphisms. As $E$, $H$, and $K$ will always be clear from context, we will generally simply write
\[
\inc
:=
\inc_H^K(E)
\qquad
\text{and}
\qquad
\trf
:=
\trf_H^K(E)
~.
\]
\end{definition}

\begin{remark}
In \Cref{defn.inc.and.trf}, we use the term ``inclusion'' instead of the more familiar term ``restriction'' because the latter is already quite overloaded. However, note that this morphism is \textit{not} generally a monomorphism (indeed, in a stable $\infty$-category, all monomorphisms are equivalences).
\end{remark}

\begin{observation}
Suppose that $G$ is a finite group. Given subgroups $H \leq K \leq G$, in the case that $H$ and $K$ are both normal in $G$ (e.g.\! when $G$ is abelian), the morphisms $\inc_H^K$ and $\trf_H^K$ admit canonical lifts from $\cM$ to $\cM^{\htpy (K/H)}$.\footnote{More generally, these morphisms are equivariant for the \textit{relative Weyl group} (Definition S.\ref{strat:defn.relative.Weyl.group}).} We use this fact without further comment.
\end{observation}

\begin{notation}
\label{notn.htpy.inc}
In line with the notation introduced in \Cref{defn.inc.and.trf}, we simply write $\hinc$ for any inclusion morphism on homotopy fixedpoints.
\end{notation}

\begin{observation}
The equivalence $\cM^{\gen G} \simeq \Mack_G(\cM)$ of \Cref{obs.mackey.for.genuine.G.objects} is compatible with restriction, in the sense that the diagram
\[ \begin{tikzcd}[row sep=1.5cm, column sep=1.5cm]
\cM^{\gen G}
\arrow[leftrightarrow]{r}{\sim}
\arrow{d}[swap]{\Res^G_H}
&
\Mack_G(\cM)
\arrow{d}{\left(\Ind_H^G \right)^*}
\\
\cM^{\gen H}
\arrow[leftrightarrow]{r}[swap]{\sim}
&
\Mack_H(\cM)
\end{tikzcd} \]
commutes. It is also compatible with categorical fixedpoints, in the sense that the diagram
\[ \begin{tikzcd}[row sep=1.5cm, column sep=1.5cm]
\cM^{\gen G}
\arrow[leftrightarrow]{r}{\sim}
\arrow{d}{\Res^G_{\Normzer(H)}}
\arrow[bend right=40]{dd}[swap]{(-)^H}
&
\Mack_G(\cM)
\arrow{d}{\left( \Ind_{\Normzer(H)}^G \right)^*}
\\
\cM^{\gen \Normzer(H)}
\arrow[leftrightarrow]{r}[swap]{\sim}
\arrow{d}{(-)^H}
&
\Mack_{\Normzer(H)}(\cM)
\arrow{d}{\left( \Res^{\Weyl(H)}_{\Normzer(H)} \right)^*}
\\
\cM^{\gen \Weyl(H)}
\arrow[leftrightarrow]{r}[swap]{\sim}
&
\Mack_{\Weyl(H)}(\cM)
\end{tikzcd} \]
commutes. We use these facts without further comment.
\end{observation}

\begin{notation}
We denote by $\tensoring$ the action on $\cM^{\gen G}$ of $\Spaces_*^{\gen G}$ (via that of $\Spectra^{\gen G}$).
\end{notation}

\begin{observation}
\label{obs.geom.fps.on.gen.G.objects.from.catl.fps}
It follows directly from the definitions that the diagram
\[ \begin{tikzcd}[column sep=2cm, row sep=1.5cm]
\cM^{\gen G}
\arrow{r}{\wEFgeomHfps \tensoring (-) }
\arrow{rd}[sloped, swap]{\Phi^H_\gen}
&
\cM^{\gen G}
\arrow{d}{(-)^H}
\\
&
\cM^{\gen \Weyl(H)}
\end{tikzcd} \]
commutes (see Definition S.\ref{strat:defn.isotropy.separation.sequence}).
\end{observation}

\section{The geometric stratification of genuine $G$-objects}
\label{section.stratn.of.tensor.products}

In this brief section, we introduce the \bit{geometric stratification} of genuine $G$-objects (\Cref{defn.geometric.stratification.of.genuine.G.objects}). Throughout it, we refer freely to the notions introduced in \cite{AMR-strat} (recall \Cref{rmk.use.strat.freely}).

\begin{observation}
\label{obs.tensored.up.stratn}
Let $\cX$ be a presentable stable $\infty$-category.  A closed subcategory $\cZ \in \Cls_\cX$ is precisely the data of an adjunction
\[ \begin{tikzcd}[column sep=2cm]
\cZ
\arrow[hook, transform canvas={yshift=0.9ex}]{r}{i_L}
\arrow[leftarrow, transform canvas={yshift=-0.9ex}]{r}[yshift=-0.2ex]{\bot}[swap]{y}
&
\cX
\end{tikzcd} \]
in $\PrLSt$ whose unit is an equivalence.  It follows that taking the tensor product in $\PrLSt$ with any presentable stable $\infty$-category $\cM$ determines a functor
\begin{equation}
\label{functor.tensor.with.M}
\begin{tikzcd}[row sep=0cm, column sep=1.5cm]
\Cls_\cX
\arrow{r}{- \otimes \cM}
&
\Cls_{\cX \otimes \cM}
\\
\rotatebox{90}{$\in$}
&
\rotatebox{90}{$\in$}
\\
\cZ
\arrow[maps to]{r}
&
\cZ \otimes \cM
\end{tikzcd}~.
\end{equation}
Observe further that the functor \Cref{functor.tensor.with.M} preserves colimits and finite products.  It follows that for any stratification $\pos \xra{\cZ_\bullet} \Cls_\cX$, postcomposition determines a stratification
\begin{equation}
\label{stratn.of.X.tensor.M}
\begin{tikzcd}[row sep=0cm, column sep=1.5cm]
\pos
\arrow{r}{\cZ_\bullet}
&
\Cls_\cX
\arrow{r}{- \otimes \cM}
&
\Cls_{\cX \otimes \cM}
\\
\rotatebox{90}{$\in$}
&
&
\rotatebox{90}{$\in$}
\\
p
\arrow[maps to]{rr}
&
&
\cZ_p \otimes \cM
\end{tikzcd}
\end{equation}
of $\cX \otimes \cM$: the factorizations guaranteed by the stratification condition define commutative squares in $\PrLSt$, and so persist upon tensoring with $\cM$.  Moreover, because the functor $\PrLSt \xra{- \otimes \cM} \PrLSt$ preserves colimits, for each $p \in \pos$ we may identify the $p\th$ stratum of the stratification \Cref{stratn.of.X.tensor.M} as
\[
(\cX \otimes \cM)_p
\simeq
\cX_p \otimes \cM
~.
\]
It follows immediately that for each $p \in \pos$ we may identify the $p\th$ geometric localization functor of the stratification \Cref{stratn.of.X.tensor.M} as
\[
\cX \otimes \cM
\xra{ \Phi_p \otimes \cM}
\cX_p \otimes \cM
~.
\]
\end{observation}

\begin{remark}
In the situation of \Cref{obs.tensored.up.stratn}, if $\cM$ is compactly generated then by \Cref{obs.tensoring.with.cpctly.gend} we may identify the $p\th$ localization adjunction of the stratification \Cref{stratn.of.X.tensor.M} as 
\[ \begin{tikzcd}[column sep=3cm]
\cX \otimes \cM
\simeq
\Fun((\cM^\omega)^\op,\cX)
\arrow[transform canvas={yshift=0.9ex}]{r}{\Fun^\ex((\cM^\omega)^\op,\Phi_p)}
\arrow[hookleftarrow, transform canvas={yshift=-0.9ex}]{r}[yshift=-0.2ex]{\bot}[swap]{{\Fun^\ex((\cM^\omega)^\op,\rho^p)}}
&
\Fun((\cM^\omega)^\op,\cX_p)
\simeq
\cX_p \otimes \cM
\end{tikzcd}~. \]
In fact, in this case the entire gluing diagram of $\cX \otimes \cM$ is tensored up from that of $\cX$, in the sense that it is given by the composite
\[
\begin{tikzcd}[column sep=1.5cm]
\GD(\cX \otimes \cM)
:
\pos
\arrow{r}[description, yshift=-0.05cm]{\llax}{\GD(\cX)}
&
\PrSt
\arrow{r}{\Fun^\ex((\cM^\omega)^\op,-)}
&[1.5cm]
\PrSt
\end{tikzcd}
~.
\]
\end{remark}

\begin{observation}
\label{obs.tensoring.preserves.closed.ideals}
Let $\cX$ and $\cM$ be presentably symmetric monoidal stable $\infty$-categories.  Then, $\cX \otimes \cM$ is canonically a presentably symmetric monoidal stable $\infty$-category as well, and the functor \Cref{functor.tensor.with.M} admits a refinement
\[ \begin{tikzcd}[column sep=1.5cm]
\Cls_\cX
\arrow{r}{- \otimes \cM}
&
\Cls_{\cX \otimes \cM}
\\
\Idl_\cX
\arrow[hook]{u}
\arrow[dashed]{r}[swap]{-\otimes \cM}
&
\Idl_{\cX \otimes \cM}
\arrow[hook]{u}
\end{tikzcd}~. \]
\end{observation}

\begin{definition}
\label{defn.geometric.stratification.of.genuine.G.objects}
The \bit{geometric stratification} of the presentable stable $\infty$-category $\cM^{\gen G}$ of genuine $G$-objects in $\cM$ is the composite functor
\[ \begin{tikzcd}[row sep=0cm, column sep=1.5cm]
\pos_G
\arrow{r}{\Spectra^{\gen G}_{^\leq \bullet}}
&
\Cls_{\Spectra^{\gen G}}
\arrow{r}{- \otimes \cM}
&
\Cls_{\cM^{\gen G}}
\\
\rotatebox{90}{$\in$}
&
&
\rotatebox{90}{$\in$}
\\
H
\arrow[maps to]{rr}
&
&
\Spectra^{\gen G}_{^\leq H} \otimes \cM
\end{tikzcd}~, \]
where the first functor is the symmetric monoidal geometric stratification of genuine $G$-spectra of Theorem S.\ref{strat:intro.thm.gen.G.spt} (cf.\! Definition S.\ref{strat:defn.geometric.prestratn.of.gen.G.spt}); the fact that this composite functor is a stratification follows from \Cref{obs.tensored.up.stratn}. In the case that $\cM$ is a presentably symmetric monoidal stable $\infty$-category, we use the same name to refer to the symmetric monoidal stratification
\begin{equation}
\label{sym.mon.geometric.stratn.in.general}
\pos_G
\xra{\Spectra^{\gen G}_{^\leq \bullet}}
\Idl_{\Spectra^{\gen G}}
\xra{- \otimes \cM}
\Idl_{\cM^{\gen G}}
\end{equation}
guaranteed by \Cref{obs.tensoring.preserves.closed.ideals}.
\end{definition}

\section{Some algebra and analysis of stable $\infty$-categories}
\label{section.stable.quotients}

In this section, we establish a number of technical results regarding the interplay between small and presentable stable $\infty$-categories, which we use in our proof that the gluing functors for the geometric stratification of genuine $G$-objects are proper Tate constructions (\Cref{prop.proper.Tate.from.genuine.G.objs}).

\begin{definition}
The \bit{stable quotient} of an exact functor
\[
\cA
\xlongra{F}
\cB
\]
between (small or large) stable $\infty$-categories is the cofiber
\[
\cB /^\St \cA
:=
\cofib \left( \cA \xlongra{F} \cB \right)
~,
\]
considered among (resp.\! small or large) stable $\infty$-categories. (Stable quotients always exist, by Observations \ref{obs.stable.quotient.only.depends.on.image} \and \ref{obs.construct.stable.quotient}.)
\end{definition}

\begin{warning}
We will be taking stable quotients of functors between \textit{presentable} (and in particular, large) stable $\infty$-categories.
\end{warning}

\begin{observation}
\label{obs.stable.quotient.only.depends.on.image}
The stable quotient of an exact functor between stable $\infty$-categories only depends on its image, because it is merely a condition (as opposed to additional data) for an exact functor among stable $\infty$-categories to be zero. In other words, to understand stable quotients it suffices to understand stable quotients by full stable subcategories. Likewise, it suffices to consider stable quotients by full stable subcategories that are closed under retracts.
\end{observation}

\begin{observation}
\label{obs.construct.stable.quotient}
Suppose that
\[ \begin{tikzcd}
\cA
\arrow[hook]{r}{i}
&
\cB
\end{tikzcd} \]
is a fully faithful exact functor between small stable $\infty$-categories. Applying the functor
\[
\St
\xra{\Ind}
\PrLSt
~,
\]
we obtain a fully faithful functor
\[ \begin{tikzcd}[column sep=1.5cm]
\Ind(\cA)
\arrow[hook]{r}{\Ind(i) = i_!}
&
\Ind(\cB)
\end{tikzcd} \]
between presentable stable $\infty$-categories.  In fact, this is the inclusion of a closed subcategory: the functor $i_!$ preserves colimits, as does its right adjoint $i^*$. Hence, we obtain a recollement
\[ \begin{tikzcd}[column sep=1.5cm]
\Ind(\cA)
\arrow[hook, bend left=30]{r}[description]{i_!}
\arrow[leftarrow]{r}[transform canvas={yshift=0.1cm}]{\bot}[swap,transform canvas={yshift=-0.1cm}]{\bot}[description]{i^*}
\arrow[bend right=30, hook]{r}[description]{i_*}
&
\Ind(\cB)
\arrow[bend left=30]{r}[description, pos=0.5]{p_L}
\arrow[hookleftarrow]{r}[transform canvas={yshift=0.1cm}]{\bot}[swap,transform canvas={yshift=-0.1cm}]{\bot}[description]{\nu}
\arrow[bend right=30]{r}[description, pos=0.5]{p_R}
&
\Ind(\cB)
&[-1.8cm]
/ \Ind(\cA)
\end{tikzcd} \]
(see e.g.\! Definition S.\ref{strat:defn.recollement.in.intro}). Observe that the functor $p_L$ preserves compact objects (because the functor $i^*$ preserves filtered colimits). From here, it is straightforward to see that the composite
\[ \begin{tikzcd}
\cA
\arrow[hook]{r}{i}
&
\cB
\arrow{r}{p}
&
(\Ind(\cB) / \Ind(\cA))^\omega
\end{tikzcd} \]
is a cofiber sequence among small stable $\infty$-categories after idempotent completion, where $p$ denotes the restriction of $p_L$ to $\cB \subseteq \Ind(\cB)$, and thereafter that the stable quotient of $i$ itself is the full stable subcategory
\[
\cB /^\St \cA
=
p(\cB)
\subseteq
(\Ind(\cB) / \Ind(\cA))^\omega
~,
\]
the image of $p$ (which is automatically stable).\footnote{In particular, $(\Ind(\cB)/\Ind(\cA))^\omega$ is the idempotent completion of $\cB /^\St \cA$.} Moreover, by enlarging our Grothendieck universe, we can apply this same construction to exact functors between not-necessarily-small stable $\infty$-categories.
\end{observation}

\begin{remark}
\label{remark.stable.quotient.by.presentable.subcategory.is.presentable.quotient}
Let $\cX$ be a presentable stable $\infty$-category and let $\cZ \subseteq \cX$ be a full presentable stable subcategory. Then, the stable and presentable quotients of $\cX$ by $\cZ$ coincide: the canonical morphism
\[
\cX /^\St \cZ
\longra
\cX / \cZ
\]
is an equivalence. Indeed, the presentable quotient satisfies the universal property of the stable quotient: writing
\[ \begin{tikzcd}[column sep=1.5cm]
\cZ
\arrow[hook, yshift=0.9ex]{r}{i}
\arrow[leftarrow, yshift=-0.9ex]{r}[yshift=-0.2ex]{\bot}[swap]{i^R}
&
\cX
\arrow[yshift=0.9ex]{r}{j^L}
\arrow[hookleftarrow, yshift=-0.9ex]{r}[yshift=-0.2ex]{\bot}[swap]{j}
&
\cX/\cZ
\end{tikzcd} \]
for the resulting diagram in $\Cat$ (see Definition S.\ref{strat:defn.presentable.quotient} and Observation S.\ref{strat:obs.right.orthogonal.of.full.presentable.stable.subcat.is.presentable.quotient}), given any stable $\infty$-category $\cC$ and any exact functor $\cX \xra{F} \cC$ such that $Fi \simeq 0$, the morphism
\[
F
\longra
F j j^L
\]
is an equivalence (because for each $X \in \cX$ the cofiber sequence $i i^R X \ra X \ra j j^L X$ is carried by $F$ to a cofiber sequence).
\end{remark}

\begin{definition}
Let $\cC$ be a stably symmetric monoidal $\infty$-category. A full stable subcategory $\cI \subseteq \cC$ is a \bit{thick ideal} if it is closed under retracts and contagious under the symmetric monoidal structure.
\end{definition}

\begin{notation}
Let $\cC$ be a stably symmetric monoidal $\infty$-category. Given a set $\{ K_s \in \cC \}_{s \in S}$ of objects of $\cC$, we write $\brax{ K_s }_{s \in S}^{\St,\otimes} \subseteq \cC$ for the thick ideal that they generate.
\end{notation}

\begin{observation}
\label{obs.ind.of.a.thick.ideal.is.a.closed.ideal}
Suppose that $\cC \in \CAlg(\St)$ is a stably symmetric monoidal $\infty$-category and $\cI \subseteq \cC$ is a thick ideal. Then, $\Ind(\cI) \subseteq \Ind(\cC)$ is an ideal (see Definition S.\ref{strat:defn.ideal.subcat}). Moreover, for any set of objects $\{K_s \in \cC \}_{s \in S}$ we have
\[
\Ind(\brax{K_s}^{\St,\otimes}_{s \in S} )
=
\brax{K_s}^\otimes_{s \in S}
\]
(see Notation S.\ref{strat:notn.for.ideal.generated.by.stuff}).
\end{observation}

\section{The proper Tate construction}
\label{section.proper.tate.in.htpy.G.objects}

In this section, we introduce the proper Tate construction (\Cref{defn.proper.tate}) and study its basic features. In fact, we introduce a generalization of the proper Tate construction, which makes reference to a family of closed subgroups; this is no more difficult to study, and has the added benefit of recovering the ordinary Tate construction as a special case (\Cref{obs.proper.F.tate.recovers.tate} \and \Cref{rmk.proper.F.tate.recovers.tate.in.any.presentably.sym.mon}).

\begin{local}
In this section, we fix a family $\ms{F} \in \Down_{\pos_G}$ of closed subgroups of $G$ (i.e.\! a collection of closed subgroups of $G$ that is closed under subconjugacy), and we fix a presentably symmetric monoidal stable $\infty$-category $\cR$.
\end{local}

\begin{notation}
\label{notn.free.htpy.G.object.on.object.of.R.and.htpy.G.space}
Given an object $E \in \cR$ and a homotopy $G$-space $X \in \Spaces^{\htpy G}$, we write
\[
E^\htpy \brax{X}
:=
(\Sigma^\infty X_+ \otimes E )
\in
\Spectra^{\htpy G} \otimes \cR
\simeq
\cR^{\htpy G}
~.
\]
\end{notation}

\begin{warning}
We write $G/H$ both for the object of $\Spaces^{\gen G}$ and its image under the forgetful functor $\Spaces^{\gen G} \xra{U} \Spaces^{\htpy G}$. However, it will always be clear from context which object is being referred to.
\end{warning}

\begin{notation}
We write
\[
\cI^\htpy_\ms{F}
:=
\brax{ \uno_\cR^\htpy \brax{G/H} }^{\St,\otimes}_{H \in \ms{F}}
\subseteq
\cR^{\htpy G}
\]
for the thick ideal of $\cR^{\htpy G}$ generated by the objects $\{ \uno_\cR^\htpy \brax{G/H} \in \cR^{\htpy G} \}_{H \in \ms{F}}$.
\end{notation}

\begin{definition}
\label{defn.proper.tate}
For any closed subgroup $H \in \pos_G$, the \bit{$\ms{F}$-$H$-Tate construction} is the composite functor
\[
(-)^{\tate_\ms{F} H}
:
\cR^{\htpy G}
\xlongra{p}
\cR^{\htpy G} /^\St \cI_\ms{F}^\htpy
\xra{\ulhom(p(\uno_\cR^\htpy \brax{G/H}),-)}
\cR^{\htpy \Weyl(H)}
~,
\]
where $p$ denotes the canonical functor to the stable quotient. In the special case where $\ms{F} = (^{\not\geq} H)$, we simply write
\[
(-)^{\tate H}
:=
(-)^{\tate_{(^{\not\geq} H)} H}
\]
and refer to this functor as the \bit{proper $H$-Tate construction}.
\end{definition}

\begin{remark}
We refer the reader to \Cref{prop.proper.Tate.from.genuine.G.objs} for the relationship between the $\ms{F}$-$H$-Tate construction and genuine $G$-objects, and to \Cref{prop.proper.tate.formula} for an explicit formula for the proper $G$-Tate construction.
\end{remark}

\begin{remark}
If $G$ is a finite group that is not of prime-power order, then $(-)^{\tate G}$ is zero. This follows from the argument of \cite[Lemma 7.15(i)]{Kaledin-Mack} (see also \cite[Lemma II.6.7]{NS}).
\end{remark}

\begin{observation}
\label{obs.proper.F.tate.recovers.tate}
Suppose that $G$ is a finite group and that $\cR = \Mod_R$ for some commutative ring spectrum $R \in \CAlg(\Spectra)$. Then, the $\{ e \}$-$G$-Tate construction recovers the usual $G$-Tate construction
\[ (-)^{\tate_{\{ e \}} G} \simeq (-)^{{ \sf t} G} := \cofib \left( (-)_{\htpy G} \xra{\Nm_G} (-)^{\htpy G} \right) ~. \]
This is proved as \cite[Lemma I.3.8(iii)]{NS} in the case that $R = \SS$ (so that $\cR = \Mod_\SS(\Spectra) \simeq \Spectra$), and the same proof applies verbatim in general.
\end{observation}

\begin{remark}
\label{rmk.proper.F.tate.recovers.tate.in.any.presentably.sym.mon}
In fact, for a finite group $G$, the $\{ e \}$-$G$-Tate construction recovers the usual $G$-Tate construction for any presentably symmetric monoidal stable $\infty$-category $\cR$.\footnote{We do not need this fact, and so we do not prove it. Nevertheless, we include the present discussion in order to motivate \Cref{defn.H.tate}.} To verify this, it suffices to prove the analog of the first equivalence of \cite[Lemma I.3.8(iii)]{NS}, namely that for any object $X \in \cR^{\htpy G}$ the morphism
\[
\left( \colim_{Y \in (\cR^{\htpy G}_\Ind)_{/X}} Y \right)
\longra
X
\]
is an equivalence, where we write $\cR^{\htpy G}_\Ind \subseteq \cR^{\htpy G}$ for the stable subcategory generated by the image of the induction functor $\cR \xra{\Ind_e^G} \cR^{\htpy G}$. To see this, observe first that this subcategory is the stable envelope of the Kleisli $\infty$-category associated to the monadic adjunction
\[ \begin{tikzcd}[column sep=1.5cm]
\cR
\arrow[transform canvas={yshift=0.9ex}]{r}{\Ind_e^G}
\arrow[leftarrow, transform canvas={yshift=-0.9ex}]{r}[yshift=-0.2ex]{\bot}[swap]{\Res^G_e}
&
\cR^{\htpy G}
\end{tikzcd}~. \]
Hence, the claim follows from the spectrally-enriched analog of \Cref{lem.T.alg.is.colim.of.frees} and the fact that for any stable $\infty$-category $\cC$ and any spectrally-enriched $\infty$-category $\cI$ the diagram
\[ \begin{tikzcd}
\Fun^\ex(\Env(\cI),\cC)
\arrow{rr}{\sim}
\arrow{rd}[sloped, swap]{\colim_{\Env(\cI)}}
&
&
\ul{\Fun} ( \cI , \cC)
\arrow{ld}[sloped, swap]{\ul{\colim}_\cI}
\\
&
\cC
\end{tikzcd} \]
commutes (where $\Env(\cI)$ denotes the stable envelope of $\cI$ and the underlines signify the analogous enriched notions).
\end{remark}

\begin{lemma}
\label{lem.T.alg.is.colim.of.frees}
Let $\cC$ be an $\infty$-category and let $T \in \Alg(\Fun(\cC,\cC))$ be a monad on $\cC$. Let us write
\[ \begin{tikzcd}[column sep=1.5cm]
\cC
\arrow[transform canvas={yshift=0.9ex}]{r}{F}
\arrow[leftarrow, transform canvas={yshift=-0.9ex}]{r}[yshift=-0.2ex]{\bot}[swap]{U}
&
\Alg_T(\cC)
\end{tikzcd} \]
for the corresponding free/forget adjunction (so that $T \simeq UF$), and let us write
\[
{\sf Kl}_T(\cC) \subseteq \Alg_T(\cC)
\]
for the Kleisli $\infty$-category of $T$ (i.e.\! the full subcategory on the free $T$-algebras). Then, for any $T$-algebra $X \in \Alg_T(\cC)$ the canonical morphism
\begin{equation}
\label{morphism.from.colimit.of.free.T.algs}
\colim_{FY \in {\sf Kl}_T(\cC)_{/X}} FY
\longra
X
\end{equation}
in $\Alg_T(\cC)$ is an equivalence.
\end{lemma}

\begin{proof}
It suffices to observe that the morphism \Cref{morphism.from.colimit.of.free.T.algs} extends to a commutative diagram
\[ \begin{tikzcd}[row sep=1.5cm]
\colim_{F Y \in {\sf Kl}_T(\cC)_{/X}} { | (FU)^{\bullet+1}(FY) | }
\arrow{r}
\arrow{d}[sloped, anchor=north]{\sim}
&
{|(FU)^{\bullet+1}X|}
\arrow{d}[sloped, anchor=south]{\sim}
\arrow{ld}
\\
\colim_{FY \in {\sf Kl}_T(\cC)_{/X}} FY
\arrow{r}
&
X
\end{tikzcd} \]
in $\Alg_T(\cC)$.
\end{proof}

\begin{definition}
\label{defn.H.tate}
In view of \Cref{obs.proper.F.tate.recovers.tate} and \Cref{rmk.proper.F.tate.recovers.tate.in.any.presentably.sym.mon}, in the special case where $\ms{F} = \{e\}$, we simply write
\[
(-)^{\st H}
:=
(-)^{\tate_{\{e\}} H}
\]
and refer to this functor as the \bit{$H$-Tate construction}.
\end{definition}

\begin{notation}
Suppose that $G$ is a finite group. We write
\[
(-)_{\htpy G}
\xra{\Nm_G}
(-)^{\htpy G}
\xra{\sQ_G}
(-)^{\st G}
\]
for the cofiber sequence in $\Fun^\ex ( \cR^{\htpy G} , \cR)$ that defines the $G$-Tate construction.\footnote{We have chosen the notation ``$\sQ$'' to invoke the idea that this the map to the quotient (of homotopy $G$-fixedpoints by homotopy $G$-orbits).}
\end{notation}

\begin{observation}
\label{obs.various.properties.of.tate}
Let $G$ be a finite group and let $H \leq G$ be a normal subgroup. We record the following facts for future use.
\begin{enumerate}

\item

Parametrizing the norm map $(-)_{\htpy H} \xra{\Nm_H} (-)^{\htpy H}$ in $\Fun(\cR^{\htpy H},\cR)$ over the $\sB H$-bundle $\BG \ra \sB (G/H)$, we obtain a homotopy $(G/H)$-equivariant norm map, i.e.\! a morphism
\[
(-)_{\htpy H}
\xra{\Nm_H}
(-)^{\htpy H}
\]
in $\Fun(\cR^{\htpy G},\cR^{\htpy (G/H)})$. Hence, we obtain a cofiber sequence
\begin{equation}
\label{equivariant.cofiber.sequence.norm.and.tate}
(-)_{\htpy H}
\xra{\Nm_H}
(-)^{\htpy H}
\xra{\sQ_H}
(-)^{\st H}
\end{equation}
in $\Fun(\cR^{\htpy G},\cR^{\htpy (G/H)})$ lifting the analogous cofiber sequence in $\Fun(\cR^{\htpy G}, \cR)$ (because the forgetful functor $\cR^{\htpy (G/H)} \xra{\fgt} \cR$ is exact). In particular, we obtain a residual homotopy $(G/H)$-equivariant structure on the $H$-Tate construction. We use these facts without further comment.

\item\label{part.rlax.s.m.str.on.h.to.tate}

Consider the morphism
\begin{equation}
\label{h.to.tate.becomes.rlax.s.m}
(-)^{\htpy H}
\xra{\sQ_H}
(-)^{\st H}
\end{equation}
in $\Fun(\cR^{\htpy G},\cR^{\htpy (G/H)})$, the second morphism in the cofiber sequence \Cref{equivariant.cofiber.sequence.norm.and.tate}. Evidently, the source of the morphism \Cref{h.to.tate.becomes.rlax.s.m} is canonically right-laxly symmetric monoidal. Thereafter, there exists a canonical enhancement of the morphism \Cref{h.to.tate.becomes.rlax.s.m} to one of right-laxly symmetric monoidal functors, obtained by working fiberwise over $\sB(G/H)$ and applying \cite[Theorem I.3.1]{NS} (which guarantees that there is in fact a unique such enhancement when $H=G$).

\item\label{part.norm.maps.compose}

Suppose that $H \leq K \leq G$. Then, for any $E \in \cR^{\htpy G}$, by \cite[Proposition 4.2.2]{ambidex} we have the following natural commutative diagram in $\cR^{\htpy (G/K)}$:
\[ \begin{tikzcd}[row sep=1.5cm, column sep=1.5cm]
(E_{\htpy H})_{\htpy (K/H)}
\arrow{rrr}{\Nm_H(E)_{\htpy (K/H)}}
\arrow{ddd}[swap]{\Nm_{K/H}(E_{\htpy H})}
&
&
&
(E^{\htpy H})_{\htpy (K/H)}
\arrow{ddd}{\Nm_{K/H}(E^{\htpy H})}
\\
&
E_{\htpy K}
\arrow{lu}[sloped]{\sim}
\arrow{rd}[sloped]{\Nm_K(E)}
\\
&
&
E^{\htpy K}
\\
(E_{\htpy H})^{\htpy (K/H)}
\arrow{rrr}[swap]{\Nm_H(E)^{\htpy (K/H)}}
&
&
&
(E^{\htpy H})^{\htpy (K/H)}
\arrow{lu}[sloped]{\sim}
\end{tikzcd}
~.
 \]

\end{enumerate}
\end{observation}

\section{The proper Tate construction from genuine $G$-objects}
\label{section.proper.tate.from.genuine.G.objects}

In \Cref{section.proper.tate.in.htpy.G.objects}, we introduced the $\ms{F}$-$H$-Tate construction for homotopy $G$-objects (\Cref{defn.proper.tate}). Our main goal in this section is to show that this functor can also be described in terms of genuine $G$-objects (\Cref{prop.proper.Tate.from.genuine.G.objs}); this description will be a key ingredient of our proof of \Cref{intro.thm.gluing.functors} in \Cref{section.gluing.functors.for.RgG}. We also use it to give an explicit formula for the proper $G$-Tate construction (\Cref{prop.proper.tate.formula}), which we state and prove directly after stating \Cref{prop.proper.Tate.from.genuine.G.objs}. We then make a number of preliminary observations before proving \Cref{prop.proper.Tate.from.genuine.G.objs} at the end of the section.

\begin{local}
In this section, we fix a compact Lie group $G$, a family $\ms{F} \in \Down_{\pos_G}$, and a rigidly-compactly generated presentably symmetric monoidal stable $\infty$-category $\cR$.
\end{local}

\begin{notation}
We define the subcategory
\[
\cI_\ms{F}
:=
\Spectra^{\gen G}_\ms{F} \otimes \cR
:=
\brax{ \Sigma^\infty_G(G/H)_+ }_{H \in \ms{F}} \otimes \cR
\subseteq
\Spectra^{\gen G} \otimes \cR
=:
\cR^{\gen G}
~.
\]
\end{notation}

\begin{observation}
\label{obs.IF.a.closed.ideal}
By Observations S.\ref{strat:obs.can.prestratn.of.SpgG.is.sm} \and \ref{obs.tensoring.preserves.closed.ideals}, the subcategory $\cI_\ms{F} \subseteq \cR^{\gen G}$ is a closed ideal.
\end{observation}

\begin{notation}
\label{notn.free.gen.G.object.on.object.of.R.and.gen.G.space}
Given an object $E \in \cR$ and a genuine $G$-space $X \in \Spaces^{\gen G}$, we write
\[
E \brax{X}
:=
(\Sigma^\infty_G X_+ \otimes E)
\in
\Spectra^{\gen G} \otimes \cR
=:
\cR^{\gen G}
~.
\]
\end{notation}

\begin{observation}
\label{obs.compatibility.of.free.objects.on.a.G.space}
\Cref{notn.free.gen.G.object.on.object.of.R.and.gen.G.space} is compatible with \Cref{notn.free.htpy.G.object.on.object.of.R.and.htpy.G.space}, in the sense that for any genuine $G$-space $X \in \Spaces^{\gen G}$ and any object $E \in \cR$ we have an equivalence
\[
U(E \brax{X})
\simeq
E^\htpy \brax{U(X)}
\]
in $\cR^{\htpy G}$ (due to the equivalence $U(\Sigma^\infty_G X_+) \simeq \Sigma^\infty U(X)_+$ in $\Spectra^{\htpy G}$).
\end{observation}

\begin{observation}
\label{obs.IF.via.ideal.generators}
There is an identification
\[
\cI_\ms{F}
=
\brax{ \uno_\cR \brax{G/H} }^\otimes_{H \in \ms{F}}
\]
among ideals of $\cR^{\gen G}$.
\end{observation}

\begin{definition}
\label{defn.F.H.geom.fps}
For any closed subgroup $H \in \pos_G$, the \bit{$\ms{F}$-$H$-geometric fixedpoints} functor is
\[
\Phi_{\ms{F}}^H
:
\cR^{\gen G}
\xlongra{p_L}
\cR^{\gen G} / \cI_\ms{F}
\xra{\ulhom_{\cR^{\gen G} / \cI_\ms{F}} ( p_L(\uno_\cR \brax{G/H}) , -)}
\cR^{\htpy \Weyl(H)}
~.
\]
\end{definition}

\begin{observation}
\Cref{defn.F.H.geom.fps} generalizes the functor $\cR^{\gen G} \xra{\Phi^H} \cR^{\htpy \Weyl(H)}$ of \Cref{notn.various.fixedpoints.on.genuine.G.objects.in.M}, as we now explain. Given a presentable stable $\infty$-category $\cC \in \PrLSt$ and a compact object $X \in \cC^\omega$, we obtain an adjunction
\[ \begin{tikzcd}[column sep=2cm]
\Spectra
\arrow[transform canvas={yshift=0.9ex}]{r}{(-) \otimes X}
\arrow[leftarrow, transform canvas={yshift=-0.9ex}]{r}[yshift=-0.2ex]{\bot}[swap]{\ulhom_\cC(X,-)}
&
\cC
\end{tikzcd} \]
in $\PrLSt$. When tensored with $\cR$, this gives an adjunction
\[ \begin{tikzcd}[column sep=2.5cm]
\cR
\arrow[transform canvas={yshift=0.9ex}]{r}{X \otimes (-)}
\arrow[leftarrow, transform canvas={yshift=-0.9ex}]{r}[yshift=-0.2ex]{\bot}[swap]{\ulhom_\cC(X,-) \otimes \cR}
&
\cC \otimes \cR
\end{tikzcd} \]
in $\Mod_\cR(\PrLSt)$, which yields a commutative triangle
\[ \begin{tikzcd}[column sep=2.5cm]
\cC \otimes \cR
\arrow{r}{\ulhom_\cC(X,-) \otimes \cR}
\arrow{rd}[sloped, swap]{\ulhom_{\cC \otimes \cR}(X \otimes \uno_\cR , - )}
&
\Spectra \otimes \cR
\arrow{d}[sloped, anchor=south]{\sim}
\\
&
\cR
\end{tikzcd}~. \]
We apply this in the case that $\cC = \Spectra^{\gen G} / \Spectra^{\gen G}_\ms{F}$ and $X = p_L(\Sigma^\infty_G(G/H)_+)$ (which is compact because $p_L$ preserves compact objects since $\nu$ preserves colimits). Using the fact that the functor $(-) \otimes \cR$ preserves colimits, we find that the $\ms{F}$-$H$-geometric fixedpoints functor may be described as the tensor product
\[
\left(
\Spectra^{\gen G}
\xlongra{p_L}
\Spectra^{\gen G} / \Spectra^{\gen G}_\ms{F}
\xra{ \ulhom_{\Spectra^{\gen G} / \Spectra^{\gen G}_\ms{F}} ( p_L ( \Sigma^\infty_G(G/H)_+ ) , - ) }
\Spectra^{\htpy \Weyl(H)}
\right)
\otimes
\cR
~.
\]
In particular, in the case that $\ms{F} = (^{\not\geq} H)$, we obtain an equivalence
\[
\Phi^H_{(^{\not\geq} H)}
\simeq
\Phi^H
\]
in $\Fun(\cR^{\gen G} , \cR^{\htpy \Weyl(H)})$. We use this fact without further comment.
\end{observation}

\begin{prop}
\label{prop.proper.Tate.from.genuine.G.objs}
For any closed subgroup $H \leq G$, we have a canonical commutative triangle
\[ \begin{tikzcd}
\cR^{\htpy G}
\arrow[hook]{rr}{\beta}
\arrow{rd}[swap, sloped]{(-)^{\tate_\ms{F} H }}
&
&
\cR^{\gen G}
\arrow{ld}[sloped, swap]{\Phi_\ms{F}^H}
\\
&
\cR^{\htpy \Weyl(H)}
\end{tikzcd}~. \]
\end{prop}

We learned the following result from Akhil Mathew.

\begin{prop}
\label{prop.proper.tate.formula}
Assume that $G$ is a finite group, and fix any homotopy $G$-object $E \in \cR^{\htpy G}$ in $\cR$.
\begin{enumerate}

\item\label{proper.tate.as.colimit}

There is a canonical equivalence
\[
E^{\tate G}
\xlongla{\sim}
\colim_{n \in \NN} \left( ( S^{n \tilde{\rho}} \tensoring E )^{\htpy G} \right)
\]
in $\cR$, in which
\begin{itemize}

\item $\tensoring$ denotes the action of $\Spaces_*^{\htpy G}$ on $\cR^{\htpy G}$,

\item $\tilde{\rho}$ denotes the reduced (real) regular representation of $G$,

\item $S^V$ denotes the representation sphere corresponding to a representation $V$, and

\item the morphisms in the colimit are induced by the inclusion $\{0\} \ra \tilde{\rho}$.

\end{itemize}

\item\label{proper.tate.as.inversion.of.euler}

Suppose that $\cR = \Mod_R$ for some commutative ring spectrum $R \in \CAlg(\Spectra)$. Then, a complex orientation of $R$ (e.g.\! a $\ZZ$-algebra structure) determines an element
\[
e \in \pi_{-2(|G|-1)}(R^{\htpy G})
~,
\]
the Euler class of the reduced complex regular representation (via the complex orientation), with respect to which the canonical morphism $E^{\tate G} \la E^{\htpy G}$ in $\cR$ exhibits $E^{\tate G}$ as the localization at $e$; that is,
\[
E^{\tate G}
\simeq
E^{\htpy G}[e^{-1}]
~.
\]

\end{enumerate}
\end{prop}

\begin{proof}
We begin with part \Cref{proper.tate.as.colimit}. Observe the equivalences
\begin{equation}
\label{begin.formula.for.proper.tate.using.genuine}
E^{\tate G} \simeq \Phi^G \beta ( E )
\simeq
( \wEF_{^{\not\geq} G} \tensoring \beta(E) )^G
~,
\end{equation}
the first by \Cref{prop.proper.Tate.from.genuine.G.objs}. By \cite[Proposition 2.7]{MNN-indres}, we have an equivalence
\begin{equation}
\label{equivalence.of.wEF.for.proper.tate.from.MNN}
\wEF_{^{\not\geq} G} \simeq \colim_{n \in \NN} ( S^{n \tilde{\rho}} )
\end{equation}
in $\Spaces^{\gen G}_*$. Combining the equivalences \Cref{begin.formula.for.proper.tate.using.genuine} \and \Cref{equivalence.of.wEF.for.proper.tate.from.MNN}, we obtain the first equivalence in the composite equivalence
\begin{equation}
\label{proper.tate.is.colim.of.categorical.fixedpoints}
E^{\tate G}
\simeq
( ( \colim_{n \in \NN} (S^{n \tilde{\rho} } ) ) \tensoring \beta(E) )^G
\xlongla{\sim}
\colim_{n \in \NN} ( ( S^{n \tilde{\rho}} \tensoring \beta(E) )^G )
~.
\end{equation}
On the other hand, since $\Sigma^\infty_G S^{n \tilde{\rho}} \in \Spectra^{\gen G}$ is dualizable, we have an equivalence
\begin{equation}
\label{rep.sphere.dzbl.hence.commutes.with.beta}
S^{n \tilde{\rho}} \tensoring \beta(E)
\xlongra{\sim}
\beta(S^{n \tilde{\rho}} \tensoring E )
\end{equation}
in $\cR^{\gen G}$. Therefore, combining equivalences \Cref{proper.tate.is.colim.of.categorical.fixedpoints} \and \Cref{rep.sphere.dzbl.hence.commutes.with.beta}, we obtain an equivalence
\[
E^{\tate G}
\simeq
\colim_{n \in \NN} ( \beta ( S^{n \tilde{\rho}} \tensoring E )^G)
\simeq
\colim_{n \in \NN}  ( ( S^{n \tilde{\rho}} \tensoring E )^{\htpy G} )
~,
\]
as asserted.

Now, part \Cref{proper.tate.as.inversion.of.euler} follows from part \Cref{proper.tate.as.colimit} along with the observation that $2 \tilde{\rho}$ is precisely the reduced complex regular representation.
\end{proof}

\begin{notation}
We write
\[ \begin{tikzcd}[column sep=1.5cm]
\Ind(\cR^{\htpy G})
\arrow[transform canvas={yshift=0.9ex}]{r}{\colim}
\arrow[hookleftarrow, transform canvas={yshift=-0.9ex}]{r}[yshift=-0.2ex]{\bot}[swap]{i}
&
\cR^{\htpy G}
\end{tikzcd} \]
for the canonical adjunction resulting from the fact that $\cR^{\htpy G}$ admits filtered colimits: its right adjoint is the canonical fully faithful inclusion, and its left adjoint is given by taking colimits of filtered diagrams.
\end{notation}

\begin{observation}
Because the functor
\[
\cR^{\gen G}
\xlongra{U}
\cR^{\htpy G}
\]
preserves colimits (being a left adjoint) and the presentable stable $\infty$-category $\cR^{\gen G}$ is compactly generated (because both $\Spectra^{\gen G}$ and $\cR$ are), there exists a canonical factorization
\begin{equation}
\label{factor.U.through.U.tilde}
\begin{tikzcd}
\cR^{\gen G}
\arrow{rr}{U}
\arrow[dashed]{rd}[sloped, swap]{\tilde{U}}
&
&
\cR^{\htpy G}
\\
&
\Ind(\cR^{\htpy G})
\arrow{ru}[swap, sloped]{\colim}
\end{tikzcd}~,
\end{equation}
namely the functor
\[
\tilde{U}
:=
\Ind \left( (\cR^{\gen G})^\omega
\longhookra
\cR^{\gen G}
\xlongra{U}
\cR^{\htpy G} \right)
~.
\]
\end{observation}

\begin{observation}
\label{obs.U.tilde.preserves.compact.objects}
By construction, the functor $\cR^{\gen G} \xra{\tilde{U}} \Ind(\cR^{\htpy G})$ preserves compact objects: given a compact object $E \in (\cR^{\gen G})^\omega \subseteq \cR^{\gen G}$, there is a canonical equivalence
\[
\tilde{U}(E)
\xlongra{\sim}
i(U(E))
~.
\]
\end{observation}

\begin{observation}
Both the source and target of the functor $\cR^{\gen G} \xra{\tilde{U}} \Ind(\cR^{\htpy G})$ are compactly generated. Combining this fact with \Cref{obs.U.tilde.preserves.compact.objects}, we see that it admits a colimit-preserving right adjoint, which we denote by
\[ \begin{tikzcd}[column sep=1.5cm]
\cR^{\gen G}
\arrow[transform canvas={yshift=0.9ex}]{r}{\tilde{U}}
\arrow[leftarrow, dashed, transform canvas={yshift=-0.9ex}]{r}[yshift=-0.2ex]{\bot}[swap]{\tilde{\beta}}
&
\Ind(\cR^{\htpy G})
\end{tikzcd}~. \]
\end{observation}

\begin{observation}
\label{obs.beta.is.beta.tilde.i}
Passing to right adjoints in the commutative diagram \Cref{factor.U.through.U.tilde}, we obtain a commutative diagram
\[
\begin{tikzcd}
\cR^{\gen G}
\arrow[hookleftarrow]{rr}{\beta}
&
&
\cR^{\htpy G}
\\
&
\Ind(\cR^{\htpy G})
\arrow[hookleftarrow]{ru}[sloped, swap]{i}
\arrow{lu}[sloped, swap]{\tilde{\beta}}
\end{tikzcd}~.
\]
\end{observation}

\begin{observation}
\label{obs.gen.G.objects.rigid.if.R.is}
The presentably symmetric monoidal stable $\infty$-category $\cR^{\gen G} := \Spectra^{\gen G} \otimes \cR$ is rigidly-compactly generated (since the functor $\St \xra{\Ind} \PrLSt$ is symmetric monoidal).
\end{observation}

\begin{observation}
\label{obs.U.tilde.sym.mon}
The functor $\cR^{\gen G} \xra{\tilde{U}} \Ind(\cR^{\htpy G})$ is canonically symmetric monoidal, because the composite $(\cR^{\gen G})^\omega \hookra \cR^{\gen G} \xra{U} \cR^{\htpy G}$ is symmetric monoidal by \Cref{obs.gen.G.objects.rigid.if.R.is}.
\end{observation}

\begin{observation}
\label{obs.beta.tilde.is.RgG.linear}
Because $\cR^{\gen G}$ is rigidly-compactly generated by \Cref{obs.gen.G.objects.rigid.if.R.is} and $\tilde{U}$ is symmetric monoidal by \Cref{obs.U.tilde.sym.mon}, by \cite[Chapter 1, Lemma 9.3.6]{GR} the right adjoint $\tilde{\beta}$ is $\cR^{\gen G}$-linear: in other words, for any $E \in \cR^{\gen G}$ and $F \in \Ind(\cR^{\htpy G})$ we have the projection formula
\[
\tilde{\beta}(\tilde{U}(E) \otimes F)
\simeq
E \otimes \tilde{\beta}(F)
~.
\]
\end{observation}

\begin{observation}
\label{obs.ideal.gend.by.Utilde.of.IF.is.a.closed.ideal.and.coaugmented.idempotents.agree}
By \Cref{obs.IF.a.closed.ideal} \and Corollary S.\ref{strat:cor.ideal.gend.by.image.of.closed.ideal.is.closed}, the ideal
\[
\brax{ \tilde{U} ( \cI_\ms{F}) }^\otimes
\subseteq
\Ind(\cR^{\htpy G} )
\]
is a closed ideal and moreover
\[
\nu \left( \uno_{\Ind(\cR^{\htpy G}) / \brax{ \tilde{U} ( \cI_\ms{F}) }^\otimes} \right)
\simeq
\tilde{U} ( \nu ( \uno_{\Spectra^{\gen G} / \cI_\ms{F}} ) )
~.
\]
\end{observation}

\begin{observation}
\label{obs.ideal.gend.by.Utilde.of.IF.is.Ind.of.IhF}
We have the string of identifications
\begin{align}
\label{use.identification.of.IF.as.image.of.SpgGF}
\brax{ \tilde{U}(\cI_\ms{F}) }^\otimes
&=
\brax{ \tilde{U} ( \uno_\cR \brax{ G / H } ) }^\otimes_{H \in \ms{F}}
\\
\label{use.that.Utilde.is.iU.on.cpcts}
&=
\brax{ iU ( \uno_\cR \brax{ G/H} ) }^\otimes_{H \in \ms{F}}
\\
\label{use.compatibility.of.genuine.and.htpy.free.objects}
&=
\brax{ i ( \uno_\cR^\htpy \brax{G/H} ) }^\otimes_{H \in \ms{F}}
\\
\label{use.that.ind.of.a.thick.ideal.has.same.generators}
&=
\Ind(\cI^\htpy_\ms{F})
\end{align}
among subcategories of $\Ind(\cR^{\htpy G})$, where identification \Cref{use.identification.of.IF.as.image.of.SpgGF} follows from \Cref{obs.IF.via.ideal.generators}, identification \Cref{use.that.Utilde.is.iU.on.cpcts} follows from \Cref{obs.U.tilde.preserves.compact.objects}, identification \Cref{use.compatibility.of.genuine.and.htpy.free.objects} follows from \Cref{obs.compatibility.of.free.objects.on.a.G.space}, and identification \Cref{use.that.ind.of.a.thick.ideal.has.same.generators} follows from \Cref{obs.ind.of.a.thick.ideal.is.a.closed.ideal}.
\end{observation}

\begin{proof}[Proof of \Cref{prop.proper.Tate.from.genuine.G.objs}]
Our proof takes place within the context of the diagram in \Cref{figure.for.proof.of.proper.tate.from.genuine},
\begin{figure}[h]
\[ \begin{tikzcd}[row sep=1.5cm, column sep=1.5cm]
\cI^\htpy_\ms{F}
\arrow[hook]{r}
\arrow[hook]{d}
&
\cR^{\htpy G}
\arrow{r}{p}
\arrow[hook]{d}{i}
\arrow[bend left=40]{dddd}[pos=0.8]{\beta}
\arrow[leftarrow, bend right=5]{lddddd}[pos=0.6]{U}
&
\cR^{\htpy G} /^\St \cI^\htpy_\ms{F}
\arrow[hook]{d}{i'}
\\
\Ind(\cI^\htpy_\ms{F})
\arrow[hook, crossing over]{r}
&
\Ind(\cR^{\htpy G})
\arrow[crossing over]{r}{p'}
&
\Ind(\cR^{\htpy G}) / \Ind(\cI^\htpy_\ms{F})
\\[-1.5cm]
\rotatebox{90}{$\simeq$}
&
\rotatebox{90}{$\simeq$}
&
\rotatebox{90}{$\simeq$}
\\[-1.5cm]
\brax{\tilde{U}(\cI_\ms{F})}^\otimes
\arrow[hook, crossing over]{r}
&
\Ind(\cR^{\htpy G})
\arrow[crossing over, yshift=0.9ex]{r}{p_L}
\arrow[crossing over, hookleftarrow, yshift=-0.9ex]{r}[yshift=-0.2ex]{\bot}[swap]{\nu}
&
\Ind(\cR^{\htpy G}) \left/ \brax{\tilde{U}(\cI_\ms{F})}^\otimes \right.
\\
\cI_\ms{F}
\arrow[crossing over, hook]{r}
\arrow{u}
&
\cR^{\gen G}
\arrow[transform canvas={xshift=-0.9ex}]{u}{\tilde{U}}
\arrow[leftarrow, transform canvas={xshift=0.9ex}]{u}[sloped, anchor=south, yshift=-0.2ex]{\bot}[swap]{\tilde{\beta}}
\arrow[transform canvas={yshift=0.9ex}]{r}{p_L}
\arrow[hookleftarrow, transform canvas={yshift=-0.9ex}]{r}[yshift=-0.2ex]{\bot}[swap]{\nu}
&
\cR^{\gen G} / \cI_\ms{F}
\\
(\cR^{\gen G})^\omega
\arrow[hook]{ru}
\end{tikzcd} \]
\caption{The proof of \Cref{prop.proper.Tate.from.genuine.G.objs} takes places within this diagram.}
\label{figure.for.proof.of.proper.tate.from.genuine}
\end{figure}
in which
\begin{itemize}
\item the upper row is a stable quotient sequence,
\item the functor $i'$ is the fully faithful inclusion into $\Ind(\cR^{\htpy G} /^\St \cI^\htpy_\ms{F} ) \simeq \Ind(\cR^{\htpy G}) / \Ind(\cI^\htpy_\ms{F})$ (note that the functor $\St \xra{\Ind} \PrLSt$ preserves colimits),\footnote{By \Cref{remark.stable.quotient.by.presentable.subcategory.is.presentable.quotient}, the functor $i'$ can also be seen as the induced functor on stable quotients.}
\item the lower two rows are presentable quotient sequences in which the kernels are closed ideals by Observations \ref{obs.ideal.gend.by.Utilde.of.IF.is.a.closed.ideal.and.coaugmented.idempotents.agree} \and \ref{obs.IF.a.closed.ideal},
\item the back triangles commute by Observations \ref{obs.beta.is.beta.tilde.i} \and \ref{obs.U.tilde.preserves.compact.objects}, and
\item the equivalences in the middle row follow from \Cref{obs.ideal.gend.by.Utilde.of.IF.is.Ind.of.IhF}.
\end{itemize}
Namely, for any $E \in \cR^{\htpy G}$, we have natural equivalences
\begin{align}
\nonumber
E^{\tate_\ms{F} H}
& :=
\ulhom_{\cR^{\htpy G} /^\St \cI_\ms{F}^\htpy} ( p(\uno_\cR^\htpy \brax{G/H}), p E)
\\
\nonumber
& \simeq
\ulhom_{\Ind(\cR^{\htpy G}) / \Ind(\cI_\ms{F}^\htpy)} ( i ' p ( \uno_\cR^\htpy \brax{G/H} ) , i ' p E )
\\
\nonumber
& \simeq
\ulhom_{\Ind(\cR^{\htpy G}) / \Ind(\cI_\ms{F}^\htpy)} ( ( p' i ( \uno_\cR^\htpy \brax{G/H} ) ) , p' i E )
\\
\nonumber
& \simeq
\ulhom_{\Ind(\cR^{\htpy G})} ( i ( \uno^\htpy_\cR \brax{G/H} ) , \nu p_L i E )
\\
\label{use.that.free.gen.and.htpy.G.objects.on.a.G.space.are.compatible}
& \simeq
\ulhom_{\Ind(\cR^{\htpy G})} ( iU(\uno_\cR \brax{G/H} ) , \nu p_L i E )
\\
\label{use.that.uno.R.brax.G.mod.H.is.compact}
& \simeq
\ulhom_{\Ind(\cR^{\htpy G})} ( \tilde{U}(\uno_\cR \brax{G/H} ) , \nu p_L i E )
\\
\nonumber
& \simeq
\ulhom_{\Ind(\cR^{\htpy G})} ( \tilde{U}(\uno_\cR \brax{G/H} ) , (\nu p_L \uno_{\Ind(\cR^{\htpy G})} ) \otimes (iE ) )
\\
\label{use.that.U.tilde.preserves.relevant.coaugmented.idempotents}
& \simeq
\ulhom_{\Ind(\cR^{\htpy G})} ( \tilde{U}(\uno_\cR \brax{G/H} ) , ( \tilde{U} \nu p_L \uno_{\cR^{\gen G}} ) \otimes (iE ) )
\\
\nonumber
& \simeq
\ulhom_{\cR^{\gen G}} ( \uno_\cR \brax{G/H} , \tilde{\beta} ( ( \tilde{U} \nu p_L \uno_{\cR^{\gen G}} ) \otimes (iE ) ) )
\\
\label{use.projection.formula.for.beta.tilde}
& \simeq
\ulhom_{\cR^{\gen G}} ( \uno_\cR \brax{G/H} , ( \nu p_L \uno_{\cR^{\gen G}} ) \otimes ( \tilde{\beta} i E ) )
\\
\nonumber
& \simeq
\ulhom_{\cR^{\gen G}} ( \uno_\cR \brax{G/H} , ( \nu p_L \uno_{\cR^{\gen G}} ) \otimes ( \beta E ) )
\\
\nonumber
& \simeq
\ulhom_{\cR^{\gen G}} ( \uno_\cR \brax{G/H} , \nu p_L \beta E )
\\
\nonumber
& \simeq
\ulhom_{\cR^{\gen G} / \cI_\ms{F}} ( p_L ( \uno_\cR \brax{G/H}) , p_L \beta E )
\\
\nonumber
& =:
\Phi^H_\ms{F} \beta E
\end{align}
in $\cR^{\htpy \Weyl(H)}$, in which
\begin{itemize}
\item equivalence \Cref{use.that.free.gen.and.htpy.G.objects.on.a.G.space.are.compatible} follows from \Cref{obs.compatibility.of.free.objects.on.a.G.space},
\item equivalence \Cref{use.that.uno.R.brax.G.mod.H.is.compact} follows from the fact that $\uno_\cR \brax{G/H} \in \cR^{\gen G}$ is compact (because $\uno_\cR \in \cR$ and $\Sigma^\infty_G(G/H)_+ \in \Spectra^{\gen G}$ are compact),
\item equivalence \Cref{use.that.U.tilde.preserves.relevant.coaugmented.idempotents} follows from \Cref{obs.ideal.gend.by.Utilde.of.IF.is.a.closed.ideal.and.coaugmented.idempotents.agree}, and
\item equivalence \Cref{use.projection.formula.for.beta.tilde} follows from \Cref{obs.beta.tilde.is.RgG.linear}. \qedhere
\end{itemize}
\end{proof}

\section{Gluing functors for the geometric stratification of genuine $G$-objects}
\label{section.gluing.functors.for.RgG}

In this section, we prove \Cref{intro.thm.gluing.functors} as \Cref{thm.stratn.of.genuine.G.objects}.

\begin{local}
In this section, we fix a finite group $G$ and a rigidly-compactly generated presentably symmetric monoidal stable $\infty$-category $\cR$.
\end{local}

\begin{notation}
Given subgroups $K,H \subseteq G$, we define the subset
\[
\tilde{C}(H,K)
:=
\{ g \in G
:
H \subseteq gKg^{-1}
\subseteq 
\Normzer(H)
\}
\subseteq
G
\]
of those elements of $G$ that conjugate $K$ to lie between $H$ and the normalizer of $H$.
\end{notation}

\begin{observation}
Considering $G$ as a $(G,G)$-bimodule set by left and right multiplication, the subset $\tilde{C}(H,K) \subseteq G$ inherits a $(\Normzer(H),\Normzer(K))$-bimodule structure: that is, it is carried into itself by left multiplication by elements of $\Normzer(H)$ and by right multiplication by elements of $\Normzer(K)$. We use this fact without further comment.
\end{observation}

\begin{notation}
\label{notn.C.sub.H.upper.K}
We write
\[
C(H,K)
:=
\tilde{C}(H,K) / K
\]
for the set obtained by quotienting $\tilde{C}(H,K)$ by its right $K$-action.
\end{notation}

\begin{observation}
Of course, $C(H,K)$ naturally inherits a $(\Normzer(H),\Weyl(K))$-bimodule structure. Moreover, its left $H$-action is trivial, so that this descends to a $(\Weyl(H),\Weyl(K))$-bimodule structure.\footnote{See \Cref{lemma.identify.double.coset} for a more conceptual description of the double quotient $\Weyl(H) \backslash C(H,K) / \Weyl(K)$.} We use these facts without further comment.
\end{observation}

\begin{theorem}
\label{thm.stratn.of.genuine.G.objects}
For any morphism $H \leq K$ in $\pos_G$, the gluing functor
\[
\cR^{\htpy \Weyl(H)}
\xra{\Gamma^H_K}
\cR^{\htpy \Weyl(K)}
\]
of the geometric stratification \Cref{sym.mon.geometric.stratn.in.general} of $\cR^{\gen G}$ is given by the formula
\begin{equation}
\label{formula.for.gluing.functor}
E
\longmapsto
\bigoplus_{[g] \in \Weyl(H) \backslash C(H,K) / \Weyl(K)}
\Ind_{(\Normzer(H) \cap \Normzer(gKg^{-1})) / (gKg^{-1})}^{\Weyl(K)}
E^{\tate ( gKg^{-1} ) / H}
\end{equation}
(where we implicitly use the isomorphism $\Weyl(gKg^{-1}) \xra{\cong} \Weyl(K)$ given by conjugation by $g^{-1}$).
\end{theorem}

\begin{remark}
The set $C(H,K)$ is empty unless we have $H \leq K \leq \Normzer(H)$ in $\pos_G$. At the other extreme, if we have $H \subseteq K \subseteq G$ with $H$ normal in $K$ and $K$ normal in $G$, then $C(H,K) = G/K$. In this case, it is not hard to see that the formula \Cref{formula.for.gluing.functor} reduces to the formula
\[
E
\longmapsto
\Ind_{\Normzer(H)/K}^{G/K} E^{\tate (K/H)}
~.
\]
If furthermore $H$ is normal in $G$, then the formula \Cref{formula.for.gluing.functor} reduces to the even simpler formula
\[
E
\longmapsto
E^{\tate (K/H)}
~.
\]
\end{remark}

\begin{remark}
While the description of the functor $\Gamma^H_K$ given by \Cref{thm.stratn.of.genuine.G.objects} is fairly explicit, it has the drawback of making reference to specific representatives of objects of $\pos_G$ (i.e.\! conjugacy classes of subgroups of $G$). Here are two alternative descriptions that are more invariant.
\begin{enumerate}

\item If we replace the double quotient appearing in formula \Cref{formula.for.gluing.functor} with the \textit{groupoid} double quotient, we obtain a span
\[
\sB \Weyl(H)
\longla
C(H,K)_{\htpy (\Weyl(H) \times \Weyl(K))}
\longra
\sB \Weyl(K)
~,
\]
from which $\Gamma^H_K$ may be obtained as a sort of pull-push operation: ordinary pullback along the leftwards functor followed by the fiberwise proper Tate construction and the indicated induction.

\item A variant of the argument used to prove \Cref{prop.proper.Tate.from.genuine.G.objs} shows that $\Gamma^H_K$ can also be described as the composite
\[
\cR^{\htpy \Weyl(H)}
\xlongra{p}
\cR^{\htpy \Weyl(H)} /^\St \cI^\htpy_\ms{F}
\xra{ \ulhom \left( p \left( \uno_\cR \brax{ (G/H)^K } \right) , (-) \right) }
\cR^{\htpy \Weyl(K)}
~,
\]
where $\ms{F} \in \Down_{\pos_{\Weyl(H)}}$ denotes the family defined in \Cref{local.define.the.special.family.for.gluing.functors}\Cref{item.define.the.special.family}. (Lemmas \ref{lemma.inclusion.from.CHK.to.GmodKupperH.becomes.iso} \and \ref{lemma.map.from.quotient.of.product.to.CHKg.is.an.iso} below can be applied to give an explicit description of the object $p(\uno_\cR \brax{ (G/H)^K } ) \in \cR^{\htpy \Weyl(H)} /^\St \cI^\htpy_\ms{F}$.)

\end{enumerate}
\end{remark}

\begin{example}[genuine $\sA_4$-spectra]
\label{example.A.four.spectra}
Let $\sA_4$ denote the alternating group on four letters. We describe the reconstruction theorem for genuine $\sA_4$-spectra that results from combining \Cref{thm.stratn.of.genuine.G.objects} with Theorems S.\ref{strat:intro.thm.cosms}\ref{strat:intro.main.thm.macrocosm} and S.\ref{strat:intro.thm.gen.G.spt}.\footnote{For brevity, we assume familiarity with the latter results. We refer the reader to \S S.\ref{strat:subsection.examples.of.SpgG} for a number of simpler examples. (In particular, the present example has a number of similarities with Example S.\ref{strat:ex.gen.Sthree.spectra} of genuine $\sS_3$-spectra.)} The poset of conjugacy classes of closed subgroups of $\sA_4$ is
\[
\sP_{\sA_4}
=
\left\{
\begin{tikzcd}
e
\arrow{r}
\arrow{d}
&
\sC_2
\arrow{d}
\\
\sC_3
\arrow{rd}
&
\sV_4
\arrow{d}
\\
&
\sA_4
\end{tikzcd}
\right\}
=:
\left\{
\begin{tikzcd}
1
\arrow{r}
\arrow{d}
&
2
\arrow{d}
\\
3
\arrow{rd}
&
4
\arrow{d}
\\
&
12
\end{tikzcd}
\right\}
~,
\]
where $\sV_4 \cong \sC_2 \times \sC_2$ denotes the Klein four-group and we label strata by cardinality. By Observation S.\ref{strat:obs.NS.proper.tate.is.p.primary}, the gluing functors $\Gamma^{e}_{\sA_4}$ and $\Gamma^{\sC_2}_{\sA_4}$ vanish. Moreover, by Observation S.\ref{strat:H.to.K.gluing.is.zero.when.K.not.leq.NH}, the gluing functor $\Gamma^{\sC_3}_{\sA_4}$ also vanishes. Altogether, we obtain an equivalence
\[
\Spectra^{\gen \sA_4}
\underset{\sim}{\xlongra{\gd}}
\lim^\rlax
\left(
\begin{tikzcd}[row sep=1.5cm, column sep=1.5cm]
\Spectra^{\htpy \sA_4}
\arrow{r}{(-)^{\st \Cyclic_2}}
\arrow{d}[swap]{(-)^{\st \Cyclic_3}}
\arrow{rd}[sloped, swap]{(-)^{\tate \sV_4}}[sloped, xshift=0.2cm, yshift=0.5cm]{\Uparrow}
&
\Spectra^{\htpy \sC_2}
\arrow{d}{(-)^{\st \Cyclic_2}}
\\
\Spectra
&
\Spectra^{\htpy \sC_3}
\arrow{d}{(-)^{\st \Cyclic_3}}
\\
&
\Spectra
\end{tikzcd}
\right)
~.
\]
So, a genuine $\sA_4$-spectrum $E \in \Spectra^{\gen \sA_4}$ is equivalently specified by the data of
\begin{itemize}

\item the objects
\[
\hspace{-2.3cm}
E_1 := \Phi^{e} E \in \Spectra^{\htpy \sA_4}
~,
\qquad
E_2 := \Phi^{\sC_2} E \in \Spectra^{\htpy \Cyclic_2}
~,
\qquad
E_3 := \Phi^{\sC_3} E \in \Spectra
~,
\qquad
E_4 := \Phi^{\sV_4} E \in \Spectra^{\htpy \sC_3}
~,
\qquad
\text{and}
\qquad
E_{12} := \Phi^{\sA_4} E \in \Spectra
~,
\]

\item the morphisms
\[
E_2
\xra{\gamma_{1,2}}
E_1^{\st \Cyclic_2}
~,
\qquad
E_3
\xra{\gamma_{1,3}}
E_1^{\st \Cyclic_3}
~,
\qquad
E_1^{\tate \sV_4}
\xla{\gamma_{1,4}}
E_4
\xra{\gamma_{2,4}}
E_2^{\st \Cyclic_2}
~,
\qquad
\text{and}
\qquad
E_{12}
\xra{\gamma_{4,12}}
E_4^{\st \Cyclic_3}
~,
\]
and

\item a homotopy making the square
\[ \begin{tikzcd}
E_4
\arrow{r}{\gamma_{2,4}}
\arrow{d}[swap]{\gamma_{1,4}}
&
E_2^{\st \Cyclic_2}
\arrow{d}{\gamma_{1,2}^{\st \Cyclic_2}}
\\
E_1^{\tate \sV_4}
\arrow{r}
&
(E_1^{\st \Cyclic_2})^{\st \Cyclic_2}
\end{tikzcd} \]
commute (in which the lower horizontal morphism is the canonical one).

\end{itemize}
\end{example}

\begin{remark}
We prove \Cref{thm.stratn.of.genuine.G.objects} at the end of this section, after some preliminary work. These preliminaries decompose into three main logically distinct pieces, which are organized into environments as follows.
\begin{enumerate}

\item[\ref{local.define.the.special.family.for.gluing.functors}-\ref{lemma.inclusion.from.CHK.to.GmodKupperH.becomes.iso}:] The proof of \Cref{thm.stratn.of.genuine.G.objects} will begin by applying a projection formula to replace the family $\ms{F}_{^{\not\geq} K} \in \Down_{\pos_G}$ with a certain family $\ms{F} \in \Down_{\pos_{\Weyl(H)}}$. The relevance of the set $C(H,K)$ is that it is clearly a subset of $(G/K)^H$, and it turns out to consist of precisely those orbits that aren't annihilated by the functor $\wEF \wedge (-)_+$.

\item[\ref{observation.double.coset.space.of.CHK}-\ref{lemma.map.from.quotient.of.product.to.CHKg.is.an.iso}:] We decompose $C(H,K)$ into $(\Weyl(H),\Weyl(K))$-orbits by identifying $\Weyl(H) \backslash C(H,K) / \Weyl(K)$ with conjugacy classes of stabilizer subgroups in $\Weyl(H)$ (with respect to the left action). We then describe these orbits explicitly, in terms of elements $g \in G$.

\item[\ref{obs.function.from.NH.mod.L.to.NH.times.NL.mod.L}-\ref{lemma.arbitrary.subgroup.between.H.and.NH}:] We exhibit each $(\Weyl(H),\Weyl(K))$-orbit in $C(H,K)$ as being itself induced from an orbit: namely, it is induced up from $\Weyl_{\Normzer(H)}(gKg^{-1}) = (\Normzer(H) \cap \Normzer(gKg^{-1})) / (gKg^{-1})$.

\end{enumerate}
\end{remark}

\begin{local}
\label{local.define.the.special.family.for.gluing.functors}
We use the following notation through the remainder of this section.
\begin{enumerate}

\item We fix subgroups $H \subseteq K \subseteq G$.\footnote{It is clear that in proving \Cref{thm.stratn.of.genuine.G.objects} we may assume without real loss of generality that $H$ is contained in $K$, and doing so leads to some notational simplification.}

\item We write $\Normzer(H) \xra{p} \Weyl(H)$ for the quotient homomorphism.

\item\label{item.define.the.special.family}

We write
\[
\ms{F}
:=
\lim \left(
\begin{tikzcd}
&
&
\ms{F}_{^{\not\geq} K}
\\
\pos_{\Weyl(H)}
\arrow{r}[swap]{p^{-1}}
&
\pos_{\Normzer(H)}
\arrow{r}
&
\pos_G
\arrow[hookleftarrow]{u}
\end{tikzcd}
\right)
\cong
\left\{
L \subseteq \Weyl(H)
:
{\def\arraystretch{1}
\begin{array}{c}
K \textup{ is not subconjugate}
\\
\textup{to } p^{-1}(L) \textup{ in } G
\end{array}
}
\right\}
\in
\Down_{\pos_{\Weyl(H)}}
\]
for the indicated family of subgroups of $\Weyl(H)$.

\end{enumerate}
\end{local}

\begin{warning}
The functor $\pos_{\Normzer(H)} \ra \pos_G$ is not generally fully faithful.\footnote{For example, take $G = \sS_4$ to be the symmetric group on four letters, define elements $h,j \in G$ by $h = (12)(34)$ and $j = (13)(24)$ (using cycle notation), and define subgroups $H = \brax{h}$ and $J = \brax{j}$ of $G$. Then we have $J \subseteq \Normzer(H)$, and moreover $J \not\cong H$ in $\pos_{\Normzer(H)}$ but $J \cong H$ in $\pos_G$.} (Whenever it is fully faithful, the set $\Weyl(H) \backslash C(H,K) / \Weyl(K)$ is a singleton; see \Cref{lemma.identify.double.coset}.)
\end{warning}

\begin{observation}
\label{obs.inclusion.from.CHK.to.GmodKupperH}
The defining inclusion $\tilde{C}(H,K) \hookra G$ extends to a diagram
\[ \begin{tikzcd}
\tilde{C}(H,K)
\arrow[hook]{r}
\arrow[two heads]{d}
&
G
\arrow[two heads]{d}
\\
C(H,K)
\arrow[hook]{r}
\arrow[hook, dashed]{rd}
&
G/K
\\
&
(G/K)^H
\arrow[hook]{u}
\end{tikzcd} \]
of $(\Normzer(H),\Normzer(K))$-bimodule sets: the dashed factorization arises from the defining fact that if $gK \in C(H,K)$ then $H \subseteq gKg^{-1}$. In particular, we obtain a canonical inclusion
\begin{equation}
\label{inclusion.from.CHK.to.G.mod.K.Hfixed}
C(H,K)
\longhookra
(G/K)^H
\end{equation}
of $(\Weyl(H),\Weyl(K))$-sets. We regard this as a morphism of homotopy $\Weyl(K)$-objects in genuine $\Weyl(H)$-spaces.
\end{observation}

\begin{lemma}
\label{lemma.inclusion.from.CHK.to.GmodKupperH.becomes.iso}
The morphism
\[
\wEF
\wedge
\Cref{inclusion.from.CHK.to.G.mod.K.Hfixed}_+
:=
\wEF
\wedge
\left( C(H,K) \longhookra (G/K)^H \right)_+
\]
in $(\Spaces^{\gen \Weyl(H)}_*)^{\htpy \Weyl(K)}$ is an equivalence.
\end{lemma}

\begin{proof}
We check that the composite functor
\[
\Orb_{\Weyl(H)}
\longhookra
\Spaces^{\gen \Weyl(H)}
\xra{\wEF \wedge (-)_+}
\Spaces^{\gen \Weyl(H)}_*
\]
annihilates (i.e.\! carries to the zero object) all $\Weyl(H)$-orbits in $(G/K)^H$ that do not lie in the image of the inclusion \Cref{inclusion.from.CHK.to.G.mod.K.Hfixed}.

Consider an arbitrary element $gK \in G/K$. Observe that its stabilizer under the $G$-action is $gKg^{-1}$. Moreover, it is $H$-fixed if and only if $H \subseteq gKg^{-1}$. In this case, the stabilizer of the element $gK \in (G/K)^H$ under the $\Weyl(H)$-action is
\[
(gKg^{-1} \cap \Normzer(H))/H
~.
\]
So, the orbit $\Weyl(H) \cdot gK \subseteq (G/K)^H$ is annihilated by the functor $\wEF \wedge(-)_+$ if (and only if) we have
\[
(gKg^{-1} \cap \Normzer(H))/H \in \ms{F} \in \Down_{\pos_{\Weyl(H)}}
~,
\]
which by definition is equivalent to the assertion that
\[
(gKg^{-1} \cap \Normzer(H))
\in
\ms{F}_{^{\not\geq} K}
\in
\Down_{\pos_G}
~,
\]
i.e.\! that $K$ is not subconjugate to $(gKg^{-1} \cap \Normzer(H))$ in $G$. So, it suffices to show that $gKg^{-1} \not\subseteq \Normzer(H)$ implies that $K \not\leq (gKg^{-1} \cap \Normzer(H))$. We show the contrapositive, namely that $K \leq (gKg^{-1} \cap \Normzer(H))$ implies that $gKg^{-1} \subseteq \Normzer(H)$.

Suppose there exists some $x \in G$ such that $xKx^{-1} \subseteq (gKg^{-1} \cap \Normzer(H))$. On the one hand, this implies that $xKx^{-1} \subseteq gKg^{-1}$, which because $K$ is finite implies that $xKx^{-1} = gKg^{-1}$. On the other hand, this implies that $xKx^{-1} \subseteq \Normzer(H)$. Combining these two implications yields the desired conclusion that $gKg^{-1} = xKx^{-1} \subseteq \Normzer(H)$.
\end{proof}

\begin{observation}
\label{observation.double.coset.space.of.CHK}
Carrying each element of $C(H,K)$ to its stabilizer under the left $\Weyl(H)$-action determines a function
\begin{equation}
\label{fxn.on.und.sets.from.CHK.to.PWH}
\begin{tikzcd}[row sep=0cm]
C(H,K)
\arrow{r}
&
\pos_{\Weyl(H)}
\\
\rotatebox{90}{$\in$}
&
\rotatebox{90}{$\in$}
\\
gK
\arrow[maps to]{r}
&
(gKg^{-1})/H
\end{tikzcd}
\end{equation}
on underlying sets. We note the following two properties of the function \Cref{fxn.on.und.sets.from.CHK.to.PWH}.
\begin{enumerate}

\item It takes values in those conjugacy classes of subgroups of $\Weyl(H)$ whose preimages in $\Normzer(H)$ are conjugate in $G$ to $K$.

\item It is invariant with respect to both the left $\Weyl(H)$-action and the right $\Weyl(K)$-action on $C(H,K)$.

\end{enumerate}
The function \Cref{fxn.on.und.sets.from.CHK.to.PWH} therefore admits a factorization
\begin{equation}
\label{factorizn.from.CHK.through.double.coset}
\begin{tikzcd}
C(H,K)
\arrow{r}
\arrow[two heads]{d}
&
\pos_{\Weyl(H)}
\\
\Weyl(H) \backslash C(H,K) / \Weyl(K)
\arrow[dashed]{r}
&
\pos_{\Weyl(H)} \times_{\pos_G} \{ K \}
\arrow[hook]{u}
\end{tikzcd}~.
\end{equation}
We use these facts without further comment.
\end{observation}

\begin{lemma}
\label{lemma.identify.double.coset}
The factorization \Cref{factorizn.from.CHK.through.double.coset} is an isomorphism
\[
\Weyl(H) \backslash C(H,K) / \Weyl(K)
\xlongra{\cong}
\pos_{\Weyl(H)} \times_{\pos_G} \{ K \}
\]
among sets.
\end{lemma}

\begin{proof}
We first verify that the function \Cref{factorizn.from.CHK.through.double.coset} is surjective. For this, choose any subgroup $J \subseteq \Weyl(H)$ such that $g K g^{-1} = p^{-1}(J)$ for some $g \in G$. Then, the function \Cref{factorizn.from.CHK.through.double.coset} carries the equivalence class of the element $gK \in C(H,K)$ to the equivalence class of $J$ in $\pos_{\Weyl(H)} \times_{\pos_G} \{K\}$. So indeed, the function \Cref{factorizn.from.CHK.through.double.coset} is surjective.

We now verify that the function \Cref{factorizn.from.CHK.through.double.coset} is injective. For this, consider a pair of elements $g_1,g_2 \in \tilde{C}(H,K)$. Then, observe the equivalence of the conditions
\begin{enumerate}

\item\label{first.condtn.in.identifying.double.coset.space}
 $(g_1 K g_1^{-1})/H$ and $(g_2 K g_2^{-1})/H$ are conjugate in $\Weyl(H)$,

\item $g_1 K g_1^{-1}$ and $g_2 K g_2^{-1}$ are conjugate in $\Normzer(H)$,

\item there exists some $x \in \Normzer(H)$ such that $xg_1 K g_1^{-1}x^{-1} = g_2 K g_2^{-1}$,

\item there exists some $x \in \Normzer(H)$ such that $g_2^{-1} xg_1 \in \Normzer(K)$,

\item there exist some $x \in \Normzer(H)$ and $y \in \Normzer(K)$ such that $g_2 = x g_1 y$,

\end{enumerate}
and
\begin{enumerate}

\setcounter{enumi}{5}

\item\label{last.condtn.in.identifying.double.coset.space}
 $[g_1] = [g_2]$ in $\Weyl(H) \backslash C(H,K) / \Weyl(K)$.

\end{enumerate}
The equivalence $\Cref{first.condtn.in.identifying.double.coset.space} \Leftrightarrow \Cref{last.condtn.in.identifying.double.coset.space}$ implies that the function \Cref{factorizn.from.CHK.through.double.coset} is injective.
\end{proof}

\begin{observation}
\label{obs.fiber.over.elt.in.double.quotient.is.a.subbimodule}
For each element $[g] \in \Weyl(H) \backslash C(H,K) / \Weyl(K)$, the fiber
\[ \begin{tikzcd}
C(H,K)_{[g]}
\arrow[hook]{r}
\arrow{d}
&
C(H,K)
\arrow{d}
\\
\{ [g] \}
\arrow[hook]{r}
&
\Weyl(H) \backslash C(H,K) / \Weyl(K)
\end{tikzcd} \]
is a $(\Weyl(H),\Weyl(K))$-subbimodule of $C(H,K)$.
\end{observation}

\begin{observation}
Fix an element $g \in \tilde{C}(H,K)$. We have a commutative diagram
\begin{equation}
\label{functions.from.product.of.subgroups.to.CtildeHK.and.CHK}
\begin{tikzcd}[column sep=2cm]
G
\arrow{r}{z \longmapsto zg}
&
G
\\
\Normzer(H) \cdot \Normzer(gKg^{-1})
\arrow[hook]{u}
\arrow[dashed, hook]{r}{xy \longmapsto xyg}
\arrow[two heads]{d}
&
\tilde{C}(H,K)
\arrow[hook]{u}
\arrow[two heads]{d}
\\
(\Normzer(H) \cdot \Normzer(gKg^{-1}))/(gKg^{-1})
\arrow[dashed, hook]{r}{xy(gKg^{-1}) \longmapsto xygK}
\arrow[dashed]{d}
&
C(H,K)
\arrow[two heads]{d}
\\
\{ [g] \}
\arrow[hook]{r}
&
\Weyl(H) \backslash C(H,K) / \Weyl(K)
\end{tikzcd}
\end{equation}
of sets, in which
\begin{itemize}

\item all functions are left $\Normzer(H)$-equivariant,

\item the left column is right $\Normzer(gKg^{-1})$-equivariant,

\item the right column is right $\Normzer(K)$-equivariant, and

\item the horizontal functions are globally right-equivariant with respect to the conjugation isomorphism $\Normzer(gKg^{-1}) \xra{g^{-1}(-)g} \Normzer(K)$.

\end{itemize}
Hence, by considering the sets in the left column of diagram \Cref{functions.from.product.of.subgroups.to.CtildeHK.and.CHK} as right $\Normzer(K)$-sets via the conjugation isomorphism $\Normzer(K) \xra{g(-)g^{-1}} \Normzer(gKg^{-1})$, we may consider the entire diagram \Cref{functions.from.product.of.subgroups.to.CtildeHK.and.CHK} as one of $(\Normzer(H),\Normzer(K))$-bimodule sets. In particular, we obtain an injective function
\begin{equation}
\label{function.from.product.mod.gKginverse.to.CHKg}
(\Normzer(H) \cdot \Normzer(gKg^{-1})) / (gKg^{-1})
\longhookra
C(H,K)_{[g]}
\end{equation}
of $(\Normzer(H),\Normzer(K))$-bimodule sets, which using \Cref{obs.fiber.over.elt.in.double.quotient.is.a.subbimodule} we consider as an injective function of $(\Weyl(H),\Weyl(K))$-bimodule sets.
\end{observation}

\begin{lemma}
\label{lemma.map.from.quotient.of.product.to.CHKg.is.an.iso}
For every $g \in \tilde{C}(H,K)$, the function \Cref{function.from.product.mod.gKginverse.to.CHKg} is an isomorphism among $(\Normzer(H),\Normzer(K))$-bimodule sets.
\end{lemma}

\begin{proof}
It remains to show that the function \Cref{function.from.product.mod.gKginverse.to.CHKg} is surjective. For this, we fix an element $zK \in C(H,K)$, and study the condition
\begin{enumerate}

\item\label{first.condtn.for.being.in.CHKg}

$zK \in C(H,K)_{[g]}$.

\end{enumerate}
By definition, condition \Cref{first.condtn.for.being.in.CHKg} is equivalent to the condition
\begin{enumerate}

\setcounter{enumi}{1}

\item\label{second.condtn.for.being.in.CHKg}

$(zKz^{-1})/H$ is conjugate to $(gKg^{-1})/H$ in $\Weyl(H)$,

\end{enumerate}
which is equivalent to the condition
\begin{enumerate}

\setcounter{enumi}{2}

\item\label{third.condtn.for.being.in.CHKg}

$zKz^{-1}$ is conjugate to $gKg^{-1}$ in $\Normzer(H)$,

\end{enumerate}
which is equivalent to the condition
\begin{enumerate}

\setcounter{enumi}{3}

\item\label{fourth.condtn.for.being.in.CHKg}

there exists some $x \in \Normzer(H)$ such that $xzKz^{-1}x^{-1} = gKg^{-1}$.

\end{enumerate}
Because $xzKz^{-1}x^{-1} = (xzg^{-1})(gKg^{-1})(xzg^{-1})^{-1}$, condition \Cref{fourth.condtn.for.being.in.CHKg} is equivalent to the condition
\begin{enumerate}

\setcounter{enumi}{4}

\item\label{fifth.condtn.for.being.in.CHKg}

there exists some $x \in \Normzer(H)$ such that $xzg^{-1} \in \Normzer(gKg^{-1})$.

\end{enumerate}
Now, we claim that condition \Cref{fifth.condtn.for.being.in.CHKg} is equivalent to the condition
\begin{enumerate}

\setcounter{enumi}{5}

\item\label{sixth.condtn.for.being.in.CHKg}

$zg^{-1} \in \Normzer(H) \cdot \Normzer(gKg^{-1})$.

\end{enumerate}
On the one hand, we have that $\Cref{fifth.condtn.for.being.in.CHKg} \Rightarrow \Cref{sixth.condtn.for.being.in.CHKg}$ via left multiplication by $x^{-1} \in \Normzer(H)$. On the other hand, if $zg^{-1} \in \Normzer(H) \cdot \Normzer(gKg^{-1})$, then there exist elements $x^{-1} \in \Normzer(H)$ and $y \in \Normzer(gKg^{-1})$ such that $zg^{-1} = x^{-1} y$, which implies that $xzg^{-1} = y \in \Normzer(gKg^{-1})$. So indeed, $\Cref{sixth.condtn.for.being.in.CHKg} \Rightarrow \Cref{fifth.condtn.for.being.in.CHKg}$. Finally, condition \Cref{sixth.condtn.for.being.in.CHKg} is clearly equivalent to the condition
\begin{enumerate}

\setcounter{enumi}{6}

\item\label{seventh.condtn.for.being.in.CHKg}

$z \in \Normzer(H) \cdot \Normzer(gKg^{-1}) \cdot g \subseteq \tilde{C}(H,K)$.

\end{enumerate}
So, via the composite equivalence $\Cref{first.condtn.for.being.in.CHKg} \Leftrightarrow \Cref{seventh.condtn.for.being.in.CHKg}$ and by inspection of diagram \Cref{functions.from.product.of.subgroups.to.CtildeHK.and.CHK}, we find that every element $zK \in C(H,K)_{[g]}$ is indeed in the image of the function \Cref{function.from.product.mod.gKginverse.to.CHKg}.
\end{proof}

\begin{observation}
\label{obs.function.from.NH.mod.L.to.NH.times.NL.mod.L}
Fix a subgroup $L \subseteq G$ such that $H \subseteq L \subseteq \Normzer(H)$. We have a commutative diagram
\begin{equation}
\label{diagram.getting.map.to.be.induced.for.arbitrary.subgroup.L.between.H.and.NL}
\begin{tikzcd}[column sep=2.5cm]
\Normzer(H)
\arrow[hook]{r}{x \longmapsto x \cdot e = x}
\arrow[two heads]{d}
&
\Normzer(H) \cdot \Normzer(L)
\arrow[two heads]{d}
\\
\Normzer(H)/L
\arrow[dashed, hook]{r}[swap]{x L \longmapsto (x \cdot e ) L = xL}
&
(\Normzer(H) \cdot \Normzer(L)) / L
\end{tikzcd}
\end{equation}
of $(\Normzer(H),\Normzer(H) \cap \Normzer(L))$-bimodule sets. Furthermore, the left $H$-actions on both sets in the lower row of diagram \Cref{diagram.getting.map.to.be.induced.for.arbitrary.subgroup.L.between.H.and.NL} are trivial. So, we may consider the lower morphism of diagram \Cref{diagram.getting.map.to.be.induced.for.arbitrary.subgroup.L.between.H.and.NL} as a morphism
\begin{equation}
\label{the.map.to.be.induced.for.arbitrary.subgroup.L.between.H.and.NL.before.getting.WL.action}
\Normzer(H) / L
\longra
(\Normzer(H) \cdot \Normzer(L)) / L
\end{equation}
of $(\Weyl(H),(\Normzer(H) \cap \Normzer(L))/L)$-bimodule sets. Moreover, the $(\Weyl(H),(\Normzer(H) \cap \Normzer(L))/L)$-bimodule structure on $(\Normzer(H) \cdot \Normzer(L)) / L$ extends to a $(\Weyl(H),\Weyl(L))$-bimodule structure (via the inclusion $(\Normzer(H) \cap \Normzer(L))/L \subseteq \Normzer(L)/L =: \Weyl(L)$). Hence, via the adjunction
\[ \begin{tikzcd}[column sep=2.5cm]
\BiMod_{(\Weyl(H),(\Normzer(H) \cap \Normzer(L))/L)}
\arrow[transform canvas={yshift=0.9ex}]{r}{\Ind_{(\Normzer(H) \cap \Normzer(L))/L}^{\Weyl(L)}}
\arrow[leftarrow, transform canvas={yshift=-0.9ex}]{r}[yshift=-0.2ex]{\bot}[swap]{\Res_{(\Normzer(H) \cap \Normzer(L))/L}^{\Weyl(L)}}
&
\BiMod_{(\Weyl(H),\Weyl(L))}
\end{tikzcd}~,
\]
the morphism \Cref{the.map.to.be.induced.for.arbitrary.subgroup.L.between.H.and.NL.before.getting.WL.action} upgrades to a morphism
\[
\Normzer(H) / L
\longra
\Res_{(\Normzer(H) \cap \Normzer(L))/L}^{\Weyl(L)} ( (\Normzer(H) \cdot \Normzer(L)) / L )
\]
in $\BiMod_{(\Weyl(H),(\Normzer(H) \cap \Normzer(L))/L)}$, which corresponds to a morphism
\begin{equation}
\label{map.from.induction.to.NH.int.NL.mod.L}
\Ind_{(\Normzer(H) \cap \Normzer(L))/L}^{\Weyl(L)} ( \Normzer(H) / L )
\longra
(\Normzer(H) \cdot \Normzer(L)) / L
\end{equation}
in $\BiMod_{(\Weyl(H),\Weyl(L))}$.
\end{observation}

\begin{lemma}
\label{lemma.arbitrary.subgroup.between.H.and.NH}
For any subgroup $L \subseteq G$ such that $H \subseteq L \subseteq \Normzer(H)$, the morphism \Cref{map.from.induction.to.NH.int.NL.mod.L} is an isomorphism.
\end{lemma}

\begin{proof}
As the forgetful functor $\BiMod_{(\Weyl(H),\Weyl(L))} \ra \RMod_{\Weyl(L)}$ is conservative, it suffices to show that the morphism \Cref{map.from.induction.to.NH.int.NL.mod.L} becomes an isomorphism in $\RMod_{\Weyl(L)}$. Now, note that we have a commutative square
\[ \begin{tikzcd}[column sep=2.5cm, row sep=1.5cm]
\RMod_{\Normzer(H) \cap \Normzer(L)}
\arrow{r}{\Ind_{\Normzer(H) \cap \Normzer(L)}^{\Normzer(L)}}
\arrow{d}[swap]{(-)/L}
&
\RMod_{\Normzer(L)}
\arrow{d}{(-)/L}
\\
\RMod_{(\Normzer(H) \cap \Normzer(L))/L}
\arrow{r}[swap]{\Ind_{(\Normzer(H) \cap \Normzer(L))/L}^{\Weyl(L)}}
&
\RMod_{\Weyl(L)}
\end{tikzcd} \]
(which commutes because it clearly commutes upon passing to right adjoints). Therefore, it suffices to observe the isomorphism
\[
\Ind_{\Normzer(H) \cap \Normzer(L)}^{\Normzer(L)}
\Normzer(H)
\cong
( \Normzer(H) \times \Normzer(L) ) / (\Normzer(H) \cap \Normzer(L))
\cong
\Normzer(H) \cdot \Normzer(L)
\]
in $\RMod_{\Normzer(L)}$.
\end{proof}

\begin{observation}
\label{obs.rho.g.H.followed.by.coInd.from.NH.to.G.is.RgG.linear}
Consider the composite adjunction
\begin{equation}
\label{composite.adjunction.from.RgG.to.RgWH}
\begin{tikzcd}[column sep=2cm, row sep=0cm]
\cR^{\gen G}
\arrow[transform canvas={yshift=0.9ex}]{r}{\Res^G_{\Normzer (H)}}
\arrow[leftarrow, transform canvas={yshift=-0.9ex}]{r}[yshift=-0.2ex]{\bot}[swap]{\coInd^G_{\Normzer (H)}}
&
\cR^{\gen \Normzer (H)}
\arrow[transform canvas={yshift=0.9ex}]{r}{\Phi^H_\gen}
\arrow[hookleftarrow, transform canvas={yshift=-0.9ex}]{r}[yshift=-0.2ex]{\bot}[swap]{\rho^H_\gen}
&
\cR^{\gen \Weyl (H)}
\end{tikzcd}
\end{equation}
(compare with Observation S.\ref{strat:obs.geom.fixedpoints.by.any.other.name}). Because $\cR^{\gen G}$ is rigidly-compactly generated by \Cref{obs.gen.G.objects.rigid.if.R.is} and the composite left adjoint $\Phi^H_\gen \circ \Res^G_{\Normzer(H)}$ is symmetric monoidal, by \cite[Chapter 1, Lemma 9.3.6]{GR} the composite right adjoint $\coInd^G_{\Normzer(H)} \circ \rho^H_\gen$ is $\cR^{\gen G}$-linear: in other words, for any $E \in \cR^{\gen G}$ and $F \in \cR^{\gen \Weyl(H)}$ we have the projection formula
\[
\coInd^G_{\Normzer(H)} \left( \rho^H_\gen \left( \Phi^H_\gen \left( \Res^G_{\Normzer(H)} ( E ) \right) \otimes F \right) \right)
\simeq
E \otimes \coInd^G_{\Normzer(H)} \left( \rho^H_\gen ( F) \right)
~.
\]
\end{observation}

\begin{proof}[Proof of \Cref{thm.stratn.of.genuine.G.objects}]
It is clear that we have an identification
\[
\Phi^K
\simeq
\ulhom_{\cR^{\gen G}} \left( \uno_\cR \brax{G/K} , (-) \otimes \uno_\cR \brax{\wEF_{^{\not\geq} K}} \right)
\]
in $\Fun ( \cR^{\gen G} , \cR^{\htpy \Weyl(K)})$ (see Definition S.\ref{strat:defn.geom.H.fps}). It then follows from Observation S.\ref{strat:obs.gluing.functors.for.SpgG} that we may identify the gluing functor $\Gamma^H_K$ as the composite
\[
\cR^{\htpy \Weyl(H)}
\xhookra{\beta_{\Weyl(H)}}
\cR^{\gen \Weyl(H)}
\xhookra{\rho_\gen^H}
\cR^{\gen \Normzer(H)}
\xra{\coInd_{\Normzer(H)}^G}
\cR^{\gen G}
\xra{ \ulhom_{\cR^{\gen G}} \left( \uno_\cR \brax{G/K} , (-) \otimes \uno_\cR \brax{\wEF_{^{\not\geq} K}} \right) }
\cR^{\htpy \Weyl(K)}
~.
\]
In other words, for any $E \in \cR^{\htpy \Weyl(H)}$ we have a natural equivalence
\begin{equation}
\label{gluing.functor.equivce.one}
\Gamma^H_K(E)
\simeq
\ulhom_{\cR^{\gen G}} \left( \uno_\cR \brax{G/K} , \coInd_{\Normzer(H)}^G ( \rho_\gen^H ( \beta_{\Weyl(H)} ( E ) ) ) \otimes \uno_\cR \brax{ \wEF_{^{\not\geq} K } } \right)
\end{equation}
in $\cR^{\htpy \Weyl(K)}$. By \Cref{obs.rho.g.H.followed.by.coInd.from.NH.to.G.is.RgG.linear}, we have an equivalence
\[
\coInd_{\Normzer(H)}^G ( \rho_\gen^H ( \beta_{\Weyl(H)} ( E ) ) ) \otimes \uno_\cR \brax{ \wEF_{^{\not\geq} K } }
\simeq
\coInd_{\Normzer(H)}^G \left( \rho_\gen^H \left( \beta_{\Weyl(H)} ( E ) \otimes \Phi^H_\gen \left( \Res_{\Normzer(H)}^G \left( \uno_\cR \brax{ \wEF_{^{\not\geq} K } } \right) \right) \right) \right)
\]
in $\cR^{\gen G}$. Therefore, by the composite adjunction \Cref{composite.adjunction.from.RgG.to.RgWH} we find that
\begin{align}
\nonumber
\Cref{gluing.functor.equivce.one}
& \simeq
\ulhom_{\cR^{\gen \Weyl(H)}} \left( \Phi^H_\gen \left( \Res^G_{\Normzer(H)} \left( \uno_\cR \brax{G/K} \right) \right) , \beta_{\Weyl(H)} ( E ) \otimes \Phi^H_\gen \left( \Res_{\Normzer(H)}^G \left( \uno_\cR \brax{ \wEF_{^{\not\geq} K } } \right) \right) \right)
\\
\label{gluing.functor.equivce.two}
& \simeq
\ulhom_{\cR^{\gen \Weyl(H)}} \left( \uno_\cR \brax{ (G/K)^H  } , \beta_{\Weyl(H)} ( E ) \otimes \uno_\cR \brax{ \wEF } \right)
\end{align}
(recall \Cref{local.define.the.special.family.for.gluing.functors}\Cref{item.define.the.special.family}). Thereafter, we find that 
\begin{align}
\label{use.idempotence.of.wEF.first.time}
\Cref{gluing.functor.equivce.two}
& \simeq
\ulhom_{\cR^{\gen \Weyl(H)}} \left( \uno_\cR \brax{\wEF} \otimes \uno_\cR \brax{ (G/K)^H  } , \beta_{\Weyl(H)} ( E ) \otimes \uno_\cR \brax{ \wEF } \right)
\\
\label{use.lemma.inclusion.from.CHK.to.GmodKupperH.becomes.iso}
& \simeq
\ulhom_{\cR^{\gen \Weyl(H)}} \left( \uno_\cR \brax{\wEF} \otimes \uno_\cR \brax{ C(H,K)  } , \beta_{\Weyl(H)} ( E ) \otimes \uno_\cR \brax{ \wEF } \right)
\\
\label{gluing.functor.equivce.three}
& \simeq
\ulhom_{\cR^{\gen \Weyl(H)}} \left( \uno_\cR \brax{ C(H,K)  } , \beta_{\Weyl(H)} ( E ) \otimes \uno_\cR \brax{ \wEF } \right)
~,
\end{align}
where equivalences \Cref{use.idempotence.of.wEF.first.time} and \Cref{gluing.functor.equivce.three} follow from the fact that $\wEF \in \Spaces^{\gen \Weyl(H)}_*$ is idempotent and equivalence \Cref{use.lemma.inclusion.from.CHK.to.GmodKupperH.becomes.iso} follows from \Cref{lemma.inclusion.from.CHK.to.GmodKupperH.becomes.iso}. Then, using \Cref{lemma.map.from.quotient.of.product.to.CHKg.is.an.iso} we obtain an equivalence
\begin{align}
\nonumber
\Cref{gluing.functor.equivce.three}
& \simeq
\ulhom_{\cR^{\gen \Weyl(H)}} \left( \uno_\cR \brax{ \coprod_{[g] \in \Weyl(H) \backslash C(H,K) / \Weyl(K)} C(H,K)_{[g]}  } , \beta_{\Weyl(H)} ( E ) \otimes \uno_\cR \brax{ \wEF } \right)
\\
\label{gluing.functor.equivce.four}
& \simeq
\bigoplus_{[g] \in \Weyl(H) \backslash C(H,K) / \Weyl(K)} \ulhom_{\cR^{\gen \Weyl(H)}} \left( \uno_\cR \brax{ C(H,K)_{[g]}  } , \beta_{\Weyl(H)} ( E ) \otimes \uno_\cR \brax{ \wEF } \right)
~.
\end{align}
To simplify our notation, we fix an element $[g] \in \Weyl(H) \backslash C(H,K) / \Weyl(K)$ and study the corresponding summand
\begin{equation}
\label{gluing.functor.equivce.five}
\ulhom_{\cR^{\gen \Weyl(H)}} \left( \uno_\cR \brax{ C(H,K)_{[g]}  } , \beta_{\Weyl(H)} ( E ) \otimes \uno_\cR \brax{ \wEF } \right)
\end{equation}
of \Cref{gluing.functor.equivce.four}. Then, we obtain equivalences
\begin{align}
\label{use.lemma.map.from.quotient.of.product.to.CHKg.is.an.iso}
\Cref{gluing.functor.equivce.five}
& \simeq
\ulhom_{\cR^{\gen \Weyl(H)}} \left( \uno_\cR \brax{ ( \Normzer(H) \cdot \Normzer( gKg^{-1}) ) / (gKg^{-1} ) } , \beta_{\Weyl(H)} ( E ) \otimes \uno_\cR \brax{ \wEF } \right)
\\
\label{use.lemma.arbitrary.subgroup.between.H.and.NH}
& \simeq
\ulhom_{\cR^{\gen \Weyl(H)}} \left( \uno_\cR \brax{ \Ind_{(\Normzer(H) \cap \Normzer(gKg^{-1}))/(gKg^{-1})}^{\Weyl(gKg^{-1})} ( \Normzer(H) / (gKg^{-1}) ) } , \beta_{\Weyl(H)} ( E ) \otimes \uno_\cR \brax{ \wEF } \right)
\\
\nonumber
& \simeq
\coInd_{(\Normzer(H) \cap \Normzer(gKg^{-1}))/(gKg^{-1})}^{\Weyl(gKg^{-1})}
\left(
\ulhom_{\cR^{\gen \Weyl(H)}} \left( \uno_\cR \brax{ \Normzer(H) / (gKg^{-1}) } , \beta_{\Weyl(H)} ( E ) \otimes \uno_\cR \brax{ \wEF } \right)
\right)
\\
\label{gluing.functor.equivce.six}
& \simeq
\Ind_{(\Normzer(H) \cap \Normzer(gKg^{-1}))/(gKg^{-1})}^{\Weyl(gKg^{-1})}
\left(
\ulhom_{\cR^{\gen \Weyl(H)}} \left( \uno_\cR \brax{ \Normzer(H) / (gKg^{-1}) } , \beta_{\Weyl(H)} ( E ) \otimes \uno_\cR \brax{ \wEF } \right)
\right)
~,
\end{align}
where equivalence \Cref{use.lemma.map.from.quotient.of.product.to.CHKg.is.an.iso} follows from \Cref{lemma.map.from.quotient.of.product.to.CHKg.is.an.iso} and equivalence \Cref{use.lemma.arbitrary.subgroup.between.H.and.NH} follows from \Cref{lemma.arbitrary.subgroup.between.H.and.NH}. To simplify our notation, we study the object
\begin{equation}
\label{gluing.functor.equivce.seven}
\ulhom_{\cR^{\gen \Weyl(H)}} \left( \uno_\cR \brax{ \Normzer(H) / (gKg^{-1} ) } , \beta_{\Weyl(H)} ( E ) \otimes \uno_\cR \brax{ \wEF } \right)
\end{equation}
of $\cR^{\htpy ( \Normzer(H) \cap \Normzer(gKg^{-1})) / (gKg^{-1})}$ (whose induction to $\Weyl(gKg^{-1})$ is \Cref{gluing.functor.equivce.six}). Specifically, we compute that
\begin{align}
\nonumber
\Cref{gluing.functor.equivce.seven}
& \simeq
\ulhom_{\cR^{\gen \Weyl(H)}} \left( \uno_\cR \brax{ \Weyl(H) /  ( (gKg^{-1}) / H ) } , \beta_{\Weyl(H)} ( E ) \otimes \uno_\cR \brax{ \wEF } \right)
\\
\label{gluing.functor.equivce.eight}
& =:
\Phi^{(gKg^{-1})/H}_\ms{F} \beta_{\Weyl(H)} (E)
~.
\end{align}
Finally, observe that $(gKg^{-1}) / H \in \pos_{\Weyl(H)} \backslash \ms{F}$ is a minimal element. Hence, by Observation S.\ref{strat:obs.geom.fixedpoints.by.any.other.name} we have an equivalence
\[
\Phi^{(gKg^{-1})/H}_\ms{F}
\simeq
\Phi^{(gKg^{-1}) / H}_{^{\not\geq} (gKg^{-1})/H}
=:
\Phi^{(gKg^{-1}) / H}
\]
in $\Fun ( \cR^{\gen \Weyl(H)} , \cR^{\htpy ( \Normzer(H) \cap \Normzer(gKg^{-1}) ) / (gKg^{-1}) } )$. Therefore, by \Cref{prop.proper.Tate.from.genuine.G.objs} we have equivalences
\[
\Cref{gluing.functor.equivce.eight}
\simeq
\Phi^{(gKg^{-1} ) / H} \beta_{\Weyl(H)} E
\simeq
E^{\tate (gKg^{-1}) / H}
~,
\]
completing the proof.
\end{proof}

\part{The Picard-graded $\Cyclic_{p^n}$-equivariant cohomology of a point}
\label{part.coh}

\section{The geometric stratification of genuine $\Cyclic_{p^n}$-$\ZZ$-modules}
\label{section.stratn.of.gen.Cpn.Z.mods}

In this section, we apply the material of \Cref{part.smstrat} to describe the symmetric monoidal \textit{geometric stratification} (see \Cref{defn.geometric.stratification.of.genuine.G.objects}) of the presentably symmetric monoidal stable $\infty$-category
\[
\Mod^{\gen \Cyclic_{p^n}}_\ZZ
:=
\Spectra^{\gen \Cyclic_{p^n}} \otimes \Mod_\ZZ
\simeq
\Mack_{\Cyclic_{p^n}}(\Mod_\ZZ)
\]
of genuine $\Cyclic_{p^n}$-$\ZZ$-modules (see \Cref{defn.gen.G.objects}\Cref{item.defn.gen.G.objects} and \Cref{obs.mackey.for.genuine.G.objects}), where $p$ is prime. This description (originally stated as \Cref{intro.thm.gen.Cpn.Z.mods}) is recorded as \Cref{thm.stratn.of.genuine.Cpn.Z.mods}; its consequences are unpacked more explicitly in \Cref{consequences.of.stratn.of.genuine.Cpn.Z.mods} (after its proof). We also describe categorical fixedpoints in these terms as \Cref{prop.fixedpts.res.and.trf.for.general.Cpn.Z.mods}.

\needspace{2\baselineskip}
\begin{notation}
\label{notn.fix.prime.and.n.et.al}
\begin{enumerate}
\item[]

\item\label{part.fix.prime.and.n} We fix a prime $p$ and a nonnegative integer $n \geq 0$; these determine a finite cyclic group $\Cyclic_{p^n}$ of order $p^n$.

\item\label{part.write.brax.n} To ease our notation, we use the identification $\pos_{\Cyclic_{p^n}} \cong [n]$ for the poset of subgroups of $\Cyclic_{p^n}$, and for any $s \in [n]$ we may use the identification $\Cyclic_{p^n} / \Cyclic_{p^s} \cong \Cyclic_{p^{n-s}}$.\footnote{We use the letter ``s'' because it stands for the word ``stratum''.}

\item

Given a genuine $\Cyclic_{p^n}$-$\ZZ$-module $E \in \Mod^{\gen \Cyclic_{p^n}}_\ZZ$, for any $s \in [n]$ we may simply write
\[
E_s
:=
\Phi^{\Cyclic_{p^s}}(E)
\in
\Mod^{\htpy \Cyclic_{p^{n-s}}}_\ZZ
\]
for its $\Cyclic_{p^s}$-geometric fixedpoints.

\item We fix a generator $\sigma \in \Cyclic_{p^n}$. For any $s \in [n]$, we also denote by $\sigma \in \Cyclic_{p^{n-s}}$ its image under the quotient homomorphism $\Cyclic_{p^n} \ra \Cyclic_{p^{n-s}}$.

\item\label{part.write.otimes.for.otimes.Z} In the interest of brevity, we simply write $\otimes := \otimes_\ZZ$ for the symmetric monoidal structure of $\Mod_\ZZ$.

\item\label{part.write.Z.for.trivial.action} For any $i \geq 0$, we simply write $\ZZ \in \Mod^{\htpy \Cyclic_{p^i}}_\ZZ$ for the trivial $\Cyclic_{p^i}$-action on $\ZZ \in \Mod_\ZZ$.

\end{enumerate}
\end{notation}

\begin{theorem}
\label{thm.stratn.of.genuine.Cpn.Z.mods}
The symmetric monoidal geometric stratification
\begin{equation}
\label{s.m.stratn.of.gen.Cpn.Z.mods}
[n]
\xra{\Spectra^{\gen \Cyclic_{p^n}}_{^\leq \bullet}}
\Idl_{\Spectra^{\gen \Cyclic_{p^n}}}
\xra{- \otimes \Mod_\ZZ}
\Idl_{\Mod^{\gen \Cyclic_{p^n}}_\ZZ}
\end{equation}
of $\Mod^{\gen \Cyclic_{p^n}}_\ZZ$ has the following features.
\begin{enumerate}

\item Its underlying stratification is strict.\footnote{That is, the stratification is convergent (as guaranteed by Theorem S.\ref{strat:intro.thm.cosms} because the poset $[n]$ is finite) and moreover the gluing functors compose strictly (as opposed to left-laxly). See \S S.\ref{strat:subsection.strict.stratns} for more discussion of this notion.}

\item Its symmetric monoidal gluing diagram is the functor
\[
[n]
\xra{ \GD^\otimes ( \Mod^{\gen \Cyclic_{p^n}}_\ZZ ) }
\CAlg^\rlax(\Cat)
\]
selecting the diagram
\begin{equation}
\label{s.m.gluing.diagram.of.genuine.Cpn.Z.mods}
\Mod_\ZZ^{\htpy \Cyclic_{p^n}}
\xra{(-)^{{\sf t} \Cyclic_p}}
\Mod_\ZZ^{\htpy \Cyclic_{p^{n-1}}}
\xra{(-)^{{\sf t} \Cyclic_p}}
\cdots
\xra{(-)^{{\sf t} \Cyclic_p}}
\Mod_\ZZ
~.
\end{equation}

\item All nontrivial composite gluing functors in its symmetric monoidal gluing diagram \Cref{s.m.gluing.diagram.of.genuine.Cpn.Z.mods} are zero.

\end{enumerate}
\end{theorem}

\begin{lemma}
\label{lem.projection.formula.categorical.then.geometric.is.geometric}
Let $\cR$ be a rigidly-compactly generated presentably symmetric monoidal stable $\infty$-category, let $G = \Cyclic_{p^n}$, and let $K < H \leq G$ be any strict containment among subgroups of $G$. Then, the composite natural transformation in the diagram
\[ \begin{tikzcd}[row sep=1.5cm, column sep=1.5cm]
\cR^{\gen G}
\arrow[bend left]{r}[pos=0.58]{(-)^K}[swap, yshift=-0.4cm]{\Downarrow}
\arrow[bend right]{r}[swap, pos=0.54]{\Phi^K_\gen}
\arrow[bend right]{rd}[swap, sloped]{\Phi^H_\gen}
&
\cR^{\gen (G/K)}
\arrow{d}{\Phi^{H/K}_\gen}
\\
&
\cR^{\gen (G/H)}
\end{tikzcd} \]
(in which the bottom triangle commutes) is an equivalence in $\Fun(\cR^{\gen G} , \cR^{\gen G/H})$: in other words, for any $E \in \cR^{\gen G}$ the canonical morphism
\begin{equation}
\label{canonical.morphism.in.projection.formula.for.totally.ordered.poset.of.closed.subgroups}
\Phi^{H/K}_\gen \left( E^K \right)
\longra
\Phi^{H/K}_\gen \left( \Phi^K_\gen(E) \right)
\simeq
\Phi_\gen^H(E)
\end{equation}
is an equivalence.
\end{lemma}

\begin{remark}
\label{rmk.projection.formula.for.Cpinfty}
In fact, \Cref{lem.projection.formula.categorical.then.geometric.is.geometric} applies to the rigidly-compactly generated presentably symmetric monoidal stable $\infty$-categories
\[
\Spectra^{\gen \Cyclic_{p^\infty}}
:=
\lim \left(
\cdots
\xra{\Res^{\Cyclic_{p^3}}_{\Cyclic_{p^2}}}
\Spectra^{\gen \Cyclic_{p^2}}
\xra{\Res^{\Cyclic_{p^2}}_{\Cyclic_{p}}}
\Spectra^{\gen \Cyclic_{p}}
\xra{\Res^{\Cyclic_{p}}_{e}}
\Spectra
\right)
\qquad
\textup{and}
\qquad
\Spectra^{\gen^{<_p} \TT}
:=
\brax{ \Sigma^\infty_\TT ( \TT / \Cyclic_{p^i})_+ }_{i \geq 0}
\subseteq
\Spectra^{\gen \TT}
~,
\]
both of which are stratified over the poset $\ZZ_{\geq 0}$ (under the requirement that the subgroups $K < H \leq G$ be closed): the key point is that this poset is totally ordered. Alternatively, these statements can be readily deduced from \Cref{lem.projection.formula.categorical.then.geometric.is.geometric} itself: for $\Spectra^{\gen \Cyclic_{p^\infty}}$ this is immediate, and for $\Spectra^{\gen^{<_p}\TT}$ it suffices to observe that the forgetful functor $\Spectra^{\gen^{<_p}\TT} \ra \Spectra^{\gen \Cyclic_{p^\infty}}$ is conservative.
\end{remark}

\begin{proof}[Proof of \Cref{lem.projection.formula.categorical.then.geometric.is.geometric}]
We begin by observing the commutative diagram
\[ \begin{tikzcd}[column sep=1.5cm, row sep=1.5cm]
\Spaces^{\gen (G/K)}
\arrow{r}{\Res^{G/K}_G}
\arrow{d}[swap]{\uno_\cR \brax{-}}
&
\Spaces^{\gen G}
\arrow{d}{\uno_\cR\brax{-}}
\\
\cR^{\gen (G/K)}
\arrow{r}[swap]{\Res^{G/K}_G}
&
\cR^{\gen G}
\end{tikzcd} \]
in $\CAlg(\PrL)$. We observe too that in the adjunction
\[ \begin{tikzcd}[column sep=1.5cm]
\cR^{\gen (G/K)}
\arrow[transform canvas={yshift=0.9ex}]{r}{\Res^{G/K}_G}
\arrow[leftarrow, transform canvas={yshift=-0.9ex}]{r}[yshift=-0.2ex]{\bot}[swap]{(-)^K}
&
\cR^{\gen G}
\end{tikzcd}~, \]
because $\cR^{\gen (G/K)}$ is rigidly-compactly generated by \Cref{obs.gen.G.objects.rigid.if.R.is} and the left adjoint $\Res^{G/K}_G$ is symmetric monoidal, by \cite[Chapter 1, Lemma 9.3.6]{GR} the right adjoint $(-)^K$ is $\cR^{\gen (G/K)}$-linear: in other words, for any $F \in \cR^{\gen (G/K)}$ and $E \in \cR^{\gen G}$ we have the projection formula
\[
F \otimes E^K
\simeq
\left( \Res^{G/K}_G(F) \otimes E \right)^K
~.
\]
Using these two observations, we respectively identify the source and target of the morphism \Cref{canonical.morphism.in.projection.formula.for.totally.ordered.poset.of.closed.subgroups} as
\begin{align}
\label{equivce.of.source.use.description.of.geom.fps}
\Phi^{H/K}_\gen \left( E^K \right)
& \simeq
\left( \wEFgeomXfps{H/K} \tensoring E^K \right)^{H/K}
\\
\nonumber
& \simeq
\left( \uno_\cR \brax{ \wEFgeomXfps{H/K} } \otimes E^K \right)^{H/K}
\\
\nonumber
& \simeq
\left( \left( \Res^{G/K}_G \left( \uno_\cR \brax{ \wEFgeomXfps{H/K} } \right) \otimes E \right) ^K \right)^{H/K}
\\
\nonumber
& \simeq
\left( \Res^{G/K}_G \left( \uno_\cR \brax{ \wEFgeomXfps{H/K} } \right) \otimes E \right)^H
\\
\nonumber
& \simeq
\left( \Res^{G/K}_G \left( \wEFgeomXfps{H/K} \right) \tensoring E \right)^H
\end{align}
and
\begin{align}
\label{equivce.of.target.use.description.of.geom.fps}
\Phi^{H/K}_\gen \left( \Phi^K_\gen ( E ) \right)
& \simeq
\left( \wEFgeomXfps{H/K} \tensoring \left( \wEFgeomXfps{K} \tensoring E \right)^K \right)^{H/K}
\\
\nonumber
& \simeq
\left( \uno_\cR \brax{ \wEFgeomXfps{H/K} } \otimes \left( \uno_\cR \brax{ \wEFgeomXfps{K} } \otimes E \right)^K \right)^{H/K}
\\
\nonumber
& \simeq
\left( \left( \Res^{G/K}_G \left( \uno_\cR \brax{ \wEFgeomXfps{H/K} } \right) \otimes \uno_\cR \brax{ \wEFgeomXfps{K} } \otimes E \right)^K \right)^{H/K}
\\
\nonumber
& \simeq
\left( \Res^{G/K}_G \left( \uno_\cR \brax{ \wEFgeomXfps{H/K} } \right) \otimes \uno_\cR \brax{ \wEFgeomXfps{K} } \otimes E \right)^H
\\
\nonumber
& \simeq
\left( \left( \Res^{G/K}_G \left( \wEFgeomXfps{H/K} \right) \wedge \wEFgeomXfps{K} \right) \tensoring E \right)^H
~,
\end{align}
where equivalences \Cref{equivce.of.source.use.description.of.geom.fps} \and \Cref{equivce.of.target.use.description.of.geom.fps} follow from \Cref{obs.geom.fps.on.gen.G.objects.from.catl.fps}. It is now clear that the morphism \Cref{canonical.morphism.in.projection.formula.for.totally.ordered.poset.of.closed.subgroups} itself may be obtained by applying the composite functor
\[
\Spaces^{\gen G}_*
\xra{(-) \tensoring E }
\cR^{\gen G}
\xra{(-)^H}
\cR^{\gen (G/H)}
\]
to the evident morphism
\begin{equation}
\label{morphism.in.pointed.gen.G.mod.K.spaces}
\Res^{G/K}_G \left( \wEFgeomXfps{H/K} \right)
\simeq
\Res^{G/K}_G \left( \wEFgeomXfps{H/K} \right) \wedge S^0
\longra
\Res^{G/K}_G \left( \wEFgeomXfps{H/K} \right) \wedge \wEFgeomXfps{K}
\end{equation}
in $\Spaces^{\gen G}_*$, which it therefore suffices to show is an equivalence. For this, fix an arbitrary closed subgroup $J \leq G$. Now, for any genuine $G/K$-space $X \in \Spaces^{\gen (G/K)}_*$ we have that
\[
\Res^{G/K}_G(X)^J
\simeq
X^{J / (K \cap J)}
\]
(note that $J / (K \cap J)$ is the image of the composite $J \hookra G \thra G/K$). Hence, it suffices to observe that if $J/(J \cap K) \geq H/K$ in $\pos_{G/K}$ then $J \geq K$, which follows from the fact that the poset $\pos_G$ is totally ordered.
\end{proof}

\begin{observation}
\label{obs.tate.vanishing.for.Z.mods}
For any $E \in \Mod^{\htpy \Cyclic_{p^2}}_\ZZ$, all three terms in the cofiber sequence
\[ \left(
E_{\htpy \Cyclic_p}
\xra{\Nm_{\Cyclic_p}(E)}
E^{\htpy \Cyclic_p}
\xra{\sQ_{\Cyclic_p}(E)}
E^{\st \Cyclic_p}
\right)^{\st \Cyclic_p}
\]
in $\Mod_\ZZ$ are zero; this follows from \cite[Footnote 9]{NS} (see also \cite[Lemma I.2.7]{NS}). Equivalently, the norm maps
\[
\hspace{-1.2cm}
(E_{\htpy \Cyclic_p})_{\htpy \Cyclic_p} \xra{\Nm_{\Cyclic_p}(E_{\htpy \Cyclic_p})} (E_{\htpy \Cyclic_p})^{\htpy \Cyclic_p}
~,
\qquad
(E^{\htpy \Cyclic_p})_{\htpy \Cyclic_p} \xra{\Nm_{\Cyclic_p}(E^{\htpy \Cyclic_p})} (E^{\htpy \Cyclic_p})^{\htpy \Cyclic_p}
~,
\qquad
\text{and}
\qquad
(E^{\st \Cyclic_p})_{\htpy \Cyclic_p} \xra{\Nm_{\Cyclic_p}(E^{\st \Cyclic_p})} (E^{\st \Cyclic_p})^{\htpy \Cyclic_p}
\]
are equivalences.
\end{observation}

\begin{observation}
\label{obs.if.tate.Cpr.vanishes.then.tate.Cprplusone.vanishes}
For any $E \in \Mod^{\htpy \Cyclic_{p^n}}_\ZZ$ and any $1 \leq s \leq n$, if $E^{\st \Cyclic_{p^s}} \simeq 0$ then $E^{\st \Cyclic_{p^i}} \simeq 0$ for all $s \leq i \leq n$. To see this, by induction it suffices to verify the case that $i=s+1$ (assuming that $s< n$, otherwise the assertion is vacuously true). And indeed, we have that $E^{\st \Cyclic_{p^{s+1}}} \simeq 0$ as a result of the diagram
\[ \begin{tikzcd}[row sep=1.5cm, column sep=0.5cm]
&
E_{\htpy \Cyclic_{p^{s+1}}}
\arrow{rr}{\Nm_{\Cyclic_{p^{s+1}}}(E)}
\arrow{ld}[sloped]{\sim}
&
&
E^{\htpy \Cyclic_{p^{s+1}}}
\\
(E_{\htpy \Cyclic_{p^s}})_{\htpy \Cyclic_p}
\arrow{rrd}[sloped, swap]{\Nm_{\Cyclic_p}(E_{\htpy \Cyclic_{p^s}})}[sloped]{\sim}
&
&
&
&
(E^{\htpy \Cyclic_{p^s}})^{\htpy \Cyclic_p}
\arrow{lu}[sloped]{\sim}
\\
&
&
(E_{\htpy \Cyclic_{p^s}})^{\htpy \Cyclic_p}
\arrow{rru}[sloped, swap]{\Nm_{\Cyclic_{p^s}}(E)^{\htpy \Cyclic_p}}[sloped]{\sim}
\end{tikzcd} \]
in $\Mod^{\htpy \Cyclic_{p^{n-s+1}}}_\ZZ$, which commutes by \Cref{obs.various.properties.of.tate}\Cref{part.norm.maps.compose} and in which the lower left morphism is an equivalence by \Cref{obs.tate.vanishing.for.Z.mods} (and the equivalence $E_{\htpy \Cyclic_{p^s}} \xra{\sim} (E_{\htpy \Cyclic_{p^{s-1}}})_{\htpy \Cyclic_p}$) and the lower right morphism is an equivalence by assumption.
\end{observation}

\begin{proof}[Proof of \Cref{thm.stratn.of.genuine.Cpn.Z.mods}]
We note that the assertions only make reference to the underlying stratification of the symmetric monoidal stratification \Cref{s.m.stratn.of.gen.Cpn.Z.mods}, so it suffices to verify them at that level. For each $s \in [n]$, its $s\th$ stratum is given by
\[
(\Mod^{\gen \Cyclic_{p^n}}_\ZZ)_s
:=
(\Spectra^{\gen \Cyclic_{p^n}} \otimes \Mod_\ZZ)_s
\simeq
(\Spectra^{\gen \Cyclic_{p^n}})_s \otimes \Mod_\ZZ
\simeq
\Spectra^{\htpy \Cyclic_{p^{n-s}}}
\otimes
\Mod_\ZZ
\simeq
\Mod^{\htpy \Cyclic_{p^{n-s}}}_\ZZ
~,
\]
where the first equivalence follows from \Cref{obs.tensored.up.stratn} and the second equivalence follows from Theorem S.\ref{strat:thm.geom.stratn.of.SpgG}. Thereafter, it follows from \Cref{prop.proper.Tate.from.genuine.G.objs} \and \Cref{obs.tensoring.with.cpctly.gend} (and the fact that $\Mod_\ZZ$ is compactly generated) that for each morphism $i \ra j$ in $[n]$ its corresponding gluing functor
\[
(\Mod^{\gen \Cyclic_{p^n}}_\ZZ)_i
\xlongra{\Gamma^i_j}
(\Mod^{\gen \Cyclic_{p^n}}_\ZZ)_j
\]
is given by the proper $\Cyclic_{p^{j-i}}$-Tate construction
\[
\Mod^{\htpy \Cyclic_{p^{n-i}}}_\ZZ
\xra{(-)^{\tate \Cyclic_{p^{j-i}}}}
\Mod^{\htpy \Cyclic_{p^{n-j}}}_\ZZ
~.
\]
Now, assuming $j-i \geq 1$, then for any $E \in \Mod^{\htpy \Cyclic_{p^{n-i}}}_\ZZ$ we have that
\begin{align}
\label{use.identification.of.proper.tate.for.genuine.G.objects}
E^{\tate \Cyclic_{p^{j-i}}}
& \simeq
\Phi^{\Cyclic_{p^{j-i}}} ( \beta E)
\\
\label{use.rewriting.of.large.geom.fps.as.mostly.catl.fps}
& \simeq
\Phi^{\Cyclic_p} \left( ( \beta E )^{\Cyclic_{p^{j-i-1}}} \right)
\\
\nonumber
& \simeq
\Phi^{\Cyclic_p} \left( \beta \left( E^{\htpy \Cyclic_{p^{j-i-1}}} \right) \right)
\\
\label{again.use.identification.of.proper.tate.for.genuine.G.objects}
& \simeq
\left( E^{\htpy \Cyclic_{p^{j-i-1}}} \right)^{\tate \Cyclic_p}
\\
\label{use.that.proper.tate.reduces.to.tate}
& \simeq
\left( E^{\htpy \Cyclic_{p^{j-i-1}}} \right)^{\st \Cyclic_p}
~,
\end{align}
where
\begin{itemize}
\item equivalences \Cref{use.identification.of.proper.tate.for.genuine.G.objects} \and \Cref{again.use.identification.of.proper.tate.for.genuine.G.objects} follow from \Cref{prop.proper.Tate.from.genuine.G.objs},
\item equivalence \Cref{use.rewriting.of.large.geom.fps.as.mostly.catl.fps} follows from \Cref{lem.projection.formula.categorical.then.geometric.is.geometric} when $j-i > 1$ and is trivially true when $j-i=1$, and
\item equivalence \Cref{use.that.proper.tate.reduces.to.tate} follows from \Cref{obs.proper.F.tate.recovers.tate}.
\end{itemize}
All three claims now follow from \Cref{obs.tate.vanishing.for.Z.mods}.
\end{proof}

\begin{notation}
\label{notn.zigzag.category}
We define the full subcategory
\[
\Zig_n
:=
\{ (i \ra j) \in \TwAr([n]) : j-i \leq 1 \}
\subseteq
\TwAr([n])
~,
\]
which we depict as
\[
\Zig_n
=
\left(
\begin{tikzcd}
(0 \ra 0)
\arrow{rd}
&
(1 \ra 1)
\arrow{d}
\arrow{rd}
&
\cdots
\arrow{d}
&
\cdots
\arrow{rd}
&
(n \ra n)
\arrow{d}
\\
&
(0 \ra 1)
&
(1 \ra 2)
&
\cdots
&
((n-1) \ra n)
\end{tikzcd} \right)~. \]
\end{notation}

\begin{remark}
When possible, we use the ``sawtooth'' depiction of $\Zig_n$ (as in \Cref{notn.zigzag.category}) and of diagrams indexed thereover: the columns will correspond to the strata of the geometric stratification of $\Mod^{\gen \Cyclic_{p^n}}_\ZZ$ of \Cref{thm.stratn.of.genuine.Cpn.Z.mods}. However, in order to depict natural transformations between diagrams indexed over $\Zig_n$ (such as in \Cref{prop.fixedpts.res.and.trf.for.general.Cpn.Z.mods}), we will use a more symmetric depiction.
\end{remark}

\begin{observation}
The inclusion
\[
\Zig_n
\longhookra
\TwAr([n])
\]
is initial. We use this fact without further comment.
\end{observation}

\begin{observation}
\label{consequences.of.stratn.of.genuine.Cpn.Z.mods}
We unpack the following consequences of \Cref{thm.stratn.of.genuine.Cpn.Z.mods}.

\begin{enumerate}

\item\label{macrocosm.consequences.of.stratn.of.genuine.Cpn.Z.mods}

At the macrocosm level, applying Theorems S.\ref{strat:macrocosm.thm} \and S.\ref{strat:thm.s.m.reconstrn} along with Observation S.\ref{strat:obs.strict.stratn.gives.macrocosm.regluing.over.sdP} to the strict symmetric monoidal stratification \Cref{s.m.stratn.of.gen.Cpn.Z.mods}, we obtain an identification
\begin{align*}
\Mod^{\gen \Cyclic_{p^n}}_\ZZ
& \simeq
\Glue^\otimes(\Mod^{\gen \Cyclic_{p^n}}_\ZZ)
:=
\lim^\rlax_{[n]} \left( \GD^\otimes( \Mod^{\gen \Cyclic_{p^n}}_\ZZ ) \right)
\simeq
\Gamma_{[n]^\op} \left( \GD^\otimes( \Mod^{\gen \Cyclic_{p^n}}_\ZZ )^\cocartdual \right)
\\
& \simeq
\lim \left(
\begin{tikzcd}[ampersand replacement=\&]
\Mod^{\htpy \Cyclic_{p^n}}_\ZZ
\arrow{rd}[sloped, swap]{(-)^{\st \Cyclic_p}}
\&
\Ar \left(\Mod^{\htpy \Cyclic_{p^{n-1}}}_\ZZ \right)
\arrow{d}{t}
\arrow{rd}[sloped, swap]{(-)^{\st \Cyclic_p} \circ s}
\&
\cdots
\arrow{d}{t}
\&
\cdots
\arrow{rd}[sloped, swap]{(-)^{\st \Cyclic_p} \circ s}
\&
\Ar \left( \Mod_\ZZ \right)
\arrow{d}{t}
\\
\&
\Mod^{\htpy \Cyclic_{p^{n-1}}}_\ZZ
\&
\Mod^{\htpy \Cyclic_{p^{n-2}}}_\ZZ
\&
\cdots
\&
\Mod_\ZZ
\end{tikzcd} \right)~,
\end{align*}
where the limit is taken in $\CAlg^\rlax(\Cat)$.

\item\label{microcosm.consequences.of.stratn.of.genuine.Cpn.Z.mods}

We now proceed to the microcosm level.

\begin{enumeratesub}

\item\label{microcosm.consequences.of.stratn.of.genuine.Cpn.Z.mods.gluing.diagrams}

By Observations S.\ref{strat:obs.stratn.strict.iff.all.objects.strict} \and S.\ref{strat:obs.strict.microcosm.gluing.diagram.iff.factors.from.sd.to.TwAr}, the fact that the stratification \Cref{s.m.stratn.of.gen.Cpn.Z.mods} is strict implies that for each object $E \in \Mod^{\gen \Cyclic_{p^n}}_\ZZ$ we have a natural diagram of equivalences
\[
E
\xlongra{\sim}
\lim_{\TwAr([n])}(\gd(E))
\xlongra{\sim}
\lim_{\Zig_n} ( \gd(E) )
~.
\]
In particular, the object $E \in \Mod^{\gen \Cyclic_{p^n}}_\ZZ$ is recorded by the data of its gluing diagram
\[ \left( \begin{tikzcd}
E_0
\arrow[maps to]{rd}
&
E_1
\arrow{d}{\gamma^E_{0,1}}
\arrow[maps to]{rd}
&
\cdots
\arrow{d}{\gamma^E_{1,2}}
&
\cdots
\arrow[maps to]{rd}
&
E_n
\arrow{d}{\gamma^E_{n-1,n}}
\\
&
(E_0)^{\st \Cyclic_p}
&
(E_1)^{\st \Cyclic_p}
&
\cdots
&
(E_{n-1})^{\st \Cyclic_p}
\end{tikzcd} \right)~. \]

\item\label{microcosm.consequences.of.stratn.of.genuine.Cpn.Z.mods.symm.mon.str}

The fact that the stratification \Cref{s.m.stratn.of.gen.Cpn.Z.mods} is symmetric monoidal implies that for any objects $E,F \in \Mod^{\gen \Cyclic_{p^n}}_\ZZ$, their tensor product $(E \otimes F) \in \Mod^{\gen \Cyclic_{p^n}}_\ZZ$ has $(E \otimes F)_s \simeq E_s \otimes F_s$ for all $s \in [n]$ and gluing morphisms the composites
\[
\gamma_{s-1,s}^{E \otimes F}
:
E_s \otimes F_s
\xra{\gamma_{s-1,s}^E \otimes \gamma_{s-1,s}^F}
(E_{s-1})^{\st \Cyclic_p} \otimes (F_{s-1})^{\st \Cyclic_p}
\longra
(E_{s-1} \otimes F_{s-1})^{\st \Cyclic_p}
~,
\]
where the second morphism arises from the fact that the functor
\[
\Mod^{\htpy \Cyclic_{p^{n-s+1}}}_\ZZ
\xra{(-)^{\st \Cyclic_p}}
\Mod^{\htpy \Cyclic_{p^{n-s}}}_\ZZ
\]
is right-laxly symmetric monoidal by \Cref{obs.various.properties.of.tate}\Cref{part.rlax.s.m.str.on.h.to.tate}.

\end{enumeratesub}

\item\label{nanocosm.consequences.of.stratn.of.genuine.Cpn.Z.mods}

At the nanocosm level, again by Observation S.\ref{strat:obs.strict.microcosm.gluing.diagram.iff.factors.from.sd.to.TwAr}, 
for each $E,F \in \Mod^{\gen \Cyclic_{p^n}}_\ZZ$ we have an equivalence
\begin{equation}
\label{nanocosm.equivce.for.g.Cpn.Z}
\ulhom_{\Mod^{\gen \Cyclic_{p^n}}_\ZZ} ( F , E )
\xlongra{\sim}
\lim_{(i \ra j) \in \Zig_n} \ulhom_{\Mod^{\htpy \Cyclic_{p^{n-j}}}_\ZZ} ( F_j , (E_i)^{\st \Cyclic_{p^{j-i}} } )
\end{equation}
in $\Mod_\ZZ$, where (by definition of $\Zig_n$) we have that $j-i$ is either $0$ or $1$. More explicitly, the diagram $\Zig_n \ra \Mod_\ZZ$ whose limit is the target of the equivalence \Cref{nanocosm.equivce.for.g.Cpn.Z} is
\[ \hspace{-2.5cm} \left( \begin{tikzcd}
\ulhom_{\Mod^{\htpy \Cyclic_{p^{n}}}_\ZZ} ( F_0 , E_0 )
\arrow{rd}
&
\ulhom_{\Mod^{\htpy \Cyclic_{p^{n-1}}}_\ZZ} ( F_1 , E_1 )
\arrow{d}
\arrow{rd}
&
\cdots
\arrow{d}
&
\cdots
\arrow{rd}
&
\ulhom_{\Mod_\ZZ} ( F_n , E_n )
\arrow{d}
\\
&
\ulhom_{\Mod^{\htpy \Cyclic_{p^{n-1}}}_\ZZ} ( F_1 , (E_0)^{\st \Cyclic_p} )
&
\ulhom_{\Mod^{\htpy \Cyclic_{p^{n-2}}}_\ZZ} ( F_2 , (E_1)^{\st \Cyclic_p} )
&
\cdots
&
\ulhom_{\Mod_\ZZ} ( F_n , (E_{n-1})^{\st \Cyclic_p} )
\end{tikzcd} \right)~, \]
where for all $1 \leq s \leq n$ the $s\th$ diagonal morphism is the composite
\[
\ulhom_{\Mod^{\htpy \Cyclic_{p^{n-s+1}}}_\ZZ} ( F_{s-1} , E_{s-1} )
\xra{(-)^{\st \Cyclic_p}}
\ulhom_{\Mod^{\htpy \Cyclic_{p^{n-s}}}_\ZZ} ( (F_{s-1})^{\st \Cyclic_p} , (E_{s-1})^{\st \Cyclic_p} )
\xra{\gamma^F_{s-1,s}}
\ulhom_{\Mod^{\htpy \Cyclic_{p^{n-s}}}_\ZZ} ( F_s , (E_{s-1})^{\st \Cyclic_p} )
\]
and the $s\th$ vertical morphism is
\[
\ulhom_{\Mod^{\htpy \Cyclic_{p^{n-s}}}_\ZZ} ( F_s , E_s )
\xra{\gamma^E_{s-1,s}}
\ulhom_{\Mod^{\htpy \Cyclic_{p^{n-s}}}_\ZZ} ( F_s , (E_{s-1})^{\st \Cyclic_p} )
~.
\]

\end{enumerate}
\end{observation}

\begin{remark}
In what follows, we record a number of basic facts about categorical fixedpoints of genuine $\Cyclic_{p^n}$-$\ZZ$-modules. In effect, we obtain these by applying the discussion of \S S.\ref{strat:subsection.categorical.fixedpoints} (regarding categorical fixedpoints of genuine $G$-spectra for an arbitrary finite group $G$). The present situation is simpler both because $\Cyclic_{p^n}$ is abelian and because we work $\ZZ$-linearly.
\end{remark}

\begin{observation}
\label{obs.inclusion.of.cat.fixedpts}
For any $0 \leq a < n$, the inclusion morphisms
\[
E^{\Cyclic_{p^{a+1}}}
\xra{\inc}
E^{\Cyclic_{p^a}}
\]
for $E \in \Mod^{\gen \Cyclic_{p^n}}_\ZZ$ define a natural transformation
\begin{equation}
\label{nat.trans.as.inclusion.of.categorical.fixedpoints}
\begin{tikzcd}[row sep=0.5cm, column sep=1.5cm]
&
\Mod^{\htpy \Cyclic_{p^{n-a-1}}}_\ZZ
\arrow{dd}{\triv}[swap, xshift=-0.8cm]{\rotatebox{50}{$\Leftarrow$}}
\\
\Mod^{\gen \Cyclic_{p^n}}_\ZZ
\arrow{ru}[sloped]{(-)^{\Cyclic_{p^{a+1}}}}
\arrow{rd}[sloped, swap]{(-)^{\Cyclic_{p^a}}}
\\
&
\Mod^{\htpy \Cyclic_{p^{n-a}}}_\ZZ
\end{tikzcd}~.
\end{equation}
Indeed, the natural transformation in diagram \Cref{nat.trans.as.inclusion.of.categorical.fixedpoints} is corepresented by the morphism
\[
\ZZ \brax{ \Cyclic_{p^n}/\Cyclic_{p^a} \longra \Cyclic_{p^n} / \Cyclic_{p^{a+1}} }
\]
in $\Mod^{\gen \Cyclic_{p^n}}_\ZZ$.
\end{observation}

\begin{observation}
\label{obs.transfer.of.cat.fixedpts}
For any $0 \leq a < n$, the transfer morphisms
\[
E^{\Cyclic_{p^a}}
\xlongra{\trf}
E^{\Cyclic_{p^{a+1}}}
\]
for $E \in \Mod^{\gen \Cyclic_{p^n}}_\ZZ$ define a natural transformation 
\begin{equation}
\label{nat.trans.as.transfer.of.categorical.fixedpoints}
\begin{tikzcd}[row sep=0.5cm, column sep=1.5cm]
&
\Mod^{\htpy \Cyclic_{p^{n-a-1}}}
\arrow{dd}{\triv}[swap, xshift=-0.8cm]{\rotatebox{50}{$\Rightarrow$}}
\\
\Mod^{\gen \Cyclic_{p^n}}_\ZZ
\arrow{ru}[sloped]{(-)^{\Cyclic_{p^{a+1}}}}
\arrow{rd}[sloped, swap]{(-)^{\Cyclic_{p^a}}}
\\
&
\Mod^{\htpy \Cyclic_{p^{n-a}}}
\end{tikzcd}~.
\end{equation}
Indeed, the natural transformation in diagram \Cref{nat.trans.as.transfer.of.categorical.fixedpoints} is corepresented by the morphism
\begin{equation}
\label{morphism.corepresenting.transfer.in.gCpn.Z.mods}
\ZZ \brax{ \Cyclic_{p^n}/\Cyclic_{p^{a+1}} }
\longra
\ZZ \brax{ \Cyclic_{p^n}/\Cyclic_{p^{a}} }
\end{equation}
in $\Mod^{\gen \Cyclic_{p^n}}_\ZZ$ obtained by applying the functor $\Mod^{\gen \Cyclic_{p^{a+1}}}_\ZZ \xra{\Ind_{\Cyclic_{p^{a+1}}}^{\Cyclic_{p^n}}} \Mod^{\gen \Cyclic_{p^n}}_\ZZ$ to the morphism
\[
\ZZ
\simeq
\ZZ \brax{ \Cyclic_{p^{a+1}} / \Cyclic_{p^{a+1}} }
\longra
\ZZ \brax{ \Cyclic_{p^{a+1}} / \Cyclic_{p^{a}} }
\simeq
\coInd_{\Cyclic_{p^a}}^{\Cyclic_{p^{a+1}}} ( \ZZ \brax{ \Cyclic_{p^a} / \Cyclic_{p^a} } )
\]
corresponding to the identity morphism
\[
\Res_{\Cyclic_{p^a}}^{\Cyclic_{p^{a+1}}} ( \ZZ \brax{  \Cyclic_{p^{a+1}} / \Cyclic_{p^{a+1}} } )
\longra
\ZZ \brax{ \Cyclic_{p^a} / \Cyclic_{p^a} }
\]
in $\Mod^{\gen \Cyclic_{p^a}}_\ZZ$.
\end{observation}

\begin{definition}
\label{defn.htrf}
The \bit{homotopy transfer} (\bit{for $\Cyclic_p$}) is the morphism
\[
\id_{\Mod^{\htpy \Cyclic_p}_\ZZ}
\xra{\htrf_{\Cyclic_p}}
\triv
\circ
(-)^{\htpy \Cyclic_p}
\]
in $\Fun^\ex(\Mod^{\htpy \Cyclic_p}_\ZZ,\Mod^{\htpy \Cyclic_p}_\ZZ)$ that is corepresented by the morphism
\begin{equation}
\label{morphism.coreping.htrf}
\ZZ
\longra
\ZZ^\htpy \brax{ \Cyclic_p / e }
~,
\end{equation}
i.e.\! the morphism $U\Cref{morphism.corepresenting.transfer.in.gCpn.Z.mods}$ in the case that $n=1$ and $a=0$. We may simply write
\[
\htrf
:=
\htrf_{\Cyclic_p}(E)
\]
for its component at an object $E \in \Mod^{\htpy \Cyclic_p}_\ZZ$ (in line with the notation introduced in \Cref{defn.inc.and.trf}).
\end{definition}

\begin{observation}
\label{obs.htrf.as.a.composite}
The morphism \Cref{morphism.coreping.htrf} is between discrete objects in $\Mod^{\htpy \Cyclic_p}_\ZZ$ (i.e.\! abelian groups with $\Cyclic_p$-action) which are cyclic as discrete $\ZZ[\Cyclic_p]$-modules, and it is characterized by the fact that it carries $1 \in \ZZ$ to the element
\[
N
:=
(1 + \sigma + \cdots + \sigma^{p-1} )
\in
\ZZ^\htpy \brax{\Cyclic_p/e}
\]
(see \Cref{defn.norm.elt}). So, for any object $E \in \Mod^{\htpy \Cyclic_p}_\ZZ$, the morphism $\ulhom_{\Mod^{\htpy \Cyclic_p}_\ZZ}(\Cref{morphism.coreping.htrf},E)$ is the composite
\[
E
\longra
E_{\htpy \Cyclic_p}
\xra{\Nm_{\Cyclic_p}(E)}
E^{\htpy \Cyclic_p}
~.
\]
\end{observation}

\begin{observation}
\label{obs.describe.corepresenting.map.for.transfer.for.Cpn}
Fix any $0 \leq a < n$. We claim that for any $s \in [n]$, applying the functor
\[
\Mod^{\gen \Cyclic_{p^n}}_\ZZ \xra{\Phi^{\Cyclic_{p^s}}} \Mod^{\htpy \Cyclic_{p^{n-s}}}_\ZZ
\]
to the morphism \Cref{morphism.corepresenting.transfer.in.gCpn.Z.mods} gives
\begin{itemize}

\item the morphism
\[
0
\longra
0
\]
when $s>a+1$;

\item the morphism
\[
\ZZ^\htpy \brax{ \Cyclic_{p^{n-s}} / \Cyclic_{p^{a+1-s}} }
\longra
0
\]
when $s=a+1$; and

\item the morphism
\begin{equation}
\label{morphism.coreping.hfps.of.htrf}
\ZZ^\htpy \brax{ \Cyclic_{p^{n-s}} / \Cyclic_{p^{a+1-s}} }
\longra
\ZZ^\htpy \brax{ \Cyclic_{p^{n-s}} / \Cyclic_{p^{a-s}} }
\end{equation}
when $s \leq a$.

\end{itemize}
The only nontrivial case to verify is when $s \leq a$. Then, the morphism \Cref{morphism.coreping.hfps.of.htrf} is between discrete objects in $\Mod^{\htpy \Cyclic_{p^{n-s}}}_\ZZ$ (i.e.\! abelian groups with $\Cyclic_{p^{n-s}}$-action) which are cyclic as discrete $\ZZ[\Cyclic_{p^{n-s}}]$-modules, and it is characterized by the fact that it carries the element $1 \in \ZZ^\htpy \brax{ \Cyclic_{p^{n-s}} / \Cyclic_{p^{a+1-s}}}$ to the element
\[
N
:=
(1 + \sigma + \cdots + \sigma^{p-1} )
\in
\ZZ^\htpy \brax{\Cyclic_{p^{n-s}}/\Cyclic_{p^{a-s}}}
~.
\]
So by \Cref{obs.htrf.as.a.composite}, the morphism \Cref{morphism.coreping.hfps.of.htrf} corepresents the natural morphism
\begin{equation}
\label{the.composite.that.is.htrf}
E^{\htpy \Cyclic_{p^{a-s}}}
\xra{\htrf_{\Cyclic_p}(E)^{\htpy \Cyclic_{p^{a-s}}}}
E^{\htpy \Cyclic_{p^{a+1-s}}}
\end{equation}
in $\Mod^{\htpy \Cyclic_{p^{n-a}}}_\ZZ$ for any object $E \in \Mod^{\htpy \Cyclic_{p^{n-s}}}_\ZZ$.
\end{observation}

\begin{notation}
\label{notn.htpy.trf}
For any $0 \leq s \leq a < n$ and any $E \in \Mod^{\htpy \Cyclic_{p^{n-s}}}_\ZZ$, we may simply write
\[
\htrf
:=
\htrf_{\Cyclic_p}(E)^{\htpy \Cyclic_{p^{a-s}}}
\]
for the morphism \Cref{the.composite.that.is.htrf}.
\end{notation}

\begin{observation}
\label{obs.nullhtpy.of.htrf.followed.by.Q}
For any $i \geq 0$ and any $E \in \Mod^{\htpy \Cyclic_{p^{i+1}}}_\ZZ$, we have a canonical commutative diagram
\[
\begin{tikzcd}[row sep=1.5cm, column sep=1.5cm]
E
\arrow{d}
\arrow{rd}[sloped]{\htrf_{\Cyclic_p}(E)}
\\
E_{\htpy \Cyclic_p}
\arrow{r}[swap]{\Nm_{\Cyclic_p}(E)}
\arrow{d}
&
E^{\htpy \Cyclic_p}
\arrow{d}{\sQ_{\Cyclic_p}(E)}
\\
0
\arrow{r}
&
E^{\st \Cyclic_p}
\end{tikzcd}
\]
in $\Mod^{\htpy \Cyclic_{p^i}}_\ZZ$, in which the upper triangle commutes by \Cref{obs.htrf.as.a.composite} and the square is the defining pushout.
\end{observation}

\begin{prop}
\label{prop.fixedpts.res.and.trf.for.general.Cpn.Z.mods}
Fix any genuine $\Cyclic_{p^n}$-$\ZZ$-module $E \in \Mod^{\gen \Cyclic_{p^n}}_\ZZ$.
\begin{enumerate}

\item\label{prop.fixedpts.res.and.trf.for.general.Cpn.Z.mods.the.fixedpts}

For any $0 \leq a \leq n$, its $\Cyclic_{p^a}$-fixedpoints $E^{\Cyclic_{p^a}} \in \Mod^{\htpy \Cyclic_{p^{n-a}}}_\ZZ$ is the limit of the diagram
\[
\Zig_a
\xra{\sD_a(E)}
\Mod^{\htpy \Cyclic_{p^{n-a}}}_\ZZ
\]
given by
\begin{equation}
\label{zigzag.diagram.whose.limit.is.genuine.Cpa.fixedpoints}
\begin{tikzcd}[row sep=2cm, column sep=1.5cm]
(E_0)^{\htpy \Cyclic_{p^a}}
\arrow{rd}[sloped, swap]{\sQ_{\Cyclic_p}(E_0)^{\htpy \Cyclic_{p^{a-1}}}}
&
(E_1)^{\htpy \Cyclic_{p^{a-1}}}
\arrow{rd}[sloped, swap]{\sQ_{\Cyclic_p}(E_1)^{\htpy \Cyclic_{p^{a-2}}}}
\arrow{d}[swap, pos=0.2]{(\gamma_{0,1}^E)^{\htpy \Cyclic_{p^{a-1}}}}
&
\cdots
\arrow{d}[swap, pos=0.2]{(\gamma_{1,2}^E)^{\htpy \Cyclic_{p^{a-2}}}}
&
\cdots
\arrow{rd}[sloped, swap]{\sQ_{\Cyclic_p}(E_{a-1})^{\htpy \Cyclic_{p^0}}}
&
(E_a)^{\htpy \Cyclic_{p^0}}
\arrow{d}[swap, pos=0.2]{(\gamma_{a-1,a}^E)^{\htpy \Cyclic_{p^0}}}
\\
&
((E_0)^{\st \Cyclic_p})^{\htpy \Cyclic_{p^{a-1}}}
&
((E_1)^{\st \Cyclic_p})^{\htpy \Cyclic_{p^{a-2}}}
&
\cdots
&
((E_{a-1})^{\st \Cyclic_p})^{\htpy \Cyclic_{p^0}}
\end{tikzcd}
~.
\end{equation}
Moreover, for any $0 \leq a < n$ the functors $\sD_a(E)$ and $\sD_{a+1}(E)$ participate in a commutative square
\begin{equation}
\label{square.relating.DaE.and.DaplusoneE}
\begin{tikzcd}[row sep=1.5cm, column sep=1.5cm]
\Zig_a
\arrow[hook]{r}{\Zig_{i \mapsto i}}
\arrow{d}[swap]{\sD_a(E)}
&
\Zig_{a+1}
\arrow{d}{\sD_{a+1}(E)}
\\
\Mod^{\htpy \Cyclic_{p^{n-a}}}
\arrow{r}[swap]{(-)^{\htpy \Cyclic_p}}
&
\Mod^{\htpy \Cyclic_{p^{n-a-1}}}
\end{tikzcd}
\end{equation}
(in which the functor $\Zig_{i \mapsto i}$ is the evident fully faithful inclusion).

\item\label{prop.fixedpts.res.and.trf.for.general.Cpn.Z.mods.the.restriction}

\newcommand{\septtfromcdotsforhtpical}{0.8cm}
\newcommand{\sepcdotsfromcdotsforhtpical}{1.5cm}

For any $0 \leq a < n$, the inclusion morphism
\[
E^{\Cyclic_{p^{a+1}}}
\xra{\inc}
E^{\Cyclic_{p^a}}
\]
in $\Mod^{\htpy \Cyclic_{p^{n-a}}}_\ZZ$ is the limit of the morphism
\begin{equation}
\label{general.description.of.inclusion.map.via.zigzags}
\begin{tikzcd}[row sep=1cm, column sep=-0.2cm]
(E_0)^{\htpy \Cyclic_{p^{a+1}}}
\arrow{rd}
\arrow{dd}[swap, pos=0.7]{\hinc}
&[\septtfromcdotsforhtpical]
&[\sepcdotsfromcdotsforhtpical]
\cdots
\arrow{ld}
\arrow{rd}
\arrow{dd}[swap, pos=0.7]{\hinc}
&[\septtfromcdotsforhtpical]
&
(E_a)^{\htpy \Cyclic_{p^1}}
\arrow{ld}
\arrow{rd}
\arrow{dd}[swap, pos=0.7]{\hinc}
&
&
(E_{a+1})^{\htpy \Cyclic_{p^0}}
\arrow{ld}
\arrow{dd}
\\
&
\cdots
\arrow{dd}[swap, pos=0.4]{\hinc}
&
&
((E_{a-1})^{\st \Cyclic_p})^{\htpy \Cyclic_{p^1}}
\arrow{dd}[swap, pos=0.4]{\hinc}
&
&
((E_a)^{\st \Cyclic_p})^{\htpy \Cyclic_{p^0}}
\arrow{dd}
\\
(E_0)^{\htpy \Cyclic_{p^a}}
\arrow{rd}
&
&
\cdots
\arrow{ld}
\arrow{rd}
&
&
(E_a)^{\htpy \Cyclic_{p^0}}
\arrow{ld}
\arrow{rd}
&
&
0
\arrow{ld}
\\
&
\cdots
&
&
((E_{a-1})^{\st \Cyclic_p})^{\htpy \Cyclic_{p^0}}
&
&
0
\end{tikzcd}
\end{equation}
in $\Fun ( \Zig_{a+1} , \Mod^{\htpy \Cyclic_{p^{n-a}}} )$ in which the nontrivial non-vertical morphisms are as in diagram \Cref{zigzag.diagram.whose.limit.is.genuine.Cpa.fixedpoints}.\footnote{Recall from \Cref{notn.htpy.inc} that $\hinc$ denotes an inclusion morphism among homotopy fixedpoints.}\footnote{That is, it arises from the functoriality of limits for the diagram
\[ \begin{tikzcd}[row sep=1.5cm, column sep=1.5cm, ampersand replacement=\&]
\Zig_{a}
\arrow{d}[swap]{\sD_{a}(E)}
\arrow[hook]{r}{\Zig_{i \mapsto i}}
\&
\Zig_{a+1}
\arrow{d}{\sD_{a+1}(E)}
\\
\Mod^{\htpy \Cyclic_{p^{n-a}}}_\ZZ
\arrow{r}{(-)^{\htpy \Cyclic_p}}[swap, xshift=-0.5cm, yshift=-0.5cm]{\rotatebox{30}{$\Leftarrow$}}
\arrow{d}[swap]{\id}
\&
\Mod^{\htpy \Cyclic_{p^{n-a-1}}}_\ZZ
\arrow{ld}[sloped, swap]{\triv}
\\
\Mod^{\htpy \Cyclic_{p^{n-a}}}_\ZZ
\end{tikzcd} \]
consisting of the commutative square \Cref{square.relating.DaE.and.DaplusoneE} and the counit of the adjunction $\triv \adj (-)^{\htpy \Cyclic_p}$.}

\item\label{prop.fixedpts.res.and.trf.for.general.Cpn.Z.mods.the.transfer}

For any $0 \leq a < n$, the transfer morphism
\[
E^{\Cyclic_{p^a}}
\xra{\trf}
E^{\Cyclic_{p^{a+1}}}
\]
in $\Mod^{\htpy \Cyclic_{p^{n-a}}}_\ZZ$ is the limit of the morphism
\begin{equation}
\label{general.description.of.trasfer.map.via.zigzags}
\begin{tikzcd}[row sep=1cm, column sep=-0.2cm]
(E_0)^{\htpy \Cyclic_{p^{a+1}}}
\arrow{rd}
&[\septtfromcdotsforhtpical]
&[\sepcdotsfromcdotsforhtpical]
\cdots
\arrow{ld}
\arrow{rd}
&[\septtfromcdotsforhtpical]
&
(E_a)^{\htpy \Cyclic_{p^1}}
\arrow{ld}
\arrow{rd}
&
&
(E_{a+1})^{\htpy \Cyclic_{p^0}}
\arrow{ld}
\\
&
\cdots
&
&
((E_{a-1})^{\st \Cyclic_p})^{\htpy \Cyclic_{p^1}}
&
&
((E_a)^{\st \Cyclic_p})^{\htpy \Cyclic_{p^0}}
\\
(E_0)^{\htpy \Cyclic_{p^a}}
\arrow{rd}
\arrow{uu}[pos=0.3]{\htrf}
&
&
\cdots
\arrow{ld}
\arrow{rd}
\arrow{uu}[pos=0.3]{\htrf}
&
&
(E_a)^{\htpy \Cyclic_{p^0}}
\arrow{ld}
\arrow{rd}
\arrow{uu}[pos=0.3]{\htrf}
&
&
0
\arrow{ld}
\arrow{uu}
\\
&
\cdots
\arrow{uu}[pos=0.6]{\htrf}
&
&
((E_{a-1})^{\st \Cyclic_p})^{\htpy \Cyclic_{p^0}}
\arrow{uu}[pos=0.6]{\htrf}
&
&
0
\arrow{uu}
\end{tikzcd}
\end{equation}
in $\Fun ( \Zig_{a+1} , \Mod^{\htpy \Cyclic_{p^{n-a}}})$, in which the nontrivial non-vertical morphisms are as in diagram \Cref{zigzag.diagram.whose.limit.is.genuine.Cpa.fixedpoints} and in which the second to rightmost square commutes by \Cref{obs.nullhtpy.of.htrf.followed.by.Q}.\footnote{That is, it arises from the functoriality of limits for a diagram
\[ \begin{tikzcd}[ampersand replacement=\&, row sep=1.5cm, column sep=1.5cm]
\Zig_a
\arrow[hook]{r}{\Zig_{i \mapsto i}}
\arrow{d}[swap]{\sD_a(E)}
\&
\Zig_{a+1}
\arrow{d}{\sD_{a+1}(E)}
\arrow{ld}[sloped]{(\Zig_{i \mapsto i})_* ( \sD_a(E)) }[sloped, swap, xshift=0.1cm, yshift=-0.5cm]{\Downarrow}
\\
\Mod^{\htpy \Cyclic_{p^{n-a}}}_\ZZ
\&
\Mod^{\htpy \Cyclic_{p^{n-a-1}}}_\ZZ
\arrow{l}{\triv}
\end{tikzcd} \]
in which upper triangle is a right Kan extension (i.e.\! extension by zero).}

\end{enumerate}
\end{prop}

\begin{proof}
Part \Cref{prop.fixedpts.res.and.trf.for.general.Cpn.Z.mods.the.fixedpts} is a special case of \Cref{consequences.of.stratn.of.genuine.Cpn.Z.mods}\Cref{nanocosm.consequences.of.stratn.of.genuine.Cpn.Z.mods} (with the commutative square \Cref{square.relating.DaE.and.DaplusoneE} following from its naturality in $F \in \Mod^{\gen \Cyclic_{p^n}}_\ZZ$). Part \Cref{prop.fixedpts.res.and.trf.for.general.Cpn.Z.mods.the.restriction} then follows by applying Observation S.\ref{strat:obs.restriction.between.catl.fixedpts}, and part \Cref{prop.fixedpts.res.and.trf.for.general.Cpn.Z.mods.the.transfer} follows by applying Observations S.\ref{strat:obs.transfer.between.catl.fixedpts}, \ref{obs.transfer.of.cat.fixedpts}, \and \ref{obs.describe.corepresenting.map.for.transfer.for.Cpn}.
\end{proof}

\section{The Picard group of genuine $\Cyclic_{p^n}$-$\ZZ$-modules}
\label{section.Pic.of.gen.Cpn.Z.mods}

In this section, we prove the first part of \Cref{intro.thm.Pic.of.gen.Cpn.Z.mods}: namely, we use our symmetric monoidal stratification of $\Mod^{\gen \Cyclic_{p^n}}_\ZZ$ (as described in \Cref{section.stratn.of.gen.Cpn.Z.mods}) to compute its Picard group (under the assumption that $p$ is odd). In order to help the reader appreciate the flow of the computation, we state its output as quickly as possible as \Cref{thm.picard.group}, and then proceed to unpack it; its proof appears at the end of the section.

\begin{local}
Through \Cref{section.Cpn.eqvrt.cohomology} (i.e.\! for the remainder of the paper except for \Cref{section.homological.algebra}), we assume that the prime $p$ that was fixed in \Cref{notn.fix.prime.and.n.et.al}\Cref{part.fix.prime.and.n} is odd.
\end{local}

\begin{notation}
Given a symmetric monoidal $\infty$-category $\cC$, we write $\pi_0 ( \iota_0 ( \cC ) )$ for the commutative monoid of equivalence classes of objects of $\cC$. We make no notational distinction between an object of $\cC$ and its equivalence class in $\pi_0 (\iota_0 ( \cC ) )$.
\end{notation}

\begin{definition}
The \bit{Picard group} of a symmetric monoidal $\infty$-category $\cC$ is the maximal subgroup
\[
\Pic(\cC)
\subseteq
\pi_0 ( \iota_0 ( \cC ) )
~,
\]
i.e.\! the abelian group of those equivalence classes of objects $c \in \cC$ such that there exist an object $c' \in \cC$ and an equivalence $c \otimes c' \simeq \uno_\cC$. We refer to an object of $\cC$ whose equivalence class lies in $\Pic(\cC)$ as a \bit{Picard element}.
\end{definition}

\begin{theorem}
\label{thm.picard.group}
The abelian group homomorphism
\begin{equation}
\label{abgrp.hom.to.Pic.of.gen.Cpn.Z.mods}
\ZZ^{\oplus (n+1)}
\oplus
\left(
\bigoplus_{s=1}^n
(\ZZ/p^{n-s+1})^\times
\right)
\xra{L^\bullet}
\Pic(\Mod^{\gen \Cyclic_{p^n}}_\ZZ)
\end{equation}
of \Cref{notn.factorize.K.bullet.to.L.bullet} is surjective, and descends to an isomorphism
\[
\ZZ^{\oplus (n+1)}
\oplus
\left(
\bigoplus_{s=1}^n
(\ZZ/p^{n-s+1})^\times / \{ \pm 1 \}
\right)
\xlongra{\cong}
\Pic(\Mod^{\gen \Cyclic_{p^n}}_\ZZ)
~.
\]
\end{theorem}

\begin{observation}
\label{obs.picard.elts.must.be.shifts.of.Z.at.each.stratum}
For any $s \in [n]$, the functors
\[
\Mod^{\gen \Cyclic_{p^n}}_\ZZ
\xra{\Phi^{\Cyclic_{p^s}}}
\Mod^{\htpy \Cyclic_{p^{n-s}}}_\ZZ
\xra{\fgt}
\Mod_\ZZ
\]
are symmetric monoidal, and therefore carry Picard elements to Picard elements. We use this to make some basic deductions regarding Picard elements of $\Mod^{\gen \Cyclic_{p^n}}_\ZZ$.

First of all, every Picard element of $\Mod_\ZZ$ is equivalent to $\Sigma^\alpha \ZZ$ for some $\alpha \in \ZZ$. Moreover, as a result of the equivalences $\Aut_{\Mod_\ZZ}(\Sigma^\alpha \ZZ) \simeq \Aut_{\Mod_\ZZ}(\ZZ) \simeq \{ \pm 1 \}$ and the fact that $p$ is odd, there are no nontrivial $\Cyclic_{p^{n-s}}$-actions on $\Sigma^\alpha \ZZ \in \Mod_\ZZ$ for any $\alpha \in \ZZ$. So, every Picard element of $\Mod_\ZZ^{\htpy \Cyclic_{p^{n-s}}}$ is equivalent to $\Sigma^\alpha \ZZ$ for some $\alpha \in \ZZ$, where we simply write $\ZZ \in \Mod_\ZZ^{\htpy \Cyclic_{p^{n-s}}}$ for the trivial $\Cyclic_{p^{n-s}}$-action on $\ZZ \in \Mod_\ZZ$ (as specified by \Cref{notn.fix.prime.and.n.et.al}\Cref{part.write.Z.for.trivial.action}).

So, given any Picard element $E \in \Mod^{\gen \Cyclic_{p^n}}_\ZZ$, for each $s \in [n]$ there must exist an equivalence $\Phi^{\Cyclic_s} E \simeq \Sigma^{\alpha_s} \ZZ$ for some $\alpha_s \in \ZZ$. Choosing such equivalences, by \Cref{thm.stratn.of.genuine.Cpn.Z.mods} (and \Cref{consequences.of.stratn.of.genuine.Cpn.Z.mods}\Cref{microcosm.consequences.of.stratn.of.genuine.Cpn.Z.mods}\Cref{microcosm.consequences.of.stratn.of.genuine.Cpn.Z.mods.gluing.diagrams}), this Picard element is recorded by the data of a gluing diagram
\begin{equation}
\label{gluing.diagram.for.K.alpha.gamma}
\left( \begin{tikzcd}
\Sigma^{\alpha_0} \ZZ
\arrow[maps to]{rd}
&
\Sigma^{\alpha_0 + \alpha_1} \ZZ
\arrow{d}{\gamma_1}
\arrow[maps to]{rd}
&
\cdots
\arrow{d}{\gamma_2}
&
\cdots
\arrow[maps to]{rd}
&
\Sigma^{\alpha_0 + \alpha_1 + \cdots + \alpha_n} \ZZ
\arrow{d}{\gamma_n}
\\
&
\Sigma^{\alpha_0} \ZZ^{\st \Cyclic_p}
&
\Sigma^{\alpha_0 + \alpha_1} \ZZ^{\st \Cyclic_p}
&
\cdots
&
\Sigma^{\alpha_0 + \alpha_1 + \cdots + \alpha_{n-1}} \ZZ^{\st \Cyclic_p}
\end{tikzcd} \right)~.
\end{equation}
For all $1 \leq s \leq n$, we consider the $s\th$ gluing morphism of the gluing diagram \Cref{gluing.diagram.for.K.alpha.gamma} as an element
\begin{align*}
\gamma_s
&\in
\pi_0 \ulhom_{\Mod^{\htpy \Cyclic_{p^{n-s}}}_\ZZ} ( \Sigma^{\alpha_0 + \alpha_1 + \cdots + \alpha_s} \ZZ , \Sigma^{\alpha_0 + \alpha_1 + \cdots + \alpha_{s-1}} \ZZ^{\st \Cyclic_p} )
\\
&\cong
\pi_0 \ulhom_{\Mod^{\htpy \Cyclic_{p^{n-s}}}_\ZZ} ( \Sigma^{\alpha_s} \ZZ , \ZZ^{\st \Cyclic_p} )
\\
& \cong
\pi_0 \ulhom_{\Mod_\ZZ} ( \Sigma^{\alpha_s} \ZZ , (\ZZ^{\st \Cyclic_p})^{\htpy \Cyclic_{p^{n-s}}} )
\\
& \cong
\pi_{\alpha_s} ( (\ZZ^{\st \Cyclic_p})^{\htpy \Cyclic_{p^{n-s}}} )
~.
\end{align*}
\end{observation}

\begin{notation}
\label{notn.potential.picard.elt.K}
Given a pair of elements
\[
\vec{\alpha}
=
(\alpha_0,\ldots, \alpha_n)
\in
\ZZ^{\oplus (n+1)}
\qquad
\text{and}
\qquad
\vec{\gamma}
=
\left( \gamma_s \in \pi_{\alpha_s} ( ( \ZZ^{\st \Cyclic_p} )^{\htpy \Cyclic_{p^{n-s}}} )
\right)_{1 \leq s \leq n}
~,
\]
we write
\[
K^{(\vec{\alpha},\vec{\gamma})} \in \pi_0 ( \iota_0 ( \Mod^{\gen \Cyclic_{p^n}}_\ZZ) )
\]
for the equivalence class of genuine $\Cyclic_{p^n}$-$\ZZ$-module corresponding to the gluing diagram \Cref{gluing.diagram.for.K.alpha.gamma} via \Cref{thm.stratn.of.genuine.Cpn.Z.mods} (and \Cref{consequences.of.stratn.of.genuine.Cpn.Z.mods}\Cref{microcosm.consequences.of.stratn.of.genuine.Cpn.Z.mods}\Cref{microcosm.consequences.of.stratn.of.genuine.Cpn.Z.mods.gluing.diagrams}). Evidently, such pairs $(\vec{\alpha},\vec{\gamma})$ may be equivalently considered as elements of the product
\[
\left( \alpha_0 \in \ZZ , (\gamma_s \in \pi_{\alpha_s}  ( ( \ZZ^{\st \Cyclic_p} )^{\htpy \Cyclic_{p^{n-s}}} ) )_{1 \leq s \leq n} \right)
\in
\ZZ
\times
\prod_{s=1}^n
\pi_* ( ( \ZZ^{\st \Cyclic_p} )^{\htpy \Cyclic_{p^{n-s}}} )_\homog
\]
(where the subscript denotes the restriction to homogeneous elements), through which identification this construction assembles as a function
\[
\ZZ
\times
\prod_{s=1}^n
\pi_* ( ( \ZZ^{\st \Cyclic_p} )^{\htpy \Cyclic_{p^{n-s}}} )_\homog
\xlongra{K^\bullet}
\pi_0 ( \iota_0 ( \Mod^{\gen \Cyclic_{p^n}}_\ZZ ) )
\]
between sets.
\end{notation}

\begin{observation}
\label{obs.multiplicativity.of.K.bullet}
By \Cref{obs.various.properties.of.tate}\Cref{part.rlax.s.m.str.on.h.to.tate}, for each $1 \leq s \leq n$ the object $(\ZZ^{\st \Cyclic_p})^{\htpy \Cyclic_{p^{n-s}}} \in \Mod_\ZZ$ admits a canonical lift to an object of $\CAlg(\Mod_\ZZ)$, so that the graded abelian group $\pi_*((\ZZ^{\st \Cyclic_p})^{\htpy \Cyclic_{p^{n-s}}})$ acquires the structure of a graded-commutative ring. Thereafter, the source of the function
\[
\ZZ
\times
\prod_{s=1}^n
\pi_* ( ( \ZZ^{\st \Cyclic_p} )^{\htpy \Cyclic_{p^{n-s}}} )_\homog
\xlongra{K^\bullet}
\pi_0 ( \iota_0 ( \Mod^{\gen \Cyclic_{p^n}}_\ZZ ) )
\]
obtained in \Cref{notn.potential.picard.elt.K} naturally inherits the structure of a monoid. By \Cref{thm.stratn.of.genuine.Cpn.Z.mods} (and \Cref{consequences.of.stratn.of.genuine.Cpn.Z.mods}\Cref{microcosm.consequences.of.stratn.of.genuine.Cpn.Z.mods}\Cref{microcosm.consequences.of.stratn.of.genuine.Cpn.Z.mods.symm.mon.str}), with respect to this structure the function $K^\bullet$ is a monoid homomorphism. We use these facts without further comment.
\end{observation}

\begin{observation}
\label{obs.htpy.ring.of.Z.tCp.hCpnminusr}
By \Cref{lem.compute.htpy.ring.of.htpy.of.tate}, for each $1 \leq s \leq n$ we have an isomorphism
\[
\pi_* ( ( \ZZ^{\st \Cyclic_p})^{\htpy \Cyclic_{p^{n-s}}} )
\cong
(\ZZ/p^{n-s+1})[(c_{n-s+1})^\pm]
\]
of graded-commutative rings, where $|c_{n-s+1}| = -2$.
\end{observation}

\begin{notation}
\label{notn.introduce.monoid.M}
For each $1 \leq s \leq n$, we define the commutative submonoid
\[
\MM_s
:=
\{ (\alpha ,\gamma) \in \ZZ \oplus \ZZ/p^{n-s+1} : \textup{if } \alpha \textup{ is odd then } \gamma = 0 \}
\subseteq
\left( \ZZ \oplus \ZZ/p^{n-s+1} \right)
:=
\left( (\ZZ,+) \oplus (\ZZ/p^{n-s+1} , ~\cdot~ ) \right)
\]
of the indicated commutative monoid (namely the direct sum of $\ZZ$ (considered as a commutative monoid via addition) and $\ZZ/p^{n-s+1}$ (considered as a commutative monoid via multiplication)). Projection to the first summand gives the commutative monoid $\MM_s$ the structure of a graded-commutative monoid. 
We also define the commutative monoid
\[
\MM
:=
\ZZ \oplus \left( \bigoplus_{s=1}^n \MM_s \right)
~.
\]
Via the injection
\[
\begin{tikzcd}[row sep=0cm]
\MM
:=
&[-1.8cm]
\ZZ \oplus \left( \bigoplus_{s=1}^n \MM_s \right)
&[-1.8cm]
\subseteq
\ZZ \oplus \left( \bigoplus_{s=1}^n \ZZ \oplus (\ZZ/p^{n-s+1})^\times \right)
\cong
&[-1cm]
\ZZ^{\oplus (n+1)} \oplus
\left(
\bigoplus_{s=1}^n
(\ZZ/p^{n-s+1})^\times
\right)
\\
&
\rotatebox{90}{$\in$}
&
&
\rotatebox{90}{$\in$}
\\
&
(\alpha_0 , (\alpha_1,\gamma_1) , \ldots,(\alpha_n,\gamma_n))
\arrow[maps to]{rr}
&
&
((\alpha_0,\ldots,\alpha_n),(\gamma_1,\ldots,\gamma_n))
\end{tikzcd}
~,
\]
we denote an element in $\MM$ as a pair of vectors 
$
(\vec{\alpha},\vec{\gamma}) 
$.
We denote the identity element in this group as $( \vec{0} , \vec{1})$.

\end{notation}

\begin{observation}
\label{obs.identify.monoids.of.homogenous.elts}
For each $1 \leq s \leq n$, by \Cref{obs.htpy.ring.of.Z.tCp.hCpnminusr} we have an isomorphism
\[ \begin{tikzcd}[row sep=0cm]
\MM_s
\arrow{r}{\cong}
&
\pi_* ( ( \ZZ^{\st \Cyclic_p} )^{\htpy \Cyclic_{p^{n-s}}} )_\homog
\\
\rotatebox{90}{$\in$}
&
\rotatebox{90}{$\in$}
\\
(\alpha,\gamma)
\arrow[mapsto]{r}
&
(c_{n-s+1})^{- \frac{\alpha}{2}} \gamma
\end{tikzcd} \]
of graded-commutative monoids, where we take the convention that $(c_{n-s+1})^{- \frac{\alpha}{2}} := 0$ whenever $\alpha$ is odd. Thereafter, we obtain a likewise isomorphism
\[ \begin{tikzcd}[row sep=0cm]
\MM
\arrow{r}{\cong}
&
{\displaystyle
\ZZ
\times
\prod_{s=1}^n
\pi_* ( ( \ZZ^{\st \Cyclic_p} )^{\htpy \Cyclic_{p^{n-s}}} )_\homog
}
\\
\rotatebox{90}{$\in$}
&
\rotatebox{90}{$\in$}
\\
(\vec{\alpha},\vec{\gamma})
\arrow[maps to]{r}
&
\left( \alpha_0 , (c_n)^{- \frac{\alpha_1}{2}} \gamma_1 , \ldots, (c_1)^{- \frac{\alpha_n}{2}} \gamma_n \right)
\end{tikzcd} \]
of monoids. We employ both of these isomorphisms -- and in particular the fact that all monoids under consideration are in fact commutative -- without further comment.
\end{observation}

\begin{observation}
\label{obs.describe.maxl.subgp.of.the.monoid}
For any $1 \leq s \leq n$, the inclusion
\[ \begin{tikzcd}[row sep=0cm]
\ZZ \oplus (\ZZ/p^{n-s+1})^\times
\arrow[hook]{r}
&
\MM_s
\arrow{r}{\cong}
&
\pi_* ( ( \ZZ^{\st \Cyclic_p} )^{\htpy \Cyclic_{p^{n-s}}} )_\homog
\\
\rotatebox{90}{$\in$}
&
\rotatebox{90}{$\in$}
&
\rotatebox{90}{$\in$}
\\
(\beta , \gamma)
\arrow[maps to]{r}
&
(2 \beta , \gamma)
\arrow[maps to]{r}
&
(c_{n-s+1})^{-\beta} \gamma
\end{tikzcd} \]
is that of the subgroup of invertible elements. Thereafter, the inclusion
\[
\hspace{-2cm}
\begin{tikzcd}[row sep=0cm]
&[-1.5cm]
\ZZ^{\oplus (n+1)}
\oplus
\left(
\bigoplus_{s=1}^n
(\ZZ/p^{n-s+1})^\times
\right)
\arrow[hook]{r}
&
\MM
\arrow{r}{\cong}
&
\ZZ \times \prod_{s=1}^n \pi_* ( ( \ZZ^{\st \Cyclic_p} )^{\htpy \Cyclic_{p^{n-s}}} )_\homog
\\
&
\rotatebox{90}{$\in$}
&
\rotatebox{90}{$\in$}
&
\rotatebox{90}{$\in$}
\\
(\vec{\beta},\vec{\gamma})
:=
&
((\beta_0,\ldots,\beta_n),(\gamma_1,\ldots,\gamma_n))
\arrow[maps to]{r}
&
((\beta_0,2\beta_1,\ldots,2\beta_n),(\gamma_1,\ldots,\gamma_n))
\arrow[maps to]{r}
&
\left( \beta_0 , (c_n)^{-\beta_1} \gamma_1 , \ldots, ( c_1 )^{-\beta_n} \gamma_n \right)
\end{tikzcd} \]
is likewise that of the subgroup of invertible elements.
\end{observation}

\begin{notation}
\label{notn.P.tilde.and.P}
We define the abelian groups $\KK$, $\tilde{\PP}$, and $\PP$ according to the commutative diagram
\[
\begin{tikzcd}[row sep=1.5cm]
\KK :=
&[-2.4cm]
&[-2.3cm]
&[-1cm]
0^{\oplus (n+1)} \oplus \left( \bigoplus_{s=1}^n \{ \pm 1 \}^\times \right)
\arrow{rr}{\cong}
\arrow[hook]{d}
&
&
\left\{ \left( 0 \in \ZZ , \left( \pm 1 \in \pi_0  ( ( \ZZ^{\st \Cyclic_p} )^{\htpy \Cyclic_{p^{n-s}}} ) \right) \right)  \right\}_{1 \leq s \leq n}
\arrow[hook]{d}
\\
&
\tilde{\PP} :=
&
&
\ZZ^{\oplus (n+1)}
\oplus
\left(
\bigoplus_{s=1}^n
(\ZZ/p^{n-s+1})^\times
\right)
\arrow[hook]{r}
\arrow[two heads]{d}
&
\MM
\arrow{r}{\cong}
&
\ZZ
\times
\prod_{s=1}^n
\pi_* ( ( \ZZ^{\st \Cyclic_p} )^{\htpy \Cyclic_{p^{n-s}}} )_\homog
\\
&
&
\PP :=
&
\ZZ^{\oplus (n+1)}
\oplus
\left(
\bigoplus_{s=1}^n
(\ZZ/p^{n-s+1})^\times / \{ \pm 1 \}
\right)
\end{tikzcd}~,\footnote{Actually, $\tilde{\PP}$ and $\PP$ were already implicitly introduced in the statement of \Cref{thm.picard.group}.}
\]
in which the middle horizontal inclusion is that of the subgroup of invertible elements by \Cref{obs.describe.maxl.subgp.of.the.monoid} and the left vertical composite is an exact sequence among abelian groups.
\end{notation}

\begin{notation}
\label{notn.factorize.K.bullet.to.L.bullet}
We write
\begin{equation}
\label{square.defining.L.as.a.factorization.to.Pic}
\begin{tikzcd}[column sep=3.5cm]
\tilde{\PP}
\arrow[hook]{d}
\arrow[dashed]{r}{L^\bullet := ((\vec{\beta},\vec{\gamma}) \longmapsto L^{(\vec{\beta},\vec{\gamma})})}
&
\Pic(\Mod^{\gen \Cyclic_{p^n}}_\ZZ)
\arrow[hook]{d}
\\
\MM
\arrow{r}[swap]{K^\bullet := ((\vec{\alpha},\vec{\gamma}) \longmapsto K^{(\vec{\alpha},\vec{\gamma})})}
&
\pi_0 ( \iota_0 ( \Mod^{\gen \Cyclic_{p^n}}_\ZZ ) )
\end{tikzcd}
\end{equation}
for the induced homomorphism on subgroups of invertible elements.
\end{notation}

\begin{lemma}
\label{lemma.presentations.of.unit.object.in.gen.Cpn.Z.mods}
The kernel of the commutative monoid homomorphism
\[
\MM
\xlongra{K^\bullet}
\pi_0 (\iota_0 ( \Mod^{\gen \Cyclic_{p^n}}_\ZZ ) )
\]
is the subgroup $\KK \subseteq \MM$.
\end{lemma}

\begin{proof}
First of all, it is clear that we have an equivalence
\[
K^{(\vec{0},\vec{1})}
\simeq
\uno_{\Mod^{\gen \Cyclic_{p^n}}_\ZZ}
\]
in ${\Mod^{\gen \Cyclic_{p^n}}_\ZZ}$. To prove the claim, we will show that for any $(\vec{\alpha},\vec{\gamma}) \in \MM$ there exists an equivalence
\begin{equation}
\label{possible.equivalence.from.K.alpha.gamma.to.unit}
K^{(\vec{\alpha},\vec{\gamma})}
\longra
K^{(\vec{0},\vec{1})}
\end{equation}
in $\Mod^{\gen \Cyclic_{p^n}}_\ZZ$ if and only if
\[
\alpha_s = 0
\in
\ZZ
\]
for all $0 \leq s \leq n$ and
\[
\gamma_s = \pm 1
\in
\pi_{\alpha_s} ( ( \ZZ^{\st \Cyclic_p})^{\htpy \Cyclic_{p^{n-s}}} )
=
\pi_0 ( ( \ZZ^{\st \Cyclic_p})^{\htpy \Cyclic_{p^{n-s}}} )
\]
for all $1 \leq s \leq n$.

First of all, observe that for each $s \in [n]$ there exists an equivalence
\[
\Phi^{\Cyclic_{p^s}} K^{(\vec{\alpha},\vec{\gamma})}
\simeq
\Sigma^{\alpha_0 + \cdots + \alpha_s} \ZZ
\longra
\ZZ
\simeq
\Phi^{\Cyclic_{p^s}} K^{(\vec{0},\vec{1})}
\]
in $\Mod^{\htpy \Cyclic_{p^{n-s}}}_\ZZ$ if and only if $\alpha_0 + \cdots + \alpha_s = 0$. Hence, in order for there to exist an equivalence \Cref{possible.equivalence.from.K.alpha.gamma.to.unit}, it must indeed be the case that $\alpha_s = 0$ for all $0 \leq s \leq n$. So we assume this, and proceed.

Now, by \Cref{thm.stratn.of.genuine.Cpn.Z.mods} (and \Cref{consequences.of.stratn.of.genuine.Cpn.Z.mods}\Cref{nanocosm.consequences.of.stratn.of.genuine.Cpn.Z.mods}), we have an equivalence
\[
\ulhom_{\Mod^{\gen \Cyclic_{p^n}}_\ZZ} ( K^{(\vec{0},\vec{\gamma})} , K^{(\vec{0},\vec{1})} )
\xlongra{\sim}
\lim_{(i \ra j) \in \Zig_n} \ulhom_{\Mod^{\htpy \Cyclic_{p^{n-j}}}_\ZZ} ( \ZZ , \ZZ^{\st \Cyclic_{p^{j-i}} } )
\]
in $\Mod_\ZZ$; more explicitly, the limit is of the diagram $\Zig_n \ra \Mod_\ZZ$ given by
\begin{equation}
\label{zigzag.diagram.whose.limit.is.homs.between.potential.picard.elts}
\hspace{-1cm}
\left( \begin{tikzcd}[column sep=0.5cm]
\ulhom_{\Mod^{\htpy \Cyclic_{p^{n}}}_\ZZ} ( \ZZ , \ZZ )
\arrow{rd}[sloped, swap]{\gamma_1}
&
\ulhom_{\Mod^{\htpy \Cyclic_{p^{n-1}}}_\ZZ} ( \ZZ , \ZZ )
\arrow{d}{1}
\arrow{rd}[sloped, swap]{\gamma_2}
&
\cdots
\arrow{d}{1}
&
\cdots
\arrow{rd}[sloped, swap]{\gamma_n}
&
\ulhom_{\Mod_\ZZ} ( \ZZ , \ZZ )
\arrow{d}{1}
\\
&
\ulhom_{\Mod^{\htpy \Cyclic_{p^{n-1}}}_\ZZ} ( \ZZ , \ZZ^{\st \Cyclic_p} )
&
\ulhom_{\Mod^{\htpy \Cyclic_{p^{n-2}}}_\ZZ} ( \ZZ , \ZZ^{\st \Cyclic_p} )
&
\cdots
&
\ulhom_{\Mod_\ZZ} ( \ZZ , \ZZ^{\st \Cyclic_p} )
\end{tikzcd} \right)~,
\end{equation}
where for all $1 \leq s \leq n$ we slightly abuse notation by writing
\[
\ulhom_{\Mod^{\htpy \Cyclic_{p^{n-s+1}}}_\ZZ} ( \ZZ , \ZZ )
\xra{\gamma_s}
\ulhom_{\Mod^{\htpy \Cyclic_{p^{n-s}}}_\ZZ} ( \ZZ , \ZZ^{\st \Cyclic_p} )
\]
for the composite
\[
\ulhom_{\Mod^{\htpy \Cyclic_{p^{n-s+1}}}_\ZZ} ( \ZZ , \ZZ )
\xra{(-)^{\st \Cyclic_p}}
\ulhom_{\Mod^{\htpy \Cyclic_{p^{n-s}}}_\ZZ} ( \ZZ^{\st \Cyclic_p} , \ZZ^{\st \Cyclic_p} )
\xra{\gamma_s}
\ulhom_{\Mod^{\htpy \Cyclic_{p^{n-s}}}_\ZZ} ( \ZZ , \ZZ^{\st \Cyclic_p} )
~.
\]
Evidently, we may rewrite the diagram \Cref{zigzag.diagram.whose.limit.is.homs.between.potential.picard.elts} more simply as
\begin{equation}
\label{rewrite.zigzag.diagram.whose.limit.is.homs.between.potential.picard.elts}
\left(
\begin{tikzcd}
\ZZ^{\htpy \Cyclic_{p^n}}
\arrow{rd}[sloped, swap]{\gamma_1}
&
\ZZ^{\htpy \Cyclic_{p^{n-1}}}
\arrow{d}{1}
\arrow{rd}[sloped, swap]{\gamma_2}
&
\cdots
\arrow{d}{1}
&
\cdots
\arrow{rd}[sloped, swap]{\gamma_n}
&
\ZZ
\arrow{d}{1}
\\
&
(\ZZ^{\st \Cyclic_p})^{\htpy \Cyclic_{p^{n-1}}}
&
(\ZZ^{\st \Cyclic_p})^{\htpy \Cyclic_{p^{n-2}}}
&
\cdots
&
\ZZ^{\st \Cyclic_p}
\end{tikzcd}
\right)
~.
\end{equation}
Hence, we obtain the composite isomorphism
\begin{align}
\nonumber
\pi_0 ( \ulhom_{\Mod^{\gen \Cyclic_{p^n}}_\ZZ} ( K^{(\vec{0},\vec{\gamma})} , K^{(\vec{0},\vec{1})} ) )
& \xlongra{\cong}
\pi_0 ( \lim_{\Zig_n} \Cref{rewrite.zigzag.diagram.whose.limit.is.homs.between.potential.picard.elts} )
\\
\label{use.vanishing.pi.one.of.terms.on.lower.row}
& \xlongra{\cong}
\lim_{\Zig_n} ( \pi_0 \Cref{rewrite.zigzag.diagram.whose.limit.is.homs.between.potential.picard.elts} )
\\
\label{use.pi.zero.of.htpy.fixedpoints.and.also.tate.computation}
& \cong
\lim
\left( \begin{tikzcd}[ampersand replacement=\&]
\ZZ
\arrow{rd}[sloped, swap]{\gamma_1}
\&
\ZZ
\arrow{d}{1}
\arrow{rd}[sloped, swap]{\gamma_2}
\&
\cdots
\arrow{d}{1}
\&
\cdots
\arrow{rd}[sloped, swap]{\gamma_n}
\&
\ZZ
\arrow{d}{1}
\\
\&
\ZZ/p^n
\&
\ZZ/p^{n-1}
\&
\cdots
\&
\ZZ/p
\end{tikzcd} \right)
\end{align}
among abelian groups, in which
\begin{itemize}

\item isomorphism \Cref{use.vanishing.pi.one.of.terms.on.lower.row} follows from the fact that all of the $\ZZ$-modules on the lower row of diagram \Cref{rewrite.zigzag.diagram.whose.limit.is.homs.between.potential.picard.elts} have vanishing $\pi_1$ by \Cref{obs.htpy.ring.of.Z.tCp.hCpnminusr}, and

\item isomorphism \Cref{use.pi.zero.of.htpy.fixedpoints.and.also.tate.computation} follows from the evident isomorphisms $\pi_0(\ZZ^{\htpy \Cyclic_{p^{n-s}}}) \cong \ZZ$ for all $s \in [n]$ as well as the isomorphisms $\pi_0 (( \ZZ^{\st \Cyclic_p} )^{\htpy \Cyclic_{p^{n-s}}}) \cong \ZZ / p^{n-s+1}$ for all $1 \leq s \leq n$ of \Cref{obs.htpy.ring.of.Z.tCp.hCpnminusr}.

\end{itemize}
By \Cref{thm.stratn.of.genuine.Cpn.Z.mods} (and \Cref{consequences.of.stratn.of.genuine.Cpn.Z.mods}\Cref{macrocosm.consequences.of.stratn.of.genuine.Cpn.Z.mods}), the functor
\[
\Mod^{\gen \Cyclic_{p^n}}_\ZZ
\xra{ \prod_{s \in [n]} \Phi^{\Cyclic_{p^s}} }
\prod_{s \in [n]}
\Mod^{\htpy \Cyclic_{p^{n-s}}}_\ZZ
\]
is conservative.  Hence, an element of $\pi_0 ( \ulhom_{\Mod^{\gen \Cyclic_{p^n}}_\ZZ} ( K^{(\vec{0},\vec{\gamma})} , K^{(\vec{0},\vec{1})} ) )$ is an isomorphism in $\ho(\Mod^{\gen \Cyclic_{p^n}}_\ZZ)$ if and only if its image under the composite homomorphism
\[
\pi_0 ( \ulhom_{\Mod^{\gen \Cyclic_{p^n}}_\ZZ} ( K^{(\vec{0},\vec{\gamma})} , K^{(\vec{0},\vec{1})} ) )
\xlongra{\cong}
\pi_0 ( \lim_{\Zig_n} ( \Cref{rewrite.zigzag.diagram.whose.limit.is.homs.between.potential.picard.elts} ) )
\longra
\prod_{s \in [n]} \ZZ
\]
(in which the last morphism is the projection to the upper factors in the diagram appearing in isomorphism \Cref{use.pi.zero.of.htpy.fixedpoints.and.also.tate.computation}) lies in the subset $\prod_{s \in [n]} \ZZ^\times \subseteq \prod_{s \in [n]} \ZZ$. From here, we see inductively that there exists an element of $\pi_0 ( \ulhom_{\Mod^{\gen \Cyclic_{p^n}}_\ZZ} ( K^{(\vec{0},\vec{\gamma})} , K^{(\vec{0},\vec{1})} ) )$ that is an isomorphism in $\ho(\Mod^{\gen \Cyclic_{p^n}}_\ZZ)$ if and only if $\gamma_s = \pm 1$ for all $1 \leq s \leq n$.
\end{proof}

\begin{proof}[Proof of \Cref{thm.picard.group}]
By \Cref{obs.picard.elts.must.be.shifts.of.Z.at.each.stratum}, the image of the commutative monoid homomorphism $K^\bullet$ contains the Picard group. Moreover, by \Cref{lemma.presentations.of.unit.object.in.gen.Cpn.Z.mods} its kernel is a subgroup of $\MM$ (i.e.\! it only contains invertible elements). Together, these two facts imply that the abelian group homomorphism $L^\bullet$ is surjective (and in fact that the commutative square \Cref{square.defining.L.as.a.factorization.to.Pic} is a pullback). Appealing again to \Cref{lemma.presentations.of.unit.object.in.gen.Cpn.Z.mods} completes the proof.
\end{proof}

\section{From virtual $\Cyclic_{p^n}$-representations to Picard genuine $\Cyclic_{p^n}$-$\ZZ$-modules}
\label{section.fxn.reps.to.Pic}

In this section, we describe the composite abelian group homomorphism
\[
\RO(\Cyclic_{p^n})
\xra{V \longmapsto \SS^V}
\Pic(\Spectra^{\gen \Cyclic_{p^n}})
\xra{(-) \otimes \ZZ}
\Pic(\Mod^{\gen \Cyclic_{p^n}}_\ZZ)
\]
in terms of the identification of the target $\Pic(\Mod^{\gen \Cyclic_{p^n}}_\ZZ)$ given by \Cref{thm.picard.group}. More precisely, in \Cref{thm.from.reps.to.Pic} (which is proved at the end of the section, and which was originally stated as the second part of \Cref{intro.thm.Pic.of.gen.Cpn.Z.mods}) we describe in these terms its values on irreducible $\Cyclic_{p^n}$-representations, which freely generate $\RO(\Cyclic_{p^n})$ (as an abelian group).

\begin{definition}
We use the term \bit{$\Cyclic_{p^n}$-representation} to mean a finite-dimensional real orthogonal $\Cyclic_{p^n}$-representation (or equivalently a homomorphism $\Cyclic_{p^n} \ra \sO(d)$ for some $d \geq 0$). Isomorphism classes of $\Cyclic_{p^n}$-representations form a commutative monoid under direct sum, whose group completion we denote by $\RO(\Cyclic_{p^n})$. We refer to the elements of $\RO(\Cyclic_{p^n})$ as \bit{virtual representations}.
\end{definition}

\begin{notation}
For any virtual $\Cyclic_{p^n}$-representation $V \in \RO(\Cyclic_{p^n})$ we write
\[
\ZZ^V
:=
\SS^V \otimes \ZZ
\in
\Pic(\Mod^{\gen \Cyclic_{p^n}}_\ZZ)
\]
for the corresponding Picard element of $\Mod^{\gen \Cyclic_{p^n}}_\ZZ$. Moreover, for any representation $\Cyclic_{p^n} \xra{\rho} \sO(d)$ we write
\[
\ZZ^\rho
:=
\ZZ^{(\RR^d,\rho)}
\in
\Pic(\Mod^{\gen \Cyclic_{p^n}}_\ZZ)
~.
\]
\end{notation}

\begin{notation}
\label{notn.for.elts.of.tilde.P}
Recall the abelian group
\[
\tilde{\PP}
:=
\ZZ^{\oplus (n+1)}
\oplus
\left(
\bigoplus_{s=1}^n
(\ZZ/p^{n-s+1})^\times
\right)
\]
of \Cref{notn.P.tilde.and.P}.
\begin{enumerate}

\item We express elements of $\tilde{\PP}$ as ordered pairs, according to the indicated direct sum decomposition.

\item We write
\[
\{ \vec{e}_0 , \vec{e}_1 , \ldots , \vec{e}_n \} \subseteq \ZZ^{ \oplus (n+1)} \subseteq \tilde{\PP}
\]
for the standard basis of the torsionfree subgroup of $\tilde{\PP}$.

\item For any $1 \leq s \leq n$, we implicitly consider the inclusion
\[
(\ZZ/p^{n-s+1})^\times
\subseteq
\left(
\bigoplus_{s=1}^n
(\ZZ/p^{n-s+1})^\times
\right)
\subseteq
\tilde{\PP}
~,
\]
and in doing so, for $\gamma \in (\ZZ/p^{n-s+1})^\times$ we simply denote by $\gamma$ the element $(1,\dots,1,\gamma,1,\dots,1)~\in~ \tilde{\PP}$ in which the $s\th$ coordinate is $\gamma$ and every other coordinate is $1$.

\item We write
\[
\vec{1}
\in
\left(
\bigoplus_{s=1}^n
(\ZZ/p^{n-s+1})^\times
\right)
\subseteq
\tilde{\PP}
\]
for the identity element of the torsion subgroup of $\tilde{\PP}$.

\end{enumerate}
So for instance, given an element $\gamma \in (\ZZ/p^n)^\times$, we write
\[
(3\vec{e}_0 - 5 \vec{e}_2 , \gamma )
:=
( (3 , 0 , -5 , 0 , \ldots , 0 ) , ( \gamma , 1 , \ldots , 1 ) )
\in
\ZZ^{\oplus (n+1)}
\oplus
\left(
\bigoplus_{s=1}^n
(\ZZ/p^{n-s+1})^\times
\right)
=:
\tilde{\PP}
~.
\]
\end{notation}

\begin{observation}
Because $p$ is odd, every $\Cyclic_{p^n}$-representation is orientable, and moreover every irreducible $\Cyclic_{p^n}$-representation has dimension 1 or 2. We use these facts without further comment.
\end{observation}

\needspace{2\baselineskip}
\begin{notation}
\label{notn.for.Cpn.irreps}
\begin{enumerate}
\item[]

\item We write
\[
\Cyclic_{p^n}
\xra{\rho_\triv}
\SO(1)
\]
for the unique 1-dimensional $\Cyclic_{p^n}$-representation.

\item For brevity, we use the identification
\[
\SO(2) \cong \sU(1)
~.
\]

\item\label{item.define.rho.j}

For each $j \in \{ 1 , \ldots, p^n-1 \}$, we write
\[
\Cyclic_{p^n}
\xlongra{\rho_j}
\sU(1)
\]
for the 2-dimensional $\Cyclic_{p^n}$-representation characterized by the fact that $\rho_j(\sigma) = e^{2\pi i j / p^n}$, where $\sigma \in \Cyclic_{p^n}$ is the standard generator.  

\item\label{item.define.k.and.gamma.from.j}

For each $j \in \{ 1 , \ldots, p^n-1 \}$, we write $\nu(j) := \nu_p(j)$ for the $p$-adic valuation of $j$, and 
\[
\gamma(j)
:=
\frac{j}{p^{\nu(j)}}
~.\footnote{In other words, the positive integers $\nu(j)$ and $\gamma(j)$ are characterized by the fact that $j = p^{\nu(j)} \cdot \gamma(j)$ where $\gamma(j)$ is coprime to $p$.}
\]
We note that $0 \leq \nu(j) < n$, and we often implicitly consider $\gamma(j) \in (\ZZ / p^{n-\nu(j)})^\times$.

\end{enumerate}
\end{notation}

We prove the following result at the end of this subsection.
\begin{theorem}
\label{thm.from.reps.to.Pic}
The Picard elements of $\Mod^{\gen \Cyclic_{p^n}}_\ZZ$ determined by the irreducible $\Cyclic_{p^n}$-representations are given by the formulas
\[
\ZZ^{\rho_\triv} = L^{(\vec{e}_0,\vec{1})}
\qquad
\text{and}
\qquad
\ZZ^{\rho_j} = L^{(2\vec{e}_0 - \vec{e}_{\nu(j)+1}, \gamma(j))}
~.
\]
\end{theorem}

\begin{definition}
We define the $\infty$-categories of \bit{naive $\Cyclic_{p^n}$-spectra} and \bit{naive $\Cyclic_{p^n}$-$\ZZ$-modules} to be
\[
\Spectra^{\naive \Cyclic_{p^n}}
:=
\Fun ( \Orb_{\Cyclic_{p^n}}^\op , \Spectra)
\qquad
\text{and}
\qquad
\Mod^{\naive \Cyclic_{p^n}}_\ZZ
:=
\Fun ( \Orb_{\Cyclic_{p^n}}^\op , \Mod_\ZZ)
~.
\]
\end{definition}

\begin{observation}
\label{obs.naive.is.tensored.up}
We have equivalences
\[
\Spectra^{\naive \Cyclic_{p^n}}
\simeq
\Spaces^{\gen \Cyclic_{p^n}}_* \otimes \Spectra
\qquad
\text{and}
\qquad
\Mod^{\naive \Cyclic_{p^n}}_\ZZ
\simeq
\Spaces^{\gen \Cyclic_{p^n}}_* \otimes \Mod_\ZZ
\]
in $\PrL$ (in fact in $\CAlg(\PrL)$).
\end{observation}

\begin{notation}
We define the functors $\Psi_\SS$, $\Psi_\ZZ$, $\tilde{\sC}(-)$, and $\tilde{\ZZ}\brax{-}$ as those appearing in the commutative diagram
\[ \begin{tikzcd}[column sep=1.5cm]
\Spaces^{\gen \Cyclic_{p^n}}_*
\arrow{r}{\Sigma^\infty}
\arrow{rd}[swap, sloped]{\Sigma^\infty_{\Cyclic_{p^n}}}
\arrow[bend left]{rr}[sloped]{\tilde{\sC}(-)}
\arrow[bend right=50]{rrd}[sloped, swap]{\tilde{\ZZ}\brax{-}}
&
\Spectra^{\naive \Cyclic_{p^n}}
\arrow{r}{(-) \otimes \ZZ}
\arrow{d}{\Psi_\SS}
&
\Mod^{\naive \Cyclic_{p^n}}_\ZZ
\arrow{d}{\Psi_\ZZ}
\\
&
\Spectra^{\gen \Cyclic_{p^n}}
\arrow{r}[swap]{(-) \otimes \ZZ}
&
\Mod^{\gen \Cyclic_{p^n}}_\ZZ
\end{tikzcd} \]
in $\PrL$, in which the functors $\Psi_\SS$ and $\Psi_\ZZ$ respectively arise from the adjunctions
\[
\begin{tikzcd}[column sep=1.5cm]
\PrL
\arrow[transform canvas={yshift=0.9ex}]{r}{(-) \otimes \Spectra}
\arrow[leftarrow, transform canvas={yshift=-0.9ex}]{r}[yshift=-0.2ex]{\bot}[swap]{\fgt}
&
\Mod_{\Spectra}(\PrL)
\end{tikzcd}
\qquad
\text{and}
\qquad
\begin{tikzcd}[column sep=1.5cm]
\PrL
\arrow[transform canvas={yshift=0.9ex}]{r}{(-) \otimes \Mod_\ZZ}
\arrow[leftarrow, transform canvas={yshift=-0.9ex}]{r}[yshift=-0.2ex]{\bot}[swap]{\fgt}
&
\Mod_{\Mod_\ZZ}(\PrL)
\end{tikzcd}
\]
using \Cref{obs.naive.is.tensored.up}.
\end{notation}

\needspace{2\baselineskip}
\begin{notation}
\begin{enumerate}
\item[]

\item

For any $\infty$-category $\cC$, we depict an object $X \in \Fun(\Orb_{\Cyclic_{p^n}}^\op , \cC)$ by the diagram
\begin{equation}
\label{typical.object.of.Fun.Orb.Cpn.op.to.C}
\begin{tikzcd}
X_n
\arrow{r}{\delta_n}
\arrow[loop below, distance=0.5cm]{}{e}
&
X_{n-1}
\arrow{r}{\delta_{n-1}}
\arrow[loop below, distance=0.5cm]{}{\Cyclic_p}
&
\cdots
\arrow{r}{\delta_2}
&
X_1
\arrow{r}{\delta_1}
\arrow[loop below, distance=0.5cm]{}{\Cyclic_{p^{n-1}}}
&
X_0
\arrow[loop below, distance=0.5cm]{}{\Cyclic_{p^n}}
\end{tikzcd}
~,
\end{equation}
in which 
\begin{itemize}

\item for every $s \in [n]$, we write $X_s := X(( \Cyclic_{p^n} / \Cyclic_{p^s} )^\circ) \in \cC$,

\item the curved arrows schematically depict actions of Weyl groups, and

\item for every $1 \leq s \leq n$, the morphism $\delta_s$ denotes the value of $X$ on the distinguished morphism
\begin{equation}
\label{morphism.in.Orb.Cpn.op.opposite.to.quotient}
(\Cyclic_{p^n} / \Cyclic_{p^s})^\circ
\longra
(\Cyclic_{p^n} / \Cyclic_{p^{s-1}} )^\circ
\end{equation}
in $\Orb_{\Cyclic_{p^n}}^\op$ opposite to the evident quotient homomorphism among quotient groups of $\Cyclic_{p^n}$.

\end{itemize}

\item

We take the convention that if the $s\th$ horizontal morphism (counting from the left) in a diagram such as \Cref{typical.object.of.Fun.Orb.Cpn.op.to.C} is unlabeled, then it corresponds to the identity morphism on underlying objects and moreover the $\Cyclic_{p^s}$-action on its target is pulled back from the $\Cyclic_{p^{s-1}}$-action on its source via the quotient homomorphism.

\item

Given an object $X \in \Fun ( \Orb_{\Cyclic_{p^n}}^\op , \cC)$ as depicted in diagram \Cref{typical.object.of.Fun.Orb.Cpn.op.to.C}, for any $1 \leq s \leq n$ we denote by
\[ \begin{tikzcd}
&
(X_{s-1})^{\htpy \Cyclic_p}
\arrow{rd}
\\
X_s
\arrow[dashed]{ru}[sloped]{\tilde{\delta}_s}
\arrow{rr}[swap]{\delta_s}
&
&
X_{s-1}
\end{tikzcd} \]
the canonical factorization (determined by the fact that the morphism \Cref{morphism.in.Orb.Cpn.op.opposite.to.quotient} is $\Cyclic_p$-equivariant with respect to the trivial $\Cyclic_p$-action on its source), a morphism in $\cC^{\htpy \Cyclic_{p^{n-s}}} := \Fun ( \sB \Cyclic_{p^{n-s}} , \cC )$.

\end{enumerate}
\end{notation}

\begin{observation}
\label{obs.genuine.Cpn.Z.module.from.naive.one}
The functor
\[
\Mod^{\naive \Cyclic_{p^n}}_\ZZ
\xra{\Psi_\ZZ}
\Mod^{\gen \Cyclic_{p^n}}_\ZZ
\]
carries the object \Cref{typical.object.of.Fun.Orb.Cpn.op.to.C} (with $\cC = \Mod_\ZZ$) to the genuine $\Cyclic_{p^n}$-$\ZZ$-module whose gluing diagram via \Cref{thm.stratn.of.genuine.Cpn.Z.mods} (and \Cref{consequences.of.stratn.of.genuine.Cpn.Z.mods}\Cref{microcosm.consequences.of.stratn.of.genuine.Cpn.Z.mods}\Cref{microcosm.consequences.of.stratn.of.genuine.Cpn.Z.mods.gluing.diagrams}) is
\[
\left(
\begin{tikzcd}
X_0
\arrow[maps to]{rdd}
&
X_1
\arrow{d}{\tilde{\delta}_1}
\arrow[maps to]{rdd}
&
\cdots
&
\cdots
\arrow[maps to]{rdd}
&
X_n
\arrow{d}{\tilde{\delta}_n}
\\
&
(X_0)^{\htpy \Cyclic_p}
\arrow{d}{\sQ_{\Cyclic_p}(X_0)}
&
\cdots
&
\cdots
&
(X_{n-1})^{\htpy \Cyclic_p}
\arrow{d}{\sQ_{\Cyclic_p}(X_{n-1})}
\\
&
(X_0)^{\st \Cyclic_p}
&
(X_1)^{\st \Cyclic_p}
&\cdots
&
(X_{n-1})^{\st \Cyclic_p}
\end{tikzcd}
\right)
~.
\]
Indeed, this follows from tom Dieck splitting along with the defining equivalence $\Psi_\ZZ \circ \tilde{\sC}(-) \simeq \tilde{\ZZ}\brax{-}$.
\end{observation}

\begin{proof}[Proof of \Cref{thm.from.reps.to.Pic}]
It is clear that
\[
\ZZ^{\rho_\triv}
:=
\SS^{\rho_\triv}
\otimes 
\ZZ
=
\Sigma \ZZ
=
L^{(\vec{e}_0,\vec{1})}
\in
\Pic(\Mod^{\gen \Cyclic_{p^n}}_\ZZ)
~.
\]
So, we fix an element $j \in \{ 1 , \ldots , p^n - 1\}$ and turn our attention to the corresponding Picard element
\[
\ZZ^{\rho_j} \in \Pic(\Mod^{\gen \Cyclic_{p^n}}_\ZZ)~.
\]
For convenience, we introduce the following notation.
\begin{itemize}

\item We simply write $\rho := \rho_j$, $\nu := \nu(j)$, and $\gamma := \gamma(j)$. So by definition, $\ZZ^\rho := \ZZ^{\rho_j}$.

\item We respectively write $S^\rho \in \Spaces^{\gen \Cyclic_{p^n}}_*$ and $S(\rho) \in \Spaces^{\gen \Cyclic_{p^n}}$ for the representation sphere and unit sphere of $\rho$.\footnote{That is, $S^\rho$ denotes the one-point compactification of $\RR^2$ and $S(\rho)$ denotes its unit circle, both considered as genuine $\Cyclic_{p^n}$-spaces (the former pointed).} So by definition, $\ZZ^\rho := \tilde{\ZZ}\brax{S^\rho}$.

\item We define $\tau \in \Cyclic_{p^{n-\nu}}$ to be the unique element such that $\tau^\gamma = \sigma$.\footnote{Explicitly, we set $\tau := \sigma^{\gamma^{-1}}$, where $\gamma^{-1}$ denotes the multiplicative inverse of the unit $\gamma$ in the commutative ring $\ZZ/p^{n-\nu}$.}

\end{itemize}

Now, we begin by observing the equivalence
\begin{equation}
\label{S.to.the.rho.is.suspension.of.S.of.rho}
S^\rho
\simeq
\Sigma S(\rho)
\end{equation}
in $\Spaces^{\gen \Cyclic_{p^n}}_*$,\footnote{This follows from the fact that at the level of topological spaces with $\Cyclic_{p^n}$-action, the representation sphere is the \textit{unreduced} suspension of the unit sphere.} as well as the equivalence
\begin{equation}
\label{CW.presentation.of.unit.sphere.in.rho}
S(\rho)
\simeq
\colim \left(
\begin{tikzcd}
\Cyclic_{p^n} / \Cyclic_{p^\nu}
\arrow[yshift=0.9ex]{r}{\id}
\arrow[yshift=-0.9ex]{r}[swap]{\tau}
&
\Cyclic_{p^n} / \Cyclic_{p^\nu}
\end{tikzcd}
\right)
\end{equation}
in $\Spaces^{\gen \Cyclic_{p^n}}$.\footnote{Indeed, the homomorphism $\Cyclic_{p^n} \xra{\rho} \SO(2) \cong \sU(1)$ factors through the subgroup $\mu_{p^{n-\nu}} \subseteq \sU(1)$ of $(p^{n-\nu})\th$ roots of unity according to the formula $\sigma \mapsto e^{2\pi i\gamma/p^{n-\nu}}$, and the unit circle has an evident $\mu_{p^{n-\nu}}$-CW complex structure that affords its description as the coequalizer
\[
\colim \left(
\begin{tikzcd}[ampersand replacement=\&, column sep=1.5cm]
\mu_{p^{n-\nu}} / e
\arrow[yshift=0.9ex]{r}{\id}
\arrow[yshift=-0.9ex]{r}[swap]{e^{2\pi i / p^{n-\nu}}}
\&
\mu_{p^{n-\nu}} / e
\end{tikzcd}
\right)
\in \Spaces^{\gen \mu_{p^{n-\nu}}}
~.
\]
} We use these to give an explicit presentation of the object $\tilde{\sC}(S^\rho) \in \Mod^{\naive \Cyclic_{p^n}}_\ZZ$. Namely, we define an object $\tilde{\ttC}(S^\rho) \in \Fun ( \Orb_{\Cyclic_{p^n}}^\op , \Ch_\ZZ)$ as follows. First of all, we define the object
\[
\ttY
:=
\left(
\cdots
\longra
0
\longra
\ZZ [ \Cyclic_{p^{n-\nu}} ]
\xra{1 - \tau}
\ZZ [ \Cyclic_{p^{n-\nu}} ]
\xra{\sigma \mapsto 1}
\uwave{\ZZ}
\longra
0 
\longra
\cdots
\right)
\in
\Ch_{\ZZ[\Cyclic_{p^{n-\nu}}]}
\simeq
\Fun ( \sB \Cyclic_{p^{n-\nu}} , \Ch_\ZZ )
\]
(where in degree 0 we endow $\ZZ$ with the trivial $\Cyclic_{p^{n-\nu}}$-action). Moreover, we simply write
\[
\uwave{\ZZ}
:=
\left( \cdots \longra 0 \longra \uwave{\ZZ} \longra 0 \longra \cdots \right)
\in
\Ch_\ZZ
~.
\]
Now, consider the morphism
\[
\uwave{\ZZ}
\xlongra{\ttz}
\ttY
\]
in $\Ch_\ZZ$ that is the identity in degree 0, which is evidently $\Cyclic_{p^{n-\nu}}$-equivariant with respect to the trivial $\Cyclic_{p^{n-\nu}}$-action on the source. We then define the functor
\[
\tilde{\ttC}(S^{\rho})
:
\Orb_{\Cyclic_{p^n}}^\op
\longra
(\sB \Cyclic_{p^{n-\nu}})^\lcone
\xlongra{\ttz}
\Ch_\ZZ
~,
\]
where the first morphism is the right adjoint retraction onto the full subcategory on the objects
\[
(\Cyclic_{p^n} / \Cyclic_{p^n})^\circ , (\Cyclic_{p^n} / \Cyclic_{p^\nu})^\circ \in \Orb_{\Cyclic_{p^n}}^\op
~.\footnote{So we may depict $\tilde{\ttC}(S^\rho) \in \Fun ( \Orb_{\Cyclic_{p^n}}^\op , \Ch_\ZZ )$ as the diagram
\[ \begin{tikzcd}[ampersand replacement=\&]
\uwave{\ZZ}
\arrow{r}
\arrow[loop below, distance=0.5cm]{}{e}
\&
\uwave{\ZZ}
\arrow{r}
\arrow[loop below, distance=0.5cm]{}{\Cyclic_p}
\&
\cdots
\arrow{r}
\&
\uwave{\ZZ}
\arrow{r}{\ttz}
\arrow[loop below, distance=0.5cm]{}{\Cyclic_{p^{n-\nu-1}}}
\&
\ttY
\arrow{r}
\arrow[loop below, distance=0.5cm]{}{\Cyclic_{p^{n-\nu}}}
\&
\cdots
\arrow{r}
\&
\ttY
\arrow[loop below, distance=0.5cm]{}{\Cyclic_{p^n}}
\end{tikzcd}
~.
\]
}
\]
It is clear from the equivalences \Cref{S.to.the.rho.is.suspension.of.S.of.rho} \and \Cref{CW.presentation.of.unit.sphere.in.rho} that $\tilde{\ttC}(S^\rho)$ is indeed a presentation of $\tilde{\sC}(S^\rho)$, i.e.\! that we have the assignment
\[ \begin{tikzcd}[row sep=0cm]
\Fun ( \Orb_{\Cyclic_{p^n}}^\op , \Ch_\ZZ )
\arrow{r}
&
\Fun ( \Orb_{\Cyclic_{p^n}}^\op , \Mod_\ZZ )
\\
\rotatebox{90}{$\in$}
&
\rotatebox{90}{$\in$}
\\
\tilde{\ttC}(S^{\rho})
\arrow[maps to]{r}
&
\tilde{\sC}(S^{\rho})
\end{tikzcd}~. \]

In order to proceed, we introduce the following additional notation.
\begin{itemize}

\item We write $Y := \Pi_\infty(\ttY) \in \Mod_\ZZ^{\htpy \Cyclic_{p^{n-\nu}}}$ for the underlying object of $\ttY \in \Ch_{\ZZ[\Cyclic_{p^{n-\nu}}]}$.

\item We write $\ZZ \xra{z} Y$ for the underlying morphism in $\Mod_\ZZ$ of the morphism $\uwave{\ZZ} \xra{\ttz} \ttY$ in $\Ch_\ZZ$, which is likewise $\Cyclic_{p^{n-\nu}}$-equivariant with respect to the trivial $\Cyclic_{p^{n-\nu}}$-action on the source.\footnote{So, we may depict $\tilde{\sC}(S^\rho) \in \Mod^{\naive \Cyclic_{p^n}}_\ZZ$ as the diagram
\[ \begin{tikzcd}[ampersand replacement=\&]
\ZZ
\arrow{r}
\arrow[loop below, distance=0.5cm]{}{e}
\&
\ZZ
\arrow{r}
\arrow[loop below, distance=0.5cm]{}{\Cyclic_p}
\&
\cdots
\arrow{r}
\&
\ZZ
\arrow{r}{z}
\arrow[loop below, distance=0.5cm]{}{\Cyclic_{p^{n-\nu-1}}}
\&
Y
\arrow{r}
\arrow[loop below, distance=0.5cm]{}{\Cyclic_{p^{n-\nu}}}
\&
\cdots
\arrow{r}
\&
Y
\arrow[loop below, distance=0.5cm]{}{\Cyclic_{p^n}}
\end{tikzcd}
~.
\]
}

\end{itemize}
Now, applying \Cref{obs.genuine.Cpn.Z.module.from.naive.one}, we find that the genuine $\Cyclic_{p^n}$-$\ZZ$-module
\[
\ZZ^\rho
:=
\tilde{\ZZ}\brax{S^\rho}
\simeq
\Psi_\ZZ ( \tilde{\sC}(S^\rho))
\in
\Mod^{\gen \Cyclic_{p^n}}_\ZZ
\]
has gluing diagram
\[
\left(
\begin{tikzcd}
Y
\arrow[maps to]{rdd}
&
Y
\arrow{d}{\tilde{\id}}
\arrow[maps to]{rdd}
&
\cdots
&
Y
\arrow[maps to]{rdd}
&
\ZZ
\arrow{d}{\tilde{z}}
\arrow[maps to]{rdd}
&
\ZZ
\arrow{d}{\tilde{\id}}
\arrow[maps to]{rdd}
&
\cdots
&
\ZZ
\arrow[maps to]{rdd}
&
\ZZ
\arrow{d}{\tilde{\id}}
\\
&
Y^{\htpy \Cyclic_p}
\arrow{d}{\sQ_{\Cyclic_p}(Y)}
&
\cdots
&
\cdots
&
Y^{\htpy \Cyclic_p}
\arrow{d}{\sQ_{\Cyclic_p}(Y)}
&
\ZZ^{\htpy \Cyclic_p}
\arrow{d}{\sQ_{\Cyclic_p}(\ZZ)}
&
\cdots
&
\cdots
&
\ZZ^{\htpy \Cyclic_p}
\arrow{d}{\sQ_{\Cyclic_p}(\ZZ)}
\\
&
Y^{\st \Cyclic_p}
&
Y^{\st \Cyclic_p}
&
\cdots
&
Y^{\st \Cyclic_p}
&
\ZZ^{\st \Cyclic_p}
&
\ZZ^{\st \Cyclic_p}
&
\cdots
&
\ZZ^{\st \Cyclic_p}
\end{tikzcd}
\right)
~.
\]
Noting the equivalence $Y \simeq \Sigma^2 \ZZ \in \Mod^{\htpy \Cyclic_{p^{n-\nu}}}_\ZZ$, we find that
\[
\ZZ^\rho
=
L^{(2\vec{e}_0-\vec{e}_{\nu+1},\sQ_{\Cyclic_p}(Y) \circ \tilde{z} )}
\in
\Pic(\Mod^{\gen \Cyclic_{p^n}}_\ZZ)
~.
\]
So it remains to show that
\[
\sQ_{\Cyclic_p}(Y) \circ \tilde{z}
=
\gamma
\in
(\ZZ/p^{n-\nu})^\times
~.
\]
For this, we define a quasi-isomorphism
\[
\ttY
\xlongra{\approx}
\ttY'
\]
in $\Ch_{\ZZ[\Cyclic_{p^{n-\nu}}]}$ as
\[ \begin{tikzcd}
\cdots
\arrow{r}
&
0
\arrow{r}
\arrow{d}
&
\ZZ[\Cyclic_{p^{n-\nu}}]
\arrow{r}{1-\tau}
\arrow{d}{1}
&
\ZZ[\Cyclic_{p^{n-\nu}}]
\arrow{r}{\sigma \mapsto 1}
\arrow{d}{1+\tau + \cdots + \tau^{\gamma-1}}
&
\uwave{\ZZ}
\arrow{r}
\arrow{d}{\gamma N}
&
0
\arrow{r}
\arrow{d}
&
\cdots
\\
\cdots
\arrow{r}
&
0
\arrow{r}
&
\ZZ[\Cyclic_{p^{n-\nu}}]
\arrow{r}[swap]{1-\sigma}
&
\ZZ[\Cyclic_{p^{n-\nu}}]
\arrow{r}[swap]{N}
&
\uwave{\ZZ[\Cyclic_{p^{n-\nu}}]}
\arrow{r}[swap]{1-\sigma}
&
\ZZ[\Cyclic_{p^{n-\nu}}]
\arrow{r}[swap]{N}
&
\cdots
\end{tikzcd}~. \]
From this, we see that the composite morphism
\[
\uwave{\ZZ}
\xlongra{\ttz}
\ttY
\xlongra{\approx}
\ttY'
\]
in $\Ch_{\ZZ[\Cyclic_{p^{n-\nu}}]}$ (where the source is endowed with the trivial $\Cyclic_{p^{n-\nu}}$-action) selects the element
\[
\gamma \cdot c_{n-\nu}
\in
\pi_{-2} ( \ZZ^{\htpy \Cyclic_{p^{n-\nu}}})
\cong
\pi_0 \ulhom_{\Mod_\ZZ} ( \ZZ , \Sigma^2 \ZZ^{\htpy \Cyclic_{p^{n-\nu}}} )
\cong
\pi_0 \ulhom_{\Mod^{\htpy \Cyclic_{p^{n-\nu}}}_\ZZ} ( \ZZ , \Sigma^2 \ZZ )
~,
\]
which proves the claim.
\end{proof}

\section{The constant Mackey functor at $\ZZ$}
\label{section.Z.underline}

In this section, we describe the constant Mackey functor at $\ZZ$ (i.e.\! the coefficients for equivariant cohomology) in terms of our stratification; this is given as \Cref{prop.describe.gluing.diagram.of.Z.underline}. We also record a resulting identification of $\THH(\FF_p)$ as \Cref{cor.identify.THH.FP}.

\begin{notation}
\label{notn.constant.mackey.functor.at.Z}
We write
\[
\ul{\ZZ}
\in
\Mack_{\Cyclic_{p^n}}(\Ab)
\subseteq
\Mack_{\Cyclic_{p^n}}(\Mod_\ZZ)
\simeq
\Mod_\ZZ^{\gen \Cyclic_{p^n}}
\]
for the constant Mackey functor at $\ZZ$, considered as a genuine $\Cyclic_{p^n}$-$\ZZ$-module (using \Cref{obs.mackey.for.genuine.G.objects}): its categorical fixedpoints for any subgroup $\Cyclic_{p^s} \leq \Cyclic_{p^n}$ are
\[
\ul{\ZZ}^{\Cyclic_{p^s}}
:=
\ZZ
\in
\Mod^{\htpy \Cyclic_{p^{n-s}}}_\ZZ
\]
(i.e.\! the abelian group $\ZZ \in \Ab \subseteq \Mod_\ZZ$ equipped with the trivial $\Cyclic_{p^{n-s}}$-action), and for any morphism
\[
\Cyclic_{p^n}/\Cyclic_{p^s}
\longra
\Cyclic_{p^n}/\Cyclic_{p^t}
\]
in $\Orb_{\Cyclic_{p^n}}$ its corresponding restriction and transfer maps are respectively
\[
\ul{\ZZ}^{\Cyclic_{p^t}}
:=
\ZZ
\xlongra{1}
\ZZ
=:
\ul{\ZZ}^{\Cyclic_{p^s}}
\qquad
\text{and}
\qquad
\ul{\ZZ}^{\Cyclic_{p^s}}
:=
\ZZ
\xra{p^{t-s}}
\ZZ
=:
\ul{\ZZ}^{\Cyclic_{p^t}}
~.\footnote{The factor $p^{t-s}$ arises as the index $|\Cyclic_{p^t} : \Cyclic_{p^r}|$.}
\]
For convenience, more generally for any $i \geq 0$ we also simply write
\[
\ul{\ZZ} \in \Mack_{\Cyclic_{p^i}}(\Ab) \subseteq \Mod^{\gen \Cyclic_{p^i}}_\ZZ
\]
for the constant Mackey functor at $\ZZ$.
\end{notation}

\begin{observation}
For every $s \in [n]$, there is an evident identification
\[ \begin{tikzcd}[row sep=0cm]
\Mod^{\gen \Cyclic_{p^n}}_\ZZ
\arrow{r}{(-)^{\Cyclic_{p^s}}}
&
\Mod^{\gen \Cyclic_{p^{n-s}}}_\ZZ
\\
\rotatebox{90}{$\in$}
&
\rotatebox{90}{$\in$}
\\
\ul{\ZZ}
\arrow[maps to]{r}
&
\ul{\ZZ}
\end{tikzcd}~. \]
We use this fact without further comment.
\end{observation}

\begin{notation}
Let $\cC$ be a stable $\infty$-category equipped with a t-structure. We write $\cC \xra{\tau_{\geq 0}} \cC$ for the connective cover functor, and for any object $E \in \cC$ we write
\[
\tau_{\geq 0} E
\xra{\varepsilon_{\geq 0}(E)}
E
\]
for the canonical morphism to it from its connective cover. For simplicity, we write $\varepsilon_{\geq 0} := \varepsilon_{\geq 0}(E)$ or even $\varepsilon := \varepsilon_{\geq 0}$ when the meaning is clear from context.
\end{notation}

\begin{prop}
\label{prop.describe.gluing.diagram.of.Z.underline}
Via the geometric stratification of $\Mod^{\gen \Cyclic_{p^n}}_\ZZ$ of \Cref{thm.stratn.of.genuine.Cpn.Z.mods} (and \Cref{consequences.of.stratn.of.genuine.Cpn.Z.mods}\Cref{microcosm.consequences.of.stratn.of.genuine.Cpn.Z.mods}\Cref{microcosm.consequences.of.stratn.of.genuine.Cpn.Z.mods.gluing.diagrams}), the object $\ul{\ZZ} \in \Mod^{\gen \Cyclic_{p^n}}_\ZZ$ is recorded by the data of the gluing diagram
\[ \begin{tikzcd}
\ZZ
\arrow[maps to]{rd}
&
\tau_{\geq 0} \ZZ^{\st \Cyclic_p}
\arrow[maps to]{rd}
\arrow{d}{\varepsilon}
&
\tau_{\geq 0} \ZZ^{\st \Cyclic_p}
\arrow[maps to]{rd}
\arrow{d}{\varepsilon}
&
\cdots
&
\cdots
\arrow[maps to]{rd}
&
\tau_{\geq 0} \ZZ^{\st \Cyclic_p}
\arrow{d}{\varepsilon}
\\
&
\ZZ^{\st \Cyclic_p}
&
\ZZ^{\st \Cyclic_p}
&
\ZZ^{\st \Cyclic_p}
&
\cdots
&
\ZZ^{\st \Cyclic_p}
\end{tikzcd}~. \]
\end{prop}

\begin{observation}
\label{obs.morphism.between.isotropy.separation.sequences}
Fix any $i \geq 1$. Consider the unit morphism
\begin{equation}
\label{unit.morphism.from.Z.underline.to.its.borel.completion}
\ul{\ZZ}
\longra
\beta U \ul{\ZZ}
 \simeq
\beta \ZZ
\end{equation}
in $\Mod^{\gen \Cyclic_{p^i}}_\ZZ$. Taking categorical $\Cyclic_p$-fixedpoints of the morphism \Cref{unit.morphism.from.Z.underline.to.its.borel.completion}, we obtain the morphism
\[
\ul{\ZZ}^{\Cyclic_p}
:=
\ZZ
\xlongra{\varepsilon}
\ZZ^{\htpy \Cyclic_p}
\simeq
(\beta \ZZ)^{\Cyclic_p}
\]
in $\Mod^{\htpy \Cyclic_{p^{i-1}}}_\ZZ$. Hence, on isotropy separation sequences the morphism \Cref{unit.morphism.from.Z.underline.to.its.borel.completion} determines a commutative diagram
\begin{equation}
\label{morphism.between.isotropy.separation.sequences}
\begin{tikzcd}[column sep=1.5cm]
&
\ul{\ZZ}^{\Cyclic_p}
\arrow{r}
&
\Phi^{\Cyclic_p} \ul{\ZZ}
\\[-0.7cm]
&
\rotatebox{90}{$=:$}
&
\rotatebox{90}{$\simeq$}
\\[-0.7cm]
\ZZ_{\htpy \Cyclic_p}
\arrow{r}
\arrow[equals]{d}
&
\ZZ
\arrow{r}
\arrow{d}{\varepsilon}
&
\tau_{\geq 0} \ZZ^{\st \Cyclic_p}
\arrow{d}{\varepsilon}
\\
\ZZ_{\htpy \Cyclic_p}
\arrow{r}[swap]{\Nm_{\Cyclic_p}(\ZZ)}
&
\ZZ^{\htpy \Cyclic_p}
\arrow{r}[swap]{\sQ_{\Cyclic_p}(\ZZ)}
&
\ZZ^{\st \Cyclic_p}
\\[-0.7cm]
&
\rotatebox{90}{$\simeq$}
&
\rotatebox{90}{$\simeq$}
\\[-0.7cm]
&
(\beta \ZZ)^{\Cyclic_p}
\arrow{r}
&
\Phi^{\Cyclic_p} \beta \ZZ
\end{tikzcd}
\end{equation}
in $\Mod^{\htpy \Cyclic_{p^{i-1}}}_\ZZ$ (and in particular the equivalence $\Phi^{\Cyclic_p} \ul{\ZZ} \simeq \tau_{\geq 0} \ZZ^{\st \Cyclic_p}$), in which both horizontal composites are cofiber sequences and hence the middle right square is a pullback.
\end{observation}

\begin{notation}
Fix any $i \geq 1$. We write
\[
\ZZ_{\htpy \Cyclic_p}
\xra{\genNm}
\ZZ
\xra{\genQuot}
\tau_{\geq 0} \ZZ^{\st \Cyclic_p}
\]
for the morphisms in the upper horizontal composite in diagram \Cref{morphism.between.isotropy.separation.sequences} in $\Mod^{\htpy \Cyclic_{p^{i-1}}}_\ZZ$.
\end{notation}

\begin{observation}
\label{obs.tate.Cp.of.Z.to.conn.cover.of.Z.tate.Cp.is.an.equivalence}
Fix any $i \geq 2$. By Observations \ref{obs.morphism.between.isotropy.separation.sequences} \and \ref{obs.tate.vanishing.for.Z.mods}, applying the $\Cyclic_p$-Tate construction to the morphism
\[
\ZZ
\xra{\genQuot}
\tau_{\geq 0} \ZZ^{\st \Cyclic_p}
\]
in $\Mod^{\htpy \Cyclic_{p^{i-1}}}_\ZZ$ yields an equivalence
\[
\ZZ^{\st \Cyclic_p}
\xra[\sim]{\genQuot^{\st \Cyclic_p}}
(\tau_{\geq 0} \ZZ^{\st \Cyclic_p})^{\st \Cyclic_p}
\]
in $\Mod^{\htpy \Cyclic_{p^{i-2}}}_\ZZ$.
\end{observation}

\begin{proof}[Proof of \Cref{prop.describe.gluing.diagram.of.Z.underline}]
We give the proof assuming that $n \geq 3$, proving the cases $0 \leq n \leq 2$ implicitly along the way. It is clear that the gluing diagram of $\ul{\ZZ} \in \Mod^{\gen \Cyclic_{p^n}}_\ZZ$ restricts as asserted over $0 \in [n]$. The fact that it restricts as asserted over $1 \in [n]$ follows from \Cref{obs.morphism.between.isotropy.separation.sequences} (in the case that $i=n$). Thereafter, the fact that it restricts as asserted over $2 \in [n]$ follows from the commutative diagram
\[ \begin{tikzcd}
\Phi^{\Cyclic_p} \ul{\ZZ}
\arrow{d}
&[-1cm]
=:
\Phi^{\Cyclic_p} ( \ul{\ZZ}^{\Cyclic_p})
\arrow{r}{\sim}
&
\Phi^{\Cyclic_p} (\Phi^{\Cyclic_p}_\gen \ul{\ZZ})
\arrow{d}
\\
\ZZ^{\st \Cyclic_p}
&
\hspace{-0.4cm}
=:
(\ul{\ZZ}^{\Cyclic_p})^{\st \Cyclic_p}
\arrow{r}[swap]{\sim}
&
(\Phi^{\Cyclic_p} \ul{\ZZ})^{\st \Cyclic_p}
\end{tikzcd} \]
in $\Mod^{\htpy \Cyclic_{p^{n-2}}}_\ZZ$, in which the upper horizontal morphism is an equivalence by \Cref{lem.projection.formula.categorical.then.geometric.is.geometric} and the lower horizontal morphism is an equivalence by Observations \ref{obs.morphism.between.isotropy.separation.sequences} \and \ref{obs.tate.Cp.of.Z.to.conn.cover.of.Z.tate.Cp.is.an.equivalence} (in the case that $i=n$). Thereafter, the fact that it restricts as asserted over $s \in [n]$ for any $s \geq 3$ follows from the commutative diagram 
\[ \begin{tikzcd}
\Phi^{\Cyclic_{p^2}}(\ul{\ZZ})
\arrow{d}
&[-1.1cm]
=:
\Phi^{\Cyclic_p} ( \Phi^{\Cyclic_p}_\gen ( \ul{\ZZ}^{\Cyclic_{p^{s-2}}} ) )
\arrow{r}{\sim}
&
\Phi^{\Cyclic_{p^s}} \ul{\ZZ}
\arrow{d}
\\
(\Phi^{\Cyclic_p} \ul{\ZZ})^{\st \Cyclic_p}
&
=:
( \Phi^{\Cyclic_p} ( \ul{\ZZ}^{\Cyclic_{p^{s-2}}} ) )^{\st \Cyclic_p}
\arrow{r}[swap]{\sim}
&
(\Phi^{\Cyclic_{p^{s-1}}} \ul{\ZZ})^{\st \Cyclic_p}
\end{tikzcd} \]
in $\Mod^{\htpy \Cyclic_{p^{n-s}}}_\ZZ$, in which both horizontal morphisms are equivalences by \Cref{lem.projection.formula.categorical.then.geometric.is.geometric}.
\end{proof}

\begin{cor}
\label{cor.identify.THH.FP}
Considering $\ul{\ZZ}$ as an object of the $\infty$-category
\[
\Spectra^{\gen \Cyclic_{p^\infty}}
:=
\lim
\left(
\cdots
\xra{\Res^{\Cyclic_{p^3}}_{\Cyclic_{p^2}}}
\Spectra^{\gen \Cyclic_{p^2}}
\xra{\Res^{\Cyclic_{p^2}}_{\Cyclic_{p}}}
\Spectra^{\gen \Cyclic_p}
\xra{\Res^{\Cyclic_{p}}_{e}}
\Spectra
\right)
\]
in the evident way, there is an equivalence
\[
\Phi^{\Cyclic_p}(\ul{\ZZ})
\simeq
\THH(\FF_p)
\]
in $\Spectra^{\gen \Cyclic_{p^\infty}}$.
\end{cor}

\begin{proof}
This follows by combining \Cref{prop.describe.gluing.diagram.of.Z.underline} with \cite[Corollary IV.4.16 et seq.]{NS}.
\end{proof}

\begin{remark}
\label{rmk.wonder.about.significance.of.Bokstedt}
In fact, the object $\ul{\ZZ} \in \Spectra^{\gen \Cyclic_{p^\infty}}$ naturally carries the structure of a $p$-typical topological Cartier module in the sense of \cite{AN-topCart}; thereafter, \Cref{cor.identify.THH.FP} gives an equivalence of $p$-cyclotomic spectra. We would be very interested in a conceptual explanation of this equivalence, which would give a conceptual explanation for B\"okstedt's computation \cite{Bok-THHFp} of $\pi_* \THH(\FF_p)$.
\end{remark}

\section{The $\Pic(\Mod^{\gen \Cyclic_{p^n}}_\ZZ)$-graded cohomology of a point}
\label{section.Cpn.eqvrt.cohomology}

In this section, we establish our explicit chain-level description of equivariant cohomology (\Cref{intro.thm.cohomology}) as \Cref{thm.describe.ho.Mod.Z.valued.mackey.functor}. We begin by recalling the notation established in \Cref{subsection.intro.computations} as \Cref{notn.main.body.of.paper.eqvrt.coh}.

\begin{remark}
In this section, we freely use the material of \Cref{section.homological.algebra}. We refer the reader to Notations \ref{notation.all.chain.level.stuff} \and \ref{notn.shifted.endo.of.T} for Lemmas \ref{lemma.values.of.mackey.functor} \and \ref{lemma.res.of.mackey.functor}, and additionally to Notations \ref{notn.S.chain.complex}, \ref{notn.B.tilde.complex}, \and \ref{notn.for.nullhtpy} for \Cref{lemma.trf.of.mackey.functor}. Furthermore, we refer the reader to \Cref{rmk.mnemonics.for.ch.cxes} for some mnemonics regarding these various notations.
\end{remark}

\begin{notation}
\label{notn.main.body.of.paper.eqvrt.coh}
Given a finite group $G$, a genuine $G$-space $X \in \Spaces^{\gen G}$, a subgroup $H \leq G$, and a Picard element $L \in \Pic(\Mod^{\gen G}_\ZZ)$, we write
\[
\sC^L_G(X)(H)
:=
\ulhom_{\Mod^{\gen G}_\ZZ} ( \Sigma^\infty_G (X \times G/H)_+ \otimes \ZZ , L \otimes \ul{\ZZ} )
\in
\Mod_\ZZ
~.
\]
This assembles as a $\Mod_\ZZ$-valued Mackey functor for $G$, i.e.\! an additive functor
\[
\Burn_G
\xra{\sC^L_G(X)}
\Mod_\ZZ
\]
carrying $G/H \in \Burn_G$ to $\sC^L_G(X)(H) \in \Mod_\ZZ$. Thereafter, for any integer $i \in \ZZ$ we write
\[
\sH^{i+L}_G(X)(H)
:=
\pi_{-i} ( \sC^L_G(X)(H))
\in
\Ab
~.
\]
This assembles as an $\Ab$-valued Mackey functor for $G$, i.e.\! an additive functor
\[
\Burn_G
\xra{\sH^{i+L}_G(X)}
\Ab
\]
carrying $G/H \in \Burn_G$ to $\sH^{i+L}_G(X)(H) \in \Ab$.
\end{notation}

\begin{observation}
Let $\cA$ be an additive category. Then, $\cA$-valued Mackey functors for $G$ are equivalent data to their values on $G/H$ along with inclusion and transfer morphisms satisfying a double coset formula. We use this fact without further comment.
\end{observation}

\begin{local}
We fix a Picard element $L \in \Pic(\Mod^{\gen \Cyclic_{p^n}}_\ZZ)$, as well as a presentation $L = L^{(\vec{\beta},\vec{\gamma})}$ for some $(\vec{\beta},\vec{\gamma}) \in \tilde{\PP}$ as afforded by \Cref{thm.picard.group} (see Notations \ref{notn.P.tilde.and.P}, \ref{notn.factorize.K.bullet.to.L.bullet}, and \ref{notn.for.elts.of.tilde.P}). We furthermore fix arbitrary lifts $\tilde{\gamma}_s \in \ZZ$ of the elements $\gamma_s \in (\ZZ/p^{n-s+1})^\times$.
\end{local}

\begin{theorem}
\label{thm.describe.ho.Mod.Z.valued.mackey.functor}
The $\ho(\Mod_\ZZ)$-valued Mackey functor
\[
\Burn_{\Cyclic_{p^n}}
\xra{\sC^L_{\Cyclic_{p^n}}(\pt)}
\Mod_\ZZ
\longra
\ho(\Mod_\ZZ)
\]
admits an explicit chain-level description as follows: its values, restriction maps, and transfer maps are respectively as described in Lemmas \ref{lemma.values.of.mackey.functor}, \ref{lemma.res.of.mackey.functor}, \and \ref{lemma.trf.of.mackey.functor}.
\end{theorem}

\begin{proof}
This follows from Lemmas \ref{lemma.values.of.mackey.functor}, \ref{lemma.res.of.mackey.functor}, \and \ref{lemma.trf.of.mackey.functor}.
\end{proof}


\begin{notation}
For any $1 \leq i \leq n$, we write
\[
\beta_{\leq i}
:=
\beta_0 + 2(\beta_1 + \cdots + \beta_i)
~,
\]
for simplicity.
\end{notation}

\begin{lemma}
\label{lemma.values.of.mackey.functor}
For any $0 \leq a \leq n$, diagram \Cref{zigzag.diagram.whose.limit.is.genuine.Cpa.fixedpoints} in $\Mod^{\htpy \Cyclic_{p^{n-a}}}_\ZZ$ in the case that $E = L^{(\vec{\beta},\vec{\gamma})} \otimes \ul{\ZZ}$ is presented by the diagram
\begin{equation}
\label{chain.level.zigzag.diagram.whose.limit.is.genuine.Cpa.fixedpoints}
\begin{tikzcd}[row sep=2cm, column sep=1.5cm]
\Sigma^{\beta_0} \ttZ^a_{n-a}
\arrow{rd}[sloped, swap]{\ttq^{a-1}_{n-a}}
&
\Sigma^{\beta_{\leq 1}} \ttC^{a-1}_{n-a}
\arrow{rd}[sloped, swap]{\ttg^{a-2}_{n-a}}
\arrow{d}
&
\cdots
\arrow{d}
&
\cdots
\arrow{rd}[sloped, swap]{\ttg^0_{n-a}}
&
\Sigma^{\beta_{\leq a}} \ttC^0_{n-a}
\arrow{d}
\\
&
\Sigma^{\beta_0} \ttT^{a-1}_{n-a}
&
\Sigma^{\beta_{\leq 1}} \ttT^{a-2}_{n-a}
&
\cdots
&
\Sigma^{\beta_{\leq(a-1)}} \ttT^0_{n-a}
\end{tikzcd}
\end{equation}
in $\Ch_{\ZZ[\Cyclic_{p^{n-a}}]}$, in which for all $1 \leq s \leq a$ the $s\th$ vertical morphism is the composite
\begin{equation}
\label{description.of.sth.vertical.morphism.in.diagram.for.categorical.fixedpoints.as.a.composite}
\Sigma^{\beta_{\leq s}} \ttC^{a-s}_{n-a}
\xra{\tte^{a-s}_{n-a}}
\Sigma^{\beta_{\leq s}} \ttT^{a-s}_{n-a}
\xra{\tilde{\gamma}_s (\ttc^{a-s}_{n-a})^{\beta_s}}
\Sigma^{\beta_{\leq (s-1)}} \ttT^{a-s}_{n-a}
~.
\end{equation}
In particular, by \Cref{prop.fixedpts.res.and.trf.for.general.Cpn.Z.mods}\Cref{prop.fixedpts.res.and.trf.for.general.Cpn.Z.mods.the.fixedpts} the object
\[
\sC^L_{\Cyclic_{p^n}}(\pt)(\Cyclic_{p^a})
:=
(L^{(\vec{\beta},\vec{\gamma})} \otimes \ul{\ZZ})^{\Cyclic_{p^a}}
\in
\ho(\Mod_\ZZ)
\]
is the homotopy limit of the diagram \Cref{chain.level.zigzag.diagram.whose.limit.is.genuine.Cpa.fixedpoints}.
\end{lemma}

\begin{proof}
The fact that the constituents and the diagonal morphisms of diagram \Cref{chain.level.zigzag.diagram.whose.limit.is.genuine.Cpa.fixedpoints} in $\Ch_{\ZZ[\Cyclic_{p^{n-a}}]}$ present those of diagram \Cref{zigzag.diagram.whose.limit.is.genuine.Cpa.fixedpoints} in $\Mod^{\htpy \Cyclic_{p^{n-a}}}_\ZZ$ follows from \Cref{prop.describe.gluing.diagram.of.Z.underline} and \Cref{lem.ABX.htpy.invariantly}. So, it remains to show that for any $1 \leq s \leq a$ the morphism \Cref{description.of.sth.vertical.morphism.in.diagram.for.categorical.fixedpoints.as.a.composite} in $\Ch_{\ZZ[\Cyclic_{p^{n-a}}]}$ is a presentation of the morphism
\begin{equation}
\label{homotopical.sth.vertical.map.for.gluing.diagram.for.categorical.fixedpoints}
\Sigma^{\beta_{\leq (s-1)}}
\left(
\Sigma^{2 \beta_s} \ZZ \otimes \tau_{\geq 0} \ZZ^{\st \Cyclic_p}
\xra{\gamma_s c_s^{\beta_s} \otimes \varepsilon}
\ZZ^{\st \Cyclic_p} \otimes \ZZ^{\st \Cyclic_p}
\xlongra{\mu}
\ZZ^{\st \Cyclic_p}
\right)^{\htpy \Cyclic_{a-s}}
\end{equation}
in $\Mod^{\htpy \Cyclic_{p^{n-a}}}_\ZZ$, in which the parenthesized morphism lies in $\Mod^{\htpy \Cyclic_{p^{n-s}}}_\ZZ$ and $\mu$ denotes the multiplication morphism resulting from \Cref{obs.various.properties.of.tate}\Cref{part.rlax.s.m.str.on.h.to.tate}. For this, observe the commutative diagram
\begin{equation}
\label{homotopical.diagram.rewriting.multiplication.as.endo}
\begin{tikzcd}[column sep=1.5cm, row sep=1.5cm]
\Sigma^{2 \beta_s} \ZZ \otimes \tau_{\geq 0} \ZZ^{\st \Cyclic_p}
\arrow{rr}{\gamma_s c_s^{\beta_s} \otimes \varepsilon}
\arrow{rd}[sloped, swap]{\tilde{\gamma}_s \otimes \varepsilon}
&
&
\ZZ^{\st \Cyclic_p} \otimes \ZZ^{\st \Cyclic_p}
\arrow{r}{\mu}
&
\ZZ^{\st \Cyclic_p}
\\
&
\Sigma^{2 \beta_s} \ZZ \otimes \ZZ^{\st \Cyclic_p}
\arrow{ru}[sloped, swap]{c_s^{\beta_s} \otimes \id_{\ZZ^{\st \Cyclic_p}}}
\arrow[dashed, bend right]{rru}[sloped, swap]{c_s^{\beta_s}}
\end{tikzcd}
\end{equation}
in $\Mod^{\htpy \Cyclic_{p^{n-s}}}_\ZZ$ (in which we use \Cref{notn.mult.by.an.eqvrt.htpy.elt} for the dashed morphism). It follows from Lemmas \ref{lem.ABX.htpy.invariantly} \and \ref{lemma.present.mult.by.chern.class.on.tate} that the lower composite in diagram \Cref{homotopical.diagram.rewriting.multiplication.as.endo} is presented by the morphism
\[
\Sigma^{2 \beta_s}
\ttC^0_{n-s}
\xra{\tilde{\gamma}_s \tte_{n-s}^0}
\Sigma^{2 \beta_s}
\ttT^0_{n-s}
\xra{(\ttc_{n-s}^0)^{\beta_s}}
\ttT^0_{n-s}
\]
in $\Ch_{\ZZ[\Cyclic_{p^{n-s}}]}$. Hence, it follows from \Cref{lemma.ABX.are.adapted} and Observations \ref{obs.fixedpoints.of.ABX.are.themselves}, \ref{obs.fixedpoints.of.chain.level.c.is.itself}, \ref{obs.basics.of.adaptedness.to.fixedpoints}\Cref{item.operations.preserve.adapted.to.fixedpoints}\Cref{subitem.adapted.to.fixedpoints.preserved.under.shifts}, \and \ref{obs.apply.tensor.with.bimodule.correct.if.to.htpy.Cpj.fixedpts} that the morphism \Cref{homotopical.sth.vertical.map.for.gluing.diagram.for.categorical.fixedpoints} in $\Mod^{\htpy \Cyclic_{p^{n-a}}}$ is indeed presented by the morphism \Cref{description.of.sth.vertical.morphism.in.diagram.for.categorical.fixedpoints.as.a.composite} in $\Ch_{\ZZ[\Cyclic_{p^{n-a}}]}$.
\end{proof}

\newcommand{\septtfromcdotsforchcxes}{1cm}
\newcommand{\sepcdotsfromcdotsforchcxes}{1.8cm}

\begin{lemma}
\label{lemma.res.of.mackey.functor}
For any $0 \leq a < n$, the morphism of diagrams \Cref{general.description.of.inclusion.map.via.zigzags} in $\Mod^{\htpy \Cyclic_{p^{n-a}}}_\ZZ$ in the case that $E = L^{(\vec{\beta},\vec{\gamma})} \otimes \ul{\ZZ}$ is presented by the morphism of diagrams
\begin{equation}
\label{coh.inclusion.map.via.zigzags}
\begin{tikzcd}[row sep=1cm, column sep=-0.2cm]
\Sigma^{\beta_0} \ttZ^{a+1}_{n-a-1}
\arrow{rd}
\arrow{dd}[swap, pos=0.7]{\tthinc}
&[\septtfromcdotsforchcxes]
&[\sepcdotsfromcdotsforchcxes]
\cdots
\arrow{ld}
\arrow{rd}
\arrow{dd}[swap, pos=0.7]{\tthinc}
&[\septtfromcdotsforchcxes]
&
\Sigma^{\beta_{\leq a}} \ttC^1_{n-a-1}
\arrow{ld}
\arrow{rd}
\arrow{dd}[swap, pos=0.7]{\tthinc}
&
&
\Sigma^{\beta_{\leq (a+1)}} \ttC^0_{n-a-1}
\arrow{ld}
\arrow{dd}
\\
&
\cdots
\arrow{dd}[swap, pos=0.4]{\tthinc}
&
&
\Sigma^{\beta_{\leq(a-1)}} \ttT^1_{n-a-1}
\arrow{dd}[swap, pos=0.4]{\tthinc}
&
&
\Sigma^{\beta_{\leq a}} \ttT^0_{n-a-1}
\arrow{dd}
\\
\Sigma^{\beta_0} \ttZ^a_{n-a}
\arrow{rd}
&
&
\cdots
\arrow{ld}
\arrow{rd}
&
&
\Sigma^{\beta_{\leq a}} \ttC^0_{n-a}
\arrow{ld}
\arrow{rd}
&
&
0
\arrow{ld}
\\
&
\cdots
&
&
\Sigma^{\beta_{\leq(a-1)}} \ttT^0_{n-a}
&
&
0
\end{tikzcd}
\end{equation}
in $\Ch_{\ZZ[\Cyclic_{p^{n-a}}]}$, in which the nontrivial non-vertical maps are as in diagram \Cref{chain.level.zigzag.diagram.whose.limit.is.genuine.Cpa.fixedpoints} and we implicitly apply \Cref{obs.fixedpoints.of.ABX.are.themselves}. In particular, by \Cref{lemma.values.of.mackey.functor} and \Cref{prop.fixedpts.res.and.trf.for.general.Cpn.Z.mods}\Cref{prop.fixedpts.res.and.trf.for.general.Cpn.Z.mods.the.restriction}, the inclusion morphism
\[
\sC^L_{\Cyclic_{p^n}}(\pt)(\Cyclic_{p^{a+1}})
\xlongra{\inc}
\sC^L_{\Cyclic_{p^n}}(\pt)(\Cyclic_{p^{a}})
\]
in $\Mod_\ZZ$ is the induced map on homotopy limits.
\end{lemma}

\begin{proof}
Note that all terms in the lower zigzag of diagram \Cref{coh.inclusion.map.via.zigzags} are adapted to homotopy $\Cyclic_p$-fixedpoints by \Cref{lemma.ABX.are.adapted} and \Cref{obs.basics.of.adaptedness.to.fixedpoints}\Cref{item.operations.preserve.adapted.to.fixedpoints}\Cref{subitem.adapted.to.fixedpoints.preserved.under.shifts}. Hence, the claim follows from \Cref{obs.point.set.res.and.tr.often.give.htpy.res.and.tr}.
\end{proof}

\begin{lemma}
\label{lemma.trf.of.mackey.functor}
For any $0 \leq a < n$, the morphism of diagrams \Cref{general.description.of.trasfer.map.via.zigzags} in $\Mod^{\htpy \Cyclic_{p^{n-a}}}_\ZZ$ in the case that $E = L^{(\vec{\beta},\vec{\gamma})} \otimes \ul{\ZZ}$ is presented by the homotopy-coherent morphism of diagrams
\begin{equation}
\label{coh.transfer.map.via.zigzags}
\begin{tikzcd}[row sep=1cm, column sep=-0.2cm]
\Sigma^{\beta_0} \ttZ^{a+1}_{n-a-1}
\arrow{rd}
&[\septtfromcdotsforchcxes]
&[\sepcdotsfromcdotsforchcxes]
\cdots
\arrow{ld}
\arrow{rd}
&[\septtfromcdotsforchcxes]
&
\Sigma^{\beta_{\leq a}} \ttC^1_{n-a-1}
\arrow{ld}
\arrow{rd}
&
&
\Sigma^{\beta_{\leq (a+1)}} \ttC^0_{n-a-1}
\arrow{ld}
\\
&
\cdots
&
&
\Sigma^{\beta_{\leq(a-1)}} \ttT^1_{n-a-1}
&
&
\Sigma^{\beta_{\leq a}} \ttT^0_{n-a-1}
\\
\Sigma^{\beta_0} \ttZ^a_{n-a}
\arrow{rd}
\arrow{uu}[pos=0.3]{\tthtrf}
&
&
\cdots
\arrow{ld}
\arrow{rd}
\arrow{uu}[pos=0.3]{\tthtrf}
&
&
\Sigma^{\beta_{\leq a}} \w{\ttC}_{n-a}
\arrow{ld}
\arrow{rd}
\arrow{uu}[pos=0.3]{\tthtrf \ttk_{n-a}}[swap, xshift=0.8cm, yshift=-0.9cm]{\tth}
&
&
0
\arrow{ld}
\arrow{uu}
\\
&
\cdots
\arrow{uu}[pos=0.6]{\tthtrf}
&
&
\Sigma^{\beta_{\leq(a-1)}} \ttT^0_{n-a}
\arrow{uu}[pos=0.6]{\tthtrf}
&
&
0
\arrow{uu}
\end{tikzcd}
\end{equation}
in $\Ch_{\ZZ[\Cyclic_{p^{n-a}}]}$, in which
\begin{itemize}

\item in the lower zigzag we replace $\ttC^0_{n-a}$ with $\w{\ttC}_{n-a}$ via the quasi-isomorphism $\ttk_{n-a}$ of \Cref{obs.cones.get.previous.cxes}\Cref{item.B.tilde.a.presents.Baa}, so that the nontrivial non-vertical morphism out of it is the composite
\[
\Sigma^{\beta_{\leq a}} \w{\ttC}_{n-a}
\xra[\approx]{\ttk_{n-a}}
\Sigma^{\beta_{\leq a}} \ttC^0_{n-a}
\xra{\tte^0_{n-a}}
\Sigma^{\beta_{\leq a}} \ttT^0_{n-a}
\xra{\tilde{\gamma}_a (\ttc^0_{n-a})^{\beta_a}}
\Sigma^{\beta_{\leq (a-1)}} \ttT^0_{n-a}
~,
\]

\item the remaining nontrivial non-vertical morphisms are as in diagram \Cref{chain.level.zigzag.diagram.whose.limit.is.genuine.Cpa.fixedpoints}, and

\item we implicitly apply \Cref{obs.fixedpoints.of.ABX.are.themselves}.
\end{itemize}
In particular, by \Cref{lemma.values.of.mackey.functor} and \Cref{prop.fixedpts.res.and.trf.for.general.Cpn.Z.mods}\Cref{prop.fixedpts.res.and.trf.for.general.Cpn.Z.mods.the.transfer}, the transfer morphism
\[
\sC^L_{\Cyclic_{p^n}}(\pt)(\Cyclic_{p^{a}})
\xlongra{\trf}
\sC^L_{\Cyclic_{p^n}}(\pt)(\Cyclic_{p^{a+1}})
\]
in $\Mod_\ZZ$ is the induced map on homotopy limits.
\end{lemma}

\begin{proof}
Note that all terms in the lower zigzag of diagram \Cref{coh.transfer.map.via.zigzags} are adapted to homotopy $\Cyclic_p$-fixedpoints by \Cref{lemma.ABX.are.adapted} and \Cref{obs.basics.of.adaptedness.to.fixedpoints}\Cref{item.operations.preserve.adapted.to.fixedpoints}\Cref{subitem.adapted.to.fixedpoints.preserved.under.shifts}. Hence, the claim follows from \Cref{obs.point.set.res.and.tr.often.give.htpy.res.and.tr} and \Cref{lemma.chain.level.nullhtpy.presents.canonical.nullhtpy}.
\end{proof}

\appendix

\section{Some homological algebra}
\label{section.homological.algebra}

In this appendix, we collect the various results in homological algebra that underlie the work carried out in the main body of the paper -- primarily our cohomology computation of the $\Pic(\Mod^{\gen \Cyclic_{p^n}}_\ZZ)$-graded cohomology of a point (\Cref{intro.thm.cohomology}). It is organized as follows.
\begin{itemize}

\item[\Cref{subsection.notation.and.conventions.for.homological.algebra}:] We lay out our notations and conventions regarding homological algebra.

\item[\Cref{subsection.adaptedness}:] We study the notion of \textit{adaptedness} of chain complexes with respect to various homotopical operations (a close analog of co/fibrancy in a model category).

\item[\Cref{subsection.some.chain.complexes}:] We introduce most of the chain complexes and chain maps that participate in our computation of equivariant cohomology, and prove that they present the correct homotopical data.

\item[\Cref{subsection.inclusion.and.transfer.stuff}:] We study the chain complexes introduced in \Cref{subsection.some.chain.complexes} in more depth, towards establishing chain-level data that presents the inclusion and transfer maps on equivariant cohomology.

\item[\Cref{subsection.a.coh.ring}:] We compute a certain Tate cohomology ring, and we prove a related result identifying a chain-level automorphism as a presentation of a resulting homotopical automorphism.

\end{itemize}

\subsection{Basic notation and conventions}
\label{subsection.notation.and.conventions.for.homological.algebra}

In this subsection, we establish the basic notation and conventions that we use for chain complexes in this paper. (These are only used in \S\S \ref{section.homological.algebra} \and \ref{section.Cpn.eqvrt.cohomology}.)

\begin{warning}
In this section, we discuss chain complexes of discrete modules over discrete rings. We simply use the word ``module'' to refer both to discrete modules and to module spectra; our meaning will always be clear from context.
\end{warning}

\begin{notation}
\label{notn.chain.cxes}
Fix an associative ring $R$.
\begin{enumerate}

\item We write $\Ch_R$ for the category of chain complexes of $R$-modules.

\item We freely identify elements of $R$ with endomorphisms of $R$-modules.

\item In depicting a chain complex, we indicate its degree-0 term with a squiggled underline.

\item We generally denote point-set objects (e.g.\! chain complexes and morphisms between them) using typerwriter text font.

\item For typographical reasons, we depict chain complexes horizontally and morphisms of chain complexes vertically.

\item When indexing a $\ZZ$-indexed family of morphisms (e.g.\! the differentials in a chain complex, the constituent morphisms in a chain map, or the constituent morphisms in a chain homotopy), we always number a morphism according to the degree of its \textit{source}. So for example, a complex $\ttM \in \Ch_R$ may be depicted as a diagram
\[
\cdots
\xra{\partial^\ttM_2}
\ttM_1
\xra{\partial^\ttM_1}
\uwave{\ttM_0}
\xra{\partial^\ttM_0}
\ttM_{-1}
\xra{\partial^\ttM_{-1}}
\cdots
~.
\]

\item\label{item.convention.for.cones} Given a morphism $\ttM \xra{\ttf} \ttN$ in $\Ch_R$, we take the convention that its cone is the complex $\ChCone(\ttf) \in \Ch_R$ such that
\[
\ChCone(\ttf)_n
:=
\left(
\def\arraystretch{1}
\begin{array}{c}
\ttM_{n-1}
\\
\oplus
\\
\ttN_n
\end{array}
\right)
\qquad
\text{and}
\qquad
\partial^{\ChCone(\ttf)}_n
:=
\left(
\def\arraystretch{1}
\begin{array}{cc}
\partial^\ttM_{n-1}
&
0
\\
(-1)^{n-1} \ttf_{n-1}
&
\partial^\ttN_n
\end{array}
\right)
~.
\]

\item\label{item.convention.for.shift} We write $\Ch_R \xra{\Sigma} \Ch_R$ for the autoequivalence given by shifting (even though it is not literally the suspension functor on the category $\Ch_R$).

\item We do not notationally distinguish between a chain map and its shifts, since (given the source and target of the shifted morphism) the meaning is unambiguous.

\end{enumerate}
\end{notation}

\begin{remark}
Note that parts \Cref{item.convention.for.cones} \and \Cref{item.convention.for.shift} of \Cref{notn.chain.cxes} are consistent: there is an evident isomorphism $\Sigma \cong \ChCone(\id_R \ra 0)$ in $\Fun ( \Ch_R,\Ch_R)$.
\end{remark}

\begin{notation}
Given a dg-algebra $\ttR \in \Alg(\Ch_\ZZ)$, we write $\Ch_\ttR := \LMod_\ttR(\Ch_\ZZ)$ for the category of (left) dg-modules over it.
\end{notation}

\begin{observation}
By \cite[Theorem 1.4]{DugShip-dga}, for any dg-algebra $\ttR \in \Alg(\Ch_\ZZ)$, freely inverting the quasi-isomorphisms in the category $\Ch_\ttR$ yields the $\infty$-category $\Mod_{\Pi_\infty(\ttR)} := \Mod_{\Pi_\infty(\ttR)}(\Mod_\ZZ)$. In particular, for any finite group $G$, freely inverting the quasi-isomorphisms in the category $\Ch_{\ZZ[G]}$ yields the $\infty$-category $\Mod_{\ZZ[G]} \simeq \Mod_\ZZ^{\htpy G}$. We use these facts without further comment.
\end{observation}

\begin{notation}
For any dg-algebra $\ttR \in \Alg(\Ch_\ZZ)$, we write $\Ch_{\ttR} \xra{\Pi_\infty} \Mod_{\Pi_\infty(\ttR)}$ for the localization functor.
\end{notation}

\subsection{Adaptedness to homotopy fixedpoints and homotopy orbits}
\label{subsection.adaptedness}

In this subsection, we study the notion of \textit{adaptedness} of chain complexes to homotopy fixedpoints and homotopy orbits. The primary output is \Cref{obs.point.set.res.and.tr.often.give.htpy.res.and.tr}, which establishes that for a chain complex that is adapted to homotopy fixedpoints, its point-set inclusion and transfer morphisms are presentations of its homotopical inclusion and transfer morphisms.

\begin{local}
In this subsection, we fix a finite group $G$ and a subgroup $H \leq G$.
\end{local}

\begin{definition}
Observe the canonical natural transformation
\[ \begin{tikzcd}[column sep=1.5cm]
\Ch_{\ZZ[G]}
\arrow{d}[swap]{\Pi_\infty}
\arrow{r}{(-)^{G}}[swap, yshift=-0.5cm]{\rotatebox{45}{$\Leftarrow$}}
&
\Ch_\ZZ
\arrow{d}{\Pi_\infty}
\\
\Mod_\ZZ^{\htpy G}
\arrow{r}[swap]{(-)^{\htpy G}}
&
\Mod_\ZZ
\end{tikzcd}~. \]
We say that a chain complex $\ttM \in \Ch_{\ZZ[G]}$ is \bit{adapted to homotopy $G$-fixedpoints} if the morphism
\[
\Pi_\infty(\ttM^G)
\longra
\Pi_\infty(\ttM)^{\htpy G}
\]
is an equivalence. More generally, we say that $\ttM \in \Ch_{\ZZ[G]}$ is adapted to homotopy $H$-fixedpoints if its image under the forgetful functor $\Ch_{\ZZ[G]} \xra{\fgt} \Ch_{\ZZ[H]}$ is so.
\end{definition}

\needspace{2\baselineskip}
\begin{observation}
\label{obs.basics.of.adaptedness.to.fixedpoints}
\begin{enumerate}
\item[]

\item\label{item.free.mods.are.adapted.to.fixedpoints}

A free $\ZZ[G]$-module concentrated in degree zero is adapted to homotopy $G$-fixedpoints.

\item\label{item.operations.preserve.adapted.to.fixedpoints}

Complexes that are adapted to homotopy $G$-fixedpoints are stable under taking

\begin{enumeratesub}

\item\label{subitem.adapted.to.fixedpoints.preserved.under.shifts}

shifts,

\item\label{subitem.adapted.to.fixedpoints.preserved.under.cones}

cones, and

\item\label{subitem.adapted.to.fixedpoints.preserved.under.holims}

homotopy limits.

\end{enumeratesub}

\item\label{item.bdd.above.free.adapted.to.fixedpoints}

Combining parts \Cref{item.free.mods.are.adapted.to.fixedpoints} and \Cref{item.operations.preserve.adapted.to.fixedpoints}, it follows that a bounded-above levelwise-free complex of $\ZZ[G]$-modules is adapted to homotopy $G$-fixedpoints. Indeed, this follows by examining its presentation as a homotopy limit of its truncations from below.

\end{enumerate}
\end{observation}

\begin{definition}
Observe the canonical natural transformation
\[ \begin{tikzcd}[column sep=1.5cm]
\Ch_{\ZZ[G]}
\arrow{d}[swap]{\Pi_\infty}
\arrow{r}{(-)_{G}}[swap, yshift=-0.5cm]{\rotatebox{45}{$\Rightarrow$}}
&
\Ch_\ZZ
\arrow{d}{\Pi_\infty}
\\
\Mod_\ZZ^{\htpy G}
\arrow{r}[swap]{(-)_{\htpy G}}
&
\Mod_\ZZ
\end{tikzcd}~. \]
We say that a chain complex $\ttM \in \Ch_{\ZZ[G]}$ is \bit{adapted to homotopy $G$-orbits} if the morphism
\[
\Pi_\infty(\ttM)_{\htpy G}
\longra
\Pi_\infty(\ttM_{G})
\]
is an equivalence. More generally, we say that $\ttM \in \Ch_{\ZZ[G]}$ is adapted to homotopy $H$-orbits if its image under the forgetful functor $\Ch_{\ZZ[G]} \xra{\fgt} \Ch_{\ZZ[H]}$ is so.
\end{definition}

\needspace{2\baselineskip}
\begin{observation}
\label{obs.basics.of.adaptedness.to.orbits}
\begin{enumerate}
\item[]

\item\label{item.free.mods.are.adapted.to.orbits}

A free $\ZZ[G]$-module concentrated in degree zero is adapted to homotopy $G$-orbits.

\item\label{item.operations.preserve.adapted.to.orbits}

Complexes that are adapted to homotopy $G$-orbits are stable under taking

\begin{enumeratesub}

\item\label{subitem.adapted.to.orbits.preserved.under.shifts}

shifts,

\item\label{subitem.adapted.to.orbits.preserved.under.cones}

cones, and

\item\label{subitem.adapted.to.orbits.preserved.under.holims}

homotopy colimits.

\end{enumeratesub}

\item\label{item.bdd.below.free.adapted.to.orbits}

Combining parts \Cref{item.free.mods.are.adapted.to.orbits} and \Cref{item.operations.preserve.adapted.to.orbits}, it follows that a bounded-below levelwise-free complex of $\ZZ[G]$-modules is adapted to homotopy $G$-orbits. Indeed, this follows by examining its presentation as a homotopy colimit of its truncations from above.

\end{enumerate}
\end{observation}

\begin{lemma}
\label{lemma.take.radjts.get.LKan}
Consider a commutative square
\[
\begin{tikzcd}
\tilde{\cD}
\arrow{d}[swap]{M}
&
\tilde{\cC}
\arrow{l}[swap]{\tilde{F}}
\arrow{d}{L}
\\
\cD
&
\cC
\arrow{l}{F}
\end{tikzcd}
\]
in $\Cat$ in which the functor $L$ is a localization and the functors $\tilde{F}$ and $F$ admit right adjoints. Then, the resulting diagram
\[
\begin{tikzcd}
\tilde{\cD}
\arrow{r}{\tilde{F}^R}
\arrow{d}[swap]{M}[xshift=0.5cm]{\rotatebox{45}{$\Leftarrow$}}
&
\tilde{\cC}
\arrow{r}{L}
&
\cC
\\
\cD
\arrow[bend right=10]{rru}[sloped, swap]{F^R}
\end{tikzcd}
\qquad
:=
\qquad
\begin{tikzcd}
\tilde{\cD}
\arrow{r}{\tilde{F}^R}[swap, yshift=-0.35cm]{\rotatebox{45}{$\Leftarrow$}}
\arrow{d}[swap]{M}
&
\tilde{\cC}
\arrow{d}{L}
\\
\cD
\arrow{r}[swap]{F^R}
&
\cC
\end{tikzcd}
\]
is a left Kan extension diagram.
\end{lemma}

\begin{proof}
Consider the diagram
\[ \begin{tikzcd}[row sep=1.5cm, column sep=1.5cm]
\Fun(\tilde{\cC} , \cC )
\arrow[dashed, transform canvas={yshift=0.9ex}]{r}
\arrow[hookleftarrow, transform canvas={yshift=-0.9ex}]{r}[yshift=-0.2ex]{\bot}[swap]{L^*}
\arrow[xshift=-0.9ex]{d}[swap]{(\tilde{F}^R)^*}
\arrow[leftarrow, xshift=0.9ex]{d}[anchor=north, sloped, yshift=0.2ex]{\top}{\tilde{F}^*}
&
\Fun ( \cC , \cC)
\arrow[xshift=-0.9ex]{d}[swap]{(F^R)^*}
\arrow[leftarrow, xshift=0.9ex]{d}[anchor=north, sloped, yshift=0.2ex]{\top}{F^*}
\\
\Fun ( \tilde{\cD} , \cC)
\arrow[dashed, transform canvas={yshift=0.9ex}]{r}
\arrow[leftarrow, transform canvas={yshift=-0.9ex}]{r}[yshift=-0.2ex]{\bot}[swap]{M^*}
&
\Fun ( \cD , \cC)
\end{tikzcd} \]
in $\Cat$, in which the right adjoints commute and the dashed horizontal left adjoints may only be partially defined. It suffices to observe that we have assignments
\[ \begin{tikzcd}
L
\arrow[maps to]{r}
\arrow[maps to]{d}
&
\id_\cC
\arrow[maps to]{d}
\\
L \circ \tilde{F}^R
&
F^R
\end{tikzcd}
~,
\]
by the uniqueness of (partially defined) left adjoints.
\end{proof}

\begin{cor}
\label{cor.tensor.with.bimodule.correct.if}
Let $\ttR,\ttS \in \Alg(\Ch_\ZZ)$ be dg-algebras, let $\ttM \in \BiMod_{(\ttS,\ttR)}(\Ch_\ZZ)$ be a dg-bimodule, and suppose that the functor
\[
\Ch_\ZZ
\xra{\ttM \otimes_\ZZ (-)}
\Ch_\ZZ
\]
preserves quasi-isomorphisms. Then, the canonical diagram
\[ \begin{tikzcd}
\Ch_\ttS
\arrow{r}{\ulhom_{\Ch_\ttS}( \ttM , - ) }
\arrow{d}[swap]{\Pi_\infty}[xshift=1cm]{\rotatebox{45}{$\Leftarrow$}}
&[0.5cm]
\Ch_\ttR
\arrow{r}{\Pi_\infty}
&
\Mod_{\Pi_\infty(\ttR)}
\\
\Mod_{\Pi_\infty(\ttS)}
\arrow[bend right=10]{rru}[sloped, swap]{\ulhom_{\Mod_{\Pi_\infty(\ttS)}}(\Pi_\infty(\ttM),-)}
\end{tikzcd} \]
is a left Kan extension.
\end{cor}

\begin{proof}
To simplify our notation, we write $R := \Pi_\infty(\ttR)$, $S := \Pi_\infty(\ttS)$, and $M := \Pi_\infty(\ttM)$.

Because the functor $\Mod_R \xra{\fgt} \Mod_\ZZ$ is colimit-preserving and conservative, it suffices to show that the diagram
\[ \begin{tikzcd}
\Ch_\ttS
\arrow{r}{\ulhom_{\Ch_\ttS}(\ttM,-)}
\arrow{d}[swap]{\Pi_\infty}[xshift=1cm]{\rotatebox{45}{$\Leftarrow$}}
&[1cm]
\Ch_\ttR
\arrow{r}{\Pi_\infty}
&
\Mod_R
\arrow{r}{\fgt}
&
\Mod_\ZZ
\\
\Mod_S
\arrow[bend right=10]{rrru}[sloped, swap]{\ulhom_{\Mod_S}(M,-)}
\end{tikzcd} \]
is a left Kan extension, which is equivalent to the condition that the diagram
\begin{equation}
\label{later.diagram.that.must.be.LKE.only.involving.Zmods.not.Rmods}
\begin{tikzcd}
\Ch_\ttS
\arrow{r}{\ulhom_{\Ch_\ttS}(\ttM,-)}
\arrow{d}[swap]{\Pi_\infty}[xshift=1cm]{\rotatebox{45}{$\Leftarrow$}}
&[1cm]
\Ch_\ZZ
\arrow{r}{\Pi_\infty}
&
\Mod_\ZZ
\\
\Mod_S
\arrow[bend right=10]{rru}[sloped, swap]{\ulhom_{\Mod_S}(M,-)}
\end{tikzcd}
\end{equation}
is a left Kan extension. Now, the canonical lax-commutative square
\[ \begin{tikzcd}[column sep=1.5cm]
\Ch_\ttS
\arrow{d}[swap]{\Pi_\infty}
&
\Ch_\ZZ
\arrow{l}[swap]{\ttM \otimes_\ZZ (-)}[yshift=-0.4cm]{\rotatebox{-45}{$\Leftarrow$}}
\arrow{d}{\Pi_\infty}
\\
\Mod_S
&
\Mod_\ZZ
\arrow{l}{M \otimes_\ZZ (-)}
\end{tikzcd} \]
commutes by our assumption (e.g.\! by taking a projective resolution of $\ttM \in \Ch_\ZZ$ to compute the derived tensor product (i.e.\! the tensor product in $\Mod_\ZZ$)). Hence, the fact that the diagram \Cref{later.diagram.that.must.be.LKE.only.involving.Zmods.not.Rmods} is a left Kan extension follows from \Cref{lemma.take.radjts.get.LKan}.
\end{proof}

\begin{observation}
\label{obs.apply.tensor.with.bimodule.correct.if.to.htpy.Cpj.fixedpts}
We apply \Cref{cor.tensor.with.bimodule.correct.if} in the case that $\ttS = \ZZ [ G ] ,\ttR = \ZZ[\Weyl(H)] \in \Alg(\Ab) \subset \Alg(\Ch_\ZZ)$ and $\ttM = \ZZ [ G/H ]$. Then, we have identifications
\[
\ulhom_{\Ch_{\ZZ[G]}}( \ZZ[G/H] , - )
\simeq
(-)^{H}
\qquad
\text{and}
\qquad
\ulhom_{\Mod_\ZZ^{\htpy G}} ( \ZZ[G/H] , - )
\simeq
(-)^{\htpy H}
~.
\]
As $\ZZ[G/H] \in \Ab \subset \Ch_\ZZ$ is a free abelian group, we find that the canonical diagram
\begin{equation}
\label{map.from.strict.Cpj.fixedpts.to.htpy.Cpj.fixedpts}
\begin{tikzcd}
\Ch_{\ZZ[G]}
\arrow{r}{(-)^H}
\arrow{d}[swap]{\Pi_\infty}[xshift=1cm]{\rotatebox{45}{$\Leftarrow$}}
&
\Ch_{\ZZ[\Weyl(H)]}
\arrow{r}{\Pi_\infty}
&
\Mod_\ZZ^{\htpy \Weyl(H)}
\\
\Mod_\ZZ^{\htpy G}
\arrow[bend right=10]{rru}[sloped, swap]{(-)^{\htpy \Cyclic_j}}
\end{tikzcd}
\end{equation}
is a left Kan extension. In particular, it follows that for any morphism $\ttM \xra{\ttf} \ttN$ in $\Ch_{\ZZ[G]}$, if $\ttM$ and $\ttN$ are adapted to homotopy $H$-fixedpoints then the morphism $\ttf^{H}$ in $\Ch_{\ZZ[\Weyl(H)]}$ is a presentation of the morphism $(\Pi_\infty(\ttf))^{\htpy H}$ in $\Mod^{\htpy \Weyl(H)}_\ZZ$.
\end{observation}

\begin{notation}
Consider the morphisms
\begin{equation}
\label{corepresent.strict.inclusion}
\ZZ[G/e]
\xra{1 \longmapsto 1}
\ZZ[G/H]
\end{equation}
and
\begin{equation}
\label{corepresent.strict.transfer}
\ZZ[G/H]
\xra{1 \longmapsto \sum_{h \in H} h}
\ZZ[G/e]
\end{equation}
in $\Mod_{\ZZ[G]}(\Ab) \subset \Mod_{\ZZ[G]}(\Ch_\ZZ)$. For any $\ttM \in \Ch_{\ZZ[G]}$, we write
\[
\left(
\ttM
\xla{\tthinc}
\ttM^H
\right)
:=
\ulhom_{\Ch_{\ZZ[G]}} ( \Cref{corepresent.strict.inclusion} , \ttM )
\qquad
\text{and}
\qquad
\left(
\ttM^H
\xla{\tthtrf}
\ttM
\right)
:=
\ulhom_{\Ch_{\ZZ[G]}} ( \Cref{corepresent.strict.transfer} , \ttM )
\]
for the indicated morphisms in $\Ch_{\ZZ[\Normzer(H)]}$.\footnote{Recall from Notations \ref{notn.htpy.inc} \and \ref{notn.htpy.trf} that in the case that $G = \Cyclic_{p^i}$ and $H = \Cyclic_p$, for any $M \in \Mod^{\htpy \Cyclic_{p^i}}_\ZZ$ we likewise simply write
\[
\left(
M
\xla{\hinc}
M^{\htpy \Cyclic_p}
\right)
:=
\ulhom_{\Mod_\ZZ^{\htpy \Cyclic_{p^i}}} ( \Cref{corepresent.strict.inclusion} , M )
\qquad
\text{and}
\qquad
\left(
M^{\htpy \Cyclic_p}
\xla{\htrf}
M
\right)
:=
\ulhom_{\Mod_\ZZ^{\htpy \Cyclic_{p^i}}} ( \Cref{corepresent.strict.transfer} , M )
\]
for the indicated morphisms.}
\end{notation}

\begin{observation}
\label{obs.manifest.functoriality.in.cor.tensor.with.bimodule.correct.if}
\Cref{cor.tensor.with.bimodule.correct.if} (and in particular \Cref{obs.apply.tensor.with.bimodule.correct.if.to.htpy.Cpj.fixedpts}) is functorial in the dg-bimodule $\ttM$. That is, there is a natural transformation
\[
(\Pi_\infty)_!
( \Pi_\infty ( \ulhom_{\Ch_\ttS} ( - , = ) ) )
\longra
\ulhom_{\Mod_{\Pi_\infty(\ttS)}}(\Pi_\infty(-) , = )
\]
in $\Fun ( \BiMod_{(\ttS,\ttR)}(\Ch_\ZZ)^\op \times \Ch_\ttS , \Mod_{\Pi_\infty(R)} )$, which is an equivalence when restricting to those dg-bimodules $\ttM$ such that the functor $\Ch_\ZZ \xra{\ttM \otimes_\ZZ (-)} \Ch_\ZZ$ preserves quasi-isomorphisms.
\end{observation}

\begin{observation}
\label{obs.point.set.res.and.tr.often.give.htpy.res.and.tr}
Fix a subgroup $H \leq G$. Using \Cref{obs.manifest.functoriality.in.cor.tensor.with.bimodule.correct.if}, we apply \Cref{obs.apply.tensor.with.bimodule.correct.if.to.htpy.Cpj.fixedpts} to the morphisms \Cref{corepresent.strict.inclusion} and \Cref{corepresent.strict.transfer} to obtain for each $\ttM \in \Ch_{\ZZ[G]}$ natural commutative squares
\[
\begin{tikzcd}[column sep=1.5cm]
\Pi_\infty(\ttM^H)
\arrow{r}{\Pi_\infty(\tthinc)}
\arrow{d}
&
\Pi_\infty(\ttM^e)
\arrow{d}[sloped, anchor=south]{\sim}
\\
\Pi_\infty(\ttM)^{\htpy H}
\arrow{r}[swap]{\hinc}
&
\Pi_\infty(\ttM)^{\htpy e}
\end{tikzcd}
\qquad
\text{and}
\qquad
\begin{tikzcd}[column sep=1.5cm]
\Pi_\infty(\ttM^e)
\arrow{r}{\Pi_\infty(\tthtrf)}
\arrow{d}[sloped, anchor=north]{\sim}
&
\Pi_\infty(\ttM^H)
\arrow{d}
\\
\Pi_\infty(\ttM)^{\htpy e}
\arrow{r}[swap]{\htrf}
&
\Pi_\infty(\ttM)^{\htpy H}
\end{tikzcd}
\]
in $\Mod_\ZZ^{\htpy G}$, in which all vertical morphisms are components of natural transformations of the form \Cref{map.from.strict.Cpj.fixedpts.to.htpy.Cpj.fixedpts}. In particular, if $\ttM$ is adapted to homotopy $H$-fixedpoints, then we obtain canonical identifications $\hinc \simeq \Pi_\infty(\tthinc)$ and $\htrf \simeq \Pi_\infty(\tthtrf)$ in $\Ar(\Mod_\ZZ^{\htpy \Weyl(H)})$.
\end{observation}

\subsection{Chain-level data for cohomology groups}
\label{subsection.some.chain.complexes}

In this subsection, we introduce and study the specific chain complexes and most of the chain maps that participate in our computation of the values of equivariant cohomology (\Cref{notation.all.chain.level.stuff}). The main output is \Cref{lem.ABX.htpy.invariantly}, which establishes the underlying homotopical content of these chain-level data.

\begin{remark}
\label{rmk.mnemonics.for.ch.cxes}
We have tried to make the notation that is introduced in the remainder of the paper mnemonical, as we now describe.
\begin{itemize}

\item The number $a$ is that appearing in \S\S\ref{section.stratn.of.gen.Cpn.Z.mods} and \ref{section.Cpn.eqvrt.cohomology}, where it participates in the categorical fixedpoints functor $(-)^{\Cyclic_{p^a}}$.

\item The letter $r$ stands for ``residual equivariance''. This is generally recorded as a subscript.

\item Superscripts generally record ``the number of times that homotopy $\Cyclic_p$-fixedpoints has been taken''. Correspondingly, we use subscripts on the left to record ``the number of times that homotopy $\Cyclic_p$-orbits has been taken''.

\item The letter $\ttZ$ stands for ``integers'', and the letter $\ttS$ is chosen because it is similar to $\ttZ$.

\item The letter $\ttT$ stands for ``Tate''.

\item The letter $\ttC$ stands for ``connective cover of Tate''.

\item The letter $\ttq$ stands for ``quotient'', and the letter $\ttg$ is chosen because it is similar to $\ttq$.

\item The letter $\tte$ is chosen because it is similar to $\varepsilon$.

\item The letter $\ttc$ stands for ``Chern class''.

\end{itemize}
\end{remark}

\begin{notation}
We fix nonnegative integers $a,r \geq 0$.
\end{notation}

\begin{definition}
\label{defn.norm.elt}
The \bit{norm element} for the group $\Cyclic_{p^r}$ is the element
\[
N
:=
N_r
:=
\left( \sum_{i=1}^{p^r} \sigma^i \right)
=
\left( 1 + \sigma + \cdots + \sigma^{p^r-1} \right)
\in
\ZZ[\Cyclic_{p^r}]
\]
of its group ring.
\end{definition}

\begin{remark}
Notations \ref{notation.all.chain.level.stuff} \and \ref{notation.of.all.chain.level.stuff.htpy} are formatted in a way that makes them well-suited for application in \Cref{section.Cpn.eqvrt.cohomology}. However, from a certain point of view this formatting is suboptimal: most of the data defined therein is independent of the choice of $a$.
\end{remark}

\begin{notation}
\label{notation.all.chain.level.stuff}
We define a diagram
\begin{equation}
\label{defining.diagram.of.ZCT}
\begin{tikzcd}
\ttZ^a_r
\arrow{rd}[sloped]{\ttq_r^{a-1}}
&
\ttC^{a-1}_r
\arrow{rd}[sloped]{\ttg_r^{a-2}}
\arrow{d}{\tte_r^{a-1}}
&
\ttC^{a-2}_r
\arrow{rd}[sloped]{\ttg_r^{a-3}}
\arrow{d}{\tte_r^{a-2}}
&
\cdots
&
\cdots
\arrow{rd}[sloped]{\ttg_r^0}
&
\ttC^0_r
\arrow{d}{\tte_r^0}
\\
&
\ttT^{a-1}_r
&
\ttT^{a-2}_r
&
\ttT^{a-3}_r
&
\cdots
&
\ttT^0_r
\end{tikzcd}
\end{equation}
in $\Ch_{\ZZ[\Cyclic_{p^r}]}$ (i.e.\! a functor $\Zig_a \ra \Ch_{\ZZ[\Cyclic_{p^r}]}$) as follows.

\begin{enumerate}

\item\label{notation.for.chain.complexes.themselves}

We define the objects in diagram \Cref{defining.diagram.of.ZCT} as follows; all are levelwise free $\ZZ[\Cyclic_{p^r}]$-modules of rank 0 or 1.

\begin{enumeratesub}

\item\label{notation.for.chain.complex.A}

We define the object
\[
\ttZ^a_r
:=
\left(
\cdots
\longra
0
\longra
0
\longra
\uwave{\ZZ[\Cyclic_{p^r}]}
\xra{1-\sigma}
\ZZ[\Cyclic_{p^r}]
\xra{p^a N}
\ZZ[\Cyclic_{p^r}]
\xra{1-\sigma}
\ZZ[\Cyclic_{p^r}]
\xra{p^a N}
\cdots
\right)
~.
\]

\item

For any $i \geq 0$, we define the object
\[
\ttC^i_r
:=
\left(
\cdots
\xra{p^{i+1} N}
\ZZ[\Cyclic_{p^r}]
\xra{1-\sigma}
\ZZ[\Cyclic_{p^r}]
\xra{p^{i+1} N}
\uwave{\ZZ[\Cyclic_{p^r}]}
\xra{1-\sigma}
\ZZ[\Cyclic_{p^r}]
\xra{p^i N}
\ZZ[\Cyclic_{p^r}]
\xra{1-\sigma}
\ZZ[\Cyclic_{p^r}]
\xra{p^i N}
\cdots
\right)
~.\footnote{Note that the differentials in nonnegative degrees differ from those in negative degrees.}
\]

\item\label{notation.for.chain.complex.X}

For any $i \geq 0$, we define the object
\[
\ttT^i_r
:=
\left(
\cdots
\xra{p^{i+1}N}
\ZZ[\Cyclic_{p^r}]
\xra{1-\sigma}
\ZZ[\Cyclic_{p^r}]
\xra{p^{i+1} N}
\uwave{\ZZ[\Cyclic_{p^r}]}
\xra{1-\sigma}
\ZZ[\Cyclic_{p^r}]
\xra{p^{i+1} N}
\ZZ[\Cyclic_{p^r}]
\xra{1-\sigma}
\cdots
\right)
~.
\]

\end{enumeratesub}

\item\label{notation.for.chain.complex.morphisms}

We define the morphisms in diagram \Cref{defining.diagram.of.ZCT} as follows.

\begin{enumeratesub}

\item\label{notation.for.chain.complex.morphism.f.one}

Assuming that $a \geq 1$, we define the morphism
\[
\ttZ^a_r
\xra{\ttq_r^{a-1}}
\ttT^{a-1}_r
\]
as
\[ \begin{tikzcd}
\cdots
\arrow{r}
&
0
\arrow{r}
\arrow{d}
&
0
\arrow{r}
\arrow{d}
&
\uwave{\ZZ[\Cyclic_{p^r}]}
\arrow{r}{1-\sigma}
\arrow{d}{1}
&
\ZZ[\Cyclic_{p^r}]
\arrow{r}{p^a N}
\arrow{d}{1}
&
\ZZ[\Cyclic_{p^r}]
\arrow{r}{1-\sigma}
\arrow{d}{1}
&
\cdots
\\
\cdots
\arrow{r}[swap]{p^a N}
&
\ZZ[\Cyclic_{p^r}]
\arrow{r}[swap]{1-\sigma}
&
\ZZ[\Cyclic_{p^r}]
\arrow{r}[swap]{p^a N}
&
\uwave{\ZZ[\Cyclic_{p^r}]}
\arrow{r}[swap]{1-\sigma}
&
\ZZ[\Cyclic_{p^r}]
\arrow{r}[swap]{p^a N}
&
\ZZ[\Cyclic_{p^r}]
\arrow{r}[swap]{1-\sigma}
&
\cdots
\end{tikzcd}~, \]
i.e.\! it is the identity in all nonpositive degrees.

\item For any $i \geq 0$, we define the morphism
\[
\ttC^i_r
\xra{\tte_r^i}
\ttT^i_r
\]
as
\[
\hspace{-1.5cm}
\begin{tikzcd}[column sep=1.25cm]
\cdots
\arrow{r}{1-\sigma}
&
\ZZ[\Cyclic_{p^r}]
\arrow{r}{p^{i+1}N}
\arrow{d}[swap]{1}
&
\uwave{\ZZ[\Cyclic_{p^r}]}
\arrow{r}{1-\sigma}
\arrow{d}[swap]{1}
&
\ZZ[\Cyclic_{p^r}]
\arrow{r}{p^i N}
\arrow{d}[swap]{1}
&
\ZZ[\Cyclic_{p^r}]
\arrow{r}{1-\sigma}
\arrow{d}[swap]{p}
&
\ZZ[\Cyclic_{p^r}]
\arrow{r}{p^i N}
\arrow{d}[swap]{p}
&
\ZZ[\Cyclic_{p^r}]
\arrow{r}{1-\sigma}
\arrow{d}[swap]{p^2}
&
\cdots
\\
\cdots
\arrow{r}[swap]{1-\sigma}
&
\ZZ[\Cyclic_{p^r}]
\arrow{r}[swap]{p^{i+1}N}
&
\uwave{\ZZ[\Cyclic_{p^r}]}
\arrow{r}[swap]{1-\sigma}
&
\ZZ[\Cyclic_{p^r}]
\arrow{r}[swap]{p^{i+1}N}
&
\ZZ[\Cyclic_{p^r}]
\arrow{r}[swap]{1-\sigma}
&
\ZZ[\Cyclic_{p^r}]
\arrow{r}[swap]{p^{i+1}N}
&
\ZZ[\Cyclic_{p^r}]
\arrow{r}[swap]{1-\sigma}
&
\cdots
\end{tikzcd}
~, \]
i.e.\! it is the identity in each positive degree and, for all $j \geq 0$, in degree $-j$ it is multiplication by $p^{\lfloor \frac{j}{2} \rfloor}$.\footnote{Here the exponent is the floor of $\frac{j}{2}$, i.e.\! the largest integer that is at most $\frac{j}{2}$.}

\item For any $i \geq 0$, we define the morphism
\[
\ttC^{i+1}_r
\xra{\ttg_r^i}
\ttT^{i}_r
\]
as
\[
\hspace{-1.5cm}
\begin{tikzcd}[column sep=1.25cm]
\cdots
\arrow{r}{1-\sigma}
&
\ZZ[\Cyclic_{p^r}]
\arrow{r}{p^{i+2} N}
\arrow{d}[swap]{p^2}
&
\ZZ[\Cyclic_{p^r}]
\arrow{r}{1-\sigma}
\arrow{d}[swap]{p}
&
\ZZ[\Cyclic_{p^r}]
\arrow{r}{p^{i+2}N}
\arrow{d}[swap]{p}
&
\uwave{\ZZ[\Cyclic_{p^r}]}
\arrow{r}{1-\sigma}
\arrow{d}[swap]{1}
&
\ZZ[\Cyclic_{p^r}]
\arrow{r}{p^{i+1} N}
\arrow{d}[swap]{1}
&
\ZZ[\Cyclic_{p^r}]
\arrow{r}{1-\sigma}
\arrow{d}[swap]{1}
&
\cdots
\\
\cdots
\arrow{r}[swap]{1-\sigma}
&
\ZZ[\Cyclic_{p^r}]
\arrow{r}[swap]{p^{i+1} N}
&
\ZZ[\Cyclic_{p^r}]
\arrow{r}[swap]{1-\sigma}
&
\ZZ[\Cyclic_{p^r}]
\arrow{r}[swap]{p^{i+1} N}
&
\uwave{\ZZ[\Cyclic_{p^r}]}
\arrow{r}[swap]{1-\sigma}
&
\ZZ[\Cyclic_{p^r}]
\arrow{r}[swap]{p^{i+1} N}
&
\ZZ[\Cyclic_{p^r}]
\arrow{r}[swap]{1-\sigma}
&
\cdots
\end{tikzcd}~, \]
i.e.\! it is the identity in each negative degree and, for all $j \geq 0$, in degree $j$ it is multiplication by $p^{\lceil \frac{j}{2} \rceil}$.\footnote{Here the exponent is the ceiling of $\frac{j}{2}$, i.e.\! the smallest integer that is at least $\frac{j}{2}$.}

\end{enumeratesub}
\end{enumerate}
\end{notation}

\begin{notation}
\label{notation.of.all.chain.level.stuff.htpy}
We define a diagram
\begin{equation}
\label{defining.diagram.of.ZCT.htpy}
\begin{tikzcd}[column sep=0.5cm]
\ZZ^{\htpy \Cyclic_{p^a}}
\arrow{rd}[sloped]{q_r^{a-1}}
&
(\tau_{\geq 0} \ZZ^{\st \Cyclic_p})^{\htpy \Cyclic_{p^{a-1}}}
\arrow{rd}[sloped]{g_r^{a-2}}
\arrow{d}{e_r^{a-1}}
&
(\tau_{\geq 0} \ZZ^{\st \Cyclic_p})^{\htpy \Cyclic_{p^{a-2}}}
\arrow{rd}[sloped]{g_r^{a-3}}
\arrow{d}{e_r^{a-2}}
&
\cdots
&
\cdots
\arrow{rd}[sloped]{g_r^0}
&
\tau_{\geq 0} \ZZ^{\st \Cyclic_p}
\arrow{d}{e_r^0}
\\
&
(\ZZ^{\st \Cyclic_p})^{\htpy \Cyclic_{p^{a-1}}}
&
(\ZZ^{\st \Cyclic_p})^{\htpy \Cyclic_{p^{a-2}}}
&
(\ZZ^{\st \Cyclic_p})^{\htpy \Cyclic_{p^{a-3}}}
&
\cdots
&
\ZZ^{\st \Cyclic_p}
\end{tikzcd}
\end{equation}
in $\Mod^{\htpy \Cyclic_{p^r}}_\ZZ$ (i.e.\! a functor $\Zig_a \ra \Mod^{\htpy \Cyclic_{p^r}}_\ZZ$) as follows.
\begin{enumeratesub}

\item

Assuming that $a \geq 1$, we define the morphism $q_r^{a-1}$ as
\[
\ZZ^{\htpy \Cyclic_{p^a}}
\xlongla{\sim}
(\ZZ^{\htpy \Cyclic_p})^{\htpy \Cyclic_{p^{a-1}}}
\xra{\sQ_{\Cyclic_p}(\ZZ)^{\htpy \Cyclic_{p^{a-1}}}}
(\ZZ^{\st \Cyclic_p})^{\htpy \Cyclic_{p^{a-1}}}
~,
\]
the homotopy $\Cyclic_{p^{a-1}}$-fixedpoints of the canonical morphism from the homotopy $\Cyclic_p$-fixedpoints to the $\Cyclic_p$-Tate construction for the object $\ZZ \in \Mod^{\htpy \Cyclic_{p^n}}_\ZZ$.

\item For any $i \geq 0$, we define the morphism $e_r^i$ as
\[
( \tau_{\geq 0} \ZZ^{\st \Cyclic_p} )^{\htpy \Cyclic_{p^i}}
\xra{\varepsilon_{\geq 0}(\ZZ^{\st \Cyclic_p})^{\htpy \Cyclic_{p^i}}}
(\ZZ^{\st \Cyclic_p} )^{\htpy \Cyclic_{p^i}}
~,
\]
the homotopy $\Cyclic_{p^i}$-fixedpoints of the canonical morphism to the $\Cyclic_p$-Tate construction on the object $\ZZ \in \Mod^{\htpy \Cyclic_{p^{r+i+1}}}_\ZZ$ from its connective cover.

\item For any $i \geq 0$, we define the morphism $g_r^i$ as
\[
(\tau_{\geq 0} \ZZ^{\st \Cyclic_p})^{\htpy \Cyclic_{p^{i+1}}}
\xlongla{\sim}
( (\tau_{\geq 0} \ZZ^{\st \Cyclic_p})^{\htpy \Cyclic_p} )^{\htpy \Cyclic_{p^{i}}}
\xra{\sQ_{\Cyclic_p}(\tau_{\geq 0} \ZZ^{\st \Cyclic_p})^{\htpy \Cyclic_{p^{i}}}}
((\tau_{\geq 0} \ZZ^{\st \Cyclic_p})^{\st \Cyclic_p})^{\htpy \Cyclic_{p^{i}}}
\xla[\sim]{(\genQuot^{\st \Cyclic_p})^{\htpy \Cyclic_{p^{i}}}}
( \ZZ^{\st \Cyclic_p})^{\htpy \Cyclic_{p^{i}}}
~,
\]
the composite of
\begin{itemize}
\item the homotopy $\Cyclic_{p^{i}}$-fixedpoints of the canonical map from the homotopy $\Cyclic_p$-fixedpoints to the $\Cyclic_p$-Tate construction for the connective cover of the $\Cyclic_p$-Tate construction on the object $\ZZ \in \Mod^{\htpy \Cyclic_{p^{r+i+2}}}_\ZZ$ and 

\item the homotopy $\Cyclic_{p^{i}}$-fixedpoints of the inverse of the equivalence $\genQuot^{\st \Cyclic_p}$ in $\Mod^{\htpy \Cyclic_{p^{r+i}}}_\ZZ$ of \Cref{obs.tate.Cp.of.Z.to.conn.cover.of.Z.tate.Cp.is.an.equivalence}.
\end{itemize}

\end{enumeratesub}
\end{notation}

\begin{lemma}
\label{lem.ABX.htpy.invariantly}
The diagram \Cref{defining.diagram.of.ZCT} in $\Ch_{\ZZ[\Cyclic_{p^r}]}$ presents the diagram \Cref{defining.diagram.of.ZCT.htpy} in $\Mod^{\htpy \Cyclic_{p^r}}_\ZZ$.
\end{lemma}

\begin{proof}
The case where $a=0$ is trivial to verify, so let us assume that $a \geq 1$. We also fix any $i \geq 0$. Then, we prove that the morphisms $q_r^{a-1}$, $g_r^i$, and $e_r^i$ in $\Mod^{\htpy \Cyclic_{p^r}}_\ZZ$ are equivalent to other morphisms that are manifestly presented by the morphisms $\ttq_r^{a-1}$, $\ttg_r^i$, and $\tte_r^i$ in $\Ch_{\ZZ[\Cyclic_{p^r}]}$, respectively. In particular, this shows that the objects
\[
\ttZ_r^a , \ttC_r^i , \ttT_r^i
\in
\Ch_{\ZZ[\Cyclic_{p^r}]}
\]
are indeed presentations of the objects
\[
\ZZ^{\htpy \Cyclic_{p^a}}
,
(\tau_{\geq 0} \ZZ^{\st \Cyclic_p})^{\htpy \Cyclic_{p^{i}}}
,
(\ZZ^{\st \Cyclic_p})^{\htpy \Cyclic_{p^{i}}}
\in
\Mod^{\htpy \Cyclic_{p^r}}_\ZZ
~,
\]
respectively.

We begin with the morphisms $\ttq_r^{a-1}$ and $\tte_r^i$. For these, note the commutative diagram
\begin{figure}[h]
\begin{equation}
\label{comm.diagram.for.f.one.and.g.r}
\hspace{-2cm}
\begin{tikzcd}[column sep=2.5cm, row sep=1.5cm]
&
(\ZZ_{\htpy \Cyclic_p})^{\htpy \Cyclic_{p^{i}}}
\arrow{r}{\genNm^{\htpy \Cyclic_{p^{i}}}}
\arrow[equals]{d}
&
\ZZ^{\htpy \Cyclic_{p^{i}}}
\arrow{r}{\genQuot^{\htpy \Cyclic_{p^{i}}}}
\arrow{d}{\varepsilon_{\geq 0}(\ZZ^{\htpy \Cyclic_p})^{\htpy \Cyclic_{p^{i}}}}
&
(\tau_{\geq 0} \ZZ^{\st \Cyclic_p})^{\htpy \Cyclic_{p^{i}}}
\arrow{d}{\varepsilon_{\geq 0}(\ZZ^{\st \Cyclic_p})^{\htpy \Cyclic_{p^{i}}} =: e_r^i}
\\
(\ZZ_{\htpy \Cyclic_p})_{\htpy \Cyclic_{p^{i}}}
\arrow{ru}[sloped]{\Nm_{\Cyclic_{p^{i}}}(\ZZ_{\htpy \Cyclic_p})}[sloped, swap]{\sim}
\arrow{r}[swap]{\Nm_{\Cyclic_{p^{i}}}(\ZZ_{\htpy \Cyclic_p})}{\sim}
&
(\ZZ_{\htpy \Cyclic_p})^{\htpy \Cyclic_{p^{i}}}
\arrow{r}{\Nm_{\Cyclic_p}(\ZZ)^{\htpy \Cyclic_{p^{i}}}}
&
(\ZZ^{\htpy \Cyclic_p})^{\htpy \Cyclic_{p^{i}}}
\arrow{r}{\sQ_{\Cyclic_p}(\ZZ)^{\htpy \Cyclic_{p^{i}}}}
\arrow{d}[sloped, anchor=south]{\sim}
&
(\ZZ^{\st \Cyclic_p})^{\htpy \Cyclic_{p^{i}}}
\arrow{d}[sloped, anchor=south]{\sim}
\\
&
\ZZ_{\htpy \Cyclic_{p^{i+1}}}
\arrow{r}[swap]{\Nm_{\Cyclic_{p^{i+1}}}(\ZZ)}
\arrow{lu}[sloped, swap]{\sim}
&
\ZZ^{\htpy \Cyclic_{p^{i+1}}}
\arrow{r}[swap]{\sQ_{\Cyclic_{p^{i+1}}}(\ZZ)}
\arrow[dashed]{ru}[sloped]{q_r^{a-1} \textup{ (when $i=a-1$)}}
&
\ZZ^{\st \Cyclic_{p^{i+1}}}
\end{tikzcd}
\end{equation}
\caption{This commutative diagram in $\Mod^{\htpy \Cyclic_{p^{n-a}}}_\ZZ$ contains the morphisms $q_r^{a-1}$ and $e_r^i$. All three horizontal composites are cofiber sequences.\label{figure.comm.diagram.for.f.one.and.g.r}}
\end{figure}
in $\Mod^{\htpy \Cyclic_{p^r}}_\ZZ$ of \Cref{figure.comm.diagram.for.f.one.and.g.r}, obtained as follows.
\begin{itemize}

\item The two upper commutative squares are obtained by applying the functor
\[
\Mod^{\htpy \Cyclic_{p^{r+i}}}_\ZZ
\xra{(-)^{\htpy \Cyclic_{p^{i}}}}
\Mod^{\htpy \Cyclic_{p^{r+i}}}_\ZZ
\]
to (the middle two rows of) the commutative diagram \Cref{morphism.between.isotropy.separation.sequences}, and so are cofiber sequences by \Cref{obs.morphism.between.isotropy.separation.sequences}.

\item The lower left pentagon commutes by \Cref{obs.various.properties.of.tate}\Cref{part.norm.maps.compose}.

\item The morphism $\Nm_{\Cyclic_{p^{i}}}(\ZZ_{\htpy \Cyclic_p})$ is an equivalence

\begin{itemize}

\item trivially if $i=0$,

\item by \Cref{obs.tate.vanishing.for.Z.mods} if $i=1$, and

\item by Observations \ref{obs.tate.vanishing.for.Z.mods} \and \ref{obs.if.tate.Cpr.vanishes.then.tate.Cprplusone.vanishes} if $i \geq 2$.

\end{itemize}

\item The lower right vertical equivalence is the induced equivalence between cofibers.

\end{itemize}
We now argue as follows.
\begin{enumeratesub}

\item In the case that $i=a-1$, the morphism $\ttq_r^{a-1}$ is evidently a presentation of the morphism $\sQ_{\Cyclic_{p^a}}(\ZZ)$, which proves that it is indeed a presentation of the morphism $q_r^{a-1}$ by diagram \Cref{comm.diagram.for.f.one.and.g.r}.

\item 

Consider the morphism of cofiber sequences
\begin{equation}
\label{g.r.as.map.on.cofibers}
\begin{tikzcd}[column sep=4cm, row sep=2cm]
\ZZ_{\htpy \Cyclic_{p^{i+1}}}
\arrow{r}{{\genNm^{\htpy \Cyclic_{p^{i}}}} \circ \Nm_{\Cyclic_{p^{i}}}(\ZZ_{\htpy \Cyclic_p})}
\arrow[equals]{d}
&
\ZZ^{\htpy \Cyclic_{p^{i}}}
\arrow{r}{\genQuot^{\htpy \Cyclic_{p^{i}}}}
\arrow{d}{\varepsilon_{\geq 0}(\ZZ^{\htpy \Cyclic_p})^{\htpy \Cyclic_{p^{i}}}}
&[-2cm]
(\tau_{\geq 0} \ZZ^{\st \Cyclic_p})^{\htpy \Cyclic_{p^{i}}}
\arrow{d}{e_r^i}
\\
\ZZ_{\htpy \Cyclic_{p^{i+1}}}
\arrow{r}[swap]{\Nm_{\Cyclic_{p^{i+1}}}(\ZZ)}
&
\ZZ^{\htpy \Cyclic_{p^{i+1}}}
\arrow{r}[swap]{\sQ_{\Cyclic_{p^{i+1}}}(\ZZ)}
&
\ZZ^{\st \Cyclic_{p^{i+1}}}
\end{tikzcd}
\end{equation}
in $\Mod^{\htpy \Cyclic_{p^r}}_\ZZ$ extracted from diagram \Cref{comm.diagram.for.f.one.and.g.r}. 
Noting that the middle vertical morphism in diagram \Cref{g.r.as.map.on.cofibers} induces an equivalence on connective covers (and that $\ZZ_{\htpy \Cyclic_{p^{i+1}}} \in \Mod^{\htpy \Cyclic_{p^r}}_\ZZ$ is connective), we see that the morphism $\tte_r^i$ is indeed a presentation of the morphism $e_r^i$.
\end{enumeratesub}

We now proceed to the morphism $\ttg_r^i$. For this, note the commutative diagram
\begin{figure}[h]
\begin{equation}
\label{comm.diagram.for.f.r}
\begin{tikzcd}[row sep=1.5cm, column sep=3cm]
&[-3.2cm]
((\ZZ_{\htpy \Cyclic_p})_{\htpy \Cyclic_p})_{\htpy \Cyclic_{p^{i}}}
\arrow{r}{(\genNm_{\htpy \Cyclic_p})_{\htpy \Cyclic_{p^{i}}}}
\arrow{d}[swap]{\Nm_{\Cyclic_{p^{i}}}((\ZZ_{\htpy \Cyclic_p})_{\htpy \Cyclic_p})}[sloped, anchor=south]{\sim}
&
(\ZZ_{\htpy \Cyclic_p})_{\htpy \Cyclic_{p^{i}}}
\arrow{d}{\Nm_{\Cyclic_{p^{i}}}(\ZZ_{\htpy \Cyclic_p})}[sloped, anchor=north]{\sim}
\\
&
((\ZZ_{\htpy \Cyclic_p})_{\htpy \Cyclic_p})^{\htpy \Cyclic_{p^{i}}}
\arrow{r}{(\genNm_{\htpy \Cyclic_p})^{\htpy \Cyclic_{p^{i}}}}
\arrow{d}[swap]{\Nm_{\Cyclic_p}(\ZZ_{\htpy \Cyclic_p})^{\htpy \Cyclic_{p^{i}}}}[sloped, anchor=south]{\sim}
&
(\ZZ_{\htpy \Cyclic_p})^{\htpy \Cyclic_{p^{i}}}
\arrow{r}{(\genQuot_{\htpy \Cyclic_p})^{\htpy \Cyclic_{p^{i}}}}
\arrow{d}{\Nm_{\Cyclic_p}(\ZZ)^{\htpy \Cyclic_{p^{i}}}}
&
(( \tau_{\geq 0} \ZZ^{\st \Cyclic_p})_{\htpy \Cyclic_p})^{\htpy \Cyclic_{p^{i}}}
\arrow{d}{\Nm_{\Cyclic_p}(\tau_{\geq 0} \ZZ^{\st \Cyclic_p})^{\htpy \Cyclic_{p^{i}}}}
\\
&
((\ZZ_{\htpy \Cyclic_p})^{\htpy \Cyclic_p})^{\htpy \Cyclic_{p^{i}}}
\arrow{r}[swap]{(\genNm^{\htpy \Cyclic_p})^{\htpy \Cyclic_{p^{i}}}}
\arrow{d}[swap]{\sQ_{\Cyclic_p}(\ZZ_{\htpy \Cyclic_p})^{\htpy \Cyclic_{p^{i}}}}
&
(\ZZ^{\htpy \Cyclic_p})^{\htpy \Cyclic_{p^{i}}}
\arrow{r}[swap]{(\genQuot^{\htpy \Cyclic_p})^{\htpy \Cyclic_{p^{i}}}}
\arrow{d}{\sQ_{\Cyclic_p}(\ZZ)^{\htpy \Cyclic_{p^{i}}}}
&
(( \tau_{\geq 0} \ZZ^{\st \Cyclic_p})^{\htpy \Cyclic_p})^{\htpy \Cyclic_{p^{i}}}
\arrow{d}{\sQ_{\Cyclic_p}(\tau_{\geq 0} \ZZ^{\st \Cyclic_p})^{\htpy \Cyclic_{p^{i}}}}
\arrow[dashed, bend left=15]{ld}[sloped, swap, pos=0.4]{g_r^i}
\\
0 \simeq 
&
((\ZZ_{\htpy \Cyclic_p})^{\st \Cyclic_p})^{\htpy \Cyclic_{p^{i}}}
\arrow{r}[swap]{(\genNm^{\st \Cyclic_p})^{\htpy \Cyclic_{p^{i}}}}
&
(\ZZ^{\st \Cyclic_p})^{\htpy \Cyclic_{p^{i}}}
\arrow{r}[swap]{(\genQuot^{\st \Cyclic_p})^{\htpy \Cyclic_{p^{i}}}}{\sim}
&
((\tau_{\geq 0} \ZZ^{\st \Cyclic_p})^{\st \Cyclic_p})^{\htpy \Cyclic_{p^{i}}}
\end{tikzcd}
\end{equation}
\caption{This commutative diagram in $\Mod^{\htpy \Cyclic_{p^r}}_\ZZ$ contains the morphism $g_r^i$. Omitting the top row, all three horizontal composites and all three vertical composites are cofiber sequences.\label{figure.comm.diagram.for.f.r}}
\end{figure}
in $\Mod^{\htpy \Cyclic_{p^r}}_\ZZ$  of \Cref{figure.comm.diagram.for.f.r}, obtained as follows.
\begin{itemize}

\item The bottom three rows are obtained from the second row of diagram \Cref{morphism.between.isotropy.separation.sequences} by applying the cofiber sequence
\[
(-)_{\htpy \Cyclic_p}
\xra{\Nm_{\Cyclic_p}}
(-)^{\htpy \Cyclic_p}
\xra{\sQ_{\Cyclic_p}}
(-)^{\st \Cyclic_p}
\]
in $\Fun^\ex( \Mod^{\htpy \Cyclic_{p^{r+1}}}_\ZZ , \Mod^{\htpy \Cyclic_{p^r}}_\ZZ)$ followed by the functor
\[
\Mod^{\htpy \Cyclic_{p^{r+i}}}_\ZZ
\xra{(-)^{\htpy \Cyclic_{p^{i}}}}
\Mod^{\htpy \Cyclic_{p^{r}}}_\ZZ
~.
\]
So indeed, the three lower vertical composites are indeed cofiber sequences, and moreover the lower three rows are cofiber sequences by \Cref{obs.morphism.between.isotropy.separation.sequences}.

\item The equivalence $0 \simeq ((\ZZ_{\htpy \Cyclic_p})^{\st \Cyclic_p})^{\htpy \Cyclic_{p^{i}}}$ follows from \Cref{obs.tate.vanishing.for.Z.mods}, and it implies that the morphisms $\Nm_{\Cyclic_p}(\ZZ_{\htpy \Cyclic_p})^{\htpy \Cyclic_{p^{i}}}$ and $(\genQuot^{\st \Cyclic_p})^{\htpy \Cyclic_{p^{i}}}$ are both equivalences.

\item The morphisms $\Nm_{\Cyclic_{p^{i}}}((\ZZ_{\htpy \Cyclic_p})_{\htpy \Cyclic_p})$ and $\Nm_{\Cyclic_{p^{i}}}(\ZZ_{\htpy \Cyclic_p})$ are equivalences
\begin{itemize}

\item trivially if $i=0$,

\item by \Cref{obs.tate.vanishing.for.Z.mods} if $i = 1$, and

\item by Observations \ref{obs.tate.vanishing.for.Z.mods} \and \ref{obs.if.tate.Cpr.vanishes.then.tate.Cprplusone.vanishes} if $i \geq 2$.

\end{itemize}

\end{itemize}
We now argue as follows.
\begin{enumeratesub}

\setcounter{enumi}{2}

\item Consider the morphism of cofiber sequences
\begin{equation}
\label{f.r.as.map.on.cofibers}
\begin{tikzcd}[column sep=5cm, row sep=2cm]
\ZZ_{\htpy \Cyclic_{p^{i+2}}}
\arrow{r}{(\genNm^{\htpy \Cyclic_p})^{\htpy \Cyclic_{p^{i}}} \circ \Nm_{\Cyclic_{p^{i+1}}}(\ZZ_{\htpy \Cyclic_p})}
\arrow{d}[swap]{\genNm_{\htpy \Cyclic_{p^{i+1}}}}
&
\ZZ^{\htpy \Cyclic_{p^{i+1}}}
\arrow{r}{\genQuot^{\htpy \Cyclic_p^{i+1}}}
\arrow[equals]{d}
&[-3cm]
(\tau_{\geq 0} \ZZ^{\st \Cyclic_p})^{\htpy \Cyclic_{p^{i+1}}}
\arrow{d}{g_r^i}
\\
\ZZ_{\htpy \Cyclic_{p^{i+1}}}
\arrow{r}[swap]{\Nm_{\Cyclic_{p^{i+1}}}(\ZZ)}
&
\ZZ^{\htpy \Cyclic_{p^{i+1}}}
\arrow{r}[swap]{\sQ_{\Cyclic_p}(\ZZ)^{\htpy \Cyclic_{p^{i}}}}
&
(\ZZ^{\st \Cyclic_p})^{\htpy \Cyclic_{p^{i}}}
\end{tikzcd}
\end{equation}
in $\Mod^{\htpy \Cyclic_{p^r}}_\ZZ$ extracted from diagram \Cref{comm.diagram.for.f.r}, where we have used \Cref{obs.various.properties.of.tate}\Cref{part.norm.maps.compose} to reidentify the left two horizontal morphisms. Noting that the left vertical morphism in diagram \Cref{f.r.as.map.on.cofibers} induces an equivalence after coconnective truncation (and that $\ZZ^{\htpy \Cyclic_{p^{i+1}}} \in \Mod^{\htpy \Cyclic_{p^r}}_\ZZ$ is coconnective), we see that the morphism $\ttg_r^i$ is indeed a presentation of the morphism $g_r^i$. \qedhere

\end{enumeratesub}
\end{proof}

\subsection{Chain-level data for inclusion and transfer}
\label{subsection.inclusion.and.transfer.stuff}

In this subsection, we establish the results that support our identification of the inclusion and transfer morphisms in equivariant cohomology. These involve some new auxiliary chain complexes and chain maps introduced in Notations \ref{notn.S.chain.complex} \and \ref{notn.B.tilde.complex}, as well as a chain homotopy introduced in \Cref{notn.for.nullhtpy}. The main results are \Cref{lemma.ABX.are.adapted} (which establishes the relevant adaptedness) and \Cref{lemma.chain.level.nullhtpy.presents.canonical.nullhtpy} (which proves that the chain homotopy of \Cref{notn.B.tilde.complex} presents the desired $\infty$-categorical homotopy).

\begin{notation}
\label{notn.S.chain.complex}
For any $i \geq 0$, we define the object ${_{i} \ttS_r} \in \Ch_{\ZZ[\Cyclic_{p^r}]}$ as
\[
{_{i} \ttS_r}
:=
\left(
\cdots
\xra{p^i N}
\ZZ[\Cyclic_{p^r}]
\xra{1-\sigma}
\ZZ[\Cyclic_{p^r}]
\xra{p^i N}
\ZZ[\Cyclic_{p^r}]
\xra{1-\sigma}
\uwave{\ZZ[\Cyclic_{p^r}]}
\longra
0
\longra
0
\longra
\cdots
\right)
~.
\]
\end{notation}

\begin{local}
In this subsection, we henceforth assume that $r \geq 1$.
\end{local}

\begin{observation}
\label{obs.fixedpoints.of.ABX.are.themselves}
For any $i \geq 0$, there are evident isomorphisms
\[
(\ttZ^a_r)^{\Cyclic_p}
\cong
\ttZ^{a+1}_{r-1}
~,
\qquad
(\ttC^i_r)^{\Cyclic_p}
\cong
\ttC^{i+1}_{r-1}
~,
\qquad
\text{and}
\qquad
(\ttT^i_r)^{\Cyclic_p}
\cong
\ttT^{i+1}_{r-1}
\]
in $\Ch_{\ZZ[\Cyclic_{p^{r-1}}]}$, as well as evident isomorphisms
\[
(\ttq_r^{a-1})^{\Cyclic_p}
\cong
\ttq_{r-1}^a
~,
\qquad
(\ttg_r^i)^{\Cyclic_p}
\cong
\ttg_{r-1}^{i+1}
~,
\qquad
\text{and}
\qquad
(\tte_r^i)^{\Cyclic_p}
\cong
\tte_{r-1}^{i+1}
\]
in $\Ar(\Ch_{\ZZ[\Cyclic_{p^{r-1}}]})$.
\end{observation}

\begin{lemma}
\label{lemma.ABX.are.adapted}
For any $i \geq 0$, the complexes $\ttZ^a_r, \ttC^i_r, \ttT^i_r \in \Ch_{\ZZ[\Cyclic_{p^r}]}$ are adapted to homotopy $\Cyclic_p$-fixedpoints.
\end{lemma}

\begin{proof}
It follows from \Cref{obs.basics.of.adaptedness.to.fixedpoints}\Cref{item.bdd.above.free.adapted.to.fixedpoints} that $\ttZ^a_r$ is adapted to homotopy $\Cyclic_p$-fixedpoints.

There is an evident diagram $\ttZ^i_r \la {_{i+1} \ttS_r} \ra \ttZ^{i+1}_r$ in $\Ch_{\ZZ[\Cyclic_{p^r}]}$ given by
\[
\hspace{-1cm}
\begin{tikzcd}[column sep=1.25cm]
\cdots
\arrow{r}
&
0
\arrow{r}
&
0
\arrow{r}
&
\uwave{\ZZ[\Cyclic_{p^r}]}
\arrow{r}{1-\sigma}
&
\ZZ[\Cyclic_{p^r}]
\arrow{r}{p^{i}N}
&
\ZZ[\Cyclic_{p^r}]
\arrow{r}{1-\sigma}
&
\cdots
\\
\cdots
\arrow{r}{1-\sigma}
&
\ZZ[\Cyclic_{p^r}]
\arrow{r}{p^{i+1}N}
\arrow{u}
\arrow{d}
&
\ZZ[\Cyclic_{p^r}]
\arrow{r}{1-\sigma}
\arrow{u}
\arrow{d}
&
\uwave{\ZZ[\Cyclic_{p^r}]}
\arrow{r}
\arrow{u}[swap]{p^{i+1}N}
\arrow{d}{p^{i+1}N}
&
0
\arrow{r}
\arrow{u}
\arrow{d}
&
0
\arrow{r}
\arrow{u}
\arrow{d}
&
\cdots
\\
\cdots
\arrow{r}
&
0
\arrow{r}
&
0
\arrow{r}
&
\uwave{\ZZ[\Cyclic_{p^r}]}
\arrow{r}[swap]{1-\sigma}
&
\ZZ[\Cyclic_{p^r}]
\arrow{r}[swap]{p^{i+1}N}
&
\ZZ[\Cyclic_{p^r}]
\arrow{r}[swap]{1-\sigma}
&
\cdots
\end{tikzcd}
\]
(reading from top to bottom) such that we have isomorphisms
\[
\ttC^i_r
\cong
\ChCone ( {_{i+1} \ttS_r} \longra \ttZ^i_r )
\qquad
\text{and}
\qquad
\ttT^i_r
\cong
\ChCone ( {_{i+1} \ttS_r } \longra \ttZ^{i+1}_r )
~.
\]
So by \Cref{obs.basics.of.adaptedness.to.fixedpoints}\Cref{item.operations.preserve.adapted.to.fixedpoints}\Cref{subitem.adapted.to.fixedpoints.preserved.under.cones}, to show that $\ttC^i_r$ and $\ttZ^i_r$ are adapted to homotopy $\Cyclic_p$-fixedpoints it suffices to show that ${_{i+1} \ttS_r}$, $\ttZ^i_r$, and $\ttZ^{i+1}_r$ are. We have just seen that the latter two are adapted to homotopy $\Cyclic_p$-fixedpoints, so it remains to prove that $\ttS := {_{i+1} \ttS_r}$ is as well. For this, let us write
\[
N_1
:=
\left(
\sum_{j=1}^p
(\sigma^{p^{r-1}})^j
\right)
=
\left(
1 + \sigma^{p^{r-1}} + \cdots + \sigma^{(p-1) \cdot p^{r-1}}
\right)
\in
\ZZ[\Cyclic_{p^r}]
\]
for the image of the norm element for $\Cyclic_p$ under the ring homomorphism $\ZZ[\Cyclic_p \hookra \Cyclic_{p^{n-a}}]$. Then, consider the evident factorization
\[ \begin{tikzcd}
\ttS
\arrow{r}{N_1}
\arrow[two heads]{d}
&
\ttS
\\
\ttS_{\Cyclic_p}
\arrow[dashed]{r}[swap]{N_1}
&
\ttS^{\Cyclic_p}
\arrow[hook]{u}
\end{tikzcd} \]
in $\Ch_{\ZZ[\Cyclic_{p^r}]}$, which determines a commutative diagram
\begin{equation}
\label{norm.of.Pi.infty.S.versus.Pi.infty.of.point.set.norm}
\begin{tikzcd}[column sep=2cm]
\Pi_\infty(\ttS)_{\htpy \Cyclic_p}
\arrow{r}{\Nm_{\Cyclic_p}(\Pi_\infty(\ttS))}
\arrow{d}
&
\Pi_\infty(\ttS)^{\htpy \Cyclic_p}
\\
\Pi_\infty(\ttS_{\Cyclic_p})
\arrow{r}[swap]{\Pi_\infty(N_1)}
&
\Pi_\infty(\ttS^{\Cyclic_p})
\arrow{u}
\end{tikzcd}
\end{equation}
in $\Mod_\ZZ^{\htpy \Cyclic_{p^r}}$. We claim that the three morphisms in diagram \Cref{norm.of.Pi.infty.S.versus.Pi.infty.of.point.set.norm} aside from the right vertical morphism are equivalences.
\begin{itemize}

\item Its upper horizontal morphism is an equivalence due to the evident equivalence $\Pi_\infty(\ttS) \simeq \ZZ_{\htpy \Cyclic_{p^{i+1}}}$ in $\Mod_\ZZ^{\htpy \Cyclic_{p^r}}$ and using \Cref{obs.tate.vanishing.for.Z.mods}.

\item Its left vertical morphism is an equivalence because $\ttS \in \Ch_{\ZZ[\Cyclic_{p^r}]}$ is adapted to homotopy $\Cyclic_p$-orbits by \Cref{obs.basics.of.adaptedness.to.orbits}\Cref{item.bdd.below.free.adapted.to.orbits}.

\item Its lower horizontal morphism is an equivalence because the morphism $\ttS_{\Cyclic_p} \xra{N_1} \ttS^{\Cyclic_p}$ in $\Ch_{\ZZ[\Cyclic_{p^r}]}$ is evidently an isomorphism.

\end{itemize}
Hence, the right vertical morphism in diagram \Cref{norm.of.Pi.infty.S.versus.Pi.infty.of.point.set.norm} is an equivalence, i.e.\! $\ttS$ is adapted to homotopy $\Cyclic_p$-fixedpoints.
\end{proof}

\begin{notation}
\label{notn.B.tilde.complex}
We define the morphisms
\[
{_{1} \ttS_r}
\xlongra{\ttgenNm}
{_{0} \ttS_r}
\xlongra{\tti}
\ttZ^0_r
\xlongra{\ttj}
\ttZ^1_{r-1}
\]
in $\Ch_{\ZZ[\Cyclic_{p^{n-a}}]}$ as
\[
\hspace{-1.5cm}
\begin{tikzcd}
\cdots
\arrow{r}{pN}
&
\ZZ[\Cyclic_{p^r}]
\arrow{r}{1-\sigma}
\arrow{d}{p}
&
\ZZ[\Cyclic_{p^r}]
\arrow{r}{pN}
\arrow{d}{p}
&
\ZZ[\Cyclic_{p^r}]
\arrow{r}{1-\sigma}
\arrow{d}{1}
&
\uwave{\ZZ[\Cyclic_{p^r}]}
\arrow{r}
\arrow{d}{1}
&
0
\arrow{r}
\arrow{d}
&
0
\arrow{r}
\arrow{d}
&
0
\arrow{r}
\arrow{d}
&
\cdots
\\
\cdots
\arrow{r}[swap]{N}
&
\ZZ[\Cyclic_{p^r}]
\arrow{r}[swap]{1-\sigma}
\arrow{d}
&
\ZZ[\Cyclic_{p^r}]
\arrow{r}[swap]{N}
\arrow{d}
&
\ZZ[\Cyclic_{p^r}]
\arrow{r}[swap]{1-\sigma}
\arrow{d}
&
\uwave{\ZZ[\Cyclic_{p^r}]}
\arrow{r}
\arrow{d}{N}
&
0
\arrow{r}
\arrow{d}
&
0
\arrow{r}
\arrow{d}
&
0
\arrow{r}
\arrow{d}
&
\cdots
\\
\cdots
\arrow{r}
&
0
\arrow{r}
\arrow{d}
&
0
\arrow{r}
\arrow{d}
&
0
\arrow{r}
\arrow{d}
&
\uwave{\ZZ[\Cyclic_{p^r}]}
\arrow{r}{1-\sigma}
\arrow{d}{p}
&
\ZZ[\Cyclic_{p^r}]
\arrow{r}{N}
\arrow{d}{p}
&
\ZZ[\Cyclic_{p^r}]
\arrow{r}{1-\sigma}
\arrow{d}{p^2}
&
\ZZ[\Cyclic_{p^r}]
\arrow{r}{N}
\arrow{d}{p^2}
&
\cdots
\\
\cdots
\arrow{r}
&
0
\arrow{r}
&
0
\arrow{r}
&
0
\arrow{r}
&
\uwave{\ZZ[\Cyclic_{p^{r-1}}]}
\arrow{r}[swap]{1-\sigma}
&
\ZZ[\Cyclic_{p^{r-1}}]
\arrow{r}[swap]{pN}
&
\ZZ[\Cyclic_{p^{r-1}}]
\arrow{r}[swap]{1-\sigma}
&
\ZZ[\Cyclic_{p^{r-1}}]
\arrow{r}[swap]{pN}
&
\cdots
\end{tikzcd}
~.
\]
Moreover, we define the chain complex
\[
\w{\ttC}_r
:=
\ChCone \left( \ttgenNm \right)
\in
\Ch_{\ZZ[\Cyclic_{p^r}]}
~,
\]
and we write
\[ \begin{tikzcd}
{_1 \ttS_r}
\arrow{r}{\ttgenNm}
\arrow[equals]{d}
&
{_0 \ttS_r}
\arrow{r}
\arrow{d}{\tti}
&
\w{\ttC}_r
\arrow[dashed]{d}{\ttk_r}
\\
{_1 \ttS_r}
\arrow{r}[swap]{\tti \circ \ttgenNm}
&
\ttZ^0_r
\arrow{r}
&
\ttC^0_r
\end{tikzcd} \]
for the induced morphism on cones.
\end{notation}

\begin{local}
For simplicity, we often omit the functor $\Ch_{\ZZ[\Cyclic_{p^r}]} \xla{\triv} \Ch_{\ZZ[\Cyclic_{p^{r-1}}]}$ from our notation. Moreover, we use the notation $(-)^\dagger$ to denote passage to adjunct morphisms in the adjunction
\[ \begin{tikzcd}[column sep=1.5cm]
\Ch_{\ZZ[\Cyclic_{p^r}]}
\arrow[transform canvas={yshift=0.9ex}]{r}{(-)_{\Cyclic_p}}
\arrow[leftarrow, transform canvas={yshift=-0.9ex}]{r}[yshift=-0.2ex]{\bot}[swap]{\triv}
&
\Ch_{\ZZ[\Cyclic_{p^{r-1}}]}
\end{tikzcd}
~.
\]
\end{local}

\needspace{2\baselineskip}
\begin{observation}
\label{obs.cones.get.previous.cxes}
\begin{enumerate}
\item[]

\item\label{item.B.tilde.a.presents.Baa}

There is a canonical isomorphism
\[
\ChCone \left( \tti \circ \ttgenNm \right)
\cong
\ttC^0_r
\]
in $\Ch_{\ZZ[\Cyclic_{p^r}]}$. Moreover, the morphism
\[
\w{\ttC}_r
:=
\ChCone \left( \ttgenNm \right)
\xlongra{\ttk_r}
\ChCone \left( \tti \circ \ttgenNm \right)
\cong
\ttC^0_r
\]
is a quasi-isomorphism, because the morphism ${_0 \ttS_r} \xra{\tti} \ttZ^0_r$ is a quasi-isomorphism.

\item\label{item.fa.is.induced.map.on.cones}

We have a canonical commutative diagram
\[ \begin{tikzcd}[column sep=1.5cm]
({_1 \ttS_r })_{\Cyclic_p}
\arrow{r}{(\ttj \circ \tti \circ \ttgenNm)^\dagger}
\arrow{d}[swap]{\ttgenNm_{\Cyclic_p}}
&
\ttZ^1_{r-1}
\arrow{r}
\arrow[equals]{d}
&
\ttC^1_{r-1}
\arrow{d}{\ttg_{r-1}^0}
\\
({_0 \ttS_r })_{\Cyclic_p}
\arrow{r}[swap]{(\ttj \circ \tti)^\dagger}
&
\ttZ^1_{r-1}
\arrow{r}
&
\ttT^0_{r-1}
\end{tikzcd} \]
in $\Ch_{\ZZ[\Cyclic_{p^{r-1}}]}$ in which both rows are cone sequences and $\ttg_{r-1}^0$ is the induced morphism on cones.

\end{enumerate}
\end{observation}

\begin{notation}
\label{notn.for.nullhtpy}
We write $\tth'$ for the canonical nullhomotopy of the composite
\[
(\w{\ttC}_r)_{\Cyclic_p}
\longra
\ttC^1_{r-1}
\xra{\ttg_{r-1}^0}
\ttT^0_{r-1}
\]
determined by the commutative triangle
\[ \begin{tikzcd}[column sep=1.5cm]
( {_1 \ttS_r } )_{\Cyclic_p}
\arrow{rr}{\ttgenNm_{\Cyclic_p}}
\arrow{rd}[sloped, swap]{(\ttj \circ \tti \circ \ttgenNm)^\dagger}
&
&
( {_0 \ttS_r} )_{\Cyclic_p}
\arrow{ld}[sloped, swap]{(\ttj \circ \tti)^\dagger}
\\
&
\ttZ^1_{r-1}
\end{tikzcd} \]
in $\Ch_{\ZZ[\Cyclic_{p^r}]}$ (using \Cref{obs.cones.get.previous.cxes}\Cref{item.fa.is.induced.map.on.cones}), and we write $\tth$ for its precomposition with the quotient morphism $\w{\ttC}_r \ra (\w{\ttC}_r)_{\Cyclic_p}$. So explicitly, $\tth$ is a sequence
\[
\tth
:=
\left(
(\w{\ttC}_r)_n
\xra{\tth_n}
(\ttT^0_{r-1})_{n+1}
\right)_{n \in \ZZ}
\]
of maps of $\ZZ[\Cyclic_{p^r}]$-modules given as follows:
\begin{itemize}

\item for $n<0$, $\tth_n$ is the zero map
\[
(\w{\ttC}_r)_n
:=
0
\longra
\ZZ[\Cyclic_{p^{r-1}}]
=:
(\ttT^0_{r-1})_{n+1}
~;
\]

\item $\tth_0$ is the quotient map
\[
(\w{\ttC}_r)_0
:=
\ZZ[\Cyclic_{p^r}]
\xra{1 \longmapsto 1}
\ZZ[\Cyclic_{p^{r-1}}]
=:
(\ttT^0_{r-1})_{n+1}
~;
\]

\item for $n > 0$, $\tth_n$ is the composite
\[
(\w{\ttC}_r)_n
:=
\left(
\def\arraystretch{1}
\begin{array}{c}
\ZZ[\Cyclic_{p^r}]
\\
\oplus
\\
\ZZ[\Cyclic_{p^r}]
\end{array}
\right)
\xra{( 0 \ 1 )}
\ZZ[\Cyclic_{p^r}]
\xra{1 \longmapsto 1}
\ZZ[\Cyclic_{p^{r-1}}]
=:
(\ttT^0_{r-1})_{n+1}
\]
of the projection onto the second factor followed by the quotient map.

\end{itemize}
\end{notation}

\begin{lemma}
\label{lemma.chain.level.nullhtpy.presents.canonical.nullhtpy}
The homotopy-commutative diagram
\[ \begin{tikzcd}[row sep=1cm, column sep=1.5cm]
\w{\ttC}_r
\arrow{rd}[sloped]{\tthtrf \circ \ttk_r}
\arrow{d}
\\
(\w{\ttC}_r)_{\Cyclic_p}
\arrow{r}
\arrow{d}
&
\ttC^1_{r-1}
\arrow{d}{\ttg_{r-1}^0}[swap, xshift=-1.1cm]{\tth'}
\\
0
\arrow{r}
&
\ttT^0_{r-1}
\end{tikzcd} \]
in $\Ch_{\ZZ[\Cyclic_{p^r}]}$ is a presentation of the commutative diagram
\begin{equation}
\label{comm.pentagon.for.tau.geq.zero.Z.tCp.of.obs.nullhtpy.of.htrf.followed.by.Q}
\begin{tikzcd}[row sep=1.5cm, column sep=2.5cm]
\tau_{\geq 0} \ZZ^{\st \Cyclic_p}
\arrow{d}
\arrow{rd}[sloped]{\htrf_{\Cyclic_p}(\tau_{\geq 0} \ZZ^{\st \Cyclic_p})}
\\
(\tau_{\geq 0} \ZZ^{\st \Cyclic_p})_{\htpy \Cyclic_p}
\arrow{r}[swap]{\Nm_{\Cyclic_p}(\tau_{\geq 0} \ZZ^{\st \Cyclic_p})}
\arrow{d}
&
(\tau_{\geq 0} \ZZ^{\st \Cyclic_p})^{\htpy \Cyclic_p}
\arrow{d}{\sQ_{\Cyclic_p}(\tau_{\geq 0} \ZZ^{\st \Cyclic_p})}
\\
0
\arrow{r}
&
(\tau_{\geq 0} \ZZ^{\st \Cyclic_p})^{\st \Cyclic_p}
\end{tikzcd}
\end{equation}
in $\Mod^{\htpy \Cyclic_{p^r}}_\ZZ$ of \Cref{obs.nullhtpy.of.htrf.followed.by.Q} (applied to $E = \tau_{\geq 0} \ZZ^{\st \Cyclic_p} \in \Mod^{\htpy \Cyclic_{p^r}}_\ZZ$). In particular, the homotopy-commutative square
\[ \begin{tikzcd}[column sep=1.5cm, row sep=1cm]
\w{\ttC}_r
\arrow{r}{\tthtrf \circ \ttk_r}
\arrow{d}
&
\ttC^1_{r-1}
\arrow{d}{\ttg_{r-1}^0}[swap, xshift=-1.1cm]{\tth}
\\
0
\arrow{r}
&
\ttT^0_{r-1}
\end{tikzcd} \]
in $\Ch_{\ZZ[\Cyclic_{p^r}]}$ is a presentation of the commutative square
\[ \begin{tikzcd}[column sep=2.5cm, row sep=1.5cm]
\tau_{\geq 0} \ZZ^{\st \Cyclic_p}
\arrow{r}{\htrf_{\Cyclic_p}(\tau_{\geq 0} \ZZ^{\st \Cyclic_p})}
\arrow{d}
&
(\tau_{\geq 0} \ZZ^{\st \Cyclic_p})^{\htpy \Cyclic_p}
\arrow{d}{\sQ_{\Cyclic_p}(\tau_{\geq 0} \ZZ^{\st \Cyclic_p})}
\\
0
\arrow{r}
&
(\tau_{\geq 0} \ZZ^{\st \Cyclic_p})^{\st \Cyclic_p}
\end{tikzcd} \]
in $\Mod^{\htpy \Cyclic_{p^r}}_\ZZ$.
\end{lemma}

\begin{proof}
Consider the commutative diagram
\begin{equation}
\label{homotopical.MNP.sequence.in.specific.case}
\begin{tikzcd}[column sep=2cm, row sep=1.5cm]
\ZZ_{\htpy \Cyclic_p}
\arrow{r}{\genNm}
\arrow{d}
&
\ZZ
\arrow{d}
\arrow{rd}
\\
\ZZ_{\htpy \Cyclic_{p^2}}
\arrow{r}[swap]{\genNm_{\htpy \Cyclic_p}}
&
\ZZ_{\htpy \Cyclic_p}
\arrow{r}[swap]{\Nm_{\Cyclic_p}(\ZZ)}
&
\ZZ^{\htpy \Cyclic_p}
\end{tikzcd}
\end{equation}
in $\Mod^{\htpy \Cyclic_{p^r}}_\ZZ$. Passing to cofibers as indicated, this yields a commutative diagram
\begin{equation}
\label{cofib.presentation.of.comm.pentagon.for.tau.geq.zero.Z.tCp.of.obs.nullhtpy.of.htrf.followed.by.Q}
\begin{tikzcd}[row sep=1.5cm]
\cofib ( \genNm )
\arrow{d}
\arrow{rd}
\\
\cofib ( \genNm_{\htpy \Cyclic_p} )
\arrow{r}
\arrow{d}
&
\cofib ( \Nm_{\Cyclic_p}(\ZZ) \circ \genNm_{\htpy \Cyclic_p} )
\arrow{d}
\\
0
\arrow{r}
&
\cofib ( \Nm_{\Cyclic_p}(\ZZ) )
\end{tikzcd}
\end{equation}
in $\Mod^{\htpy \Cyclic_{p^r}}$ in which the square is a pushout. Moreover, it follows from the proof of \Cref{lem.ABX.htpy.invariantly} that the commutative diagram \Cref{cofib.presentation.of.comm.pentagon.for.tau.geq.zero.Z.tCp.of.obs.nullhtpy.of.htrf.followed.by.Q} is precisely the commutative diagram \Cref{comm.pentagon.for.tau.geq.zero.Z.tCp.of.obs.nullhtpy.of.htrf.followed.by.Q}. So to conclude, it suffices to show that the commutative diagram \Cref{homotopical.MNP.sequence.in.specific.case} in $\Mod^{\htpy \Cyclic_{p^r}}_\ZZ$ is presented by the commutative diagram
\[ \begin{tikzcd}[column sep=2cm, row sep=1.5cm]
{_1 \ttS_r}
\arrow{r}{\ttgenNm}
\arrow{d}
&
{_0 \ttS_r}
\arrow{d}
\arrow{rd}[sloped]{\ttj \circ \tti}
\\
({_1 \ttS_r})_{\Cyclic_p}
\arrow{r}{\ttgenNm_{\Cyclic_p}}
\arrow[bend right=20]{rr}[swap]{(\ttj \circ \tti \circ \ttgenNm)^\dagger}
&
({_0 \ttS_r})_{\Cyclic_p}
\arrow{r}{(\ttj \circ \tti)^\dagger}
&
\ttZ^1_{r-1}
\end{tikzcd} \]
in $\Ch_{\ZZ[\Cyclic_{p^r}]}$. For this, we note that it follows from \Cref{obs.cones.get.previous.cxes}\Cref{item.B.tilde.a.presents.Baa} that the morphism $\ttgenNm$ presents the morphism $\genNm$ and moreover that ${_1 \ttS_r}$ and ${_0 \ttS_r}$ are adapted to homotopy $\Cyclic_p$-orbits by \Cref{obs.basics.of.adaptedness.to.orbits}\Cref{item.bdd.below.free.adapted.to.orbits}.
\end{proof}

\subsection{Multiplicative structure of Tate cohomology}
\label{subsection.a.coh.ring}

In this subsection, we study the multiplicative structure on Tate cohomology. Namely, in \Cref{lem.compute.htpy.ring.of.htpy.of.tate} we compute the ring structure on the homotopy groups of $(\ZZ^{\st \Cyclic_p})^{\htpy \Cyclic_{p^{a-1}}} \in \Mod_\ZZ$, and in \Cref{lemma.present.mult.by.chern.class.on.tate} we give a chain-level presentation of an endomorphism of $\ZZ^{\st \Cyclic_p} \in \Mod^{\htpy \Cyclic_{p^r}}_\ZZ$ given by multiplying by a $\Cyclic_{p^r}$-equivariant homotopy element.

\begin{notation}
\label{notn.chern.class.elts}
We respectively write
\[
c_a'
\in
\pi_{-2} ( \ZZ^{\htpy \Cyclic_{p^a}})
\qquad
\text{and}
\qquad
c_a
\in
\pi_{-2} ( (\ZZ^{\st \Cyclic_p})^{\htpy \Cyclic_{p^{a-1}}})
\]
for the elements represented by the cycles
\[
1
\in
\ZZ
=:
(\ttZ_0^a)_{-2}
\qquad
\text{and}
\qquad
1
\in
\ZZ
=:
(\ttT^{a-1}_0)_{-2}
\]
(using \Cref{lem.ABX.htpy.invariantly}).\footnote{It is not hard to see that more generally these same elements are respectively represented by the cycles
\[
N
\in
\ZZ[\Cyclic_{p^r}]
=:
(\ttZ_r^a)_{-2}
\qquad
\text{and}
\qquad
N
\in
\ZZ[\Cyclic_{p^r}]
=:
(\ttT_r^{a-1})_{-2}
\]
(again using \Cref{lem.ABX.htpy.invariantly}).}
\end{notation}

\begin{lemma}
\label{lem.compute.htpy.ring.of.htpy.of.tate}
The homomorphism
\begin{equation}
\label{map.on.htpy.gr.comm.rings.from.ZhCpa.to.ZtCphCpaminusone}
\pi_* ( \ZZ^{\htpy \Cyclic_{p^a}} )
\xra{\pi_*(q_0^{a-1})}
\pi_* ( (\ZZ^{\st \Cyclic_p})^{\htpy \Cyclic_{p^{a-1}}} )
\end{equation}
of graded-commutative rings guaranteed by \Cref{obs.various.properties.of.tate}\Cref{part.rlax.s.m.str.on.h.to.tate} (which lifts $q_0^{a-1}$ from a morphism in $\Mod_\ZZ$ to a morphism in $\CAlg ( \Mod_\ZZ )$) is the homomorphism
\begin{equation}
\label{hom.of.gr.comm.rings}
\ZZ[c_a'] / (p^a c_a')
\longra
(\ZZ / p^a)[c_a^\pm]
\end{equation}
of graded-commutative rings characterized by the fact that it carries $c_a'$ to $c_a$.
\end{lemma}

\begin{proof}
By \Cref{lem.ABX.htpy.invariantly}, the morphism $\ttq_0^{a-1}$ in $\Ch_{\ZZ}$ presents the morphism $q_0^{a-1}$ in $\Mod_\ZZ$. From this, we easily identify the morphisms \Cref{map.on.htpy.gr.comm.rings.from.ZhCpa.to.ZtCphCpaminusone} and \Cref{hom.of.gr.comm.rings} of graded abelian groups. Moreover, it is clear that the isomorphism
\[
\pi_* ( \ZZ^{\htpy \Cyclic_{p^a}} )
\cong
\ZZ[c_a'] / (p^a c_a')
\]
of graded abelian groups is in fact one of graded-commutative rings. It follows that the homomorphism \Cref{map.on.htpy.gr.comm.rings.from.ZhCpa.to.ZtCphCpaminusone} of graded $\pi_*(\ZZ^{\htpy \Cyclic_{p^a}})$-modules must coincide with the homomorphism \Cref{hom.of.gr.comm.rings} of graded $\ZZ[c_a'] / (p^a c_a')$-modules. Now, it suffices to observe that the commutative ring structure on the graded abelian group $(\ZZ / p^a)[c_a^\pm]$ is the only one that lifts the homomorphism \Cref{hom.of.gr.comm.rings} of graded $\ZZ[c_a'] / (p^a c_a')$-modules to one of graded-commutative rings.
\end{proof}

\begin{notation}
\label{notn.mult.by.an.eqvrt.htpy.elt}
Given a stably symmetric monoidal $\infty$-category $\cC$, a commutative algebra object $A \in \CAlg(\cC)$, an $A$-module $M \in \Mod_A(\cC)$, and a morphism $\uno_\cC \xra{f} \Sigma^k A$ in $\cC$ for any $k \in \ZZ$, we simply write $M \xra{f} \Sigma^k M$ for the morphism given by multiplication by $f$, i.e.\! the composite
\[
M
\simeq
\uno_\cC
\otimes
M
\xra{f \otimes \id_M}
\Sigma^k A
\otimes M
\longra
\Sigma^k M
~.
\]
In particular, we apply this to the commutative algebra objects $\ZZ^{\htpy \Cyclic_p} , \ZZ^{\st \Cyclic_p} \in \CAlg(\Mod^{\htpy \Cyclic_{p^r}}_\ZZ)$ (guaranteed by \Cref{obs.various.properties.of.tate}\Cref{part.rlax.s.m.str.on.h.to.tate}) and the morphisms
\[
c'_{r+1}
\in
\ulhom_{\Mod^{\htpy \Cyclic_{p^r}}_\ZZ} \left( \ZZ , \Sigma^2 \ZZ^{\htpy \Cyclic_p} \right)
\simeq
\Sigma^2 \ZZ^{\htpy \Cyclic_{p^{r+1}}}
\qquad
\text{and}
\qquad
c_{r+1}
\in
\ulhom_{\Mod^{\htpy \Cyclic_{p^r}}_\ZZ} \left( \ZZ , \Sigma^2 \ZZ^{\st \Cyclic_p} \right)
\simeq
\Sigma^2 (\ZZ^{\st \Cyclic_p})^{\htpy \Cyclic_{p^r}}
\]
of \Cref{notn.chern.class.elts}.
\end{notation}

\begin{notation}
\label{notn.shifted.endo.of.T}
For any $i,r \geq 0$ we write
\[
\ttT^i_r
\xra{\ttc_r^i}
\Sigma^2 \ttT^i_r
\]
for the evident isomorphism in $\Ch_{\ZZ[\Cyclic_{p^r}]}$ that is the identity in all degrees.
\end{notation}

\begin{observation}
\label{obs.fixedpoints.of.chain.level.c.is.itself}
For any $i \geq 0$ and $r \geq 1$ we have an evident identification
\[
\left(
\ttT^i_r
\xra{\ttc_r^i}
\Sigma^2 \ttT^i_r
\right)^{\Cyclic_p}
\cong
\left(
\ttT^{i+1}_{r-1}
\xra{\ttc^{i+1}_{r-1}}
\Sigma^2 \ttT^{i+1}_{r-1}
\right)
\]
in $\Ar(\Ch_{\ZZ[\Cyclic_{p^{r-1}}]})$.
\end{observation}

\begin{lemma}
\label{lemma.present.mult.by.chern.class.on.tate}
The morphism
\[
\ttT^0_r
\xra{\ttc^0_r}
\Sigma^2
\ttT^0_r
\]
in $\Ch_{\ZZ[\Cyclic_{p^r}]}$ is a presentation of the morphism
\[
\ZZ^{\st \Cyclic_p}
\xra{c_{r+1}}
\Sigma^2 \ZZ^{\st \Cyclic_p}
\]
in $\Mod^{\htpy \Cyclic_{p^r}}_\ZZ$ (using \Cref{notn.mult.by.an.eqvrt.htpy.elt}).
\end{lemma}

\begin{proof}
For simplicity, we write
\[
c' := c'_{r+1}
~,
\qquad
c := c_{r+1}
~,
\qquad
\text{and}
\qquad
\ttc := \ttc_r^0
~.
\]

It follows from \Cref{lem.compute.htpy.ring.of.htpy.of.tate} (and \Cref{obs.various.properties.of.tate}\Cref{part.rlax.s.m.str.on.h.to.tate}) that we have a commutative square
\begin{equation}
\label{mult.by.c.prime.and.c.commute}
\begin{tikzcd}[column sep=1.5cm]
\ZZ^{\htpy \Cyclic_p}
\arrow{r}{\sQ_{\Cyclic_p}(\ZZ)}
\arrow{d}[swap]{c'}
&
\ZZ^{\st \Cyclic_p}
\arrow{d}{c' \simeq c}
\\
\Sigma^2 \ZZ^{\htpy \Cyclic_p}
\arrow{r}[swap]{\sQ_{\Cyclic_p}(\ZZ)}
&
\Sigma^2 \ZZ^{\st \Cyclic_p}
\end{tikzcd}
\end{equation}
in $\Mod^{\htpy \Cyclic_{p^r}}_\ZZ$. By definition, the left vertical morphism in diagram \Cref{mult.by.c.prime.and.c.commute} is given by
\begin{align}
\nonumber
\ulhom_{\Mod^{\htpy \Cyclic_{p^{r+1}}}_\ZZ} \left( \ZZ[\Cyclic_{p^r}] , \ZZ \xlongra{c'} \Sigma^2 \ZZ \right)
& \simeq
\ulhom_{\Mod^{\htpy \Cyclic_{p^{r+1}}}_\ZZ} \left( \ZZ[\Cyclic_{p^r}] \otimes \left( \ZZ \xlongra{c'} \Sigma^2 \ZZ \right)^\vee  , \ZZ \right)
\\
\label{rewritten.endomorphism.in.ZhCp}
&
\simeq
\ulhom_{\Mod^{\htpy \Cyclic_{p^{r+1}}}_\ZZ} \left( \ZZ[\Cyclic_{p^r}] \otimes \left( \ZZ \xlongla{c'} \Sigma^{-2} \ZZ \right)  , \ZZ \right)
~.
\end{align}

We give a chain-level presentation of the morphism \Cref{rewritten.endomorphism.in.ZhCp}. For this, consider the commutative triangle
\begin{equation}
\label{definitional.presentation.of.c.prime.gives.shifted.endo.presentation}
\begin{tikzcd}
&
{_0 \ttS_{r+1}}
\arrow{d}[sloped, anchor=south]{\approx}
\\
\Sigma^{-2} {_0 \ttS_{r+1}}
\arrow{ru}
\arrow{r}
&
\uwave{\ZZ}
\end{tikzcd}
\end{equation}
in $\Ch_{\ZZ[\Cyclic_{p^{r+1}}]}$, in which both morphisms to $\uwave{\ZZ}$ are characterized by the fact that they act as $\ZZ[\Cyclic_{p^{r+1}}] \xra{\sigma \mapsto 1} \ZZ$ in degree 0 and the diagonal morphism is characterized by the fact that it acts as the identity in all nonnegative degrees (and the vertical morphism is evidently a quasi-isomorphism). Through the evident isomorphism
\[
\ttZ^{r+1}_0
\cong
\ulhom_{\Ch_{\ZZ[\Cyclic_{p^{r+1}}]}}( {_0 \ttS_{r+1}} , \uwave{\ZZ} )
\]
in $\Ch_\ZZ$, the horizontal morphism in diagram \Cref{definitional.presentation.of.c.prime.gives.shifted.endo.presentation} in $\Ch_{\ZZ[\Cyclic_{p^{r+1}}]}$ represents the morphism $\Sigma^{-2} \ZZ \xra{c'} \ZZ$ in $\Mod^{\htpy \Cyclic_{p^{r+1}}}_\ZZ$. Therefore, the diagonal morphism in diagram \Cref{definitional.presentation.of.c.prime.gives.shifted.endo.presentation} in $\Ch_{\ZZ[\Cyclic_{p^{r+1}}]}$ represents the same morphism in $\Mod^{\htpy \Cyclic_{p^{r+1}}}_\ZZ$. Hence, the morphism \Cref{rewritten.endomorphism.in.ZhCp} in $\Mod_{\ZZ[\Cyclic_{p^r}]}$ is represented by the morphism
\begin{equation}
\label{doubly.rewritten.endomorphism.in.ZhCp}
\ulhom_{\Ch_{\ZZ[\Cyclic_{p^{r+1}}]}} \left( \uwave{ \ZZ[\Cyclic_{p^r}]} \otimes \left( {_0 \ttS_{r+1}} \longla \Sigma^{-2} {_0 \ttS_{r+1}} \right) , \uwave{\ZZ} \right)
\end{equation}
in $\Ch_\ZZ[\Cyclic_{p^r}]$ (using that ${_0 \ttS_r}$ is levelwise free and concentrated in nonnegative degrees).

We have just shown that the morphism \Cref{doubly.rewritten.endomorphism.in.ZhCp} in $\Ch_{\ZZ[\Cyclic_{p^r}]}$ is a presentation of the left vertical morphism in diagram \Cref{mult.by.c.prime.and.c.commute} in $\Mod^{\htpy \Cyclic_{p^r}}_\ZZ$. Unwinding the definitions, we find that the former is the morphism
\[
\ttZ^1_r
\xlongra{\ttc'}
\Sigma^2 \ttZ^1_r
\]
in $\Ch_{\ZZ[\Cyclic_{p^r}]}$ characterized by the fact that it acts as the identity in all nonpositive degrees. To proceed, we define the mapping telescopes
\[
\tel(\ttc')
:=
\colim \left(
\ttZ^1_r
\xlongra{\ttc'}
\Sigma^2 \ttZ^1_r
\xlongra{\ttc'}
\Sigma^4 \ttZ^1_r
\xlongra{\ttc'}
\cdots
\right)
\in
\Ch_{\ZZ[\Cyclic_{p^r}]}
\]
and
\[
\tel(c')
:=
\colim \left(
\ZZ^{\htpy \Cyclic_p}
\xlongra{c'}
\Sigma^2 \ZZ^{\htpy \Cyclic_p}
\xlongra{c'}
\Sigma^4 \ZZ^{\htpy \Cyclic_p}
\xlongra{c'}
\cdots
\right)
\in
\Mod^{\htpy \Cyclic_{p^r}}_\ZZ
~.
\]
Now, by \Cref{lem.compute.htpy.ring.of.htpy.of.tate}, the morphism $\ZZ^{\st \Cyclic_p} \xra{c} \Sigma^2 \ZZ^{\st \Cyclic_p}$ in $\Mod^{\htpy \Cyclic_{p^r}}_\ZZ$ is an equivalence. Therefore the commutative square \Cref{mult.by.c.prime.and.c.commute} in $\Mod^{\htpy \Cyclic_{p^r}}_\ZZ$ extends to a commutative diagram
\begin{equation}
\label{mult.by.c.prime.and.c.commute.and.also.telescope}
\begin{tikzcd}
\ZZ^{\htpy \Cyclic_p}
\arrow[dashed]{r}
\arrow[bend left]{rr}{\sQ_{\Cyclic_p}(\ZZ)}
\arrow{d}[swap]{c'}
&
\tel(c')
\arrow[dashed]{r}{\sim}
\arrow[dashed]{d}[swap]{c'}
&
\ZZ^{\st \Cyclic_p}
\arrow{d}{c' \simeq c}
\\
\Sigma^2 \ZZ^{\htpy \Cyclic_p}
\arrow[dashed]{r}
\arrow[bend right]{rr}[swap]{\sQ_{\Cyclic_p}(\ZZ)}
&
\Sigma^2 \tel(c')
\arrow[dashed]{r}{\sim}
&
\Sigma^2 \ZZ^{\st \Cyclic_p}
\end{tikzcd}~,
\end{equation}
in which the equivalences also follow from \Cref{lem.compute.htpy.ring.of.htpy.of.tate}. On the other hand, using \Cref{lem.ABX.htpy.invariantly} (and the fact that the functor $\Ch_{\ZZ[\Cyclic_{p^r}]} \xra{\Pi_\infty} \Mod^{\htpy \Cyclic_{p^r}}_\ZZ$ commutes with filtered colimits) we see that the diagram \Cref{mult.by.c.prime.and.c.commute.and.also.telescope} in $\Mod^{\htpy \Cyclic_{p^r}}_\ZZ$ is presented by the diagram
\[ \begin{tikzcd}
\ttZ^1_r
\arrow{r}
\arrow[bend left]{rr}{\ttq^0_r}
\arrow{d}[swap]{\ttc'}
&
\tel(\ttc')
\arrow{r}{\cong}
\arrow{d}[swap]{\ttc'}
&
\ttT^0_r
\arrow{d}{\ttc}
\\
\Sigma^2 \ttZ^1_r
\arrow{r}
\arrow[bend right]{rr}[swap]{\ttq^0_r}
&
\Sigma^2 \tel(\ttc')
\arrow{r}{\cong}
&
\Sigma^2 \ttT^0_r
\end{tikzcd} \]
in $\Ch_{\ZZ[\Cyclic_{p^r}]}$, in which the square on the right commutes by inspection. In particular, the claim follows.
\end{proof}

\bibliographystyle{amsalpha}
\bibliography{mackey}{}

\end{document}